\documentclass[10pt]{amsart}
\usepackage{amssymb}
\usepackage{graphicx}
\usepackage{xcolor} 
\usepackage{tensor}
\usepackage{upgreek}
\usepackage{esint}
\usepackage{comment} 

\usepackage[margin=1in, marginpar=.6in]{geometry} 

\usepackage{enumitem}

\allowdisplaybreaks				

\usepackage{mathtools}
\usepackage{hyperref}

\usepackage{slashed}

\usepackage{marvosym}		

\definecolor{green}{rgb}{0,0.8,0} 

\newtheorem{theorem}{Theorem}[section]
\newtheorem{corollary}[theorem]{Corollary}
\newtheorem{lemma}[theorem]{Lemma}
\newtheorem{proposition}[theorem]{Proposition}
\theoremstyle{definition}
\newtheorem{definition}[theorem]{Definition}
\newtheorem{example}[theorem]{Example}
\theoremstyle{remark}
\newtheorem{remark}[theorem]{Remark}
\numberwithin{equation}{section}
\newcommand{\relphantom}[1]{\mathrel{\phantom{#1}}}

\makeatletter
\newcommand{\nrm}{\@ifstar{\nrmb}{\nrmi}}
\newcommand{\nrmi}[1]{\Vert{#1}\Vert}
\newcommand{\nrmb}[1]{\left\Vert{#1}\right\Vert}
\newcommand{\abs}{\@ifstar{\absb}{\absi}}
\newcommand{\absi}[1]{\vert{#1}\vert}
\newcommand{\absb}[1]{\left\vert{#1}\right\vert}
\newcommand{\brk}{\@ifstar{\brkb}{\brki}}
\newcommand{\brki}[1]{\langle{#1}\rangle}
\newcommand{\brkb}[1]{\left\langle{#1}\right\rangle}
\newcommand{\set}{\@ifstar{\setb}{\seti}}
\newcommand{\seti}[1]{\{#1\}}
\newcommand{\setb}[1]{\left\{ #1\right\}}
\makeatother

\newcommand{\td}[1]{\widetilde{#1}}
\newcommand{\br}[1]{\overline{#1}}
\newcommand{\ul}[1]{\underline{#1}}
\newcommand{\wh}[1]{\widehat{#1}}

\newcommand{\VERT}[1]{{\left\vert\kern-0.25ex\left\vert\kern-0.25ex\left\vert #1
    \right\vert\kern-0.25ex\right\vert\kern-0.25ex\right\vert}}

\DeclareMathOperator{\dist}{dist}

\DeclareMathOperator{\supp}{supp}
\DeclareMathOperator{\tr}{tr}

\DeclareMathOperator{\diam}{diam}

\let\Re\relax
\DeclareMathOperator{\Re}{Re}

\newcommand{\aeq}{\simeq}
\newcommand{\aleq}{\lesssim}

\newcommand{\ceq}{\coloneqq}

\newcommand{\lap}{\Delta}

\newcommand{\ud}{\mathrm{d}}
\newcommand{\rd}{\partial}
\newcommand{\nb}{\nabla}

\newcommand{\imp}{\Rightarrow}

\newcommand{\bb}{\Big}

\newcommand{\0}{\emptyset}

\newcommand{\peq}{\relphantom{=}}			

\newcommand{\alp}{\alpha}
\newcommand{\bt}{\beta}
\newcommand{\gmm}{\gamma}
\newcommand{\Gmm}{\Gamma}
\newcommand{\dlt}{\delta}

\newcommand{\eps}{\epsilon}

\newcommand{\kpp}{\kappa}
\newcommand{\lmb}{\lambda}

\newcommand{\vphi}{\varphi}

\newcommand{\omg}{\omega}
\newcommand{\Omg}{\Omega}

\newcommand{\zt}{\zeta}

\newcommand{\bfb}{{\bf b}}

\newcommand{\bfe}{{\bf e}}

\newcommand{\bfg}{{\bf g}}
\newcommand{\bfh}{{\bf h}}

\newcommand{\bfk}{{\bf k}}

\newcommand{\bfu}{{\bf u}}

\newcommand{\bfx}{{\bf x}}
\newcommand{\bfy}{{\bf y}}
\newcommand{\bfz}{{\bf z}}

\newcommand{\bfA}{{\bf A}}
\newcommand{\bfB}{{\bf B}}
\newcommand{\bfC}{{\bf C}}

\newcommand{\bfG}{{\bf G}}

\newcommand{\bfR}{{\bf R}}

\newcommand{\bfZ}{{\bf Z}}

\newcommand{\bfalp}{\boldsymbol{\alpha}}

\newcommand{\bfGmm}{\boldsymbol{\Gamma}}
\newcommand{\bfdlt}{\boldsymbol{\delta}}

\newcommand{\bfzt}{\boldsymbol{\zeta}}

\newcommand{\bfeta}{\boldsymbol{\eta}}

\newcommand{\bfnu}{\boldsymbol{\nu}}

\newcommand{\bfpi}{\boldsymbol{\pi}}
\newcommand{\bfPi}{\boldsymbol{\Pi}}

\newcommand{\bfomg}{\boldsymbol{\omega}}

\newcommand{\bbC}{\mathbb C}

\newcommand{\bbH}{\mathbb H}

\newcommand{\bbR}{\mathbb R}
\newcommand{\bbS}{\mathbb S}

\newcommand{\bbZ}{\mathbb Z}

\newcommand{\calA}{\mathcal A}
\newcommand{\calB}{\mathcal B}
\newcommand{\calC}{\mathcal C}
\newcommand{\calD}{\mathcal D}
\newcommand{\calE}{\mathcal E}
\newcommand{\calF}{\mathcal F}
\newcommand{\calG}{\mathcal G}

\newcommand{\calL}{\mathcal L}
\newcommand{\calM}{\mathcal M}

\newcommand{\calP}{\mathcal P}
\newcommand{\calQ}{\mathcal Q}

\newcommand{\calS}{\mathcal S}
\newcommand{\calT}{\mathcal T}

\newcommand{\weakto}{\rightharpoonup}

\newcommand{\chf}{\mathbf{1}}

\newcommand{\Ric}{\mathrm{Ric}}

\newcommand{\ueta}{\upeta}
\newcommand{\uI}{\mathrm{I}}
\newcommand{\uK}{\mathrm{K}}
\newcommand{\uS}{\mathrm{S}}
\newcommand{\ua}{\mathrm{a}}
\newcommand{\uk}{\mathrm{k}}

\newcommand{\uV}{\mathrm{V}}
\newcommand{\uW}{\mathrm{W}}

\DeclareMathOperator{\ran}{ran}

\newcommand{\nnrm}[1]{{\left\vert\kern-0.25ex\left\vert\kern-0.25ex\left\vert #1
    \right\vert\kern-0.25ex\right\vert\kern-0.25ex\right\vert}}

 \let\div\relax
\DeclareMathOperator{\div}{div}

\newcommand{\prin}{\mathrm{prin}}
\begin{document}

\title[]{Integral formulas for under/overdetermined \\ differential operators via recovery on curves and the finite-dimensional cokernel condition I: General theory}
\author{Philip Isett}%
\address{Department of Mathematics, California Institute of Technology, Pasadena, CA 91125, USA}%
\email{isett@caltech.edu}%

\author{Yuchen Mao}%
\address{Department of Mathematics, UC Berkeley, Berkeley, CA 94720, USA}%
\email{yuchen\_mao@berkeley.edu}%

\author{Sung-Jin Oh}%
\address{Department of Mathematics, UC Berkeley, Berkeley, CA 94720, USA and School of Mathematics, Korea Institute for Advanced Study, Seoul, Korea}%
\email{sjoh@math.berkeley.edu}%

\author{Zhongkai Tao}%
\address{Institut des Hautes \'Etudes Scientifiques, 35 route de Ghartres, 91440 Bures-sur-Yvette, France}%
\email{ztao@ihes.fr}%


\begin{abstract}
We introduce a new versatile method for constructing solution operators (i.e., right-inverses up to a finite rank operator) for a wide class of underdetermined PDEs $\mathcal{P} u = f$, which are regularizing of optimal order and, more interestingly, whose integral kernels have certain prescribed support properties. By duality, we simultaneously obtain integral representation formulas (i.e., left-inverses up to a finite rank operator) for overdetermined PDEs $\mathcal{P}^{\ast} v = g$ with analogous properties, which lead to Poincar\'e- or Korn-type inequalities. Our method applies to operators such as the divergence, linearized scalar curvature, and linearized Einstein constraint operators (which are underdetermined), as well as the gradient, Hessian, trace-free part of the Hessian, Killing, and conformal Killing operators (which are overdetermined).

The starting point for our construction is a condition -- dubbed the \emph{recovery on curves condition} (RC) -- that leads to Green's functions for $\mathcal{P}$ supported on prescribed curves. Then the desired integral solution operators (and, by duality, integral representation formulas) are obtained by taking smooth averages over a suitable family of curves. Our method generalizes, on the one hand, the previous formulas of Bogovskii and Oh--Tataru for the divergence operator, and on the other hand, integral representation formulas for overdetermined operators by Reshetnyak, which lead to classical inequalities of Poincar\'e and Korn.

We furthermore identify a simple algebraic sufficient condition for (RC), namely, that the principal symbol $p(x, \xi)$ of $\mathcal{P}$ is full-rank for all non-zero \emph{complex} vectors $\xi$ (as opposed to real, as in ellipticity). When the principal symbol has constant coefficients, this is equivalent to (RC) and also to the condition that the formal cokernel of $\mathcal{P}$ (without any boundary conditions) is finite dimensional; for this reason, we call it the \emph{finite-dimensional cokernel condition} (FC). We give a short proof that all the examples above satisfy (FC), and thus (RC).

Our method provides a new approach to solving a wide range of linear and nonlinear problems with operators that satisfy (FC): we may now \emph{design} integral operators tailored to each problem. Various applications will be considered in subsequent papers.

\end{abstract}

\maketitle

\setcounter{tocdepth}{2}
\tableofcontents


\section{Introduction}
Underdetermined partial differential operators $\calP$ resembling the divergence operator appear naturally in various fields of physics and geometry. Take, for instance, the Gauss law in electromagnetism, the divergence-free condition for incompressible fluids, the linearized scalar curvature operator in Riemannian geometry, the constraint equations in general relativity and gauge theories, and so on. The duals $\calP^{\ast}$ of such underdetermined operators, which are overdetermined, also play a significant role. Examples include the gradient operator (dual to divergence), the Hessian operator (dual to linearized scalar curvature, up to lower order terms), the Killing operator (symmetric part of the covariant gradient of a vector field, dual to divergence of symmetric $2$-tensor), the conformal Killing operator (trace-free symmetric part of the covariant gradient of a vector field, dual to divergence of trace-free symmetric $2$-tensor), and many others.

In this paper, we describe a new versatile method for obtaining solution operators (i.e., right-inverses up to a finite rank operator) for such underdetermined operators $\calP$ and, by duality, representation formulas (i.e., left-inverses up to a finite rank operator) for such overdetermined operators $\calP^{\ast}$. The method is based on a direct derivation of integral formulas (i.e., Green's functions) for these operators based on a property we dub \emph{recovery on curves} (see \ref{hyp:crc} below). The operators we construct are \emph{regularizing of optimal order} (i.e., they gain $m$ derivatives, where $m$ is the order of $\calP$) and, more interestingly, their integral kernels have \emph{prescribed support properties}.  This latter feature means that the support of the solutions can be prescribed, provided the data satisfy appropriate assumptions.

Furthermore, we demonstrate that our method is applicable (even in variable coefficient situations) as soon as a simple algebraic condition on the principal symbol $p(x,\xi)$ of the operator $\calP$ is satisfied:
\begin{enumerate}[label=(FC)]
\item \label{hyp:fc-0} $p(x, \xi)$ is full rank for all $x \in U$ and $\xi \in \bbC^{d} \setminus \set{0}$.
\end{enumerate}
 At a glance, \ref{hyp:fc-0} is a (strict) strengthening of the familiar notion of ellipticity, which is the same condition but only for real covectors $\xi \in \bbR^{d} \setminus \set{0}$. At a deeper level, it is a suitable variable-coefficient generalization of the condition that the \emph{formal cokernel} of $\calP$ on $U$,
\begin{equation} \label{eq:formal-coker}
	\ker \calP^{\ast}(U) \ceq \set{\bfZ \in C^{\infty}(\br{U}) : \calP^{\ast} \bfZ = 0 \hbox{ in } \calD'(U)},
\end{equation}
is finite dimensional (which is implied by \ref{hyp:crc}; see Theorem~\ref{thm:main-summary-bogovskii}). In fact, the three conditions \ref{hyp:fc-0}, \ref{hyp:crc}, and $\dim \ker \calP^{\ast} < + \infty$ are \emph{equivalent} if each row of $p^{\ast}$ is a homogeneous vector-valued polynomial (see Theorem~\ref{thm:alg-cond-intro}). For this reason, our new condition is dubbed the \emph{finite-dimensional cokernel condition} (FC). All operators mentioned above (and more) satisfy \ref{hyp:fc-0}, as the short proof of Theorem~\ref{thm:appl} below shows (see also Appendix~\ref{sec:appl}).

Our results provide a fresh approach to solving a wide range of linear and nonlinear problems with operators that satisfy \ref{hyp:fc} (and hence \ref{hyp:crc}): we may now \emph{design} integral operators tailored to each unique problem. In companion papers \cite{IMOT2, MOT2}, we give the following applications of this strategy:
\begin{itemize}
\item {\it Linear problems:} general solvability results for operators $\calP$ satisfying \ref{hyp:fc-0} under optimal (up to endpoints) assumptions on the regularity and decay properties of the coefficients; and
\item {\it Nonlinear problems:} new sharp results concerning the flexibility of general relativistic initial data sets (on a compact or asymptotically flat background), such as localization, gluing, extension, and parametrization, with or without the constant mean curvature (gauge) condition.
\end{itemize}
As an example, see Section~\ref{subsec:appl-lin} below for the precise statement of our sharp solvability result on bounded Lipschitz domains for $\calP$ with rough coefficients from \cite{IMOT2}. For applications to general relativistic initial data sets, we refer to \cite{MOT2}. We expect this approach to have numerous additional applications.

This paper was primarily motivated by the investigation of the \emph{flexibility} of solutions to \emph{underdetermined} linear and nonlinear PDEs arising in physics and geometry. Our approach generalizes classical work on the divergence operator by Bogovskii \cite{Bog} (which may be more familiar in fluid dynamics than general relativity) and clarifies explicit integral solution formulas found in recent studies of Yang-Mills initial data sets \cite{OhTat}, convex integration in fluid dynamics \cite{IseOh}, and asymptotically flat general relativistic initial data sets \cite{MaoTao, MOT1}.

By duality, the flexibility of underdetermined PDEs corresponds to the \emph{rigidity} of solutions to \emph{overdetermined} PDEs.  In this way, our work also connects to classical investigations of rigidity, notably Reshetnyak's integral representations for solutions of certain overdetermined linear differential operators (e.g., the Killing and conformal Killing operators) arising in geometric rigidity problems \cite{Resh}; see   Remark~\ref{rem:resh} for more discussion. Our method also provides a unified proof of Poincar\'e- or Friedrich-type (or rigidity) inequalities, including Korn's inequality \cite{Kor, KonOle}, for various overdetermined operators on a broad class of backgrounds, thereby bringing together previously disparate proofs (see Remark~\ref{rem:appl-rigidity}).

\subsection{Summary of the main results} \label{subsec:results}

\subsubsection{Explicit integral formulas for the divergence operator}
Before we describe our results, we first exhibit known explicit integral formulas for solving the prescribed divergence equation $\div u = f$ on $\bbR^{d}$, which are the main inspiration for our work.

In \cite{Bog}, Bogovskii wrote down a remarkable explicit integral formula for a compactly supported solution to $\div u = f$, where $f$ is a given scalar function with compact support and integral zero. Concretely, it takes the form
\begin{equation} \label{eq:bog-0}
u(x) \ceq \int_{\bbR^{d}} K_{\eta_{1}}(x, y) f(y) \, \ud y, \quad 
K_{\eta_{1}}(x, y) \ceq \frac{(x-y)^{j}}{\abs{x-y}^{d}} \left(\int_{\abs{x-y}}^{\infty} \eta_{1}(r \tfrac{x-y}{\abs{x-y}} + y)  r^{d-1} \, \ud r \right) \hbox{ for } x \neq y,
\end{equation}
where $\eta_{1} \in C^{\infty}_{c}(\bbR^{d})$ with $\int_{\bbR^{d}} \eta_{1}(y) \, \ud y = 1$. This integral formula turns out to satisfy the following properties:
\begin{enumerate}
\item (Green's function) We have $\div u (x) = f(x) - \left( \int_{\bbR^{d}} f(y) \, \ud y \right) \eta_{1}(x)$ for all $f \in C^{\infty}_{c}(\bbR^{d})$;
\item (Prescribed support) $\supp u \subseteq \cup_{y \in \supp f, \, y_{1} \in \supp \eta_{1}} (\hbox{line segment from $y$ to $y_{1}$})$;
\item (Optimal regularization) $K_{\eta_{1}} (x, y)$ is a locally integrable function such that $\rd_{x^{j}} K_{\eta_{1}}(x, y)$ is a Calder\'on--Zygmund integral kernel (and hence $f \mapsto \rd_{j} u$ is bounded on $L^{p}$ for any $1 < p < +\infty$).
\end{enumerate}
In view of (2), $u$ is indeed compactly supported if $f$ is, and we may manipulate its support property by varying $\eta_{1}$. The presence of an extra term involving $\int f \, \ud y$ in (1) is natural in view of the following necessary condition for the existence of a compactly supported solution $u$ (via the divergence theorem): $\int f \, \ud y = \int \div u \, \ud y = \lim_{R \to \infty} \int_{\rd B_{R}} u \cdot \bfnu \, \ud S = 0$. More abstractly, it is a manifestation of the fact that the formal cokernel of $\div$ (which is the pre-annihilator of the image of compactly supported distributions under $\div$), or simply the space of $C^{\infty}(\bbR^{d})$ functions with zero gradient, consists of constant functions.

In \cite{OhTat}, another explicit integral formula for a solution to $\div u = f$ was written down, where $f$ is a given scalar function with compact support (but not necessarily integral zero):
\begin{equation} \label{eq:conic-0}
	u (x) := \int_{\bbR^{d}} K_{\slashed{\eta}}(x, y) f(y) \, \ud y, \quad
	(K_{\slashed{\eta}})^{j}(x, y) = \frac{(x-y)^{j}}{\abs{x-y}^{d}}  \slashed{\eta}(\tfrac{x-y}{\abs{x-y}}) \quad \hbox{ for } x \neq y,
\end{equation}
where $\slashed{\eta} \in C^{\infty}(\bbS^{d-1})$ with $\int_{\bbS^{d-1}} \slashed{\eta} (\omg) \, \ud S(\omg) = 1$. The following properties hold:
\begin{enumerate}
\item (Green's function) We have $\div u (x) = f(x)$ for all $f \in C^{\infty}_{c}(\bbR^{d})$;
\item (Prescribed support) $\supp u \subseteq \cup_{y \in \supp f, \, \omg \in \supp \slashed{\eta}} (\hbox{ray from $y$ in the direction $\omg$})$;
\item (Optimal regularization) $K_{\slashed{\eta}} (x, y)$ is a locally integrable function such that $\rd_{x^{j}} K_{\slashed{\eta}}(x, y)$ is a Calder\'on--Zygmund integral kernel.
\end{enumerate}
By (2), $u$ is supported in the union of cones over $\supp \slashed{\eta} \subseteq \bbS^{d-1}$; for this reason, we call the operator $f \mapsto u$ in \eqref{eq:conic-0} a \emph{conic solution operator}. In this case, $u$ directly solves $\div u = f$ since it is allowed to have a non-compact support. Indeed, by $\int f \, \ud y = \int \div u \, \ud y = \liminf_{R \to \infty} \int_{\rd B_{R}} u \cdot \bfnu \, \ud S$, it necessarily has a non-compact support if $\int f \, \ud y \neq 0$. Abstractly, this is a manifestation of the fact that the cokernel of $\div$ in $C^{\infty}_{c}(\bbR^{d})$ (which is the pre-annihilator of the image of distributions under $\div$), or simply the space of $C_{c}^{\infty}(\bbR^{d})$ functions with zero gradient, is trivial.

Thanks to their simplicity and flexibility, these explicit integral formulas have proven useful in many applications: Bogovskii's operator has been extensively used in fluid dynamics (see \S\ref{subsubsec:discussion:div} for further discussions), and \cite{OhTat} used the conic operator to manipulate initial data sets for the Yang--Mills equation. See also \cite{MaoTao, MOT1}, in which these ideas were applied to the study of initial data sets in general relativity.

The results of this paper \emph{generalize such integral formulas to a large class of underdetermined differential operators} (and simultaneously, \emph{their adjoints to overdetermined differential operators}) \emph{arising from geometry and physics}. In the remainder of this subsection, we explain each component of our approach in more detail. For a systematic derivation of \eqref{eq:bog-0} and \eqref{eq:conic-0} from our viewpoint, as well as a justification of the properties stated above, we refer the reader already to Section~\ref{subsec:ideas-div} below.

\subsubsection{Integral solution and representation formulas from \ref{hyp:crc}}
Let $\calP$ be an $r_{0} \times s_{0}$-matrix-valued differential operator on an open subset $U \subseteq \bbR^{d}$ where $r_{0} \leq s_{0}$. For simplicity, assume for now that $\calP$ has $C^{\infty}(\br{U})$ coefficients (for the case of rough coefficients, see Theorem~\ref{thm:low-reg-domain}). We first introduce (in a simplified form) the \emph{recovery on curves condition} for $\calP$, which plays a basic role in this paper. For a curve $\bfx : [0, 1] \to \bbR^{d}$, it says (roughly speaking):
\begin{enumerate}[label=(RC)]
\item \label{hyp:crc} Given any $\varphi \in C^{\infty}_{c}(U)$, there exists a linear way to continuously recover $\varphi(\bfx(0))$ from (the jet of) $\calP^{\ast} \varphi$ on $\bfx$ and (the jet of) $\varphi$ at $\bfx(1)$.
\end{enumerate}
Note that \ref{hyp:crc} obviously holds on any curve for the divergence operator $\calP u = \rd_{j} u^{j}$, in the sense that $\calP^{\ast} \varphi = - \ud \varphi$ (gradient operator) and thus $\varphi(\bfx(0)) = \int_{0}^{1} (- \ud \varphi)(\dot{\bfx}(t)) \, \ud t + \varphi(\bfx(1))$. For the precise version of this condition, see Sections~\ref{subsec:crc-1} and \ref{subsec:crc-q} below.

Simply speaking, our first set of results says that \ref{hyp:crc} is all we need to construct integral formulas analogous to \eqref{eq:bog-0} and \eqref{eq:conic-0} solving $\calP u = f$. More specifically, but still informally, \ref{hyp:crc} implies the existence of a solution operator (i.e., right-inverse) for $\calP$ that is \emph{regularizing of optimal order} such that, for every $y$, its integral kernel $K(x, y)$ is supported (in the $x$-variable) in a union of curves emanating from $y$ that can be \emph{prescribed} in the sense we will explain below. By duality, \ref{hyp:crc} also implies the existence of integral representation formulas (i.e., left-inverse) for $\calP^{\ast}$ with analogous properties, which in turn imply Poincar\'e-type inequalities that control $u$ in terms of $\calP^{\ast} u$ under suitable additional conditions.

We now formulate the results more precisely. We employ the fractional Sobolev spaces $W^{s,p}(U)$ and $\td{W}^{s,p}(U)$ on domains, which are precisely defined in Section~\ref{subsec:ftn-sp}. Here, we simply point out that when $s$ is a nonnegative integer and $1 < p < +\infty$, $W^{s, p}(U)$ agrees with the usual definition (see, e.g., \cite{Evans}) and $\td{W}^{s,p}(U) = W^{s, p}_{0}(U)$, the closure of $C^{\infty}_{c}(U)$ in $W^{s, p}(U)$. For any open subset $U \subseteq \bbR^{d}$, $s \in \bbR$, and $p \in (1, \infty)$, the following duality relations hold (for details, see Lemma~\ref{lem:sob-duality}):
\begin{equation*}
	W^{-s, p'}(U) \equiv (\td{W}^{s, p}(U))^{\ast}, \quad \td{W}^{s, p}(U) \equiv (W^{-s, p'}(U))^{\ast},
\end{equation*}
where the isomorphisms (denoted by $\equiv$) are induced by the unique pairing $W^{-s, p'}(U) \times \td{W}^{s, p}(U) \to \bbR$, $(f, g) \mapsto \brk{f, g}$ that coincides with $\int_{U} \Re(\br{f} g) \, \ud x$ for $(f, g) \in C^{\infty}(\br{U}) \times C^{\infty}_{c}(U)$.

For each $K \in \set{1, \ldots, s_{0}}$, we write $m_{K}$ for the order of the operator $\varphi \mapsto (\calP^{\ast} \varphi)_{K}$. Furthermore, consider a family of curves $\bfx(y, y_{1}, s)$ ($s \in [0, 1]$), where $y$ and $y_{1}$ denote the two endpoints at $s = 0$ and $s = 1$, respectively. We assume that $\bfx(y, y_{1}, s)$ is \emph{admissible} in the sense that it behaves like (or coincides with) straight line segments $\ul{\bfx}(y, y_{1}, s) = y + s (y_{1} - y)$ when $y$ and $y_{1}$ are close (in fact, the precise conditions for admissibility consist of \ref{hyp:x-1}--\ref{hyp:x-2} in Section~\ref{subsubsec:x-q}, and \ref{hyp:x-3} in Section~\ref{subsec:1st-order-pde}).

\begin{theorem} [Conic-type solution operators, representation formulas, and Poincar\'e-type inequalities]\label{thm:main-summary-conic}
Let $U$ be an open subset of $\bbR^{d}$. Let $\calP$ satisfy \ref{hyp:crc} for an admissible family of curves $\bfx = \bfx(y, y_{1}, s)$ for all $y$ in a neighborhood of $\br{U}$ and $y_{1} \in U_{1}$ for some open subset $U_{1}$ of $\bbR^{d}$ (see Section~\ref{subsec:crc-q} for the precise formulation). Assume, moreover, that the curves are \emph{nontrapped in $U$} (i.e., the curve eventually exits $U$) in the sense that
\begin{equation*}
	\br{U} \cap \br{U}_{1} = \0.
\end{equation*}
Then the following holds.
\begin{enumerate}
\item {\bf Cokernel in $\td{W}^{-s, p'}(U)$.} For any $1 < p < +\infty$ and $s \in \bbR$, define the cokernel in $\td{W}^{-s, p'}(U)$ to be:
\begin{equation} \label{eq:sob-td-coker}
	\ker_{\td{W}^{-s, p'}(U)} \calP^{\ast} := \set{\bfZ \in \td{W}^{-s, p'}(U) : \calP^{\ast} \bfZ = 0 \hbox{ in } \calD'(\td{U})},
\end{equation}
where $\td{U}$ is an open subset of $\bbR^{d}$ such that $\br{U} \subseteq \td{U}$, and the coefficients of $\calP^{\ast}$ are extended in a smooth way to $\td{U}$ (since $\bfZ \in \td{W}^{-s, p'}(U)$, this definition is independent of these choices).
Under the assumptions of this theorem, we have
\begin{equation*}
\ker_{\td{W}^{-s, p'}(U)} \calP^{\ast} = \set{0}.
\end{equation*} 
\item {\bf Integral solution formula.} There exists a locally integrable integral kernel $K : U \times U \to \bbC^{s_{0} \times r_{0}}$ with the support property
\begin{equation*}
\supp K(\cdot, y) \subseteq \br{\bigcup_{y_{1} \in U_{1}} \bfx(y, y_{1}, [0, 1])} \quad \hbox{ for every } y \in U,
\end{equation*}
such that the integral operator $\calS f (x) \ceq \int_{U} K(x, y) f(y) \, \ud y$ for $f \in C^{\infty}_{c}(U; \bbC^{s_{0}})$ satisfies
\begin{equation*}
	\calP \calS f = f \quad \hbox{ for all } f \in C^{\infty}_{c}(U; \bbC^{s_{0}}), \quad
\end{equation*}
Moreover, for any $1 < p < +\infty$ and $s \in \bbR$, $\calS$ extends to a bounded operator from $\td{W}^{s, p}(U; \bbC^{r_{0}})$ to $W^{s+m_{1}, p}(U) \times \cdots \times W^{s+m_{s_{0}}, p}(U)$.
\item {\bf Integral representation formula \& Friedrich-type inequality.} Dually, we have
\begin{equation*}
\varphi = \calS^{\ast} \calP^{\ast} \varphi \quad \hbox{ for all } \varphi \in C^{\infty}_{c}(U; \bbC^{r_{0}}).
\end{equation*}
For any $1 < p < +\infty$ and $s \in \bbR$, $\calS^{\ast}$ extends to a bounded operator from $\td{W}^{-s-m_{1}, p'}(U) \times \cdots \times \td{W}^{-s-m_{s_{0}}, p'}(U)$ to $W^{-s, p'}(U; \bbC^{r_{0}})$. Moreover, the following \emph{Friedrich-type inequality} holds:
\begin{equation*}
	\nrm{\varphi}_{\td{W}^{-s, p'}(U; \bbC^{r_{0}})} \aleq \nrm{\calP^{\ast} \varphi}_{\td{W}^{-s-m_{1}, p'}(U) \times \cdots \times \td{W}^{-s-m_{s_{0}}, p'}(U)} \quad \hbox{ for all } \varphi \in \td{W}^{-s, p'}(U; \bbC^{r_{0}}).
\end{equation*}
\end{enumerate}
\end{theorem}

\begin{theorem} [Bogovskii-type solution operators, representation formulas, and Poincar\'e-type inequalities]\label{thm:main-summary-bogovskii}
Let $U$ be a connected bounded open subset of $\bbR^{d}$. Let $\calP$ satisfy \ref{hyp:crc} for an admissible family of curves $\bfx = \bfx(y, y_{1}, s)$ (see Section~\ref{subsec:crc-q} for the precise formulation). Assume that $\br{U}$ is \emph{$\bfx$-star-shaped} with respect to $\br{U_{1}}$, where $U_{1}$ is an open subset of $U$ such that $\br{U_{1}} \subseteq U$, in the sense that
\begin{equation} \label{eq:main-summary-x-starshaped}
	\br{\bigcup_{y \in \br{U}, \, y_{1} \in \br{U_{1}}} \bfx(y, y_{1}, [0, 1])} \subseteq \br{U}.
\end{equation}
Then the following holds.
\begin{enumerate}
\item {\bf Cokernel in $W^{-s, p'}(U)$.} For any $1 < p < +\infty$ and $s \in \bbR$, define the cokernel in $W^{-s, p'}(U)$ to be:
\begin{equation} \label{eq:sob-coker}
	\ker_{W^{-s, p'}(U)} \calP^{\ast} := \set{\bfZ \in W^{-s, p'}(U) : \calP^{\ast} \bfZ = 0 \hbox{ in } \calD'(U)}.
\end{equation}
We have the \emph{invariance property}
\begin{equation*}
\ker_{W^{-s, p'}(U)} \calP^{\ast} = \ker \calP^{\ast},
\end{equation*}
and the \emph{finite-dimensional property}
\begin{equation} \label{eq:main-summary-bogovkii-fc}
	\dim \ker \calP^{\ast}(U) < + \infty.
\end{equation}
where $\ker \calP^{\ast}$ is the formal cokernel of $\calP$ defined in \eqref{eq:formal-coker}. Moreover, for any open subset $V \subseteq U$, the restriction of $\ker \calP^{\ast}$ to $V$, i.e.,
\begin{equation*}
	\ker \calP^{\ast} |_{V} \ceq \set{\bfZ |_{V} : \bfZ \in \ker \calP^{\ast}}
\end{equation*}
has the same dimension as $\ker \calP^{\ast}$.

\item {\bf Integral solution formula under orthogonality conditions.} 
Consider a family $w_{\bfA}(x) \in C^{\infty}_{c}(U_{1}; \bbC^{r_{0}})$ $(\bfA \in \set{1, \ldots, \dim \ker \calP^{\ast}})$ satisfying $\brk{w_{\bfA}, \bfZ^{\bfA'}} = \dlt_{\bfA}^{\bfA'}$ for some basis $\set{\bfZ^{\bfA'}}$ of $\ker \calP^{\ast}$. Then there exists a locally integrable integral kernel $\td{K} : U \times U \to \bbC^{s_{0} \times r_{0}}$ with the support property
\begin{equation*}
\supp \td{K}(\cdot, y) \subseteq \br{\bigcup_{y_{1} \in \br{U_{1}}} \bfx(y, y_{1}, [0, 1])} \subseteq \br{U} \quad \hbox{ for every } y \in U,
\end{equation*}
(where the last inclusion follows simply from \eqref{eq:main-summary-x-starshaped}) such that the integral operator $\td{\calS} f (x) \ceq \int_{U} \td{K}(x, y) f(y) \, \ud y$ for $f \in C^{\infty}_{c}(U; \bbC^{s_0})$ satisfies
\begin{equation*}
	\calP \td{\calS} f = f - \sum_{\bfA \in \set{1, \ldots, \dim \ker \calP^{\ast}}} w_{\bfA} \brk{\bfZ^{\bfA}, f}  \quad \hbox{ for all } f \in C^{\infty}_{c}(U; \bbC^{s_{0}}).
\end{equation*}
Moreover, for any $1 < p < +\infty$ and $s \in \bbR$, $\td{\calS}$ extends to a bounded operator from $\td{W}^{s, p}(U; \bbC^{r_{0}})$ to $\td{W}^{s+m_{1}}(U) \times \cdots \times \td{W}^{s+m_{s_{0}}}(U)$.
\item {\bf Integral representation formula \& Poincar\'e-type inequality under orthogonality conditions.} Dually, we have
\begin{equation*}
\varphi = \td{\calS}^{\ast} \calP^{\ast} \varphi + \sum_{\bfA \in \set{1, \ldots, \dim \ker \calP^{\ast}}} \bfZ^{\bfA} \brk{w_{\bfA}, \varphi} \quad \hbox{ for all } \varphi \in C^{\infty}(\br{U}; \bbC^{r_{0}}).
\end{equation*}
Furthermore, for any $1 < p < +\infty$ and $s \in \bbR$, $\td{\calS}^{\ast}$ defines a bounded operator from $W^{-s-m_{1},p}(U) \times \cdots \times W^{-s-m_{s_{0}},p'}(U)$ to $W^{-s, p'}(U; \bbC^{r_{0}})$, and the following \emph{Poincar\'e-type inequality} holds:
\begin{equation*}
	\nrm{\varphi - \sum_{\bfA} \bfZ^{\bfA} \brk{w_{\bfA}, \varphi}}_{W^{-s, p'}(U; \bbC^{r_{0}})} \aleq \nrm{\calP^{\ast} \varphi}_{W^{-s-m_{1}, p'}(U) \times \cdots \times W^{-s-m_{s_{0}}, p'}(U)} \quad \hbox{ for all } \varphi \in W^{-s, p'}(U; \bbC^{r_{0}}).
\end{equation*}
\end{enumerate}
\end{theorem}

In Section~\ref{subsec:method} below, we give a more detailed description of the structure of the integral kernels $K(x, y)$ and $\td{K}(x, y)$, and summarize the proofs of Theorems~\ref{thm:main-summary-conic} and \ref{thm:main-summary-bogovskii}.

\begin{remark}[On the regularity of $U$]
The reader may find it amusing that neither theorem requires any regularity assumptions on the boundary of $U$. For Theorem~\ref{thm:main-summary-conic}, the nontrapping assumption is crucial. For Theorem~\ref{thm:main-summary-bogovskii}, the $\bfx$-star-shaped assumption in fact embodies some notion of regularity of $\rd U$. For instance, if $\bfx$ is the straight line segment $\ul{\bfx}(y, y_{1}, s) = y + s (y_{1} - y)$, then this assumption implies the uniform cone condition for $U$, which in turn implies that $\rd U$ is Lipschitz \cite[Section~1.2]{Gri}. See also Theorem~\ref{thm:low-reg-domain} for a result that applies to Lipschitz domains rather than those with \eqref{eq:main-summary-x-starshaped}.
\end{remark}

\subsubsection{Tools for verifying \ref{hyp:crc}: Graded augmented system and \ref{hyp:fc}}
Our second set of results provides tools for verifying \ref{hyp:crc} for a variety of under/overdetermined partial differential operators.

Our basic device is the notion of a (\emph{graded}) \emph{augmented system}, which generalizes a basic procedure for verifying \ref{hyp:crc} for the divergence operator on $\bbR^{d}$; see Section~\ref{subsec:ideas-div} and Remark~\ref{rem:intro-div:aug} below. It is also a generalization of the procedure used by Retshenyak \cite{Resh} to construct integral representation formulas for some overdetermined linear operators (see Remark~\ref{rem:resh}). Its precise formulation requires us to introduce some conventions and definitions. In what follows, we adopt the following index notation (which is consistent with Section~\ref{subsec:prelim-P}):
\begin{itemize}
\item $J \in \set{1, \ldots, r_{0}}$ (and its variants such as $J'$, etc.): index for components of $\varphi = (\varphi_{J})_{J = 1, \ldots, r_{0}}$
\item $K \in \set{1, \ldots, s_{0}}$ (and its variants such as $K'$, etc.): index for components of $\calP^{\ast} \varphi = ((\calP^{\ast} \varphi)_{K})_{K = 1, \ldots, s_{0}}$
\item $\bfA \in \calA$ (and its variants such as $\bfA'$, etc.): index for components of the augmented variables $(\Phi_{\bfA})_{\bfA \in \calA}$ (to be defined below).
\end{itemize}
As usual, we adopt the convention of summing up repeated upper and lower indices, unless otherwise stated. We also make the provision that {\bf a multi-index $\gmm$ in $\rd^{\gmm}$ is considered a lower index} (hence, $c^{\gmm} \rd^{\gmm} = \sum_{\gmm} c^{\gmm} \rd^{\gmm}$).

\begin{definition} [Graded augmented system] \label{def:aug}
Let $\calP$ be an $r_{0} \times s_{0}$-matrix-valued differential operator on an open subset $U \subseteq \bbR^{d}$, with $m_{K}$ denoting the order of $(\calP^{\ast} \varphi)_{K}$ for each $K \in \set{1, \ldots, s_{0}}$. Given $m_{K}' \in \bbZ_{\geq 0}$ for each $K \in \set{1, \ldots, s_{0}}$ and an $\bbC^{r_{0}}$-valued function $(\varphi_{J})_{J \in \set{1, \ldots, r_{0}}}$ on $U$, consider $(\Phi_{\bfA} = \Phi_{\bfA}(y))_{\bfA \in \calA}$ (called \emph{augmented variables}) satisfying the following properties:
\begin{enumerate}[label=($\Phi$-\arabic*), series=augcondition]
\item \label{hyp:aug1} {\bf $(\Phi_{\bfA})$ is an augmentation of $(\varphi_{J})$.} The index set $\calA$ contains $\set{1, \ldots, r_{0}}$; moreover, $\Phi_{J} = \varphi_{J}$ for $J \in \set{1, \ldots, r_{0}} \subseteq \calA$.
\item \label{hyp:aug2} {\bf $(\varphi_{J}) \mapsto (\Phi_{\bfA})$ is a differential operator.} There exist functions $c[\Phi_{\bfA}]^{(\alp, J)}$ on $U$, where $\alp$ is a multi-index and $J \in \set{1, \ldots, r_{0}}$, such that
\begin{equation*}
\Phi_{\bfA}(y) = c[\Phi_{\bfA}]^{(\alp, J)}(y) \rd^{\alp} \varphi_{J}(y),
\end{equation*}
for all $\bfA \in \calA$ and $y \in U$.
\item \label{hyp:aug3} {\bf $\Phi$ and $\calP^{\ast} \varphi$ solve a first-order PDE (augmented PDE system).} There exist matrix-valued $1$-forms $\tensor{(\bfB_{i})}{_{\bfA}^{\bfA'}}$ and $\tensor{(\bfC_{i})}{_{\bfA}^{(\gmm, K)}}$ on $U$, where $\gmm$ is a multi-index and $K \in \set{1, \ldots, s_{0}}$, such that
\begin{equation} \label{eq:aug-pde}
	\rd_{i} \Phi_{\bfA}(y) = \tensor{(\bfB_{i})}{_{\bfA}^{\bfA'}}(y) \Phi_{\bfA'}(y) + \tensor{(\bfC_{i})}{_{\bfA}^{(\gmm, K)}}(y) \rd^{\gmm} (\calP^{\ast} \varphi)_{K}(y),
\end{equation}
for all $\bfA \in \calA$, $i \in \set{1, \ldots, d}$ and $y \in U$. Furthermore,
\begin{equation*}
\tensor{(\bfC_{i})}{_{\bfA}^{(\gmm, K)}} = 0 \quad \hbox{ if } \abs{\gmm} > m_{K}'.
\end{equation*}
\item \label{hyp:aug4} {\bf Graded structure.} Define the \emph{degree} $d_{\bfA}$ of $\Phi_{\bfA}$ by
\begin{equation*}
d_{\bfA} \ceq - \max \set{\abs{\alp} : c[\Phi_{\bfA}]^{(\alp, J)} \neq 0 \hbox{ for some } J}.
\end{equation*}
In particular, $\varphi_{J}$ has degree $0$, i.e., $d_{J} = 0$ for $J \in \set{1, \ldots, r_{0}}$. We assume that:
\begin{itemize}
\item {\bf Graded structure for the augmented system.}
\begin{align*}
\tensor{(\bfB_{i})}{_{\bfA}^{\bfA'}} &= 0 \qquad \hbox{ if } d_{\bfA} >  d_{\bfA'} + 1, \\
\tensor{(\bfC_{i})}{_{\bfA}^{(\gmm, K)}} &= 0 \qquad \hbox{ if } d_{\bfA} > -m_{K} - \abs{\gmm} +1.
\end{align*}
\end{itemize}
\end{enumerate}

We write $N_{0} \ceq \max\set{\abs{d_{\bfA}}}_{\bfA \in \calA} + 1$ for the maximal degree\footnote{Here, $+1$ accounts for the derivative $\rd_{i}$ on the LHS of \eqref{eq:aug-pde}.} that occurs in \eqref{eq:aug-pde}.

We call a collection $(\calA, (\Phi_{\bfA})_{\bfA \in \calA}, \tensor{(\bfB_{i})}{_{\bfA}^{\bfA'}}, \tensor{(\bfC_{i})}{_{\bfA}^{(\gmm, K)}})$ satisfying \ref{hyp:aug1}--\ref{hyp:aug4} a (\emph{graded}) \emph{augmented system} for $\calP$.
\end{definition}

\begin{remark} [Graded structure]
The degree $d_{\bfA}$ is (minus) the number of derivatives falling on $\varphi$ in $\Phi_{\bfA}$. The graded structure for $\bfB_{i}$ is the natural requirement that, in the equation for one derivative of $\Phi_{\bfA} = c[\Phi_{\bfA}]^{(\alp, J)} \rd^{\alp} \varphi_{J}$, we do not see derivatives of $\varphi$ that are two orders higher. The graded structure for $\bfC_{i}$ is an analogous requirement, where we assign degree $-m_{K}-\abs{\gmm}$ to $\rd^{\gmm} (\calP^{\ast} \varphi)_{K}$.
\end{remark}

Given a graded augmented system $(\Phi_{\bfA})_{\bfA \in \calA}$, \ref{hyp:aug3} implies that each $\Phi_{\bfA}$ satisfies an ODE on each curve $\bfx(y, y_{1}, \cdot)$ of the following form:
\begin{equation} \label{eq:aug-ode}
	\frac{\ud}{\ud s} \left( \Phi_{\bfA} \circ \bfx \right) = \dot{\bfx}^{i} \left( \tensor{(\bfB_{i})}{_{\bfA}^{\bfA'}} \circ \bfx \right) \left( \Phi_{\bfA'}  \circ \bfx \right) + \dot{\bfx}^{i} \left( \tensor{(\bfC_{i})}{_{\bfA}^{(\gmm, K)}}\circ \bfx\right) \left( \rd^{\gmm} (\calP^{\ast} \varphi)_{K} \circ \bfx \right),
\end{equation}
where $\dot{\bfx}(y, y_{1}, s) = \rd_{s} \bfx(y, y_{1}, s)$. In particular, by Duhamel's principle (or variation of constants), we may express $\varphi_{J}$ at $y = \bfx(y, y_{1}, 0)$ as follows:
\begin{equation} \label{eq:varphi-duhamel}
\begin{aligned}
\varphi_{J}(y) &= - \int_{0}^{1}  {}^{(\bfx_{y, y_{1}})} \tensor{\bfPi}{_{J}^{\bfA}}(0, s) \dot{\bfx}^{i} \left( \tensor{(\bfC_{i})}{_{\bfA}^{(\gmm, K)}}  \circ \bfx \right) \left( (\rd^{\gmm} \calP^{\ast} \varphi)_{K} \circ \bfx \right) (y, y_{1}, s) \, \ud s  \\
&\peq + {}^{(\bfx_{y, y_{1}})} \tensor{\bfPi}{_{J}^{\bfA}}(0, 1) \Phi_{\bfA}(y_{1}),
\end{aligned}
\end{equation}
where ${}^{(\bfx_{y, y_{1}})} \tensor{\bfPi}{_{\bfA}^{\bfA'}}(s, t)$ is the fundamental matrix for $\frac{\ud}{\ud s} - \dot{\bfx}^{i} (\bfB_{i} \circ \bfx)$. This formula immediately implies \ref{hyp:crc}; it also gives a representation of any element $\bfZ \in \ker \calP^{\ast}$  (i.e., $\calP^{\ast} \bfZ = 0$) in terms of the values of $\Phi_{\bfA}$ at a point. In fact, we have the following result.

\begin{proposition} \label{prop:aug2crc}
Assume that $\calP$ possesses a graded augmented system $(\Phi_{\bfA})_{\bfA \in \calA}$, and that $\bfB_{i}$ and $\bfC_{i}$ and their derivatives are uniformly bounded on $U$. Then $\calP$ satisfies \ref{hyp:crc} for any admissible family of curves $\bfx$; more precisely, Theorems~\ref{thm:main-summary-conic} and \ref{thm:main-summary-bogovskii} are applicable. Moreover, for every $\bfZ \in \ker \calP^{\ast}$ and $y_{1} \in U$, we have
\begin{equation} \label{eq:aug2crc-Z}
	\bfZ_{J}(x) = {}^{(\bfx_{y, y_{1}})} \tensor{\bfPi}{_{J}^{\bfA}}(0, 1) \Phi_{\bfA}(y_{1}),
\end{equation}
In particular, $\dim \ker \calP^{\ast} \leq \# \calA$.
\end{proposition}
For the precise formulation and proof, see Section~\ref{sec:ode}, in particular, Proposition~\ref{prop:ode2crc} and Remark~\ref{rem:coeff-nrm-Ck}.

In view of the bound $\dim \calP^{\ast} \leq \# \calA$, it is of interest to ask when equality holds. The following result answers this question under reasonable assumptions:
\begin{proposition} \label{prop:zero-curv}
Let $U$ be a simply connected open subset of $\bbR^{d}$. Assume that $\calP$ possesses a graded augmented system $(\Phi_{\bfA})_{\bfA \in \calA}$, and that $\bfB_{i}$ and $\bfC_{i}$ and their derivatives are uniformly bounded on $U$. Then $\dim \ker \calP^{\ast} = \# \calA$ if and only if the following condition (called the \emph{zero curvature condition}) holds for all $x \in U$, $i, j = 1, \ldots, d$, and $\bfA, \bfA' \in \calA$:
\begin{equation} \label{eq:zero-curv}
\left(\rd_{i} \tensor{(\bfB_{j})}{_{\bfA}^{\bfA'}} - \rd_{j} \tensor{(\bfB_{i})}{_{\bfA}^{\bfA'}} 
+ \tensor{(\bfB_{i})}{_{\bfA}^{\bfA''}} \tensor{(\bfB_{j})}{_{\bfA''}^{\bfA'}}
- \tensor{(\bfB_{j})}{_{\bfA}^{\bfA''}} \tensor{(\bfB_{i})}{_{\bfA''}^{\bfA'}} \right)(x)= 0.
\end{equation}
\end{proposition}
For a proof, see Section~\ref{subsec:conn}, where we give a geometric interpretation of \eqref{eq:zero-curv} in terms the curvature of a connection on a vector bundle, a viewpoint that is of interest on its own. This motivates the following:
\begin{definition}[Completely integrable augmented systems] \label{def:aug-int}
We say that an augmented system is \emph{complete integrable} if the zero curvature condition \eqref{eq:zero-curv} is satisfied  for all $x \in U$, $i, j = 1, \ldots, d$, and $\bfA, \bfA' \in \calA$.
\end{definition}

For examples of completely integrable augmented systems, we refer to Appendix~\ref{sec:appl}, as well as Reshetnyak \cite{Resh, ReshBook} (see also Remark~\ref{rem:resh}).

Complete integrability, or more precisely $\# \calA = \dim \ker \calP^{\ast}$, leads to a simplification of the derivation of integral solution and representation formulas; see Remark~\ref{rem:intro:bogovskii}, Remark~\ref{rem:intro:full-sol-eff}, and Theorem~\ref{thm:full-sol-max} below. However, it is not necessary for the validity of Theorems~\ref{thm:main-summary-conic} and \ref{thm:main-summary-bogovskii}. In fact, to handle a larger class of operators $\calP$, it turns out to be useful to consider the other extreme case, namely, graded augmented systems with the \emph{maximal} number of augmented variables with a given maximal degree $N_{0}$. 

\begin{definition}[Maximal graded augmented system] \label{def:aug-max}
Given an integer $N_{0} \geq \max_{K} m_{K}$, a graded augmented system with augmented variables $(\Phi_{\bfA})_{\bfA \in \calA} = (\rd^{\alp} \varphi_{J})_{(\alp, J) : 1 \leq J \leq r_{0}, \, \abs{\alp} \leq N_{0} - 1}$ consisting of all partial derivatives of $\varphi$ up to order $N_{0} -1$ (i.e., $\calA = \set{(\alp, J) : 1\leq J \leq r_{0}, \, \abs{\alp} \leq N_{0} - 1}$) is called a \emph{maximal graded augmented system}.
\end{definition}

A useful property of a maximal graded augmented system is that it is \emph{stable under lower order perturbations}, i.e., the same variables constitute a graded augmented system for $\td{\calP}$ as long as $((\td{\calP}-\calP)^{\ast} \varphi)_{K}$ is of order less than $m_{K}$. Observe that such a stability property is not evident for \ref{hyp:crc}, nor for completely integrable augmented systems.

We are now ready to formulate an algebraic sufficient condition for \ref{hyp:crc} in terms of the principal symbol $p^{\ast}(x, \xi)$ of $\calP^{\ast}$, which greatly facilitates the applicability of our theory (see, for instance, Theorem~\ref{thm:appl} below). To formulate this result, we begin with a suitable definition of the principal symbol of the matrix-valued operator $\calP^{\ast}$:

\begin{definition} \label{def:prin-symb}
Let $U$ be an open subset of $\bbR^{d}$, and let $\calP$ be an $r_{0} \times s_{0}$-matrix-valued differential operator on $U$. Suppose that $\calP$ and its adjoint $\calP^{\ast}$ can be written out in the form\footnote{Note that $\tensor{c[\calP^{\ast}]}{_{K}^{(\alp, J)}}(x) = (-1)^{\abs{\alp}} \tensor{\br{c[\calP]}}{^{(\alp, J)}_{K}}(x)$.}
\begin{equation*}
(\calP u)^{J}(x) = \sum_{\alp} \tensor{c[\calP]}{^{(\alp, J)}_{K}}(x) \rd_{x}^{\alp} u^{K}(x), \qquad
	(\calP^{\ast} \varphi)_{K}(x) = \sum_{\alp} \tensor{c[\calP^{\ast}]}{_{K}^{(\alp, J)}}(x) \rd_{x}^{\alp} \varphi_{J}(x).
\end{equation*}
We define the \emph{principal parts} of $\calP$ and $\calP^{\ast}$, respectively, to be
\begin{equation*}
(\calP_{\prin} u)^{J}(x) \ceq \sum_{\alp : \abs{\alp} = m_{K}} \tensor{c[\calP]}{^{(\alp, J)}_{K}}(x) \rd_{x}^{\alp} u^{K}(x), \qquad
	(\calP_{\prin}^{\ast} \varphi)_{K}(x) = \sum_{\abs{\alp} = m_{K}} \tensor{c[\calP^{\ast}]}{_{K}^{(\alp, J)}}(x) \rd_{x}^{\alp} \varphi_{J}(x),
\end{equation*}
where we recall that $m_{K}$ is the order of the operator $\varphi \mapsto (\calP^{\ast} \varphi)_{K}$. Accordingly, we define the \emph{principal symbols} $p(x, \xi)$ and $p^{\ast}(x, \xi)$ of $\calP$ and $\calP^{\ast}$, respectively, to be
\begin{equation*}
    \tensor{p}{^{J}_{K}}(x, \xi) \ceq \sum_{\alp : \abs{\alp} = m_{K}} \tensor{c[\calP]}{^{(\alp, J)}_{K}}(x) i^{\abs{\alp}} \xi^{\alp}, \qquad
	\tensor{(p^{\ast})}{_{K}^{J}}(x, \xi) \ceq \sum_{\alp : \abs{\alp} = m_{K}} \tensor{c[\calP^{\ast}]}{_{K}^{(\alp, J)}}(x) i^{\abs{\alp}} \xi^{\alp}.
\end{equation*}
\end{definition}
In terms of this definition, we formulate the following algebraic condition:
\begin{enumerate}[label=(FC)]
\item \label{hyp:fc} For all $x \in U$ and $\xi \in \bbC^{d} \setminus \set{0}$, $p^{\ast}(x, \xi)$ is injective (or equivalently, $p^{\ast}(x, \xi)$ is full rank, or $p(x, \xi)$ is surjective, or $p(x, \xi)$ is full rank).
\end{enumerate}
Observe that \ref{hyp:fc} is stronger than over/underdetermined \emph{ellipticity}, which would be the same condition but only for $\xi \in \bbR^{d} \setminus \set{0}$. Our key result is:
\begin{theorem} \label{thm:alg-cond-intro}
Let $U$ be a connected open subset of $\bbR^{d}$, and let $\calP$ be an $r_{0} \times s_{0}$-matrix-valued differential operator on $U$ with smooth coefficients.
\begin{enumerate}
    \item Condition \ref{hyp:fc} implies the existence of a maximal graded augmented system $(\Phi_{\bfA})_{\bfA \in \calA}$. Hence, \ref{hyp:fc} implies \ref{hyp:crc} for any admissible family of curves, and Theorems~\ref{thm:main-summary-conic}--\ref{thm:main-summary-bogovskii} apply.
    \item If $p^{\ast}$ is independent of $x$ (i.e., $\calP_{\prin}$ has constant coefficients), then the following are equivalent:
    \begin{enumerate}
\item $p^{\ast}$ satisfies \ref{hyp:fc},
\item any $r_{0} \times s_{0}$-matrix-valued differential operator $\calP'$ on $U$ with principal symbol $p^{\ast}$ satisfies \ref{hyp:crc} for any admissible family of curves, and
\item the formal cokernel of $\calP_{\prin}$ (i.e., $\ker \calP_{\prin}^{\ast} \ceq \{ \bfZ \in C^{\infty}(\br{U}) : \calP^{\ast}_{\prin} \bfZ = 0 \hbox{ in } \calD'(U) \}$) is finite dimensional.
\end{enumerate}
\end{enumerate}
\end{theorem}
In light of Theorem~\ref{thm:alg-cond-intro}.(2), we refer to \ref{hyp:fc} as the \emph{finite-dimensional cokernel} condition. In order for \ref{hyp:fc} to hold, we necessarily have $r_{0} \leq s_{0}$.

A key ingredient in our proof of Theorem~\ref{thm:alg-cond-intro} is a basic result in algebraic geometry -- namely, Hilbert's Nullstellensatz (Proposition~\ref{prop:nullstellensatz}) -- which we use to construct a maximal graded augmented system with a sufficiently high maximal degree $N_{0}$ for an operator satisfying the algebraic condition \ref{hyp:fc}. We refer to Section~\ref{sec:fc2crc} for a proof of Theorem~\ref{thm:alg-cond-intro}.

\begin{remark} [Further generalization] \label{rem:gen-P}
Our setup assumes that each component of $(\varphi_{J})_{J \in \set{1, \ldots, r_{0}}}$ (or equivalently, $(f^{J})_{J \in \set{1, \ldots, r_{0}}}$) has the same degree (see, in particular, \ref{hyp:aug4}); accordingly, we look at $(\calP^{\ast} \varphi)_{K}$ for each $K \in \set{1, \ldots, r_{0}}$ to define the order $m_{K}$. This setup is sufficient for our applications in Theorem~\ref{thm:appl} and Appendix~\ref{sec:appl}. Nevertheless, we note that it is possible to develop the theory under the assumption that $\varphi_{J}$ have different degrees (i.e., $d_{J}$ are not all equal). In this case, one needs to introduce $\tensor{m}{_{K}^{J}} \in \bbZ_{\geq 0}$ in place of $m_{K}$ and alter Definition~\ref{def:prin-symb} and \ref{hyp:aug4}. We will not pursue this more generalized setup in more detail.
\end{remark}

\begin{remark} [Comparison with Reshetnyak's approach] \label{rem:resh}
The ideas presented in this part owes much to the work of Reshetnyak \cite{Resh, ReshBook}.
In our terminology, Reshetnyak introduced completely integrable graded augmented systems for certain overdetermined operators $\calP^{\ast}$ (including the Killing and conformal Killing operators on Euclidean space), and utilized them to derive integral representation formulas involving $\calP^{\ast}$ based on line segments $\ul{\bfx}(y, y_{1}, s) = y + s (y_{1} - y)$. Of this procedure, Retshenyak \cite[p.~27]{ReshBook} remarks: {\it ``The general scheme of constructing such representations is apparently beyond formalisation.''}

Among others, the most important departure of our approach from that of Reshetnyak is the relaxation of Reshetnyak's completely integrable system to Definition~\ref{def:aug}, and then furthermore considering \emph{maximal} systems with a possibly large $N_{0}$. This idea plays a crucial role in our proof of Theorem~\ref{thm:alg-cond-intro}, which in turn is key to the wide applicability of our method. Indeed, while the question of which $\calP^{\ast}$ admits completely integrable graded augmented systems seems difficult to answer in general, our result implies that, at least for homogeneous constant-coefficient $\calP^{\ast}$ (which includes all examples considered in \cite{Resh}), the existence of a maximal graded augmented system is \emph{equivalent} to the algebraic condition \ref{hyp:fc}.
\end{remark}

\subsection{Applications I: examples of $\calP$} \label{subsec:appl-ex}
With the help of \ref{hyp:fc}, one may easily check that our method applies to a variety of operators that arise naturally from physics and geometry, as well as their lower-order variable-coefficient perturbations (which include their covariant versions on Riemannian manifolds; see Appendix~\ref{sec:appl}). For the next theorem, we adopt the following conventions: on an open subset $U$ of $\bbR^{d}$, we use $\bfg_{jk}$ to refer to a metric (i.e., positive definite symmetric $2$-tensor); $(\bfg^{-1})^{jk}$ its inverse; $\varphi$ a real-valued function; $\bfu^{j}$ a vector field; $\bfomg_{j}$ a one-form, $\bfh^{jk}$ and $\bfpi^{jk}$ symmetric $2$-tensors; and $\wh{\bfh}^{jk}$ and $\wh{\bfpi}^{jk}$ symmetric $2$-tensors that are trace-free (with respect to $\bfg$).
\begin{theorem} \label{thm:appl}
Matrix-valued differential operators $\calP$ on $U$ with smooth coefficients that have the following principal parts satisfy \ref{hyp:fc}:
\begin{enumerate}
\item {\bf Divergence \& (adjoint) gradient operator, $d \geq 1$.}
\begin{equation*}
\calP_{\prin} \bfu = \rd_{j} \bfu^{j} \quad \hbox{ or equivalently } \quad \calP_{\prin}^{\ast} \varphi = - \rd_{j} \varphi,
\end{equation*}
in which case $\tensor{(p^{\ast})}{_{j}} = - i \xi_{j}$.
\item {\bf Double divergence (or linearized scalar curvature) \& (adjoint) Hessian operator, $d \geq 1$.}
\begin{equation*}
\calP_{\prin} u = \rd_{j} \rd_{k} \bfh^{j k} \quad \hbox{ or equivalently } \quad \calP_{\prin}^{\ast} \varphi = \rd_{j} \rd_{k} \varphi,
\end{equation*}
in which case $\tensor{(p^{\ast})}{_{jk}} = - \xi_{j} \xi_{k}$.
\item {\bf Trace-free double divergence \& (adjoint) trace-free Hessian operator, $d \geq 2$.}
\begin{equation*}
\calP_{\prin} \wh{\bfh} = \rd_{j} \rd_{k} \wh{\bfh}^{j k} \quad \hbox{ or equivalently } \quad  \calP_{\prin}^{\ast} \varphi = \rd_{j} \rd_{k} \varphi - \frac{1}{d} \bfg_{jk}(x) (\bfg^{-1})^{\ell m}(x) \rd_{\ell} \rd_{m} \varphi,
\end{equation*}
in which case $\tensor{(p^{\ast})}{_{j k}} = - \xi_{j} \xi_{k} + \frac{1}{d} \bfg_{jk}(x) (\bfg^{-1})^{\ell m}(x) \xi_{\ell} \xi_{m}$.
\item {\bf Symmetric divergence \& (adjoint) Killing operator, $d \geq 1$.}
\begin{equation*}
(\calP_{\prin} \bfh)_{j} = \rd_{k} \bfh^{j k} \quad \hbox{ or equivalently } \quad (\calP_{\prin}^{\ast} \bfomg)_{jk} = - \frac{1}{2} (\rd_{j} \bfomg_{k} + \rd_{k} \bfomg_{j}),
\end{equation*}
in which case $\tensor{(p^{\ast})}{_{jk}^{\ell}} = - \frac{i}{2} (\xi_{j} \dlt_{k}^{\ell} + \xi_{k} \dlt_{j}^{\ell})$.
\item {\bf Symmetric trace-free divergence \& (adjoint) conformal Killing operator, $d \geq 3$.}
\begin{equation*}
(\calP_{\prin} \wh{\bfh})_{j} = \rd_{k} \wh{\bfh}^{j k} \quad \hbox{ or equivalently } \quad (\calP_{\prin}^{\ast} \bfomg)_{jk} = - \frac{1}{2} (\rd_{j} \bfomg_{k} + \rd_{k} \bfomg_{j}) + \frac{1}{d} \bfg_{jk} (\bfg^{-1})^{\ell m} \partial_\ell \omega_m,
\end{equation*}
in which case $\tensor{(p^{\ast})}{_{j k}^{\ell}} = - \frac{i}{2} (\xi_{j} \dlt_{k}^{\ell} + \xi_{k} \dlt_{j}^{\ell}) + \frac{i}{d} \bfg_{jk}(x) (\bfg^{-1})^{\ell m}(x) \xi_{m}$.

\item {\bf Linearized Einstein vacuum constraint \& (adjoint) Killing Initial Data operator, $d \geq 1$.}
\begin{equation*}
\calP_{\prin} (\bfh, \bfpi) = (\rd_{j} \rd_{k} \bfh^{j k}, \rd_{k'} \bfpi^{j' k'}) \quad \hbox{ or equivalently } \quad \calP_{\prin}^{\ast} (\varphi, \bfomg) = \left( - \rd_{j} \rd_{k} \varphi, - \frac{1}{2} (\rd_{j'} \bfomg_{k'} + \rd_{k'} \bfomg_{j'}) \right)
\end{equation*}
in which case $p^{\ast}$ is $2 \times 2$-block-diagonal with the principal symbols of the Hessian and Killing operators as blocks.

\item {\bf Linearized Einstein vacuum constraint operator with constant mean curvature, $d \geq 3$.}
\begin{equation*}
\calP_{\prin} (\bfh, \wh{\bfpi}) = (\rd_{j} \rd_{k} \bfh^{j k}, \rd_{k'} \wh{\bfpi}^{j' k'}),
\end{equation*}
in which case $p^{\ast}$ is $2 \times 2$-block-diagonal with the principal symbols of the Hessian and conformal Killing operators as blocks.
\end{enumerate}
\end{theorem}

The proof of this theorem, which we immediately provide, consists of short algebraic computations.
\begin{proof}
For (1), it is clear that, for each $\xi \in \bbC^{d} \setminus \set{0}$, $(p^{\ast})_{j}(\xi) \varphi = \xi_{j} \varphi = 0$ implies $\varphi = 0$; similarly for (2).

To prove (3), fix $x \in U$ and assume (by passing to the normal coordinates) that $\bfg^{\ell m}(x) \bfg_{jk}(x) = \dlt^{\ell m} \dlt_{jk}$. For each $\xi \in \bbC^{d} \setminus \set{0}$, we need to show that if $\varphi \in \bbR$ satisfies
\begin{equation*}
	(p^{\ast})_{jk}(x, \xi) \varphi = \left( - \xi_{j} \xi_{k} + \frac{1}{d} \xi_{\ell} \xi_{m} \dlt^{\ell m} \dlt_{jk} \right) \varphi = 0 \quad \hbox{ for all } j, k \in \set{1, \ldots, d},
\end{equation*}
then $\varphi = 0$. In view of (2), it suffices to show that $\xi_{\ell} \xi_{m} \dlt^{\ell m} \varphi = \sum_{\ell} \xi_{\ell}^{2} \varphi = 0$. For each $j$, we have
\begin{align*}
	\xi_{j} \sum_{\ell} \xi_{\ell}^{2} \varphi
	&= \xi_{j}^{3} \varphi + \sum_{\ell : \ell \neq j} \xi_{\ell} \xi_{\ell} \xi_{j} \varphi
	= \xi_{j} \left(- (p^{\ast})_{jj}(x, \xi) \varphi + \frac{1}{d} \sum_{\ell} \xi_{\ell}^{2} \varphi \right) + \sum_{\ell : \ell \neq j} \xi_{\ell} \left( - (p^{\ast})_{\ell j}(x, \xi) \varphi\right)  \\
	&= \frac{1}{d} \xi_{j} \sum_{\ell} \xi_{\ell}^{2} \varphi,
\end{align*}
which implies $\sum_{\ell} \xi_{\ell}^{2} \varphi = 0$ as long as $d > 1$.
.

For (4), we need to show that if $\omg \in T^{\ast}_{x} \calM$ (identified with $\bbR^{d}$) and $\xi \in \bbC^{d} \setminus \set{0}$ satisfies
\begin{equation*}
	\tensor{(p^{\ast})}{_{jk}^{\ell}}(\xi) \omg_{\ell} = - \frac{i}{2} (\xi_{j} \omg_{k} + \xi_{k} \omg_{j}) = 0 \quad \hbox{ for all } j, k \in \set{1, \ldots, d},
\end{equation*}
then $\omg_{\ell} = 0$ for all $\ell \in \set{1, \ldots, d}$. Note that the previous condition implies $\xi_{j} \omg_{k} = - \xi_{k} \omg_{j}$, and in particular, $\xi_{j} \omg_{j} = 0$ for any $j, k \in \set{1, \ldots, d}$. Thus, for any $j, k \in \set{1, \ldots, d}$, we have
\begin{equation*}
	\xi_{j}^{2} \omg_{k} = - \xi_{j} \xi_{k} \omg_{j} = \xi_{k} \xi_{j} \omg_{j} = 0,
\end{equation*}
which implies $\omg_{k} = 0$ as desired.

The proof of (5) requires a bit more computation compared to the previous cases. Fix $x \in U$ and assume (by passing to the normal coordinates) that $\bfg^{\ell m}(x) \bfg_{jk}(x) = \dlt^{\ell m} \dlt_{jk}$. Fix $\xi \in \bbC^{d} \setminus \set{0}$, and assume that $\omg \in T^{\ast}_{x} \calM$ (identified with an element in $\bbR^{d}$) satisfies
\begin{equation*}
\tensor{(p^{\ast})}{_{jk}^{\ell}}(x, \xi) \omg_{\ell} = - \frac{i}{2} (\xi_{j} \omg_{k} + \xi_{k} \omg_{j}) + \frac{i}{d} \dlt^{\ell m} \xi_{m} \omg_{\ell} \dlt_{j k} = 0 \quad \hbox{ for all } j, k \in \set{1, \ldots, d}.
\end{equation*}
We need to show that $\omg_{\ell} = 0$ for all $\ell$. In view of (4), it suffices to show that $w \ceq \frac{1}{d} \dlt^{\ell m} \xi_{m} \omg_{\ell} = \frac{1}{d} \sum_{\ell} \xi_{\ell} \omg_{\ell} = 0$. The above condition implies
\begin{equation*}
	\xi_{j} \omg_{k} = - \xi_{k} \omg_{j} + 2 w \dlt_{j k} \quad \hbox{ for any } j, k \in \set{1, \ldots, d}.
\end{equation*}
and in particular, $\xi_{j} \omg_{j} = w$ for any $j \in \set{1, \ldots, d}$. We first compute $\xi_{j} w$ (for any $J \in \set{1, \ldots, d}$):
\begin{align*}
	\xi_{j} w
	=  \frac{1}{d} \xi_{j} \xi_{j} \omg_{j}
	+ \frac{1}{d} \sum_{\ell : \ell \neq j} \xi_{\ell} (\xi_{j} \omg_{\ell})
	=  \frac{1}{d} \xi_{j} w - \frac{1}{d} \sum_{\ell : \ell \neq j} \xi_{\ell}^{2} \omg_{j},
\end{align*}
which shows that $\xi_{j} w = - \frac{1}{d-1} \sum_{\ell : \ell \neq j} \xi_{\ell}^{2} \omg_{j}$. Multiplying by $\xi_{j}$, we arrive at
\begin{equation*}
	\xi_{j}^{2} w = - \frac{1}{d-1} \sum_{\ell : \ell \neq j} \xi_{\ell}^{2} w.
\end{equation*}
On the one hand, by summing up in $j$, we conclude that $\sum_{\ell} \xi_{\ell}^{2} w = 0$. On the other hand, using $\sum_{\ell : \ell \neq j} \xi_{\ell}^{2} w = - \xi_{j}^{2} w$, we arrive at $\xi_{j}^{2} w = \frac{1}{d-1} \xi_{j}^{2} w$, which implies $\xi_{j}^{2} w = 0$ for any $j \in \set{1, \ldots, d}$, since $d > 2$. Thus $w = 0$.

Finally, (6) follows from combining (2) and (4); similarly, (7) follows from (2) and (5).
\end{proof}

By Theorem~\ref{thm:alg-cond-intro}, a maximal graded augmented system exists for each example, \ref{hyp:crc} thus holds, and Theorems~\ref{thm:main-summary-conic} and \ref{thm:main-summary-bogovskii} are applicable. Alternatively, in Appendix~\ref{sec:appl}, each example is revisited with a more geometric viewpoint, and we provide a special graded augmented system that is (i)~stable under the addition of lower order terms, and (ii)~completely integrable on backgrounds with constant sectional curvature.  Using this augmented system, we also compute explicitly the Bogovskii and conic integral kernels on the flat space, which also recovers the known results from \cite{IseOh, MaoTao, MOT1, Resh}. 

While the results in Appendix~\ref{sec:appl} are of independent interest -- which is why we have worked them out -- we point out the contrast between the simplicity of the proof of Theorem~\ref{thm:appl} versus the case-by-case ingenuity required in the direct derivation of the special graded augmented systems in Appendix~\ref{sec:appl}.

\subsection{Applications II: sharp solvability results (from \cite{IMOT2})} \label{subsec:appl-lin}
Next, we discuss an application of our method to the study of differential operators satisfying \ref{hyp:fc} under optimal (up to endpoints) assumptions on the coefficients. Our overall approach is to:
\begin{enumerate}
\item first use our method to \emph{design} appropriate integral formulas for $\calP$ that have constant coefficients, are homogeneous (i.e., $\calP = \calP_{\prin}$ in the sense of Definition~\ref{def:prin-symb}); and
\item handle the general case via \emph{local perturbation} (or freezing-coefficients) techniques.
\end{enumerate}
We restrict our attention to $L^{2}$-based Sobolev spaces, but extensions to other function spaces that behave well under singular integral operators (e.g., H\"older or $L^{p}$-based Sobolev spaces with $1 < p < + \infty$) should be possible. The main result in the bounded domain case in \cite{IMOT2} is as follows (we refer again to Section~\ref{subsec:ftn-sp} for our notation and conventions concerning function spaces, and to \eqref{eq:sob-td-coker} and \eqref{eq:sob-coker} for the precise definitions of $\ker_{\td{H}^{-s}(U)} \calP^{\ast}$ and $\ker_{H^{-s}(U)} \calP^{\ast}$, respectively):

\begin{theorem} [Solvability on a bounded Lipschitz domain]\label{thm:low-reg-domain}
Fix an exponent $s_{\calP, 0}$ such that $s_{\calP, 0} > \frac{d}{2}$ and $s_{\calP, 0} \geq \frac{1}{2} \max_{K} m_{K}$. Let $U$ be a bounded open subset of $\bbR^{d}$ with a Lipschitz boundary, i.e., $\rd U$ is compact and can be covered by finite balls, in each of which $\rd U$ can be written as the graph of a Lipschitz function after suitably relabeling and rotating the coordinate axes. Let $\calP$ be a differential operator on $U$ satisfying \ref{hyp:fc} and, for some $s_{\calP} \geq s_{\calP, 0}$, assume that
\begin{align*}
\tensor{c[\calP]}{^{(\alp, J)}_{K}}(x) &\in H^{s_{\calP}-(m_{K} - \abs{\alp})}(U) & & \hbox{ for all } \abs{\alp} \leq m_{K}.
\end{align*}
Then the following statements hold:
\begin{enumerate}
\item (Cokernel in $H^{-s}(U)$) For every $s$ satisfying $-s_{\calP} \leq s \leq s_{\calP} - \max_{K} m_{K}$, we
have the \emph{finite-dimensional property}
\begin{equation*}
\dim \ker_{H^{-s}(U)} \calP^{\ast} < + \infty,
\end{equation*}
and the \emph{invariance property}
\begin{equation*}
\ker_{H^{-s}(U)} \calP^{\ast} = \ker_{H^{s_{\calP}}(U)} \calP^{\ast},
\end{equation*}
Moreover, for any open subset $V \subseteq U$, the restriction of $\ker_{H^{s_{\calP}}(U)} \calP^{\ast}$ to $V$, i.e.,
\begin{equation*}
	\ker_{H^{s_{\calP}}(U)} \calP^{\ast} |_{V} \ceq \set{\bfZ |_{V} : \bfZ \in \ker_{H^{s_{\calP}}(U)} \calP^{\ast}}
\end{equation*}
has the same dimension as $\ker_{H^{s_{\calP}}(U)} \calP^{\ast}$.

\item (Solution operators \& representation formulas associated with cokernel in $H^{-s}(U)$) Given $s_{1} \geq - s_{\calP}$, consider a family $w_{\bfA}(x) \in \td{H}^{s_{1}}(U)$ $(\bfA \in \set{1, \ldots, \dim \ker_{H^{s_{\calP}}(U)} \calP^{\ast}})$ satisfying $\brk{w_{\bfA}, \bfZ^{\bfA'}} = \dlt_{\bfA}^{\bfA'}$ for some basis $\set{\bfZ^{\bfA'}}$ of $\ker_{H^{s_{\calP}}(U)} \calP^{\ast}$. Then there exists an operator $\td{\calS} : C^{\infty}_{c}(U; \bbC^{r_{0}}) \to \calD'(U; \bbC^{s_{0}})$ independent of $s_{\calP}$ such that, for $s$ satisfying
\begin{equation} \label{eq:low-reg-bog-s}
	-s_{\calP} \leq s \leq s_{\calP} - \max_{K} m_{K}, \quad s \leq s_{1},
\end{equation}
we have $\td{\calS} : \td{H}^{s}(U) \to \td{H}^{s+m_{1}} (U)\times \cdots \times \td{H}^{s+m_{s_{0}}}(U)$ and
\begin{equation*}
	\calP \td{\calS} f = f - \sum_{\bfA \in \set{1, \ldots, \dim \ker \calP^{\ast}}} w_{\bfA} \brk{\bfZ^{\bfA}, f} \quad \hbox{ for all } f \in \td{H}^{s}(U),
\end{equation*}
and by duality,
\begin{equation*}
	\varphi = \td{\calS}^{\ast} \calP^{\ast} \varphi + \sum_{\bfA \in \set{1, \ldots, \dim \ker \calP^{\ast}}} \bfZ^{\bfA} \brk{w_{\bfA}, \varphi} \quad \hbox{ for all } \varphi \in H^{-s}(U).
\end{equation*}
In particular, the following \emph{Poincar\'e-type inequality} holds:
\begin{equation*}
	\nrm{\varphi - \sum_{\bfA} \bfZ^{\bfA} \brk{w_{\bfA}, \varphi}}_{H^{-s}(U)} \aleq \nrm{\calP^{\ast} \varphi}_{H^{-s-m_{1}}(U) \times \cdots \times H^{-s-m_{s_{0}}}(U)} \quad \hbox{ for all } \varphi \in H^{-s}(U).
\end{equation*}

\item (Cokernel in $\td{H}^{-s}(U)$) For every $s$ satisfying $-s_{\calP} \leq s \leq s_{\calP} - \max_{K} m_{K}$, we have
\begin{equation*}
\ker_{\td{H}^{-s}(U)} \calP^{\ast} = \set{0}.
\end{equation*}

\item (Solution operators \& representation formulas associated with cokernel in $\td{H}^{-s}(U)$) There exists an operator\footnote{Unlike in Theorem~\ref{thm:main-summary-conic}, we are able to define $\calS$ for $f \in H^{s}(U)$ (not $\td{H}^{s}(U)$) via the existence of a Sobolev extension operator under the Lipschitz boundary regularity assumption. This feature is useful for, say, setting up a Picard iteration scheme to solve nonlinear problems; see \cite{IMOT2, MOT2}.} $\calS : C^{\infty}(\br{U}; \bbC^{r_{0}}) \to \calD'(U; \bbC^{s_{0}})$ independent of $s_{\calP}$ such that, for $s$ satisfying
\begin{equation} \label{eq:low-reg-conic-s}
	-s_{\calP} \leq s \leq s_{\calP} - \max_{K} m_{K},
\end{equation}
we have $\calS : H^{s}(U) \to H^{s+m_{1}} (U)\times \cdots \times H^{s+m_{s_{0}}}(U)$ and
\begin{equation*}
	\calP \calS f = f \quad \hbox{ for all } f \in \td{H}^{s}(U),
\end{equation*}
and by duality,
\begin{equation*}
	\varphi = \calS^{\ast} \calP^{\ast} \varphi \quad \hbox{ for all } \varphi \in \td{H}^{-s}(U).
\end{equation*}
In particular, the following \emph{Friedrich-type inequality} holds:
\begin{equation*}
	\nrm{\varphi}_{\td{H}^{-s}(U)} \aleq \nrm{\calP^{\ast} \varphi}_{\td{H}^{-s-m_{1}}(U) \times \cdots \times \td{H}^{-s-m_{s_{0}}}(U)} \quad \hbox{ for all } \varphi \in \td{H}^{-s}(U).
\end{equation*}

\end{enumerate}
\end{theorem}

\begin{remark}[Comparison with elliptic operators] \label{rem:lipschitz-bdry}
Theorem~\ref{thm:low-reg-domain} -- especially Parts~(1) and (2) -- is analogous to the standard solvability result for the Dirichlet boundary value problem (BVP) for an elliptic equation $\td{\calP} \td{u} = f$ on a bounded domain (see, e.g., \cite[Chapter~6]{Evans}). However, some interesting differences stand out:
\begin{enumerate}
\item Given $f \in C^{\infty}_{c}(U)$, the solution $\td{\calS} f$ in Part~(2) vanishes to all possible orders at the boundary, while the solution to the elliptic BVP necessarily only vanishes to order one unless it is trivial.
\item As opposed to $\td{\calS} : \td{H}^{s}(U; \bbC^{r_{0}}) \to \td{H}^{s+m_{1}} (U)\times \cdots \times \td{H}^{s+m_{s_{0}}}(U)$ in Part~(2), it is well-known that boundary elliptic regularity, or more precisely an estimate of the form $\nrm{\td{u}}_{W^{s+m, p}(U)} \aleq \nrm{f}_{W^{s, p}(U)}$ (where $m$ is the order of $\td{\calP}$) fails in general on Lipschitz domains, the simplest case being $\td{\calP} = -\lap$ and $U \subseteq \bbR^{2}$ has a corner; see \cite[Chapters~4 and 5]{Gri}.
\end{enumerate}
Another basic but important distinction is that operators considered in Theorem~\ref{thm:low-reg-domain} are \emph{not}  Fredholm in most cases of interest, as the kernel of $\calP$ may be infinite dimensional (take, for instance, the divergence operator). Our approach therefore does \emph{not} rely on the Fredholm alternative theorem as in the usual proof of elliptic solvability results (see, e.g., \cite[Chapter~6]{Evans}).
\end{remark}

\begin{remark} [Applications to Poincar\'e- and Friedrich-type estimates] \label{rem:appl-rigidity}
Simply combining Theorem~\ref{thm:appl} with the Poincar\'e- and Friedrich-type inequalities in Theorem~\ref{thm:low-reg-domain}.(2) and (4), respectively, recovers and generalizes (to the rough-coefficients, Lipschitz domain setting) various standard inequalities in analysis, such as the standard Poincar\'e and Friedrich inequalities (for $\calP^{\ast}$ equal to the gradient), Korn's first and second inequalities (for $\calP^{\ast}$ equal to the symmetric gradient) \cite{Kor, KonOle}, Korn's inequalities for the trace-free symmetric gradient \cite{Dai}, etc.
\end{remark}

\begin{remark} \label{rem:optimal-bdry}
A natural question one might ask is what happens on domains with boundary less regular than Lipschitz. In the case of the divergence operator, Acosta--Duran--Muschietti \cite{AcoDurMus} showed that the existence of a Bogovskii-type operator that is bounded from $L^{p}$ to $W^{1, p}$ for all $1 < p < \infty$ is equivalent to $U$ being a John domain. For more discussion on this literature, see \S\ref{subsubsec:discussion:div}. It would be interesting to generalize this to general differential operators satisfying \ref{hyp:fc}.
\end{remark}

In fact, in \cite{IMOT2}, we also establish an analog of Theorem~\ref{thm:low-reg-domain} in the case of unbounded Lipschitz domains using weighted Sobolev spaces, under the additional assumptions that (i)~$\calP$ asymptotes (in a suitable way) to a homogeneous constant-coefficient operator at infinity, and (ii)~$U$ is non-degenerate (in a suitable way) towards infinity. In this case, the cokernel varies according to the weight in the function space. We refer the interested reader to that paper for more details.

\subsection{Related works} \label{subsec:discussion}

\subsubsection{Bogovskii-type operators in fluid dynamics}
We use the term Bogovskii-type operator to refer to the type of operator that yields compactly supported solutions to the underdetermined PDE in question. The original Bogovskii operator \eqref{eq:bog-0} has become a basic tool in fluid dynamics \cite{galdi2011introduction}, where the divergence free constraint plays a fundamental role in the study of incompressible fluids.  The way the operator is often used is to construct divergence free vector fields by starting with a vector field that $v$ that is approximately divergence-free and then correcting $v$ to obtain a truly divergence free vector field $V = v + u$ by solving $\div u = - \div v$.  Applying the Bogovskii operator to solve this equation does not disturb the compact support property.  A recent and profound example where the Bogovskii operator is used in this way is the recent breakthrough paper \cite{albritton2022non}, which constructs nonunique solutions of the Leray-Hopf class to the forced Navier Stokes equations on $\bbR^3$.  To give just a few other examples, see \cite{borchers1990algebraic,borchers1987semigroup,galdi1991existence,giga1988stokes,iwashita1989Lq,takahashi1992regularity,tani1991global}.

In \cite{IseOh}, the authors construct a Bogovskii-type operator for the equation $\div R = F$, where $R$ is a symmetric $2$ tensor field and $F$ is a compactly supported vector field that is orthogonal to the Killing vector fields of Euclidean space, which are spanned by translations and rotations.  This operator is key to the construction of nonunique and energy non-conserving continuous weak solutions to $3D$ incompressible Euler defined on Euclidean space.   The same operator turned out surprisingly to be crucial for the construction in \cite{isett2024endpoint} of weak (periodic) solutions to $3D$ Euler of class $\bigcap_{\epsilon > 0} C^{1/3-\epsilon}$ that fail to conserve energy.  This latter result is the best result towards the endpoint case of the famous Onsager conjecture.  We expect that the general class of Bogovskii-type operators obtained in this paper will continue to be useful for related and future applications.

\subsubsection{Underdetermined problems in general relativity}
Due to the divergence structure of the Einstein constraint equation (on a spacelike Cauchy hypersurface), divergence equations have applications in general relativity, especially in initial data construction. Corvino \cite{Cor} in his pioneering work proved rigidity estimates for the dual linearized scalar curvature operator on a compact region via variational methods and applied them to prove a gluing result for the prescribed scalar curvature problem, which corresponds to time-symmetric initial data sets. Corvino--Schoen \cite{CorSch} generalized this gluing technique to the full Einstein constraint equation, and constructed a large class of initial data sets which coincide with Kerr initial data outside a large ball.  Chru\'{s}ciel--Delay\cite{ChrDel} extended the aforementioned gluing results by establishing mapping properties of linearized constraint operator on a large class of weighted Sobolev spaces, and Delay \cite{Delay} used a similar method to study underdetermined elliptic operators and proved various gluing results for those operators. Carlotto--Schoen \cite{CarSch} were the first to construct initial data with conic support by developing a gluing scheme using variational techniques. Hintz \cite{hintz2023gluingI, hintz2024gluingII, hintz2024gluingIII}, with a geometric microlocal approach, generalized Corvino--Schoen-type gluing by considering generic initial data sets outside of a compact region.

In \cite{MaoTao} and \cite{MOT1}, the authors constructed conic- and Bogovskii-type solution operators, respectively, for the linearized Einstein vacuum constraint operator around flat initial data. Moreover, these operators were applied to simplify the proofs of existing gluing results and obtain new results, such as the existence of nontrivial initial data sets localized to degenerate cones \cite{MaoTao}, and obstruction-free gluing with a sharp positivity condition \cite{MOT1}, improving upon the seminal work of Czimek--Rodnianski \cite{CziRod}. The ideas introduced in \cite{MaoTao, MOT1} served as a precursor for the present paper.

We also mention a recent work of Chru\'{s}ciel--Cogo--N\"utzi \cite{CCN}, in which a Bogovskii-type solution operator for the linearized constant scalar curvature operator was constructed for the hyperbolic metric near infinity.

\subsubsection{Overdetermined problems in geometry} \label{subsubsec:discussion:dual}
The study of rigidity problems in geometry and elasticity has a long and rich history. By linearization, such problems often lead to overdetermined differential operators $\calP^{\ast}$, many of which turn out to satisfy \ref{hyp:fc}. We refer the interested reader to \cite{ReshBook, DeLMul, ConDolMul} (in geometry), \cite{FriJamMul} (in elasticity) and references therein. Among these works, Reshetnyak's approach to the derivation of integral representation formulas has been a major influence on this paper, as highlighted above (see, e.g., Remark~\ref{rem:resh}).

\subsubsection{Divergence operator on rough domains} \label{subsubsec:discussion:div}
There is a substantial body of work on divergence operators and Poincar\'e-type inequalities on rough domains. In the aforementioned work \cite{AcoDurMus}, Acosta--Duran--Muschietti constructed an explicit solution operator for the divergence equation on John domains using singular integral techniques. Duran--Garcia \cite{DurGar1}\cite{DurGar2} proved existence of bounded right inverse of the divergence operators on planar simply-connected H\"older-$\alp$ domains and domains with an external cusp, using singular integrals and $A_p$-weights. Duran--Muschietti--Russ--Tchamitchian \cite{DMRT} gave a general sufficient condition on invertibility of divergence operators on weighted $L^p$ spaces on an arbitrary domain via Calderon-Zygmund type arguments. An interesting question would be to generalize these investigations to the setting of differential operators satisfying \ref{hyp:fc}.

\subsubsection{Other related problems} \label{subsubsec:discussion:others}
There is a strong resemblance between our construction of the rough integral kernel (see Section~\ref{sec:crc-summary}) and a well-known proof of Poincar\'e's lemma on star-shaped domains via the construction of a chain homotopy to the De Rham chain complex over the base point (see, for instance, \cite[Theorem~4.11]{Spi}). Indeed, Bogovskii-type chain homotopy for the De Rham complexes has been found by Takahashi \cite{Tak}; since the last map in the De Rham complex is the divergence operator, this chain homotopy generalizes the original Bogovskii operator.  We also note the recent work of N\"utzi \cite{Nutzi} on the construction of a Bogovskii-type chain homotopy for an elliptic complex whose last homomorphism corresponds to the symmetric trace-free divergence operator. In this case, the chain homotopy specializes to a Bogovskii-type solution operator for the symmetric trace-free divergence operator.

An approach akin to ours for De Rham complex may be of utility in the study of (nonlinear) pullback equation for differential forms, which is the generalization of the well-known theorems of Darboux (see, for instance, \cite[Chapter~8]{daSilva}) and Dacorogna--Moser \cite{DacMos} on finding a diffeomorphism that pulls back a given differential form to another given differential form (symplectic in the case of Darboux, and volume in the case of Dacorogna--Moser). We refer to the monograph of Csato--Dacorogna--Kneuss \cite{CsaDacKne} for more on this topic.

\subsection{Structure of the paper}
The paper is structured as follows. 
\begin{itemize}
\item In Section~\ref{sec:prelim}, we collect the notation, conventions and preliminary facts from analysis and geometry used in this paper. 
\item In Section~\ref{sec:crc-summary}, we summarize the ideas behind our construction of solution operators and representation formulas via \ref{hyp:crc}, and demonstrate our approach in the simple special case of the divergence operator with a lower order term, i.e., $\calP u = \rd_{j} \bfu^{j} + B_{j} \bfu^{j}$ (which already leads to results that are, to the best of our knowledge, new). This discussion will serve as a motivation for the remainder of the paper concerning the general case. 
\item In Section~\ref{sec:sio}, we prove a proposition that will allow us to show that the integral kernel produced by our method defines a singular integral operator with good boundedness properties. 
\item In Section~\ref{sec:crc}, we give a precise formulation of \ref{hyp:crc} and describe how it leads to our construction of solution operators and representation formulas with prescribed support properties, thereby proving (precise versions of) Theorems~\ref{thm:main-summary-conic} and \ref{thm:main-summary-bogovskii}. \item In Section~\ref{sec:ode}, we study graded augmented systems and establish Propositions~\ref{prop:aug2crc} and \ref{prop:zero-curv}. 
\item In Section~\ref{sec:fc2crc}, we prove Theorem~\ref{thm:alg-cond-intro}. 
\item Finally, in Appendix~\ref{sec:appl}, we write down special graded augmented systems for operators (2)--(7) in Theorem~\ref{thm:appl} in the geometric context, which turn out to be completely integrable 
on constant sectional curvature backgrounds. On such backgrounds, we explicitly compute the fundamental matrix $\tensor{\bfPi}{_{\bfA}^{\bfA'}}(y, y_{1}, s)$ on geodesic segments. Specializing further to the flat background, we also explicitly compute the Bogovskii and conic integral kernels.
\end{itemize}

\medskip \noindent{\bf Acknowledgements.}
P. Isett was supported by a Sloan Research Fellowship and the NSF grant DMS 2346799.  S.-J.~Oh was partially supported by a Sloan Research Fellowship, an NSF CAREER Grant under NSF-DMS-1945615, and an NSF Grant under NSF-DMS-2452760. Z. Tao would like to thank Maciej Zworski and Semyon Dyatlov for helpful discussions on the finite-dimensional cokernel condition. Z. Tao was partially supported by the Simons Targeted Grant Award No. 896630.

\section{Preliminaries} \label{sec:prelim}
\subsection{Notation and conventions for the operator $\calP$} \label{subsec:prelim-P}
Some important objects used throughout this paper are as follows. Let $U$ be an open subset of $\bbR^{d}$, and let $W$ be an open subset of $\bbR^{d} \times \bbR^{d}$ whose $\bbR^{d}_{y}$-projection contains $U$, i.e.,
\begin{equation*}
	U \subseteq \set{y \in \bbR^{d} : \exists y_{1} \hbox{ such that } (y, y_{1}) \in W}.
\end{equation*}
We denote by $\bfx = \bfx(y, y_{1}, s)$ a family of curves defined for $(y, y_{1}) \in W$ and $s \in [0, 1]$ with $\rd_{s} \bfx(y, y_{1}, s) \neq 0$ for every $s \in [0, 1]$, such that $\bfx(y, y_{1}, 0) = y$ and $\bfx(y, y_{1}, 1) = y_{1}$.

Let $\calP$ be an $r_{0} \times s_{0}$ (possibly complex-)matrix-valued differential operator on $U$. Unless otherwise specified, we follow the following conventions:
\begin{itemize}
\item $u = (u^{K})_{K=1, \ldots, s_{0}}$ is an $\bbC^{s_{0}}$-valued function, $f = (f^{J})_{J=1, \ldots, r_{0}}$ is an $\bbC^{r_{0}}$-valued function, $\varphi = (\varphi_{J})_{J=1, \ldots, r_{0}}$ is an $\bbC^{r_{0}}$-valued function, and $\psi = (\psi_{K})_{K=1, \ldots, s_{0}}$ is an $\bbC^{s_{0}}$-valued function.
\item We employ the natural $L^{2}$-inner products on $(U, \ud x)$, which are
\begin{equation*}
    \brk{\psi, u} \ceq \Re \int_{U} \br{\psi}_{K} u^{K} \, \ud x, \quad
    \brk{\vphi, f} \ceq \Re \int_{U}  \br{\vphi}_{J} f^{J}  \, \ud x.
\end{equation*}
\item $\calP^{\ast}$ denotes the adjoint of $\calP$ with respect to the above $L^{2}$-inner products on $(U, \ud x)$.
\item Finally, we will often omit writing out the identity matrix in equations, e.g., $\dlt_{0}(x-y) = \dlt_{0}(x - y) I_{r_{0} \times r_{0}}$ in \eqref{eq:trun-ker-y1-gen}.
\end{itemize}

\subsection{Notation and conventions for geometry and analysis} \label{subsec:prelim-notation}
Throughout the paper, we adopt the following conventions.
\begin{itemize}
\item We write $A \aleq B$ if there exists a positive constant $C > 0$ (that may differ from expression to expression) such that $A \leq C B$, and $A \aeq B$ if $A \aleq B$ and $B \aleq A$. We specify the dependencies of $C$ by a subscript, e.g., $A \aleq_{d} B$.

\item We write $\brk{x} = (1 + \abs{x}^{2})^{\frac{1}{2}}$.
\item We adopt the Einstein summation convention, i.e., repeated upper and indices are summed.

\item $\alp, \bt, \gmm, \ldots$ usually denote multi-indices, i.e., elements of $\bbZ_{\geq 0}^{d}$ ($\bbZ_{\geq 0}$ is the set of nonnegative integers). As usual, $\rd^{\alp} = \rd_{1}^{\alp_{1}} \cdots \rd_{d}^{\alp_{d}}$, and $\abs{\alp} = \alp_{1} + \cdots + \alp_{d}$.

\item We write $X \Subset U$ for a compact subset $X$ of some topological space $U$.

\item We also use the following notation for geometric objects:
\begin{itemize}
\item $B_{R}(x)$: Ball of radius $R$ centered at $x$ in $\bbR^{d}$.
\item $C_{\Omg}$: Given $\Omg \subseteq \bbS^{d-1}$, the cone over $\Omg$ is defined as $C_{\Omg} = \set{x \in \bbR^{d} : \frac{x}{\abs{x}} \in \Omg}$.
\item $\ud S(\omg)$: the $(d-1)$-dimensional surface measure on $\bbS^{d-1} \subseteq \bbR^{d}$ (set of unit directions)
\end{itemize}


\item Fix $m_{< 1} \in C^{\infty}_{c}(\bbR^{d})$ that is nonnegative, equals $1$ on $B_{1}(0)$, and vanishes outside $B_{2}(0)$. Given $N \in 2^{\bbZ}$, define $m_{< N}(\xi) \ceq m_{<1}(N^{-1} \xi)$ and $m_{N}(\xi) \ceq m_{< 2 N}(\xi) - m_{< N}(\xi)$. These functions form a smooth partition of unity subordinate to dyadic annuli in $\bbR^{d}$, i.e., $m_{<1} + \sum_{N \in 2^{\bbZ_{\geq 0}}} m_{N} = 1$ on $\bbR^{d}$ and $\supp m_{N} \subseteq \set{N < \abs{\xi} < 4 N}$.

\item We use the following convention for the Fourier and inverse Fourier transforms on $\bbR^{d}$:
\begin{equation*}
	\calF[f](x) = \int f(y) e^{-i \xi \cdot y} \, \ud y, \quad
	\calF^{-1}[F](x) = \int F(\xi) e^{i \xi \cdot x} \, \frac{\ud \xi}{(2 \pi)^{d}}.
\end{equation*}
Given $m : \bbR^{d} \to \bbC$, the \emph{Fourier multiplier operator with symbol $m$} is defined as
\begin{equation*}
	m(D) f = \calF^{-1}[m(\xi) \calF[f](\xi)]
\end{equation*}
for every Schwartz function $f$ on $\bbR^{d}$.

\item To discuss the boundedness properties of the solution operators below, it will be convenient to also use the language of \emph{pseudodifferential operators}. Given $a : \bbR^{d} \times \bbR^{d} \to \bbC$, the \emph{right-quantized pseudodifferential operator associated with the symbol $a$} (or simply the \emph{right-quantization of $a$}) is defined as
\begin{equation*}
	a(D, x) f = \int \int a(\xi, y) f(y) e^{-i \xi \cdot (x-y)} \, \frac{\ud \xi}{(2 \pi)^{d}} \, \ud y
\end{equation*}
for every Schwartz function $f$ on $\bbR^{d}$. Note that a Fourier multiplier is a particular instance of a (right-quantized) pseudodifferential operator with $a(\xi, y) = m(\xi)$. The integral kernel $K(x, y)$ of $a(D, x)$ is related to the symbol $a(\xi, y)$ by
\begin{equation} \label{eq:psdo-symb-ker}
	\int K(z + y, y) e^{-i \xi \cdot z} \, \ud z = a(\xi, y).
\end{equation}
\end{itemize}

\subsection{Preliminaries on Sobolev spaces} \label{subsec:ftn-sp}
The classical references on this subject are Lions--Magenes \cite{LioMag1,LioMag2,LioMag3}, Grisvard \cite{Gri}, and Triebel \cite{Tri}. Our definitions and notation, however, follow closely those of McLean \cite{McLean}; see also \cite{ChaHewMoi}.

Given $s \in \bbZ_{\geq 0}$ (nonnegative integers) and $p \in (1, \infty)$, we define
\begin{equation*}
    \nrm{f}_{W^{s, p}} \ceq \bb( \sum_{\alp : \abs{\alp} \leq s} \nrm{\rd^{\alp} f}_{L^{p}}^{p} \bb)^{\frac{1}{p}}, \quad \hbox{ where } f \in \calD'(\bbR^{d}),
\end{equation*}
and write $W^{s, p}$ for the space of all distributions $f$ on $\bbR^{d}$ with $\nrm{f}_{W^{s, p}} < + \infty$. For $s \geq 0$, $W^{s, p}$ is defined by (complex) interpolation (see \cite[\S2.4.2]{Tri}\footnote{In \cite{Tri}, $W^{s, p}(U)$, $\td{W}^{s, p}(U)$, and $W^{s, p}_{0}(U)$ are denoted $H^{s}_{p}(U)$, $\td{H}^{s}_{p}(U)$, and $\mathring{H}^{s}_{p}(U)$, respectively. Note also that $H^{s}_{p} = F^{s}_{p, 2}$.})), and $W^{-s, p'} \ceq (W^{s, p})^{\ast}$, where $\frac{1}{p'} = 1 - \frac{1}{p}$. It turns out that $C^{\infty}_{c}(\bbR^{d})$ is a dense subspace of $W^{s, p}$ \cite[\S2.3.2]{Tri} (in fact, $C^{\infty}_{c}(\bbR^{d}) \subseteq W^{s, p}$ is dense even for $s < 0$ once $W^{s, p}$ is identified with a space of distributions as below). Hence, by the duality pairing
\begin{equation*}
	W^{-s, p'} \times W^{s, p} \to \bbR, (f, g) \mapsto \brk{f, g}
\end{equation*}
we may {\bf identify} $f \in W^{-s, p'}$ with a distribution on $\bbR^{d}$ (extended from $g \in C^{\infty}_{c}(\bbR^{d})$ to $g \in W^{s, p}$ uniquely via continuity). Moreover, $W^{s, p}$ is reflexive, which implies that the above pairing induces the isomorphism $W^{-s, p'} \equiv (W^{s, p})^{\ast}$ holds for all $s \in \bbR$ \cite[\S2.6.1]{Tri}.

Given an open subset $U$ of $\bbR^{d}$, $s \in \bbR$ and $p \in (1, \infty)$, we introduce the following spaces:
\begin{itemize}
\item $W^{s, p}(U)$, which consists of distributions on $U$ which arise by restricting elements of $W^{s, p}$ to $U$, i.e.,
\begin{equation} \label{eq:sob-U}
	W^{s, p}(U) \ceq \set{f \in \calD'(U) : f = \td{f} |_{U} \hbox{ for some } \td{f} \in W^{s, p}}
\end{equation}
equipped with the norm
\begin{equation*}
	\nrm{f}_{W^{s, p}(U)} \ceq \inf_{\td{f} \in W^{s, p} : \td{f} |_{U} = f} \nrm{\td{f}}_{W^{s, p}}.
\end{equation*}
\item $\td{W}^{s, p}(U)$, which is the closure of $C^{\infty}_{c}(U)$ viewed as a subspace of $W^{s, p}$, i.e.,
\begin{equation} \label{eq:sob-td}
	\td{W}^{s, p}(U) \ceq \br{C^{\infty}_{c}(U)}^{\nrm{\cdot}_{W^{s, p}}}.
\end{equation}
\end{itemize}
Note the fundamental distinction that $\td{W}^{s, p}(U) \subseteq \calD'(\bbR^{d})$ (i.e., distributions on $\bbR^{d}$), while $W^{s, p}(U) \subseteq \calD'(U)$ (i.e., distributions on $U$). Elements in $\td{W}^{s, p}(U)$ may be identified with elements in $W^{s, p}(U)$, by the natural maps $\td{W}^{s, p}(U) \to W^{s, p} \to W^{s, p}(U)$. In particular, for $\vphi \in \td{W}^{s, p}(U)$, we have $\nrm{\vphi}_{W^{s, p}(U)} \leq \nrm{\vphi}_{\td{W}^{s, p}(U)}$. Note, however, that the map $\td{W}^{s, p}(U) \to W^{s, p}(U)$ may \emph{not} be one-to-one; indeed, for $s < 0$, note that $\td{W}^{s, p}(U)$ may contain distributions (on $\bbR^{d}$) supported in $\rd U$, while such objects do not even correspond to non-trivial elements in $\calD'(U)$, and thus in $W^{s, p}(U)$.

Intuitively, for $s > 0$, $\td{W}^{s, p}(U)$ consists of elements that vanish to all possible orders on $\rd U$, while those in $W^{s, p}(U)$ are allowed to be nontrivial near $\rd U$. However, when $s < 0$, one must be careful of the distinctions discussed in the preceding paragraph. For more on these points, see Remarks~\ref{rem:sob-bdry} and \ref{rem:sob-0} below.

Given a closed subset $F \subseteq \bbR^{d}$, we introduce the closed subspace $W_{F}^{s, p}$ of $W^{s, p}$ that consists of elements that are supported in $F$, i.e.,
 \begin{equation} \label{eq:sob-F}
	W_{F}^{s, p} \ceq \set{f \in W^{s, p}(\bbR^{d}) : \supp f \subseteq F}.
\end{equation}
Clearly, we have $\td{W}^{s, p}(U) \subseteq W_{\br{U}}^{s, p}$; equality holds under additional assumptions on  $\rd U$ (see Remark~\ref{rem:sob-bdry}). Given an open subset $U$ of $\bbR^{d}$, note that we have the Banach space isomorphism
\begin{align}
	W^{s, p}(U) &\equiv W^{s, p} / W^{s, p}_{\bbR^{d} \setminus U}, \label{eq:sob-quotient}
\end{align}
where the isomorphism is given by the quotient of the restriction map $W^{s, p} \to W^{s, p}(U)$, $\td{f} \mapsto \td{f} |_{U}$ by its kernel (which is precisely $W^{s, p}_{\bbR^{d} \setminus U}$). Using $W^{-s, p'} = (W^{s, p})^{\ast}$, it is also clear that we have the identity
\begin{align}
	W^{-s, p'}_{\bbR^{d} \setminus U} &= (\td{W}^{s, p}(U))^{\perp},  \label{eq:sob-annihilator}
\end{align}
where for a subspace $Y$ of a Banach space $X$, $Y^{\perp} \ceq \set{\ell \in X^{\ast} : \brk{\ell, g} = 0 \hbox{ for all } g \in Y}$ (space of annihilators).
The following duality statement, which generalizes $(W^{s, p})^{\ast} = W^{-s, p'}$, is now obvious:
\begin{lemma} [Duality] \label{lem:sob-duality}
For any open subset $U \subseteq \bbR^{d}$, $s \in \bbR$ and $p \in (1, \infty)$, the bilinear map
\begin{equation*}
	W^{-s, p'}(U) \times \td{W}^{s, p}(U) \to \bbR, \, (f, g) \mapsto \brk{f, g},
\end{equation*}
which coincides with the usual pairing $\brk{f, g}$ for $f \in W^{-s, p'}(U) \subseteq \calD'(U)$ and $g \in C^{\infty}_{c}(U)$, is well-defined, continuous, and induces the Banach space isomorphisms
\begin{equation} \label{eq:sob-duality}
	(\td{W}^{s, p}(U))^{\ast} \equiv W^{-s, p'}(U), \quad (W^{-s, p'}(U))^{\ast} \equiv \td{W}^{s, p}(U).
\end{equation}
\end{lemma}
\begin{proof}
The first isomorphism follows from the natural identification $Y^{\ast} \equiv X^{\ast} / Y^{\perp}$ for any closed subspace $Y$ of a Banach space $X$. The second statement then follows from the fact that $\td{W}^{s, p}(U)$ is a reflexive Banach space, being a closed subspace of the Banach space $W^{s, p}$, which is reflexive.
\end{proof}

The Rellich--Kondrachov theorem holds for both scales of Sobolev spaces $W^{s, p}(U)$ and $\td{W}^{s, p}(U)$:
\begin{lemma} [Rellich--Kondrachov] \label{lem:sob-cpt}
For any bounded open subset $U \subseteq \bbR^{d}$, $s \in \bbR$, $p \in (1, \infty)$, and $\dlt > 0$, the natural embeddings
\begin{equation*}
	W^{s, p}(U) \to W^{s-\dlt, p}(U), \quad \td{W}^{s, p}(U) \to \td{W}^{s-\dlt, p}(U)
\end{equation*}
are compact.
\end{lemma}
\begin{proof}
For $\td{W}^{s, p}(U)$, this result is a quick consequence of Rellich--Kondrachov for $W^{s, p}_{\br{U}}$ (see, e.g., McLean \cite[Theorem~3.27]{McLean}). For $W^{s, p}(U)$, the result follows by duality (Lemma~\ref{lem:sob-duality}). \qedhere
\end{proof}
When $p = 2$ (and $s \in \bbR$), we write $H^{s}(U) = W^{s, 2}(U)$ and $\td{H}^{s}(U) = \td{W}^{s, 2}(U)$, etc.

\begin{remark}[Consequences of regularity of $\rd U$] \label{rem:sob-bdry}
So far, we have not used any regularity assumptions on $\rd U$. Under some regularity assumptions on $\rd U$, the spaces introduced above may be given alternative characterizations as follows.
\begin{enumerate}
\item If $\rd U$ is $C^{0}$, then $W^{s, p}_{\br{U}} = \td{W}^{s, p}(U)$.
\item If $\rd U$ is Lipschitz, $s = 1, 2, \ldots$, and $p \in (1, \infty)$, we have the equivalence
\begin{equation*}
\nrm{f}_{W^{s, p}(U)} \aeq_{s, p} \bb( \sum_{\alp : \abs{\alp} \leq s} \nrm{\rd^{\alp} f}_{L^{p}(U)}^{p} \bb)^{\frac{1}{p}}.
\end{equation*}
\end{enumerate}
For $p = 2$, the proofs of these assertions can be found in McLean \cite[Chapter~3]{McLean}; the case $p \neq 2$ can also be handled in a similar way. See also \cite[\S4.3.2, \S4.2.4]{Tri} for the case $\rd U$ is smooth.
\end{remark}

\begin{remark}[The space $W^{s, p}_{0}(U)$] \label{rem:sob-0}
Closely related to $\td{W}^{s, p}(U)$ is the space $W_{0}^{s, p}(U)$, which is the closure of $C^{\infty}_{c}(U)$ with respect to $\nrm{\cdot}_{W^{s, p}(U)}$ (as opposed to $\nrm{\cdot}_{W^{s, p}}$ as in the case of $\td{W}^{s, p}(U)$). An alternative characterization for $W^{s, p}_{0}(U)$ is in terms of vanishing boundary trace \cite[\S4.7.1]{Tri}: for $p \in (1, \infty)$, and $s = m + \alp + \frac{1}{p}$ with $m \in \bbZ_{\geq 0}$ and $\alp \in [0, 1)$, we have
\begin{equation*}
	W^{s, p}_{0}(U) = \set{f \in W^{s, p}(U) : f |_{\rd U} = \tfrac{\rd}{\rd \bfnu} f |_{\rd U} = \cdots = \tfrac{\rd^{m}}{\rd \bfnu^{m}} f |_{\rd U} = 0}.
\end{equation*}

There are subtle differences between the two spaces $\td{W}^{s, p}(U)$ and $W_{0}^{s, p}(U)$. For simplicity, assume that $\rd U$ is smooth and bounded. Then for any $p \in (1, \infty)$, we have (see, e.g., \cite[\S4.3.2]{Tri})
\begin{equation*}
	\td{W}^{s, p}(U) \subseteq W_{0}^{s, p}(U) \quad \hbox{ with equality if } \frac{1}{p} - 1 < s < \infty, \, s + \frac{1}{p} \not \in \bbZ,
\end{equation*}
where $\td{W}^{s, p}(U)$ is realized as a (non-closed) subspace of $W^{s, p}(U)$ under the map $\td{W}^{s, p}(U) \to W^{s, p} \to W^{s, p}(U)$.

The inclusion $\td{W}^{s, p}(U) \subseteq W_{0}^{s, p}(U)$ may not be strict. Indeed, for $U = (0, \infty) \subseteq \bbR$, for any $m \in \bbZ_{\geq 0}$ and $f \in C^{\infty}_{c}((-1, 1))$ with $f(0) = f'(0) = \cdots = f^{(m-1)}(0) = 0$ but $f^{(m)}(0) \neq 0$ (where the condition is $f(0) \neq 0$ if $m = 0$), we have $f |_{U} \in H_{0}^{m+\frac{1}{2}}(U) \setminus \td{H}^{m+\frac{1}{2}}(U)$. We also note that the inclusion fails in general if $s < \frac{1}{p} - 1$. Indeed, for $U = (0, \infty) \subseteq \bbR$ and $s < \frac{1}{2}$, $\dlt_{0} \in \td{H}^{s}(U)$, while $\dlt_{0}$ is not even a nontrivial element of $\calD'(U)$.

We note that $W^{s, p}_{0}(U)$ is heavily used in the classical treatments of Lions--Magenes \cite{LioMag1} and Grisvard \cite{Gri}.  In this paper, we prefer to use $\td{W}^{s, p}(U)$ as it behaves better under standard operations such as duality (see \eqref{eq:sob-duality}) and interpolation (see \cite[\S4.3.2, Theorem~2]{Tri}).
\end{remark}

\section{A preview of the recovery on curves method} \label{sec:crc-summary}
The goal of this expository section is to provide a summary and a basic example for our derivation of integral formulas (i.e., solution operators and representation formulas) from \ref{hyp:crc}, which is carried out in Section~\ref{sec:crc} in the general case. In Section~\ref{subsec:method}, we summarize the proof of Theorems~\ref{thm:main-summary-conic} and \ref{thm:main-summary-bogovskii}, and in Section~\ref{subsec:ideas-div}, we work out the basic example of the divergence operator $\calP u = \rd_{j} u^{j} + B_{j} u^{j}$ (for $u : U \to \bbR^{d}$) with possibly variable coefficients $B: U \to \bbR^{d}$ according to our general method developed in Sections~\ref{sec:sio} and \ref{sec:crc} below. As a byproduct, we provide a self-contained derivation of the Bogovskii and conic operators, \eqref{eq:bog-0} and \eqref{eq:conic-0}, respectively.

\subsection{Summary of the derivation of the integral formulas} \label{subsec:method}
Here, we summarize the proof of Theorems~\ref{thm:main-summary-conic} and \ref{thm:main-summary-bogovskii}, which concern the derivation of integral solution operators and representation formulas for $\calP$ (with $C^{\infty}(\br{U})$ coefficients) satisfying \ref{hyp:crc}. The precise results and arguments are in Section~\ref{sec:crc} below.

\smallskip
\noindent{\it Step~1: Constructing a rough integral kernel supported on curves.}
Consider curves $\bfx(y, y_{1}, s) \in \bbR^{d}$ with $s \in [0, 1]$, where $y$ and $y_{1}$ are the two end points (i.e., $\bfx(y, y_{1}, 0) = y$ and $\bfx(y, y_{1}, 1) = y_{1}$).  Recall that the~\ref{hyp:crc} condition posits that we are able to linearly recover the value of a function $\phi$ at $y$ in terms of its jet at $y_1$ and the jet of $\calP^* \phi$ along $\bfx(y,y_1,\cdot)$.  By duality, \ref{hyp:crc} on the curve $\bfx(y, y_{1}, \cdot)$ amounts to the existence of distributions $K_{y_{1}}(\cdot, y)$ and $b_{y_{1}}(\cdot, y)$ on $\bbR^{d}$ with the following properties:
\begin{enumerate}
\item (Green's function) We have
\begin{equation*}
	\calP K_{y_{1}}(x, y) = \dlt_{0}(x-y) - b_{y_{1}}(x, y) \quad \hbox{ in } U,
\end{equation*}
where $\calP$ acts in the $x$-variable; and
\item (Prescribed support) $K_{y_{1}}(\cdot, y)$ is supported on the image of the curve $\bfx(y, y_{1}, [0, 1])$, and $b_{y_{1}}(\cdot, y)$ is supported in $\set{y_{1}}$.
\end{enumerate}

See Step~1 in Section~\ref{subsec:ideas-div} for a concrete example, and Proposition~\ref{prop:crc-x-duality} for a precise formulation.
One case where this observation is immediately useful is when $y_{1}$ lies outside of the open set $U$ where we wish to solve $\calP u = f$. Then $\calP K_{y_{1}}(x, y) = \dlt_{0}(x-y)$ in $U$, so the integral operator $\calS_{y_{1}} f(x) \ceq \int K_{y_{1}}(x, y) f(y) \, \ud y$ (for $f \in C^{\infty}_{c}(U)$) is already a solution operator (i.e., right-inverse) for $\calP$!

\begin{remark} \label{rem:endpoint-x}
For the construction in the case $y_{1} \not \in U$, note that the following weaker version of the recovery on curves condition is sufficient:
\begin{enumerate}[label=(wRC)]
\item \label{hyp:w-crc} Given any $\varphi \in C^{\infty}_{c}(U)$ that vanishes in a neighborhood of $\bfx(1)$, there exists a linear way to continuously recover $\varphi(\bfx(0))$ from (the jet of) $\calP^{\ast} \varphi$ on $\bfx$.
\end{enumerate}
It is conceptually interesting to observe that the apparently weaker version \ref{hyp:w-crc}, in fact, essentially implies \ref{hyp:crc} by a truncation argument; see Remark~\ref{rem:endpoint-x-detail}. However, we also note that the stronger version \ref{hyp:crc} directly follows from an augmented system; see Section~\ref{subsec:ideas-div} below, as well as Sections~\ref{sec:ode} and \ref{sec:appl}.
\end{remark}

While we have already succeeded in finding a solution operator with an integral kernel $K_{y_{1}}(x, y)$ having prescribed support properties, it is unfortunately very singular (indeed, $K_{y_{1}}(\cdot, y)$ is merely a distribution). In particular, $\calS_{y_{1}}$ does not obey good boundedness properties in standard function spaces (e.g., Sobolev spaces), and is unsuitable for many applications (e.g., nonlinear analysis).

\smallskip
\noindent{\it Step~2: Smooth averaging.}
To remedy the issue of the singularity of $K_{y_{1}}(\cdot, y)$, we introduce a smooth function $\eta(y, y_{1})$ on $\bbR^{d} \times \bbR^{d}$ such that $\int \eta(y, y_{1}) \, \ud y_{1} = 1$ and define a new \emph{smoothly averaged integral kernel}
\begin{equation*}
	K_{\eta}(x, y) \ceq \int K_{y_{1}}(x, y) \eta(y, y_{1}) \, \ud y_{1}.
\end{equation*}
Under suitable assumptions, the smoothly averaged kernel $K_{\eta}(x, y)$ has the following properties:
\begin{enumerate}
\item (Green's function) We have
\begin{equation*}
	\calP K_{\eta}(x, y) = \dlt_{0}(x-y) - b_{\eta}(x, y) \quad \hbox{ in } U, \quad \hbox{ where } b_{\eta}(x, y) \ceq \int b_{y_{1}}(x, y) \eta(y, y_{1}) \, \ud y_{1};
\end{equation*}
\item (Prescribed support) $\supp K_{\eta}(\cdot, y) \subseteq \cup_{y_{1} \in \supp \eta(y, \cdot)} \ran \bfx(y, y_{1}, \cdot)$;
\item (Optimal regularization) $K_{\eta}$ is a singular integral kernel such that $\calS_{\eta} f \ceq \int K_{\eta}(x, y) f(y) \, \ud y$ defines a pseudodifferential operator of order $-m$ (where $m$ is the order of $\calP$).
\end{enumerate}
In particular, we may now summarize the proof of Theorem~\ref{thm:main-summary-conic}. Assuming that the $y_{1}$-support of $\eta(y, y_{1})$ lies outside of $U$, i.e.,
\begin{equation} \label{eq:intro-disjoint-supp}
\br{U} \cap \br{\cup_{y \in U} \supp \eta(y, \cdot)} = \0,
\end{equation}
the integral operator $\calS_{\eta}$ again defines a solution operator for $f \in C^{\infty}_{c}(U)$ in the sense that
\begin{equation*}
\calP \calS_{\eta} f = f \quad \hbox{ for all } f \in C^{\infty}_{c}(U),
\end{equation*}
but this time, has desirable boundedness properties (i.e., regularizing to optimal order) in standard function spaces. We emphasize that $\calS_{\eta} f$ is \emph{not} compactly supported in $U$ in general (cf.~Step~3 below). By duality, we also obtain the representation formula
\begin{equation*}
\calS_{\eta}^{\ast} \calP^{\ast} \varphi  = \varphi \quad \hbox{ for all } \varphi \in C^{\infty}_{c}(U)
\end{equation*}
where we emphasize that the compact support assumption on $\varphi$ is necessary (cf.~Step~3 below). We call $\calS_{\eta}$ a \emph{conic-type solution operator}.

\begin{remark} \label{rem:intro:conic}
As we will see in Section~\ref{subsec:ideas-div} (see also Example~\ref{ex:div-full-sol-conic}), this construction generalizes the conic solution operator introduced by Oh--Tataru \cite{OhTat} for the divergence operator $\calP u = \rd_{j} u^{j}$ on $\bbR^{d}$, which explains its name. The conic solution operator, in turn, has been generalized to the case of the linearized Einstein constraint equations around the flat space by Mao--Tao \cite{MaoTao}, who used it to simplify and improve the nonlinear construction of localized asymptotically flat initial data sets by Carlotto--Schoen \cite{CarSch} (see also Section~\ref{subsec:einstein}).
\end{remark}

\smallskip
\noindent{\it Step~3: Solutions with compact support I: completely integrable case.}
Without the simplifying assumption \eqref{eq:intro-disjoint-supp}, $\calS_{\eta}$ does not directly define a right-inverse of $\calP$. Nevertheless, under suitable assumptions, $b_{\eta}(x, y)$ turns out to be smooth and supported (in the $x$-variable) in $\supp \eta(y, \cdot)$. Therefore,
\begin{equation*}
	\calP \calS_{\eta} f = f - \calB_{\eta} f \quad \hbox{ where $\calB_{\eta} : C^{\infty}_{c}(U) \to C^{\infty}_{c}(U)$ is smoothing.}
\end{equation*}
Observe that if $f \in C^{\infty}_{c}(U)$ then $u = \calS_{\eta} f \in C^{\infty}_{c}(U)$ (under suitable assumptions on $\bfx$, $\eta$ and using the optimal regularization property).  The observation that $u$ does \emph{not} solve $\calP u = f$ is now not surprising: In order for $C^{\infty}_{c}(U)$ solution to exist, $f$ must be orthogonal to the formal cokernel of $\calP$ (i.e., $\ker \calP^{\ast}$, which consists of solutions to $\calP^{\ast} \bfZ = 0$ with $\bfZ \in C^{\infty}(\br{U})$), which is often nontrivial; see Section~\ref{subsec:ideas-div}, as well as Appendix~\ref{sec:appl}, for examples.  Specifically, if $\calP \tilde{u} = f$, and $\tilde{u}$ has compact support, then for any cokernel element $\bfZ \in \ker \calP^\ast$ one has $\int \langle f, \bfZ\rangle dx = \int \langle \calP \tilde{u}, \bfZ \rangle dx = \int \langle \tilde{u}, \calP^\ast Z \rangle dx = 0$.

Let us first discuss an important special case, namely, when the point distribution $b_{y_{1}}$ from Step~1 is already of the form
\begin{equation} \label{eq:intro-b-y1-bogovskii}
	b_{y_{1}}(x, y) = \sum_{\bfA \in \set{1, \ldots, \dim \ker \calP^{\ast}}} \bfZ^{A}(y) \zt_{\bfA}(x, y_{1})
\end{equation}
where $\set{\bfZ^{\bfA}(x)}_{\bfA \in \set{1, \ldots, \dim \ker \calP^{\ast}}} \subseteq C^{\infty}(\br{U})$ is a basis for $\ker \calP^{\ast}$ and $\zt_{\bfA}(\cdot, y_{1})$ is a distribution supported in $\set{y_{1}}$. As we will see in Section~\ref{subsec:ideas-div}, the divergence operator $\calP u = \rd_{j} u^{j}$ falls into this case \cite{Bog}, and in fact, so do any operators with a \emph{completely integrable} graded augmented system by Proposition~\ref{prop:zero-curv} (see also Appendix~\ref{sec:appl} for many concrete examples). In this case, choosing $\eta(y, y_{1}) = \eta(y_{1})$ in Step~2 immediately leads to
\begin{equation} \label{eq:full-sol-max-intro}
	\calP \calS_{\eta} f = f - \calB_{\eta} f, \quad \hbox{ with } \calB_{\eta} f(x) = \sum_{\bfA \in \set{1, \ldots, \dim \ker \calP^{\ast}}} \left( \int \bfZ^{\bfA}(y) f(y) \, \ud y \right) \brk{\zt_{\bfA}(x, \cdot), \eta(\cdot)}.
\end{equation}
In particular, if $f$ is orthogonal to $\ker \calP^{\ast}$, then $\calB_{\eta} f = 0$ and thus $\calP \calS_{\eta} f = f$, i.e., $S_{\eta}$ is a solution operator. By duality, we also have the representation formula $\varphi = \calS_{\eta}^{\ast} \calP^{\ast} \varphi + \calB_{\eta}^{\ast} \varphi$ that holds for all $\varphi \in C^{\infty}(\br{U})$ (i.e., without the compact support assumption).

\begin{remark} \label{rem:intro:bogovskii}
The construction described so far in the completely integrable case extends to a general setting the classical Bogovskii operator \cite{Bog} for $\calP u = \rd_{j} u^{j}$ on $\bbR^{d}$, as well as the integral representation formulas of Reshetnyak \cite{Resh, ReshBook} for $\calP^{\ast}$ the Killing and conformal Killing operators, etc. It has been carried out for the linearized Einstein constraint equations around the flat space in Mao--Oh--Tao \cite{MOT1}, and it was used to simplify and advance (nonlinear) initial data gluing results in the asymptotically flat setting.
\end{remark}

\smallskip
\noindent{\it Step~4: Solutions with compact support II: general case.}
We now consider Theorem~\ref{thm:main-summary-bogovskii} in the general case, when $b_{\eta}$ is not necessarily of the form \eqref{eq:intro-b-y1-bogovskii}. Our result is the existence of a smoothing operator $\calQ : C^{\infty}_{c}(U) \to C^{\infty}_{c}(U)$ (which preserves the compact support property) that deforms the integral solution operator $\calS_{\eta}$ to a correct solution operator $\td{\calS}$ (i.e., $\td{\calS} = \calS_{\eta} - \calQ$) that satisfies
\begin{equation} \label{eq:full-sol-intro}
	\calP \td{\calS} f = f - \sum_{\bfA \in \set{1, \ldots, \dim \ker \calP^{\ast}}}  \brk{f, \bfZ^{\bfA}} w_{\bfA}.
\end{equation}
Here, $\set{\bfZ^{\bfA}}_{\bfA \in \set{1, \ldots, \dim \ker \calP^{\ast}}}$ is a basis of $\ker \calP^{\ast}$ and $w_{\bfA} \in C^{\infty}_{c}(U)$ are \emph{prescribable} functions satisfying $\brk{\bfZ^{\bfA'}, w_{\bfA}} = \dlt^{\bfA'}_{\bfA}$. By duality, we also have the representation formula $\varphi = \td{\calS}^{\ast} \calP^{\ast} \varphi + \sum_{\bfA \in \set{1, \ldots, \dim \ker \calP^{\ast}}}   \bfZ^{\bfA} \brk{w_{\bfA}, \varphi}$ for any $\varphi \in C^{\infty}(\br{U})$ (i.e., without the compact support assumption). In particular, if $\ker \calP^{\ast} = \set{0}$, then the last term in \eqref{eq:full-sol-intro}is dropped and $\td{\calS}$ is a {\it bona fide} right-inverse of $\calP$; by duality, $\td{\calS}^{\ast}$ is a left-inverse of $\calP^{\ast}$. We call $\td{\calS}$ a \emph{Bogovskii-type operator}.

A key ingredient for this argument is a Poincar\'e-type (or rigidity) inequality
\begin{equation} \label{eq:poincare-intro}
\nrm{\varphi}_{H^{-s}(U)} \aleq \nrm{\calP^{\ast} \varphi}_{H^{-s-m}(U)} \quad \hbox{ for all } \varphi \in H^{-s}(U) \hbox{ with } \brk{w_{\bfA}, \varphi} = 0 \quad (\bfA \in \set{1, \ldots, \dim \ker \calP^{\ast}}).
\end{equation}
Using a standard contradiction argument, \eqref{eq:poincare-intro} follows from the weaker inequality
\begin{equation} \label{eq:poincare-eff-intro}
\nrm{\varphi}_{H^{-s}(U)} \aleq \nrm{\calP^{\ast} \varphi}_{H^{-s-m}(U)} + \nrm{\varphi}_{H^{-s-\dlt}(U)} \quad \hbox{ for all } \varphi \in H^{s}(U),
\end{equation}
where $\dlt > 0$ may be arbitrary. Inequality~\eqref{eq:poincare-eff-intro},
in turn, can be proved using the representation formula obtained via duality, and also using the optimal regularization property of $\calS_{\eta}$ and the smoothing property of $\calB_{\eta}^{}$. On the other hand, by another duality argument, \eqref{eq:poincare-intro} is equivalent to the existence of a (special) solution $u \in \td{H}^{s+m}(U)$ to $\calP u = f$ for any $f \in \td{H}^{s}(U)$ with $f \perp \ker \calP^{\ast}$. In fact, that the latter statement may then be upgraded to the existence of the desired linear operator $\calQ$, and thus of $\td{\calS}$; see Step~4 of Section~\ref{subsec:ideas-div} for a simple version of this argument, and \S\ref{subsubsec:full-sol} for our actual proof.

\begin{remark}\label{rem:intro:full-sol-eff}
In the non-completely integrable case, the bound for $\calQ$ is \emph{non-effective} in general (i.e., we know that $\calQ$ is a smoothing operator but have no quantitative relationship between its bound and other constants). But it is only because the argument relies on \eqref{eq:poincare-intro} whose implicit constant is  non-effective due to our use of a contradiction argument. In specific situations where adequate special solutions are already known, the bound for $\calQ$ may be made quantitative.
\end{remark}

\begin{remark} \label{rem:intro:poincare}
On the other hand, we note that the implicit constant in the second Poincar\'e-type inequality \eqref{eq:poincare-eff-intro} can be easily made effective. Hence, our method provides a way to establish effective Poincar\'e-type inequalities (akin to \eqref{eq:poincare-eff-intro}) for a large class of \emph{over}determined operators $\calP^{\ast}$, including the Killing operator (Section~\ref{subsec:symm-div}) and the conformal Killing operator (Section~\ref{subsec:symm-div-tf}) on curved domains. See also \S\ref{subsubsec:discussion:dual}.
\end{remark}

\subsection{A basic example: the divergence operator with variable coefficients} \label{subsec:ideas-div}
To illustrate our method with a simple concrete example, we consider the divergence operator with variable zeroth-order coefficients:
\begin{equation} \label{eq:intro:div-eq}
	\calP u = (\rd_{j} + B_{j}) u^{j} \quad \hbox{ in } U,
\end{equation}
where $U$ is an open subset of $\bbR^{d}$ ($d \geq 2$), $u^{j}$ is a vector field on $U$ and $B_{j}$ is a $1$-form on $U$ ($j=1, \ldots, d$). Its adjoint is given by
\begin{equation*}
	(\calP^{\ast} \varphi)_{j} = - \rd_{j} \varphi + B_{j} \varphi \quad \hbox{ in } U,
\end{equation*}
where $\varphi$ is a function on $U$.

\smallskip
\noindent{\it Step~1: Constructing integral kernel supported on a curve.}
We begin by verifying \ref{hyp:crc} on an arbitrary (smooth) curve $\bfx : [0, 1] \to \bbR^{d}$. In view of the formula for $\calP^{\ast}$, we immediately obtain
\begin{equation} \label{eq:intro-div:aug}
	\rd_{i} \varphi(x) = B_{i}(x)  \varphi (x) - (\calP^{\ast} \varphi)_{i} (x)
\end{equation}
Restricting to the curve $\bfx$ and contracting with $\dot{\bfx}$, we obtain the ODE
\begin{equation*}
	\frac{\ud}{\ud s} \left( \varphi(\bfx(s)) \right) = \rd_{s} \bfx^{i}(s) \rd_{i} \varphi(\bfx(s)) = \rd_{s} \bfx^{i} (B_{i} \circ \bfx)(s)  \varphi (\bfx(s)) - \rd_{s} \bfx^{i} ((\calP^{\ast} \varphi)_{i} \circ \bfx)(s),
\end{equation*}
whose integration leads to
\begin{equation} \label{eq:div-intro-rc}
\begin{aligned}
	\varphi(\bfx(0)) &= \int_{0}^{1} \exp\left(- \int_{0}^{s} \rd_{s} \bfx^{j} (B_{j} \circ \bfx) (s') \, \ud s' \right) \rd_{s} \bfx^{j}  \left( (\calP^{\ast} \varphi)_{j} \circ \bfx \right) (s) \, \ud s \\
	&\peq + \exp\left(- \int_{0}^{1} \rd_{s} \bfx^{j} (B_{j} \circ \bfx) (s') \, \ud s' \right) \varphi(\bfx(1)),
\end{aligned}
\end{equation}
which verifies \ref{hyp:crc} on $\bfx$.

\begin{remark} \label{rem:intro-div:aug}
In this example, \eqref{eq:intro-div:aug} is a (graded) \emph{augmented system} for $\calP$ in the sense of Definition~\ref{def:aug}. The immediate verification of \ref{hyp:crc} on every smooth curve segment $\bfx(s)$ (more precisely, \eqref{eq:div-intro-rc}) via \eqref{eq:intro-div:aug} is a special instance of Proposition~\ref{prop:aug2crc}. 
\end{remark}
Accordingly, given a smooth family of curves $\bfx(y, y_{1}, s) \in \bbR^{d}$ with endpoints $y$ and $y_{1}$, if we define the distributions $K_{y_{1}}(\cdot, y)$ and $b_{y_{1}}(\cdot, y)$ by (for $\psi \in C^{\infty}_{c}(U; \bbR^{d})$ and $\varphi \in C^{\infty}_{c}(U)$)
\begin{align}
	\brk{K_{y_{1}}(\cdot, y), \psi} &= \int_{0}^{1} \exp\left(- \int_{0}^{s} \rd_{s} \bfx^{j} (B_{j} \circ \bfx) (y, y_{1}, s') \, \ud s' \right) \rd_{s} \bfx^{j}  \left( \psi_{j} \circ \bfx \right) (y, y_{1}, s) \, \ud s, \label{eq:div-intro-Ky1} \\
	\brk{b_{y_{1}}(\cdot, y), \varphi} &= \exp\left(- \int_{0}^{1} \rd_{s} \bfx^{j} (B_{j} \circ \bfx) (y, y_{1}, s') \, \ud s' \right) \varphi(y_{1}), \label{eq:div-intro-by1}
\end{align}
then \eqref{eq:div-intro-rc} is equivalent to the identity $\brk{\calP K_{y_{1}}(\cdot, y), \varphi} = \brk{\dlt_{0}(x-y) - b_{y_{1}}(\cdot, y), \varphi}$ (for $\varphi \in C^{\infty}_{c}(U)$). Clearly, $\supp K_{y_{1}}(\cdot, y) \subseteq \bfx(y, y_{1}, [0, 1])$ and $\supp b_{y_{1}}(\cdot, y) \subseteq \set{y_{1}}$. In conclusion, $K_{y_{1}}(\cdot, y)$ and $b_{y_{1}}(\cdot, y)$ satisfy properties~(1) and (2) in Step~1 of Section~\ref{subsec:method}.

\smallskip
\noindent{\it Step~2: Smooth averaging.}
To illustrate the effect of smooth averaging, we consider the following special case (with a slightly modified construction\footnote{In Example~\ref{ex:div-full-sol-conic} below, the same operator is constructed following the method in Section~\ref{subsec:method} more faithfully.} for simplicity). Take $U = \bbR^{d}$, and for every $y \in \bbR^{d}$ and $\omg \in \bbS^{d-1}$, consider the family of curves
\begin{equation*}
	\bfx(y, \omg, s) \ceq y + s \omg.
\end{equation*}
Following (a slight modification of) Step~1, we define $K_{\omg}(\cdot, y)$ by (for $\psi \in C^{\infty}_{c}(U; \bbR^{d})$)
\begin{equation*}
	\brk{K_{\omg}(\cdot, y), \psi} = \int_{0}^{\infty} \exp\left(-\int_{0}^{s} \omg^{j} B_{j}(y + s' \omg) \, \ud s' \right) \omg^{j}  \psi_{j}(y + s \omg) \, \ud s,
\end{equation*}
which satisfies $\calP K_{\omg}(\cdot, y) = \dlt_{0}(x-y)$ and $\supp K_{\omg}(\cdot, y) \subseteq \bfx(y, \omg, [0, \infty))$. Next, given a smooth averaging kernel $\slashed{\eta} \in C^{\infty}_{c}(\bbS^{d-1})$ with $\int \slashed{\eta} \, \ud S(\omg) = 1$, we define the smoothly averaged kernel $K_{\slashed{\eta}}$ by (for $\psi \in C^{\infty}_{c}(U; \bbR^{d})$)
\begin{equation*}
	\brk{K_{\slashed{\eta}}(\cdot, y), \psi} = \int_{\bbS^{d-1}} \int_{0}^{\infty} \exp\left(-\int_{0}^{s} \omg^{j} B_{j}(y + s' \omg) \, \ud s' \right) \omg^{j}  \psi_{j}(y + s \omg) \slashed{\eta}(\omg) \, \ud s \ud S(\omg).
\end{equation*}
By construction, $\calP K_{\slashed{\eta}}(\cdot, y) = \dlt_{0}(x-y)$ and $\supp K_{\slashed{\eta}}(\cdot, y) \subseteq \cup_{\omg \in \supp \slashed{\eta}} \bfx(y, \omg, [0, \infty))$, which forms a cone over the angular set $\supp \slashed{\eta}$ with its tip at $y$. Hence, the operator $\calS_{\slashed{\eta}}$ with integral kernel $K_{\slashed{\eta}}$ satisfies $\calP \calS_{\slashed{\eta}} = I$ and has the property
\begin{equation*}
	\supp f \subseteq C_{\Omg} \quad \imp \quad \supp \calS_{\slashed{\eta}} f \subseteq C_{\Omg}
\end{equation*}
for any cone $C_{\Omg} \subseteq \bbR^{d}$ over an angular set $\Omg$ containing $\supp \slashed{\eta}$.

Moreover, using the polar integration formula, we can explicitly compute $K_{\slashed{\eta}}$. Indeed,
\begin{align*}
\brk{K_{\slashed{\eta}}(\cdot, y), \psi} &= \int_{0}^{\infty} \int_{\bbS^{d-1}} \exp\left(-\int_{0}^{s} \omg^{j} B_{j}(y + s' \omg) \, \ud s' \right) s^{-(d-1)}\omg^{j}  \psi_{j}(y + s \omg) \slashed{\eta}(\omg) s^{d-1}\, \ud s \ud S(\omg) \\
&= \int_{\bbR^{d}} \exp\left(-\int_{0}^{1} (x-y)^{j} B_{j}(y + s (x-y)) \, \ud s \right) \tfrac{(x-y)^{j}}{\abs{x-y}^{d}}  \slashed{\eta}(\tfrac{x-y}{\abs{x-y}}) \psi_{j}(x) \, \ud x.
\end{align*}
In conclusion, $(K_{\slashed{\eta}})^{j}(x, y)$ coincides with a locally integrable function on $\bbR^{d} \times \bbR^{d}$ with
\begin{equation*}
	(K_{\slashed{\eta}})^{j}(x, y) = \exp\left(-\int_{0}^{1} (x-y)^{j} B_{j}(y + s (x-y)) \, \ud s \right) \frac{(x-y)^{j}}{\abs{x-y}^{d}}  \slashed{\eta}(\tfrac{x-y}{\abs{x-y}}) \quad \hbox{ for } x \neq y.
\end{equation*}
When $B_{j} = 0$, this is precisely the \emph{conic operator} of Oh--Tataru \cite{OhTat} for the divergence operator. Moreover, when $B_{j} = 0$ and $\slashed{\eta}$ is constant (i.e., $\slashed{\eta} = \abs{\bbS^{d-1}}^{-1}$), we have $\calS_{\slashed{\eta}} f = - \nb (-\lap)^{-1} f$; in particular, $\calS_{\slashed{\eta}} f$ coincides with the \emph{gradient of a harmonic function} outside of $\supp f$. From this expression, it follows that $\calS_{\slashed{\eta}}$ is a singular integral operator of order $-1$ for a suitably regular $B_{j}$.

\smallskip
\noindent{\it Step~3: Solutions with compact support I: completely integrable case.}
For this step, consider an open subset $U$ of $\bbR^{d}$, which we assume to be connected. Then $\ker \calP^{\ast}$ consists of $\bfZ \in C^{\infty}(\br{U})$ satisfying $\rd_{j} \bfZ = B_{j} \bfZ$ in $U$. It is not difficult to see that any nontrivial $\bfZ \in \ker \calP^{\ast}$ (which is then nonzero everywhere on $U$ by the equation) must satisfy $\rd_{j} \log \bfZ = B_{j}$. Hence,
\begin{align*}
	\ker \calP^{\ast} = \begin{cases} \set{0} & \hbox{ if $B$ is not exact}, \\ (e^{z}) & \hbox{ if $B_{j} = \rd_{j} z$}. \end{cases}
\end{align*}

Let us first consider the case $B_{j} = \rd_{j} z$, which corresponds to the complete integrability of the graded augmented system \eqref{eq:intro-div:aug}. We introduce the shorthand $\bfZ \ceq e^{z}$, which generates $\ker \calP^{\ast}$. In this case, the integral inside the exponential in $b_{y_{1}}(\cdot, y)$ in Step~1 may be computed, and we obtain
\begin{align*}
	\brk{b_{y_{1}}(\cdot, y)} = \bfZ(y) (\bfZ^{-1}\varphi)(y_{1}).
\end{align*}
Let $K_{\eta_{1}}(x, y)$ be the smoothly averaged kernel defined with respect to a smooth function $\eta_{1} = \eta_{1}(y_{1})$ with $\int \eta_{1}(y_{1}) \, \ud y_{1} = 1$, and let $\calS_{\eta_{1}}$ be the operator with integral kernel $K_{\eta_{1}}$. A quick computation shows that
\begin{equation*}
	\calP \calS_{\eta_{1}} f = \calP \left(\int K_{\eta_{1}}(x, y) f(y) \, \ud y \right)= f(x) - (\bfZ^{-1} \eta_{1})(x) \left(\int \bfZ(y) f(y) \, \ud y\right).
\end{equation*}
Moreover, by construction, $\supp K_{\eta_{1}}(\cdot, y) \subseteq \cup_{y_{1} \in \supp \eta_{1}} \bfx(y, y_{1}, [0, 1])$. In particular, if $U$ is \emph{$\bfx$-star-shaped with respect to $\supp \eta_{1}$} in the sense that
\begin{equation*}
\bigcup_{y \in U, \, y_{1} \in \supp \eta_{1}} \bfx(y, y_{1}, [0, 1]) \subseteq U,
\end{equation*}
then $\calS_{\eta_{1}}$ has the support property
\begin{equation*}
	\supp f \subseteq U \quad \imp \quad \supp \calS_{\eta_{1}} f \subseteq U.
\end{equation*}

To illustrate the optimal regularization property, let us consider the special case
\begin{equation*}
	\bfx(y, y_{1}, s) = y + s(y_{1} - y),
\end{equation*}
i.e., $\bfx(y, y_{1}, \cdot)$ is the line segment from $y$ to $y_{1}$. Then, as in Step~2, we can explicitly compute $K_{\eta_{1}}$. Indeed,
\begin{align*}
\brk{K_{\eta_{1}}(\cdot, y), \psi} &= \int_{0}^{1} \int_{\bbR^{d}} \exp\left(- \int_{0}^{s} (y_{1} - y)^{j} \rd_{j} z (y + s' (y_{1} - y)) \, \ud s' \right) (y_{1}-y)^{j}  \psi_{j}(y + s (y_{1} - y)) \eta_{1}(y_{1}) \, \ud y_{1} \ud s  \\
&= \int_{0}^{1} \int_{\bbR^{d}} \frac{\bfZ(y)}{\bfZ(x)} (x - y)^{j}  \psi_{j}(x) \eta_{1}(y + s^{-1}(x-y)) s^{-d-1} \, \ud x \ud s  \\
&= \int_{\bbR^{d}} \frac{\bfZ(y)}{\bfZ(x)} \frac{(x - y)^{j}}{\abs{x-y}^{d}}  \psi_{j}(x) \int_{\abs{x-y}}^{\infty} \eta_{1}(y + r \tfrac{x-y}{\abs{x-y}}) r^{d-1} \ud r \, \ud x,
\end{align*}
where we made the change of variables $x = y + s (y_{1} - y)$ and $r = s^{-1} \abs{x-y}$.
In conclusion, $(K_{\eta_{1}})^{j}(x, y)$ coincides with a locally integrable function on $\bbR^{d} \times \bbR^{d}$ with
\begin{equation}  \label{eq:div-eq-bogovskii-ker}
	(K_{\eta_{1}})^{j}(x, y) = \frac{\bfZ(y)}{\bfZ(x)}  \frac{(x-y)^{j}}{\abs{x-y}^{d}} \int_{\abs{x-y}}^{\infty} \eta_{1}(r \tfrac{x-y}{\abs{x-y}} + y)  r^{d-1} \, \ud r \hbox{ for } x \neq y.
\end{equation}
When $B_{j} = 0$, we have $\bfZ \equiv 1$ and this is precisely the classical \emph{Bogovskii operator} \cite{Bog}; cf.~\eqref{eq:bog-0}. From this expression, it follows that $\calS_{\eta_{1}}$ is a singular integral operator of order $-1$ for a suitably regular $B_{j} = \rd_{j} z$ (see also the proof of Theorem~\ref{thm:low-reg-domain} in \cite{IMOT2} for an alternative construction, which works for a rough $B_{j}$).

\smallskip
\noindent{\it Step~4: Solutions with compact support II.}
Finally, consider the case when $B$ is not exact, or equivalently, $\ker \calP^{\ast} = \set{0}$; this corresponds to the non-completely integrable case. Given $s \in \bbR$, we now show\footnote{In fact, our argument in \S\ref{subsubsec:full-sol} is a slight variant of the present argument, where we construct $\td{\calS}$ that is independent of the order $s$.} the existence of a right-inverse $\td{\calS} : \td{H}^{s}(U) \to \td{H}^{s+1}(U)$ of $\calP$ that preserves the compact support property in $U$. To illustrate the ideas, we focus on the special case $\bfx(y, y_{1}, s) = y + s(y_{1} - y)$ and $\eta_{1} = \eta_{1}(y_{1})$ as before. In this case,
\begin{equation*}
	\calP \calS_{\eta_{1}} f (x)= f(x) - \int \eta_{1}(x) Z(y, x) f(y) \, \ud y, \quad \hbox{ where } Z(y, x) = \exp\left(-\int_{0}^{1} \frac{(x-y)^{j}}{\abs{x-y}} B_{j}(y + s'(x-y)) \, \ud s' \right).
\end{equation*}

We begin by approximating $\eta_{1}(x) Z(y, x)$ by a finite sum of tensor products (or possibly, a single tensor product). For instance, we may simply write
\begin{align*}
\calP \calS_{\eta_{1}} f (x)= f(x) - \calE_{0} f(x) - \eta_{1}(x) \left( \int Z(y, 0) f(y) \, \ud y \right),
\end{align*}
where $\calE_{0} f(x) \ceq \eta_{1}(x) \int (Z(y, x) - Z(y, 0)) f(y) \, \ud y$. From this expression, it is clear that we may arrange $\calE_{0}$ to have operator norm on, say, $L^{1}(U)$ (which is bounded by $\sup_{y} \int \abs{\eta_{1}(x) (Z(y, x) - Z(y, 0))} \, \ud x$) less than $1$ by making $\supp \eta_{1}$ sufficiently small depending on $\nrm{\rd Z}_{L^{\infty}(U)}$. Then $I-\calE_{0} : L^{1}(U) \to L^{1}(U)$ is invertible and we have
\begin{align*}
\calP \calS_{\eta_{1}} (I-\calE_{0})^{-1} f (x)= f(x) - \eta_{1}(x) \left( \int Z(y, 0) (I-\calE_{0})^{-1}f(y) \, \ud y \right).
\end{align*}

Next, we find a special solution $u \in \td{H}^{s+1}(U)$ with $\supp u \subseteq U$ to $\calP u = \eta_{1}$. A key ingredient is the following Poincar\'e-type inequality:
\begin{equation*}
\nrm{\varphi}_{H^{-s}(U)} \aleq \nrm{\calP^{\ast} \varphi}_{H^{-s-1}(U)} \quad \hbox{ for all } \varphi \in H^{-s}(U),
\end{equation*}
which follows from \eqref{eq:poincare-eff-intro} by a standard contradiction argument (see Proposition~\ref{prop:poincare}), the key point being that, in this case, there does not exist any nontrivial solutions $\vphi \in H^{-s}(U)$ to $\calP^{\ast} \vphi = 0$ in $\calD'(U)$. Then from the Poincar\'e-type inequality, by a duality argument involving the Hahn--Banach theorem (see Corollary~\ref{cor:abstract-solvability}), the existence of a special solution $u \in \td{H}^{s+1}(U)$ to $\calP u = \eta_{1}$ follows.

With the special solution $u \in \td{H}^{s+1}(U)$ at hand, we may conclude the construction as follows. Note that
\begin{equation*}
	\td{\calS} f(x) \ceq (\calS_{\eta_{1}} - \calQ)(I-\calE_{0})^{-1} f(x), \hbox{ where } \calQ f(x) = u(x) \left( \int Z(y, 0) f(y) \, \ud y \right),
\end{equation*}
defines a right-inverse of $\calP$. Moreover, observe that $(I - \calE_{0})^{-1} f = f + \calE_{0} (I - \calE_{0})^{-1} f$. Hence,
\begin{equation*}
\supp \calQ f \subseteq \supp u, \quad
\supp (I-\calE_{0})^{-1} f
\subseteq \supp \eta_{1} + \supp f,
\end{equation*}
and the support preserving property of $\calS_{\eta_{1}}$ (under the assumption that $U$ is $\bfx$-star-shaped with respect to $\supp \eta_{1}$), it follows that $\td{\calS}$ also preserves the compact support property in $U$. Finally, since $u \in \td{H}^{s+1}$, it follows that $\calQ$ maps into $\td{H}^{s+1}(U)$, and hence $\td{\calS} : \td{H}^{s}(U) \to \td{H}^{s+1}(U)$.

\section{Singular integral kernels} \label{sec:sio}
In this section, we perform a computation that will show that the general smoothly averaged integral kernel $K_{\eta}(x, y)$ as in Step~2 in Section~\ref{subsec:method} (see Section~\ref{sec:crc} below for the precise construction) define adequate singular integral operators.

\subsection{Assumptions} \label{subsec:sio-hyp}
Recall the setup in Steps~1 and 2 of Section~\ref{subsec:method}. Our aim here is to formulate the precise assumptions on the family of curves $\bfx(y, y_{1}, s)$, rough integral kernels $K_{y_{1}}(\cdot, y)$ and smooth averaging kernel $\eta(y, y_{1})$, which guarantees that the smoothly averaged integral kernel $K_{\eta}(x, y)$ defines a singular integral operator (or more precisely, a classical pseudodifferential operator) of suitable order.

For the purpose of this section, it is more convenient to work with the following spatial variables
\begin{equation*}
	z_{1} \ceq y_{1} - y, \quad
	z(y, z_{1}, s) \ceq \bfx(y, z_{1}+y, s) - y.
\end{equation*}

\smallskip \noindent{\it Assumptions on the family of curves.}
Let $R_{y} > 0$, $N_{0}, M_{0} \in \bbZ_{\geq 0}$, $A_{\bfz} \geq 0$ be parameters to be used below. Let $U$ and $\uV$ be open subsets of $\bbR^{d}$, and $\uW$ an open subset of $\bbR^{d} \times \bbR^{d}$, such that $U$ contains the projection of $\uW$ to the first $\bbR^{d}$, i.e.,
\begin{equation*}
	\set{y \in \bbR^{d} : (y, z_{1}) \in \uW \hbox{ for some } z_{1} \in \bbR^{d}} \subseteq U.
\end{equation*}
We assume that $\bfz : \uW \times [0, 1] \to \uV$, $\bfz = \bfz(y, z_{1}, s)$, which is a smooth family of curves in $\uV$ parametrized by $(y, z_{1}) \in \uW$, satisfies the following properties:
\begin{itemize}
\item We have $\bfz(y, z_{1}, 0) = 0$, $\bfz(y, z_{1}, 1) = z_{1}$. Moreover, $\bfz(y, z_{1}, s)$ obeys
\begin{equation*}
	(1+A_{\bfz})^{-1} \abs{z_{1}} \leq \abs{\rd_{s} \bfz(y, z_{1}, s)} \leq (1+A_{\bfz}) \abs{z_{1}}\quad \hbox{ for every } s \in (0, 1).
\end{equation*}
\item The map $z_{1} \mapsto \bfz$ is invertible for each fixed $y$ and $s \in (0, 1]$; we denote the inverse by $\bfz_{1}(y, z, s)$. We assume that $\frac{\rd \bfz_{1}(y, z, s)}{\rd z}$ obeys
\begin{equation*}
\abs*{\frac{\rd \bfz_{1}(y, z, s)}{\rd z}} \leq (1+A_{\bfz}) s^{-1},
\end{equation*}
where we used the operator norm in $\bbR^{d}$.
\item For higher derivatives, we have
\begin{align*}
	\abs*{R_{y}^{\abs{\alp}} \abs{z}^{\abs{\bt}} \rd_{y}^{\alp} \rd_{z}^{\bt} \bfz_{1}(y, z, s)} &\leq A_{\bfz}  s^{-1} \abs{z} & & \hbox{ for } \abs{\alp} \leq N_{0}, \, 1 \leq \abs{\bt} \leq 1+\abs{\gmm} + M_{0}, \, \abs{\alp} > 0 \hbox{ or } \abs{\bt} > 1.
\end{align*}
\end{itemize}

\begin{remark}[Straight line segments] \label{rem:sio-ul-z}
The simplest (yet useful) example of such a family of curves is the \emph{straight line segments},
\begin{equation*}
\ul{\bfz} (y, z_{1}, s) \ceq s z_{1},
\end{equation*}
which indeed obeys the assumptions with $A_{\bfz} = 0$ and any $R_{y} > 0$, $N_{0}, M_{0} \in \bbZ_{\geq 0}$.
\end{remark}

\smallskip \noindent {\it Assumptions on the smooth averaging kernel.}
Let $R_{z_{1}} > 0$ and $A_{\ueta} > 0$ be parameters to be used below. We assume that $\ueta : \bbR^{d} \times \bbR^{d} \to \bbR$ satisfies the following properties:
\begin{itemize}
\item $\supp \ueta \subseteq \uW$.
\item $\ueta(y, z_{1}) = 0$ if $\abs{z_{1}} \geq R_{z_{1}}$.
\item We have
\begin{equation*}
	\abs{R_{y}^{\abs{\alp}}\abs{z_{1}}^{\abs{\bt}} \rd_{y}^{\alp} \rd_{z_{1}}^{\bt}\ueta(y, z_{1})} \leq A_{\ueta} R_{z_{1}}^{-d} \quad \hbox{ for } \abs{\alp} \leq N_{0}, \, \abs{\bt} \leq \abs{\gmm} + M_{0}.
\end{equation*}
\end{itemize}
\begin{remark}
In practice, we will take $\ueta$ of the form $\ueta(y, z_{1}) = \chi_{1}(y) \chi_{2}(z_{1}) \eta(y, z_{1} + y)$, where $\eta$ is a smooth averaging kernel satisfying \ref{hyp:eta-1}--\ref{hyp:eta-smth} below and $\chi_{1}, \chi_{2}$ are additional smooth functions inserted to make $R_{y}$ and $R_{z_{1}}$ constant.
\end{remark}

\smallskip \noindent{\it Assumption on the rough integral kernel.}
Let $m > 0$ and $A_{\uS} > 0$ be parameters to be used below. Instead of $K_{y_{1}}(\cdot, y)$, we work with a rough integral kernel $\uK_{z_{1}}(\cdot, y)$ of the following form: for every $(y, z_{1}) \in \uW$, $\uK_{z_{1}}(\cdot, y) \in \calD'(\bbR^{d})$ with
\begin{equation*}
\brk{\uK_{z_{1}}(\cdot, y), \varphi} \ceq \int_{0}^{1} \uS^{\gmm}(y, z_{1}, s) \rd^{\gmm} \varphi(y + z(y, z_{1}, s)) \, \ud s \quad \hbox{ for every } \varphi \in C^{\infty}_{c}(\bbR^{d}),
\end{equation*}
for some multi-index $\gmm$ (which could be $0$) and $\uS^{\gmm} : \uW \times [0, 1] \to \bbC$. Each component of the rough integral kernel $K_{y_{1}}(\cdot, y)$ in Section~\ref{subsec:method} will be a linear combination of such distributions; see Section~\ref{sec:crc}.

We assume that the function $\uS^{\gmm}$ satisfies the following bound: for every $(y, z_{1}) \in \supp \ueta$ and $s \in [0, 1]$,
\begin{align*}
\abs*{R_{y}^{\abs{\alp}} \abs{z_{1}}^{\abs{\bt}} \rd_{y}^{\alp} \rd_{z_{1}}^{\bt} \uS^{\gmm}(y, z_{1}, s)} \leq A_{\uS} \abs{z_{1}}^{m+\abs{\gmm}} s^{m+\abs{\gmm}-1} \quad \hbox{ for } \abs{\alp} \leq N_{0}, \, \abs{\bt} \leq \abs{\gmm} + M_{0}.
\end{align*}

\begin{remark} [Discussion of the parameters $A_{\bfz}$, $A_{\ueta}$, $A_{\uS}$, $R_{y}$, and $R_{z_{1}}$] \label{rem:sio-para}
Note that $A_{\bfz}$, $A_{\ueta}$, and $A_{\uS}$ are dimensionless, whereas $R_{y}$ and $R_{z_{1}}$ have the dimension of length. The parameter $A_{\bfz}$ quantifies how for $\bfz$ is from being the straight line segments $\ul{\bfz}$; $A_{\ueta}$ and $A_{\uS}$ quantify the sizes of $\ueta$ and $\uS$. The length parameter $R_{y}$ is the $y$-characteristic scale of $\bfz$, $\ueta$, and $\uS^{\gmm}$, i.e., these objects vary slowly as $y$ varies within scale $R_{y}$. Finally, $\ueta(y, \cdot)$ is supported in $B_{R_{z_{1}}}(0)$ and is bounded by $O( R_{z_{1}}^{-d})$; this is consistent with the unit mean property \ref{hyp:eta-1} below. The power of $R_{z_{1}}$ in the assumption for $\uS^{\gmm}$ is consistent with Example~\ref{ex:conic-bog-sio} below.
\end{remark}

With the above objects, we define the \emph{averaged integral kernel} $\uK_{\ueta}(\cdot, y)$ by the following relation for every $\varphi \in C^{\infty}_{c}(\bbR^{d})$ and $y \in \bbR^{d}$:
\begin{equation} \label{eq:sio-kernel}
	\brk{\uK_{\ueta}(\cdot, y), \varphi} \ceq \int \int_{0}^{1} \uS^{\gmm}(y, z_{1}, s) \ueta(y, z_{1}) (\rd^{\gmm} \varphi)(y+\bfz(y, z_{1}, s))  \ud s \, \ud z_{1} \quad \hbox{ for every } \varphi \in C^{\infty}_{c}(U)
\end{equation}
Informally, $\uK_{\ueta}(\cdot, y) = \int \uK_{z_{1}}(\cdot, y) \ueta(y, z_{1}) \, \ud z_{1}$. Note that, thanks to $\supp \ueta \subseteq \uW$, the right-hand side is well-defined for any $y \in \bbR^{d}$, although it is trivial unless $y$ lies in $U$.

The above setup generalizes (modulo some technical modifications) the conic and Bogovskii integral kernels for $\calP u = \rd_{j} u^{j} + B_{j} u^{j}$ (see Section~\ref{subsec:ideas-div}) constructed using straight line segments, as the following example shows.
\begin{example}[Conic- and Bogovskii integral kernels] \label{ex:conic-bog-sio}
Let $U = \uV = \bbR^{d}$ and $\uW = \bbR^{d} \times \bbR^{d}$. Define, for $\gmm = 0$,
\begin{equation*}
\bfz(y, z_{1}, s) = \ul{\bfz}(y, z_{1}, s) \ceq s z_{1}, \quad \uS^{0}(y, z_{1}, s) = \exp\left(\int_{s}^{1} \rd_{s} \bfz(y, z_{1}, s') \cdot B(y + \bfz(y, z_{1}, s')) \, \ud s' \right) \rd_{s} \bfz(y, z_{1}, s).
\end{equation*}
As discussed in Remark~\ref{rem:sio-ul-z}, the assumptions for $\bfz$ are satisfied for any $N_{0}, M_{0} \in \bbZ_{\geq 0}$ with $A_{\bfz} = 0$ and an arbitrarily large $R_{y}$, $N_{0}$ and $M_{0}$. With any choice of $\ueta$ satisfying the above requirements, the assumption for $\uS^{0}$ is satisfied on $\uW$ for any $N_{0}, M_{0} \in \bbZ_{\geq 0}$ with $m = 1$, an arbitrarily large $R_{y}$, and \begin{equation*}
	A_{\uS} \aleq_{N_{0}, M_{0}, \nnrm{B}} 1,
\end{equation*}
where\footnote{The norm $\nnrm{B}$ coincides with $\nrm{B}_{\dot{G}^{N_{0} + M_{0}, 1}(\ul{\bfx}, \uW)}$ (with $\ul{\bfx} = y + \ul{\bfz}$), which will be properly introduced in Section~\ref{sec:ode} below.}
\begin{align*}
\nnrm{B}
= \sup_{(y, y_{1}-y) \in \uW} \sum_{\alp : \abs{\alp} \leq N_{0} + M_{0}} \int_{0}^{1} \abs{\rd^{\alp} B(y+s(y_{1} - y))} \abs{y_{1} - y}^{1 + \abs{\alp}} s^{\abs{\alp}} \, \ud s.
\end{align*}
The conic and Bogovskii kernels correspond to the following choices of $\ueta$ (and a parameter $R_{\ueta} > 0$):
\begin{enumerate}
\item {\it conic case.} $\ueta(y, z_{1}) = \psi(\abs{z_{1}}) \slashed{\eta}(\frac{z_{1}}{\abs{z_{1}}})$ with $\int_{\bbS^{d-1}} \slashed{\eta} \, \ud S = 1$, $\int \psi(r) r^{d-1} \, \ud r = 1$, $\supp \psi \subseteq (\frac{1}{2} R_{\ueta}, R_{\ueta})$ and $\abs{\rd^{\alp} \psi} \aleq_{\alp} R_{\ueta}^{-d-\abs{\alp}}$ for arbitrarily large $R_{\ueta}$. Then the above assumptions are satisfied for any $N_{0}, M_{0} \in \bbZ_{\geq 0}$ with $R_{z_{1}}(y) = R_{\ueta}$ (independent of $y$) and $A_{\ueta} \aleq_{N_{0}, M_{0}, \psi, \slashed{\eta}} 1$ (independent of $R_{\ueta}$).
\item {\it Bogovskii case.} $\ueta(y, z_{1}) = \chi(y) \eta_{1}(z_{1} + y)$ with $\int \eta_{1} = 1$, $\supp \eta_{1} \subseteq B_{R_{\ueta}}(0)$, and $\abs{\rd^{\alp} \eta_{1}} \aleq_{\abs{\alp}} R_{\ueta}^{-d-\abs{\alp}}$, and an auxiliary cutoff function $\chi \in C^{\infty}_{c}(\bbR^{d})$. Let $R_{\chi} = \sup_{y \in \supp \chi} \abs{y}$. Then the assumptions for $\ueta$ are satisfied for any $N_{0}, M_{0} \in \bbZ_{\geq 0}$ with $R_{z_{1}} = 1 + R_{\chi} + R_{\ueta}$ and $A_{\ueta} \aleq_{N_{0}, M_{0}, \eta} \left(\frac{R_{z_{1}}}{R_{\ueta}}\right)^{d}$; observe that such constants exist thanks to the presence of $\chi$.
\end{enumerate}
Indeed, in the first case, note that the conic kernel agrees with $\uK_{\ueta}(z + y, y)$ for $\abs{z} \leq \frac{1}{2} R_{\ueta}$, and hence globally in the limit $R_{\ueta} \to +\infty$. In the second case, the Bogovskii kernel agrees with $\uK_{\ueta}(z + y, y)$ for every $y \in \bbR^{d}$ such that $\chi(y) = 1$ (hence, in practice, we will choose $\chi$ to be equal to $1$ on the domain $U$ under consideration).
\end{example}

\subsection{Singular integral kernel and symbol bounds} \label{subsec:sio}
The main result of this section is as follows.
\begin{proposition}[Singular integral kernel bounds] \label{prop:sio}
Suppose that the assumptions for $\bfz$, $\ueta$, and $\uS^{\gmm}$ in Section~\ref{subsec:sio-hyp} hold, and let $\uK_{\ueta}$ be defined as in \eqref{eq:sio-kernel}. If $N_{0} \geq 0$ and $M_{0} \geq 0$, then $\uK_{\ueta}(x, y) \in L^{1}_{loc}(\bbR^{d} \times \bbR^{d})$. Moreover, for all $y \in \bbR^{d}$ and $z \in \bbR^{d} \setminus \set{0}$, we have the representation
\begin{equation} \label{eq:sio-ker-0}
	\uK_{\ueta}(z + y, y) = \rd_{z}^{\gmm} \int_{0}^{1} (-1)^{\abs{\gmm}} \uS^{\gmm}(y, \bfz_{1}(y, z, s), s) \ueta(y, \bfz_{1}(y, z, s)) \abs*{\det \frac{\rd \bfz_{1}}{\rd z}}  \, \ud s.
\end{equation}
In fact, we have
\begin{equation} \label{eq:sio-ker}
	\abs{R_{y}^{\abs{\alp}} \abs{z}^{\abs{\bt}} \rd_{y}^{\alp} \rd_{z}^{\bt}\uK_{\ueta}(z + y, y)} \leq A_{\alp, \bt+\gmm} \abs{z}^{-d+m}
\end{equation}
where
\begin{equation} \label{eq:sio-ker-A}
A_{\alp, \bt + \gmm} \leq C_{m, \alp, \bt+\gmm} A_{\uS} A_{\ueta} (1 + A_{\bfz})^{2(d+\abs{\alp} +\abs{\bt + \gmm}) + m + \abs{\gmm}}\quad \hbox{ for } \abs{\alp} \leq N_{0}, \, \abs{\alp} + \abs{\bt} \leq M_{0}.
\end{equation}
Moreover,
\begin{equation} \label{eq:sio-supp}
\supp \uK_{\ueta}(\cdot + y, y) \subseteq B_{(1+A_{\bfz}) R_{z_{1}}}(0).
\end{equation}
\end{proposition}

Let $\uS_{\ueta}$ be the linear operator with integral kernel $\uK_{\ueta}$:
\begin{equation} \label{eq:sio-op}
	\uS_{\ueta} f(x)  = \int \uK_{\ueta}(x, y) f(y) \, \ud y \quad \hbox{ for } f \in C^{\infty}_{c}(\bbR^{d}).
\end{equation}
A convenient way to establish the mapping properties of $\uS_{\ueta}$ is to show that it is a pseudodifferential operator of order $m$ with a classical (or Kohn--Nirenberg) symbol.
\begin{proposition}[Symbol bounds] \label{prop:sio-symb}
Suppose that the assumptions for $\bfz$, $\ueta$, and $\uS^{\gmm}$ in Section~\ref{subsec:sio-hyp} hold, and let $\uK_{\ueta}$ be defined as in \eqref{eq:sio-kernel}. If $M_{0} \geq \max\set{N_{0}, m+1}$, the symbol $\ua_{\ueta}(\xi, y) \ceq \int \uK_{\ueta}(y + z, y) e^{-i \xi \cdot z} \, \ud z$ obeys the bound
\begin{equation} \label{eq:sio-symb}
\begin{aligned}
	\abs{\rd_{y}^{\alp} \rd_{\xi}^{\bt} \ua_{\ueta}(\xi, y)} &\leq C_{m, \alp, \bt, \gmm} A_{\uS} A_{\ueta} (1+A_{\bfz})^{10(d+\abs{\alp}+\abs{\bt}+\abs{\gmm}+m)} R_{y}^{-\abs{\alp}} \left( ((1+A_{\bfz}) R_{z_{1}})^{-1} + \abs{\xi}\right)^{-m-\abs{\bt}}
 \end{aligned}
\end{equation}
for $\abs{\alp} \leq N_{0}$ and $\abs{\alp} + \abs{\bt} \leq M_{0} - m - 1$.
\end{proposition}

\begin{corollary} \label{cor:sio-bdd}
Suppose that the assumptions for $\bfz$, $\ueta$, and $\uS^{\gmm}$ in Section~\ref{subsec:sio-hyp} hold, and let $\uK_{\ueta}$ and $\uS_{\ueta}$ be defined as in \eqref{eq:sio-kernel} and \eqref{eq:sio-op}, respectively. Then there exists a constant $c_{m, d} > 0$ such that, for every $1 < p < \infty$ and $\abs{s} \leq \min\set{N_{0}, M_{0}} - c_{m, d}$, we have $\uS_{\ueta} : W^{s, p}(\bbR^{d}) \to W^{s+m, p}(\bbR^{d})$.
\end{corollary}
This immediately follows from Proposition~\ref{prop:sio-symb} and standard boundedness results for pseudodifferential operators (see, e.g., \cite{Stein}), since $\uS_{\ueta}$ is the left-quantization of the classical symbol $a_{\ueta}$ of order $-m$.
\begin{remark}
Alternatively, one may attempt to directly verify that $\rd_{x}^{\alp} \uK_{\ueta}(x, y)$ with $\abs{\alp} = m$ are Calder\'on--Zygmund kernels. In this case, it is an interesting question to ask what are the minimal regularity assumptions for $\bfz$ and $\uS^{\gmm}$ for this property to hold. In Acosta--Dur\'an--Muschietti \cite{AcoDurMus}, it was shown that for the divergence operator on a John domain, there exists a Bogovskii-type integral kernels $K_{\eta_{1}}$ such that $\rd_{x^{j}} K_{\eta_{1}}$ is a Calder\'on--Zygmund kernel for every $j$ (and in fact, it characterizes John domains).
\end{remark}

We now prove Propositions~\ref{prop:sio} and \ref{prop:sio-symb} in a sequence of lemmas. We begin with the formula \eqref{eq:sio-ker-0}.
\begin{lemma} \label{lem:sio-1}
Under the hypotheses of Proposition~\ref{prop:sio}, we have (for every $\varphi \in C^{\infty}_{c}(\bbR^{d})$)
\begin{align*}
	\brk{K_{\ueta}(\cdot, y), \varphi} = \int \int_{0}^{1} \uS^{\gmm}(y, \bfz_{1}(y, z, s), s) \ueta(y, \bfz_{1}(y, z, s)) \abs*{\det \frac{\rd \bfz_{1}}{\rd z}} (\rd^{\gmm} \varphi)(y+z) \ud s \, \ud z.
\end{align*}
\end{lemma}
\begin{proof}
This is a simple change-of-variables computation. By the definition of $\uK_{\ueta}$, we have
\begin{align*}
\brk{K_{\ueta}(\cdot, y), \varphi} &=\int \int_{0}^{1} \uS^{\gmm}(y, z_{1}, s) \ueta(y, z_{1}) (\rd^{\gmm} \varphi)(y+\bfz(y, z_{1}, s))  \ud s \, \ud z_{1} \\
&= \int \int_{0}^{1} \uS^{\gmm}(y, \bfz_{1}(y, z, s), s) \ueta(y, \bfz_{1}(y, z, s)) (\rd^{\gmm} \varphi)(y+z) \abs*{\det \frac{\rd \bfz_{1}}{\rd z}}  \ud s \, \ud z. \qedhere
\end{align*}
\end{proof}
Lemma~\ref{lem:sio-1} already shows (interpreted in the sense of distributions)
\begin{align*}
	\uK_{\ueta}(z + y, y) &= (-1)^{\abs{\gmm}} \rd_{z}^{\gmm} \int_{0}^{1} \uS^{\gmm}(y, \bfz_{1}(y, z, s), s) \ueta(y, \bfz_{1}(y, z, s)) \abs*{\det \frac{\rd \bfz_{1}}{\rd z}}  \, \ud s.
\end{align*}
The following lemma then completes the proof of Proposition~\ref{prop:sio}:
\begin{lemma} \label{lem:sio-2}
Under the hypotheses of Proposition~\ref{prop:sio}, define
\begin{equation*}
	\uI_{\alp, \bt'}(z+y, y) =  \int_{0}^{1} \rd_{y}^{\alp} \rd_{z}^{\bt'} \left[(-1)^{\abs{\gmm}} \uS^{\gmm}(y, \bfz_{1}(y, z, s), s) \ueta(y, \bfz_{1}(y, z, s)) \abs*{\det \frac{\rd \bfz_{1}}{\rd z}}\right] \, \ud s
\end{equation*}
The integral on the RHS is well-defined for every $y \in \bbR^{d}$ and $z \in \bbR^{d} \setminus \set{0}$, and we have
\begin{equation*}
	\abs{\uI_{\alp, \bt'}(z+y, y)} \leq A_{\alp, \bt'} R_{y}^{-\abs{\alp}} \abs{z}^{-d+m+\abs{\gmm}-\abs{\bt'}}\end{equation*}
where $A_{\alp, \bt'}$ satisfies \eqref{eq:sio-ker-A}. Moreover, $\supp I_{\alp, \bt'}(\cdot + y, y) \subseteq B_{(1+A_{\bfz}) R_{z_{1}}} (0)$.
\end{lemma}

\begin{proof}
We begin by estimating each factor in the definition of $\uI_{\alp, \bt'}$. By the hypothesis on $\bfz$, we have
\begin{equation}
\abs*{R_{y}^{\abs{\alp}} \abs{z}^{\abs{\bt'}} \rd_{y}^{\alp} \rd_{z}^{\bt'} \abs*{\det \frac{\rd \bfz_{1}}{\rd z}}} \aleq (1+A_{\bfz})^{d+\abs{\alp}} s^{-d}. \label{eq:sio-factor-3}
\end{equation}
for $\abs{\alp} \leq N_{0}$ and $\abs{\bt'} \leq \abs{\gmm} + M_{0}$. For the other factors, we claim that
\begin{align}
\abs*{R_{y}^{\abs{\alp}} \abs{z}^{\abs{\bt'}} \rd_{y}^{\alp} \rd_{z}^{\bt'} \ueta(y, \bfz_{1}(y, z, s))}
&\aleq A_{\ueta} (1 + A_{\bfz})^{2\abs{\alp} + 2\abs{\bt'}} R_{z_{1}}^{-d} \label{eq:sio-factor-1}\\
\abs*{R_{y}^{\abs{\alp}} \abs{z}^{\abs{\bt'}} \rd_{y}^{\alp} \rd_{z}^{\bt'} \uS^{\gmm}(y, \bfz_{1}(y, z, s), s)} &\aleq A_{\uS} (1 + A_{\bfz})^{2\abs{\alp} + 2\abs{\bt'}+m+\abs{\gmm}} \abs{z}^{m + \abs{\gmm}} s^{-1}, \label{eq:sio-factor-2}
\end{align}
as long as $\abs{\alp} \leq N_{0}$ and $\abs{\alp} + \abs{\bt'} \leq M_{0} + \abs{\gmm}$.

To establish \eqref{eq:sio-factor-1} and \eqref{eq:sio-factor-2}, we need to study compositions of the form $F(y, \bfz_{1}(y, z, s))$ (for an appropriate function $F$). First, using the hypothesis on $\rd_{s} \bfz$, observe that
\begin{equation} \label{eq:sio-z-z1-compare}
	(1+A_{\bfz})^{-1} \abs{z} \leq s \abs{\bfz_{1}(y, z, s)} \leq (1+A_{\bfz}) \abs{z}.
\end{equation}
Consider a $C^{1}$ function $F = F(y, z_{1})$. For $M_{0} + \abs{\gmm} \geq 1$, we have
\begin{align*}
	\abs{\rd_{z} F(y, \bfz_{1}(y, z, s))} &\leq \abs*{(\rd_{z_{1}} F)(y, z_{1}) |_{z_{1} = \bfz_{1}(y, z, s)}} \abs{\rd_{z} \bfz_{1}(y, z, s)} \\
	&\leq (1+A_{\bfz})^{2} \abs*{(\abs{z_{1}} \rd_{z_{1}} F)(y, z_{1}) |_{z_{1} = \bfz_{1}(y, z, s)}}.
\end{align*}
Similarly, for $N_{0} \geq 1$ and $M_{0} + \abs{\gmm} \geq 1$, we have
\begin{align*}
	\abs{\rd_{y} F(y, \bfz_{1}(y, z, s))} &\leq \abs{(\rd_{y} F)(y, z_{1}) |_{z_{1} = \bfz_{1}(y, z, s)} } + \abs*{(\rd_{z_{1}} F)(y, z_{1}) |_{z_{1} = \bfz_{1}(y, z, s)}} \abs{\rd_{y} \bfz_{1}(y, z, s)} \\
	&\leq \abs{(\rd_{y} F)(y, z_{1}) |_{z_{1} = \bfz_{1}(y, z, s)} } + A_{\bfz} (1+A_{\bfz})\abs*{(\abs{z_{1}} \rd_{z_{1}} F)(y, z_{1}) |_{z_{1} = \bfz_{1}(y, z, s)}}.
\end{align*}
Now, using the hypotheses on $\ueta$, and $\uS^{\gmm}$, \eqref{eq:sio-factor-1} and \eqref{eq:sio-factor-2} in the case $\abs{\alp} + \abs{\bt'} = 1$ follows. The general higher order case follows by a routine induction argument.

Putting together \eqref{eq:sio-factor-3}, \eqref{eq:sio-factor-1}, and \eqref{eq:sio-factor-2}, we arrive at
\begin{align*}
	\abs{\uI_{\alp, \bt'}} \aleq (1 + A_{\bfz})^{d+2(\abs{\alp}+\abs{\bt'})} \abs{z}^{m+\abs{\gmm}} R_{z_{1}}^{-d} \int_{0}^{1} \chf_{\supp_{z_{1}} \ueta(y, z_{1})}(\bfz_{1}(y, z, s)) s^{-d-1} \, \ud s
\end{align*}
To estimate the integral on RHS, we make the change of variables $s = \frac{\abs{z}}{r}$, so that
\begin{equation} \label{eq:sio-s-int}
\begin{aligned}
	\int_{0}^{1} s^{-d-1} \chf_{\supp_{z_{1}} \ueta(y, z_{1})}(\bfz_{1}(y, z, s)) \, \ud s
	= \abs{z}^{-d} \int_{\abs{z}}^{\infty} r^{d-1} \chf_{\supp_{z_{1}} \ueta(y, z_{1})}(\bfz_{1}(y, z, \tfrac{\abs{z}}{r})) \, \ud r.
\end{aligned}
\end{equation}
By \eqref{eq:sio-z-z1-compare}, it follows that $(1+A_{\bfz})^{-1} r \leq \abs{\bfz_{1}(y, z, \tfrac{\abs{z}}{r})}$.
Recalling also that $\supp \ueta(y, \cdot) \subseteq B_{R_{z_{1}}}(0)$, we have
\begin{equation*}
	\chf_{\supp_{z_{1}} \ueta(y, z_{1})}(\bfz_{1}(y, z, \tfrac{\abs{z}}{r})) \leq \chf_{[0, (1+A_{\bfz}) R_{z_{1}}]}(r).
\end{equation*}
Thus,
\begin{equation} \label{eq:sio-r-int}
\int_{\abs{z}}^{\infty} r^{d-1} \chf_{\supp_{z_{1}} \ueta(y, z_{1})}(\bfz_{1}(y, z, \tfrac{\abs{z}}{r})) \, \ud r
\leq (1+A_{\bfz})^{d} R_{z_{1}}^{d},
\end{equation}
and it vanishes for $z$ in a neighborhood of $\set{z \in \bbR^{d} : \abs{z} \geq (1+A_{\bfz}) R_{z_{1}}}$. This completes the proof. \qedhere
\end{proof}

Finally, the symbol bounds (Proposition~\ref{prop:sio-symb}) follows from Proposition~\ref{prop:sio} and the following lemma, which is of independent interest.
\begin{lemma} \label{lem:sio2symb}
Assume that $\uK(x, y) \in L^{1}_{loc}(\bbR^{d} \times \bbR^{d})$ satisfies
\begin{align*}
	\abs{\rd_{y}^{\alp} \rd_{z}^{\bt} \uK(z + y, y)} &\leq A R_{y}^{-\abs{\alp}} \abs{z}^{-d+m - \abs{\bt}} \quad \hbox{ for } \abs{\alp} \leq N, \, \abs{\bt} \leq M, \\
	\supp \mathrm{K} &\subseteq \set{(x, y) \in \bbR^{d} \times \bbR^{d} : \abs{x - y} < R_{z}},
\end{align*}
for some $M, N \in \bbZ_{\geq 0}$, $m > 0$, $A > 0$, $R_{y}, R_{z} > 0$. Then for $\abs{\alp} \leq N$ and $\abs{\bt} \leq M - m - 1$, $\ua(\xi, y) \ceq \int \uK(z+y, y) e^{-i \xi \cdot z} \, \ud z$ satisfies
\begin{equation} \label{eq:sio2symb}
	\abs{\rd_{y}^{\alp} \rd_{\xi}^{\bt} \ua(\xi, y)}
	\leq C_{m, \bt} A R_{y}^{-\abs{\alp}} (R_{z}^{-1} + \abs{\xi})^{-m-\abs{\bt}}.
\end{equation}
\end{lemma}

\begin{proof}
Let $m_{R}(z)$ denote a smooth partition of unity subordinate to dyadic annuli $A_{R} = \set{z \in \bbR^{d} : R < \abs{z} < 4R}$ as in Section~\ref{subsec:prelim-notation}. We split $\uK(z + y, y) = \sum_{R \in 2^{\bbZ}} \uK_{R}(z + y, y)$, where
\begin{align*}
\uK_{R}(z+y, y) \ceq m_{R}(z) \uK(z+y, y).
\end{align*}
We note that this sum is finite thanks to the support property of $\uK$. Clearly, the following holds for each $R$:\begin{equation} \label{eq:sio-K-alp-bt}
\begin{aligned}
\abs{\rd_{y}^{\alp} \rd_{z}^{\bt} \uK_{R}(y+z, y)} &\aleq_{\bt} A R_{y}^{-\abs{\alp}} R^{-d+m-\abs{\bt}} \quad \hbox{ for } \abs{\alp} \leq N, \, \abs{\bt} \leq M.
\end{aligned}
\end{equation}

Correspondingly, $\ua(\xi, y)$ may be split into
\begin{equation*}
	\ua(\xi, y) = \sum_{R \in 2^{\bbZ}} \ua_{R}(\xi, y) \ceq \sum_{R \in 2^{\bbZ}} \int \uK_{R}(z + y, y) e^{-i \xi \cdot z} \, \ud z.
\end{equation*}
We now estimate $\rd_{y}^{\alp} \rd_{\xi}^{\bt'} \ua_{R} = \rd_{y}^{\alp} \rd_{\xi}^{\bt'} \int \uK_{R}(y+z, y) e^{i \xi \cdot z} \, \ud z$ for each $R \in 2^{\bbZ}$. Observe that:
\begin{enumerate}[label=(\roman*)]
\item Using the identity $\rd_{\xi_{j}} e^{- i \xi \cdot z} = - i z^{j} e^{-i \xi \cdot z}$ and \eqref{eq:sio-K-alp-bt},
\begin{align*}
\abs*{\rd_{y}^{\alp} \rd_{\xi}^{\bt'} \int \uK_{R}(z+y, y) e^{i \xi \cdot z} \, \ud z} &\aleq_{\bt'} A R_{y}^{-\abs{\alp}} R^{m+\abs{\bt'}}.
\end{align*}
\item Using the identity $e^{i \xi \cdot z} = (i \xi_{j})^{-1} \rd_{z^{j}} e^{i \xi \cdot z}$, integration by parts and \eqref{eq:sio-K-alp-bt},
\begin{align*}
& \abs*{\rd_{y}^{\alp} \rd_{\xi}^{\bt'} \int \uK_{R}(z+y, y) e^{i \xi \cdot z} \, \ud z}
\aleq_{\bt'} A R_{y}^{-\abs{\alp}} R^{m-\abs{\bt'}} \abs{\xi}^{-\abs{\bt'}}.
\end{align*}
\item For $R > R_{z}$,
\begin{equation*}
\int \uK_{R}(z+y, y) e^{i \xi \cdot z} \, \ud z = 0.
\end{equation*}
\end{enumerate}
For (i) and (ii), we need $\abs{\alp} \leq N$ and $\abs{\bt'} \leq M$.

We now sum up the above bounds in $R$ to estimate $\ua$. By (i) and (iii), we immediately obtain
\begin{align*}
\abs{\rd_{y}^{\alp} \rd_{\xi}^{\bt'} \ua} \aleq_{\bt'} A R_{y}^{-\abs{\alp}} R_{z}^{m+\abs{\bt'}}
\end{align*}
On the other hand, by optimizing (i) and (ii) (which requires $M > \abs{\bt'}+m$), we obtain
\begin{align*}
\abs{\rd_{y}^{\alp} \rd_{\xi}^{\bt'} \ua} &\aleq_{\bt'} A R_{y}^{-\abs{\alp}} \abs{\xi}^{-m-\abs{\bt'}},
\end{align*}
Combining the two bounds, \eqref{eq:sio2symb} follows. \qedhere
\end{proof}
Finally, Proposition~\ref{prop:sio-symb} follows from Proposition~\ref{prop:sio} and Lemma~\ref{lem:sio2symb} with appropriate choices of $A$, $N$, $M$, the same $R_{y}$, and $R_{z} = (1+A_{\bfz}) R_{z_{1}}$, where we are being loose with the power of $(1+A_{\bfz})$ in \eqref{eq:sio-symb}.

\section{From \ref{hyp:crc} to integral formulas} \label{sec:crc}
In this section, we carry out in detail the construction outlined in Section~\ref{subsec:method}. In particular, the conditions needed for our construction, including the key \emph{recovery on curves} condition, are precisely formulated here; see Section~\ref{subsubsec:crc} (simple qualitative version), Sections~\ref{subsubsec:x-q}--\ref{subsubsec:crc-x-q} (detailed quantitative version), and Section~\ref{subsubsec:Z-eta-smth} (additional quantitative assumptions for obtaining solutions with compact support) below.

\subsection{Construction of a rough integral kernel supported on curves} \label{subsec:crc-1}
Our aim in this subsection is to formulate \emph{qualitative} conditions (including recovery on curves) that lead to the construction of a rough integral kernel $K_{y_{1}}(x, y)$ with properties outlined in Step~1 of Section~\ref{subsec:method}.

\subsubsection{Recovery on curves and duality argument, qualitative versions} \label{subsubsec:crc}
Let $W$ be an open subset of $\bbR^{d} \times \bbR^{d}$, and let $\bfx : W \times [0, 1] \to \bbR^{d}$ be a smooth family of curves with $\bfx(y, y_{1}, 0) = y$ and $\bfx(y, y_{1}, 1) = y_{1}$. We state the precise (but qualitative) formulation of \ref{hyp:crc}:

\begin{enumerate}[label=$(\mathrm{RC}^{\vee})$]
\item {\bf Recovery on Curves (with endpoint).} \label{hyp:crc-full}
For each $(y, y_{1}) \in W$ and multi-index $\gmm$, and there exists an $r_{0} \times s_{0}$-matrix-valued continuous function $s \mapsto \tensor{S}{_{J}^{(\gmm, K)}}(y, y_{1}, s)$ ($s \in [0, 1]$) and an $r_{0} \times r_{0}$-matrix-valued distribution $\tensor{(b_{y_{1}})}{^{J'}_{J}}(\cdot, y) \in \calD'(U)$ such that the following holds. For every ($\bbC^{r_{0}}$-valued) $\varphi \in C^{\infty}_{c}(U)$ (\emph{without} the vanishing condition near $y_{1}$), we have
\begin{equation} \label{eq:crc-x-trunc}
	\varphi_{J}(y) = \int_{0}^{1} \sum_{\gmm} \tensor{S}{_{J}^{(\gmm, K)}}(y, y_{1}, s) (\rd_{x}^{\gmm} \calP^{\ast} \varphi)(\bfx(y, y_{1}, s)) \, \ud s + \brk{\tensor{(b_{y_{1}})}{^{J'}_{J}}(\cdot, y), \varphi_{J'}},
\end{equation}
where $\tensor{S}{_{J}^{(\gmm, K)}} = 0$ except for finitely many multi-indices $\gmm$, and
\begin{equation} \label{eq:b-y1-properties}
	\supp \tensor{(b_{y_{1}})}{^{J'}_{J}}(\cdot, y) \subseteq \set{y_{1}}.
\end{equation}
\end{enumerate}

As outlined in Step~1 of Section~\ref{subsec:method}, \ref{hyp:crc-full} leads to (in fact, is equivalent to) the existence of a (distributional) Green's function for $\calP$ supported on the curve $\bfx(y, y_{1}, [0, 1])$.
\begin{proposition}[Duality argument] \label{prop:crc-x-duality}
Let $U$, $W$, $\bfx = \bfx(y, y_{1}, s)$, $\calP$, $\tensor{S}{_{J}^{(\gmm, K)}}(y, y_{1}, s)$, and $\tensor{(b_{y_{1}})}{_{J}^{J'}}(\cdot, y)$ satisfy \ref{hyp:crc-full}. For each $(y, y_{1}) \in W$, define the $s_{0} \times r_{0}$-matrix-valued distribution $K_{y_{1}}(\cdot, y)$ (on $\bbR^{d}$) by
\begin{equation} \label{eq:K-y1-def}
	\brk{K_{y_{1}} (\cdot, y), \psi}
	= \brk{\tensor{(K_{y_{1}})}{^{K}_{J}}(\cdot, y), \psi_{K}}= \int_{0}^{1} \sum_{\gmm} \tensor{S}{_{J}^{(\gmm, K)}}(y, y_{1}, s) \rd_{x}^{\gmm} \psi_{K}(\bfx(y, y_{1}, s)) \, \ud s.
\end{equation}
Then we have, for each $(y, y_{1}) \in W$ with both $y \in U$ and $y_{1} \in U$,
\begin{equation} \label{eq:trun-ker-y1-gen}
	\calP K_{y_{1}} (x, y) = \dlt_{0}(x-y) - b_{y_{1}}(x, y),
\end{equation}
in the sense of distributions on $U$ (where $\calP$ acts on the $x$-variable). Moreover, we have
\begin{equation} \label{eq:K-y1-supp}
	\supp K_{y_{1}}(\cdot, y) \subseteq \bfx(y, y_{1}, [0, 1]).
\end{equation}
\end{proposition}

\begin{proof}
The claim \eqref{eq:K-y1-supp} concerning the support of $K_{y_{1}}(\cdot, y)$ is clear from \eqref{eq:K-y1-def}; hence it only remains to verify \eqref{eq:trun-ker-y1-gen}. Indeed, note that \eqref{eq:trun-ker-y1-gen} is equivalent to
\begin{equation*}
	\brk{K_{y_{1}}(\cdot, y), \calP^{\ast} \varphi} = \varphi(y) - \brk{b_{y_{1}}(\cdot, y), \varphi} \quad \hbox{ for all } \varphi \in C^{\infty}_{c}(U),
\end{equation*}
which in turn is equivalent to \eqref{eq:crc-x-trunc}.
\qedhere
\end{proof}
\begin{example} \label{ex:div-rough-ker}
For $\calP u = (\rd_{j} + B_{j}) u^{j}$, $\bfx = \ul{\bfx}$ (straight line segments), we have \eqref{eq:div-intro-rc}, which leads to $K_{y_{1}}$ and $b_{y_{1}}$ given by \eqref{eq:div-intro-Ky1} and \eqref{eq:div-intro-by1}, respectively.
\end{example}

\begin{remark} [Recovery on curves without endpoint] \label{rem:endpoint-x-detail}
In view of the support property of $b_{y_{1}} (\cdot, y)$ in \eqref{eq:b-y1-properties}, \ref{hyp:crc-full} implies the following weaker version:
\begin{enumerate}[label=$(\mathrm{wRC})$]
\item {\bf Recovery on Curves (without endpoint).} \label{hyp:crc-x}
For each $(y, y_{1}) \in W$ and multi-index $\gmm$, there exists an $r_{0} \times s_{0}$-matrix-valued continuous function $s \mapsto \tensor{S}{_{J}^{(\gmm, K)}}(y, y_{1}, s)$ ($s \in [0, 1]$) such that the following holds. For every ($\bbC^{r_{0}}$-valued) $\varphi \in C^{\infty}_{c}(U)$ with $\varphi$ vanishing in a neighborhood of $y_{1}$, we have
\begin{equation} \label{eq:crc-x}
	\tensor{\varphi}{_{J}}(y) = \int_{0}^{1} \sum_{\gmm} \tensor{S}{_{J}^{(\gmm, K)}}(y, y_{1}, s) (\rd_{x}^{\gmm} (\calP^{\ast} \varphi)_{K}) (\bfx(y, y_{1}, s)) \, \ud s,
\end{equation}
where $\tensor{S}{_{J}^{(\gmm, K)}} = 0$ except for finitely many multi-indices $\gmm$.
\end{enumerate}
Amusingly, the reverse implication also holds under a mild additional assumption. Specifically, assume that \ref{hyp:crc-x} and the following holds:
\begin{enumerate}[label=$(\mathrm{R}\textrm{-}y_{1})$]
\item {\bf Regularity at $y_{1}$} (or {\bf endpoint regularity}){\bf .}  \label{hyp:endpoint-x}
For each $\gmm$ and $(y, y_{1}) \in W$, $\tensor{S}{_{J}^{(\gmm, K)}}(y, y_{1}, s)$ is $(m_{K}+\abs{\gmm}-1)$-times continuously differentiable in $s$ at $s = 1$.
\end{enumerate}
Then, for each $(y, y_{1}) \in W$, there exists an $r_{0} \times r_{0}$-matrix-valued distribution $b_{y_{1}}(\cdot, y) \in \calD'(U)$ such that \eqref{eq:b-y1-properties} holds and, for every $\varphi \in C^{\infty}_{c}(U)$ (without the vanishing condition near $y_{1}$), \eqref{eq:crc-x-trunc} holds. In fact, $b_{y_{1}}(\cdot, y)$ is given by
\begin{equation} \label{eq:b-y1}
\brk{\tensor{(b_{y_{1}})}{_{J}^{J'}}(\cdot, y), \varphi_{J'}} \ceq \lim_{\eps \to 0} \int_{0}^{1} \sum_{\gmm} \tensor{S}{_{J}^{(\gmm, K)}}(y, y_{1}, s)  ([\rd_{x}^{\gmm} \calP^{\ast}, h_{\eps}] \varphi)_{K} (\bfx(y, y_{1}, s)) \, \ud s,
\end{equation}
where $\bfnu(y, y_{1}) \ceq \frac{\rd_{s} \bfx(y, y_{1}, 1)}{\abs{\rd_{s} \bfx(y, y_{1}, 1)}}$, $h_{\eps}(x) \ceq \chi_{> 1}( \eps^{-1} (\bfnu(y, y_{1}) \cdot (y_{1} - x) ) )$, and $\chi_{>1}(s)$ is a smooth nonnegative and nondecreasing function that equals $1$ for $s > 1$ and $0$ for $s < \frac{1}{2}$. We omit the details, as this fact will not be used in the remainder of the paper.
\end{remark}

\subsection{Quantitative formulation of \ref{hyp:crc}} \label{subsec:crc-q}
We now give quantitative version of the assumptions on the basic objects needed for carrying out our method outlined in Section~\ref{subsec:method}.

\subsubsection{Conditions on $\bfx$, quantitative version} \label{subsubsec:x-q}
Let $L_{y}, L_{b} > 0$, $M_{\bfx}, M_{\bfx}' \in \bbZ_{\geq 0}$ and $A_{\bfx}, A_{\bfx}' > 0$ be parameters to be used below (note that $L_{b}$, $A_{\bfx}'$, and $M_{\bfx}'$ are only used in \ref{hyp:x-3}). In accordance with our multi-index notation, $(\rd_{y} + \rd_{y_{1}})^{\alp} = \prod_{j=1}^{d} (\rd_{y^{j}} + \rd_{y_{1}^{j}})^{\alp_{j}}$ and $(\rd_{y} + \rd_{x})^{\alp} = \prod_{j=1}^{d} (\rd_{y^{j}} + \rd_{x^{j}})^{\alp_{j}}$.

Fix an open set $W \subseteq \bbR^{d}_{y} \times \bbR^{d}_{y_{1}}$. We will consider a family of curves $\bfx : W \times [0, 1] \to \bbR^{d}$ satisfying (possibly a subset of) the following properties:
\begin{enumerate}[label=($\bfx$-\arabic*), series=xcondition]
\item \label{hyp:x-1} The map $\bfx(y, y_{1}, s)$ is continuous and obeys
\begin{gather*}
\bfx(y, y_{1}, 0) = y, \quad
\bfx(y, y_{1}, 1) = y_{1}, \\
	(1+A_{\bfx})^{-1} \abs{y_{1}-y} \leq \abs{\rd_{s} \bfx(y, y_{1}, s)} \leq (1+A_{\bfx}) \abs{y_{1}-y} \quad \hbox{ for every } s \in (0, 1).
\end{gather*}
For higher derivatives, we have
\begin{equation*}
	\abs*{L_{y}^{\abs{\alp}} \abs{y_{1}-y}^{\abs{\bt}} (\rd_{y} + \rd_{y_{1}})^{\alp} \rd_{y_{1}}^{\bt} \rd_{s} \bfx(y, y_{1}, s)} \leq  A_{\bfx} \abs{y_{1}-y} \quad \hbox{ for } \abs{\alp} + \abs{\bt} \leq M_{\bfx}, \, \abs{\alp} > 0 \hbox{ or } \abs{\bt} > 1.
\end{equation*}
\item \label{hyp:x-2} The map $y_{1} \mapsto \bfx$ is invertible for each fixed $y$ and $s \in (0, 1]$; we denote the inverse by $\bfy_{1}(y, x, s)$. We assume that $\frac{\rd \bfy_{1}(y, x, s)}{\rd x}$ obeys
\begin{equation*}
\abs*{\frac{\rd \bfy_{1}(y, x, s)}{\rd x}} \leq (1+A_{\bfx}) s^{-1},
\end{equation*}
where we used the operator norm in $\bbR^{d}$. For higher derivatives, we have
\begin{equation*}
	\abs*{L_{y}^{\abs{\alp}} \abs{x-y}^{\abs{\bt}} (\rd_{y} + \rd_{x})^{\alp} \rd_{x}^{\bt} \bfy_{1}(y, x, s)} \leq A_{\bfx} s^{-1} \abs{x-y} \quad \hbox{ for } \abs{\alp} + \abs{\bt} \leq M_{\bfx}, \, \abs{\alp} > 0 \hbox{ or } \abs{\bt} > 1.
\end{equation*}

\item \label{hyp:x-3} The map $\bfx(y, y_{1}, s)$ obeys
\begin{equation*}
	\abs*{L_{y}^{\abs{\alp}} L_{b}^{\abs{\bt}} (\rd_{y} + \rd_{y_{1}})^{\alp} \rd_{y_{1}}^{\bt} \rd_{s} \bfx(y, y_{1}, s)} \leq  A'_{\bfx}  L_{b} \quad \hbox{ for } \abs{\bt} > 0, \, \abs{\alp} + \abs{\bt} \leq M_{\bfx}',
\end{equation*}
for every $(y, y_{1}, s) \in W \times (0, 1]$.
\end{enumerate}

Hypotheses \ref{hyp:x-1} and \ref{hyp:x-2} are motivated by the assumptions on $\bfz$ in Section~\ref{sec:sio}; for \ref{hyp:x-3}, see Remark~\ref{rem:x-3}.

\begin{remark} [\ref{hyp:x-1}--\ref{hyp:x-3} for straight line segments and other $\bfx$] \label{rem:x-q}
As in Remark~\ref{rem:sio-ul-z}, a basic example of $\bfx$ is the \emph{straight line segments}, $\ul{\bfx}(y, y_{1}, s) \ceq y + (y_{1} - y) s$, for which \ref{hyp:x-1}--\ref{hyp:x-3} hold with $A_{\bfx} = 1$ and any $L_{y}, L_{b} > 0$.

Other interesting examples of $\bfx$ satisfying \ref{hyp:x-1}--\ref{hyp:x-3} can be obtained by keeping $\bfx(y, y_{1}, s)$ close to straight line segments for small $s$ (e.g., $s < \frac{\dlt}{\abs{y_{1} - y}}$ for some $0 < \dlt \ll 1$), but letting it curve for large $s$. The solution operator in \cite{MaoTao} adapted to degenerate cones may be constructed using such an $\bfx$.
\end{remark}

\begin{remark} [On the hypothesis \ref{hyp:x-3}] \label{rem:x-3}
Note that \ref{hyp:x-3} improves upon \ref{hyp:x-1} for $\abs{y_{1} - y} \aleq L_{b}$. Observe also that, since $(\rd_{y} + \rd_{y_{1}})^{\alp} \rd_{y_{1}}^{\bt} \bfx(y, y_{1}, 0) = 0$ for $\abs{\bt} > 0$ (indeed, $\bfx(y, y_{1}, 0) = y$), it follows from \ref{hyp:x-3} that
\begin{equation} \label{eq:x-Lb}
	\abs*{L_{y}^{\abs{\alp}} L_{b}^{\abs{\bt}} (\rd_{y} + \rd_{y_{1}})^{\alp} \rd_{y_{1}}^{\bt} \bfx(y, y_{1}, s)} \leq  s A_{\bfx}  L_{b} \quad \hbox{ for } \abs{\alp} + \abs{\bt} \leq M_{\bfx}',
\end{equation}
for every $(y, y_{1}, s) \in W \times (0, 1]$.

At a technical level, we remark that \ref{hyp:x-3} is \emph{not} needed in the proof of Theorem~\ref{thm:smth-ker}, our core analytic result. Its only use is to derive \ref{hyp:b-y1-smth} below from a graded augmented system; see Proposition~\ref{prop:ode2crc}.(2).
\end{remark}

\subsubsection{Recovery on curves condition, quantitative version} \label{subsubsec:crc-x-q}
In addition to \ref{hyp:crc-full}, we assume that the following assumptions on $\tensor{S}{_{J}^{(\gmm, K)}}$ hold for some $m_{K}' \in \bbZ_{\geq 0}$ ($K \in \set{1, \ldots, s_{0}}$), $M_{S} \in \bbZ_{\geq 0}$ and $A_{S} > 0$:
\begin{enumerate}[label=(RC-q)]
\item {\bf Recovery on curves, quantitative.} \label{hyp:crc-x-q}
\ref{hyp:crc-full} holds for $\tensor{S}{_{J}^{(\gmm, K)}}$ with the following bounds:
\begin{equation*}
\abs*{L_{y}^{\abs{\alp}} \abs{y_{1}-y}^{\abs{\bt}} (\rd_{y} + \rd_{y_{1}})^{\alp} \rd_{y_{1}}^{\bt} \tensor{S}{_{J}^{(\gmm, K)}}(y, y_{1}, s)} \leq A_{S} \abs{y_1-y}^{m_{K}+\abs{\gmm}} s^{m_{K}+\abs{\gmm}-1} \quad \hbox{ for } \abs{\alp} + \abs{\bt} \leq M_{S}
\end{equation*}
for all $(y, y_{1}, s) \in W \times (0, 1]$, where $m_{K}$ is the order of the operator $(\calP^{\ast} \varphi)_{K}$. Moreover, $\tensor{S}{_{J}^{(\gmm, K)}} = 0$ if $\abs{\gmm} > m_{K}'$.
\end{enumerate}
We also make (possibly a subset of) the following quantitative assumptions on $b_{y_{1}}$ for some $M_{g} , M_{Z}, M_{Z}'\in \bbZ_{\geq 0}$ and $A_{g}, A_{Z}, A_{Z}' > 0$:
\begin{enumerate}[label=($b_{y_{1}}$-\arabic*), series=by1condition]
\item {\bf Structure of $b_{y_{1}}$, quantitative.} \label{hyp:b-y1-q}
For every $(y, y_{1}) \in W$, there exists a distribution $b_{y_{1}}(\cdot, y)$ of the form
\begin{equation}
\tensor{(b_{y_{1}})}{^{J'}_{J}}(x, y) = \sum_{\bfA \in \calA} (Z^{\bfA})_{J}(y, y_{1}) (g_{\bfA})^{J'}(x, y_{1})
\end{equation}
for some finite index set $\calA$ (independent of $y, y_{1}$), which satisfies the relation \eqref{eq:crc-x-trunc} for every $(y, y_{1}) \in W$ with $y \in U$. Moreover, each $Z^{\bfA}$ obeys
\begin{align*}
	\abs{L_{y}^{\abs{\alp}} \abs{y- y_{1}}^{\abs{\bt}} (\rd_{y} + \rd_{y_{1}})^{\alp} \rd_{y_{1}}^{\bt} Z^{\bfA}(y, y_{1})} \leq A_{Z} \quad \hbox{ for } \abs{\alp} + \abs{\bt} \leq M_{Z},
\end{align*}
and each $g_{\bfA}$ takes the form (for $J'=1, \ldots, r_{0}$)
\begin{align*}
(g_{\bfA})^{J'}(x, y_{1}) = \sum_{\alp} c[g_{\bfA}]^{(\alp, J')}(y_{1}) \rd^{\alp} \dlt_{0}(x-y_{1}),
\end{align*}
where the coefficients $c[g_{\bfA}]^{(\alp, J')}$ are zero except if $\abs{\alp} \leq \max_{K} (m_{K} + m'_{K})-1$ and obey
\begin{equation*}
	\abs{\rd_{y_{1}}^{\bt} c[g_{\bfA}]^{(\alp, J')}} \leq A_{g} \quad \hbox{ for } \abs{\bt} \leq M_{g}.
\end{equation*}

\item {\bf Local regularity of $Z^{\bfA}$.} \label{hyp:b-y1-smth}
\ref{hyp:b-y1-q} holds with $Z^{\bfA}$ obeying the following additional bounds:
\begin{align*}
	\abs{L_{y}^{\abs{\alp}} L_{b}^{\abs{\bt}} (\rd_{y} + \rd_{y_{1}})^{\alp} \rd_{y_{1}}^{\bt} Z^{\bfA}(y, y_{1})} \leq A_{Z}  \quad \hbox{ for } \abs{\alp} + \abs{\bt} \leq M_{Z}'
\end{align*}
\end{enumerate}

Some remarks concerning these quantitative assumptions are in order.

\begin{remark} [On the hypotheses \ref{hyp:crc-x-q} and \ref{hyp:b-y1-q}] \label{rem:b-y1-q}
Although these assumptions look complicated, we will see in Section~\ref{sec:ode} that \ref{hyp:crc-x-q} and \ref{hyp:b-y1-q} may be derived directly on any curves $\bfx$ satisfying \ref{hyp:x-1}--\ref{hyp:x-2} from the existence of a graded augmented system (with appropriate bounds); see Proposition~\ref{prop:ode2crc}.(1). Furthermore, this direct derivation will provide us with more concrete expressions for $Z^{\bfA}$ and $g_{\bfA}$, as well as a bound for $\# \calA$.
\end{remark}

\begin{remark} [On the hypothesis \ref{hyp:b-y1-smth}] \label{rem:b-y1-smth}
Hypothesis \ref{hyp:b-y1-smth} improves upon the bounds for $Z^{\bfA}$ in \ref{hyp:b-y1-q} when $\abs{y_{1} - y} \leq L_{b}$. Given a graded augmented system, \ref{hyp:b-y1-smth} on $\bfx$ follows from an additional local (i.e., for $y_{1}$ close to $y$) regularity assumption \ref{hyp:x-3} for the curves $\bfx$.

At a technical level, we note that \ref{hyp:b-y1-smth} is \emph{not} used in the proof of Theorem~\ref{thm:smth-ker}, our core analytic result. It is only used in Section~\ref{subsec:cokernel} below.
\end{remark}

\begin{remark}[Recovery on curves condition without endpoint, quantitative version]
We note that \ref{hyp:b-y1-q} can be derived from \ref{hyp:crc-x-q} and a quantitative version of the endpoint regularity condition. Like Remark~\ref{rem:endpoint-x-detail}, the following statement only plays a conceptual role and will not be used elsewhere in the paper.

For simplicity, take $\bfx$ to be the straight line segments $\ul{\bfx}(y, y_{1}, s) = y + (y_{1} - y) s$ (so that \ref{hyp:x-1}--\ref{hyp:x-2} hold), and let $U$ and $\tensor{S}{_{J}^{(\gmm, K)}}(y, y_{1}, s)$ satisfy \ref{hyp:crc-x-q}. Assume also that
\begin{equation*}
	\bigcup_{y \in U, \, (y, y_{1}) \in W} \bfx(y, y_{1}, [0, 1]) \subseteq U.
\end{equation*}
Assume also that $\calP$ takes the form $\sum_{\alp' : \abs{\alp'} \leq m} \tensor{(c_{\calP})}{^{(\alp', J)}_{K}}(x) \rd^{\alp'}$ and obeys, for some $M' \geq \bbZ_{\geq 0}$ and $A'_{\calP} > 0$,
\begin{equation*}
	\abs{\rd_{x}^{\bt} \tensor{(c_{\calP})}{^{(\alp', J)}_{K}}} \leq A'_{\calP} \quad \hbox{ for } \abs{\bt} \leq m_{K}' + M',
\end{equation*}
and that, for some $A'_{S} > 0$,
\begin{align*}
\abs*{\abs{y_{1}-y}^{\abs{\bt}} (\rd_{y} + \rd_{y_{1}})^{\alp} \rd_{y_{1}}^{\bt} \rd_{s}^{\ell} \tensor{S}{_{J}^{(\gmm, K)}}(y, y_{1}, 1)} &\leq A'_{S} \quad \hbox{ for } \abs{\alp} + \abs{\bt} \leq M', \, 1 \leq \ell \leq m_{K} + \abs{\gmm}.
\end{align*}
Then \eqref{eq:crc-x-trunc} holds with $b_{y_{1}}(\cdot, y)$ given by \eqref{eq:b-y1}. Moreover, this $b_{y_{1}}(\cdot, y)$ satisfies \ref{hyp:b-y1-q} with $\calA = \set{(\alp, J) : \abs{\alp} \leq \max_{K} (m_{K}+m'_{K})-1, \, J = 1, \ldots, r_{0}}$ (where each $\alp$ is a multi-index), $c[g_{(\alp, J)}]^{(\bt, J')} = \dlt^{J'}_{J} \dlt^{\bt}_{\alp}$ and
\begin{equation*}
A_{Z} \leq A'_{S}, \quad
A_{g} = 1, \quad M_{Z} = M_{g} = M' - C_{m, m_{1}}.
\end{equation*}

The proof proceeds by noting that $b_{y_{1}}$ defined by \eqref{eq:b-y1} in fact takes the form
\begin{equation*} 
\tensor{(b_{y_{1}})}{^{J'}_{J}}(x, y) = \sum_{\alp : \abs{\alp} \leq m + m_{1} - 1} (Z^{(\alp, J')})_{J}(y, y_{1}) \rd^{\alp} \dlt_{0}(x-y_{1}),
\end{equation*}
where
\begin{align*}
& (Z^{(\alp, J')})_{J}(y, y_{1}) \ceq \brk*{\tensor{(b_{y_{1}})}{^{J'}_{J}}(x, y), \tfrac{(-1)^{\abs{\alp}}}{\alp!} (x-y_1)^{\alp}} \\
&= \sum_{\substack{\alp', \bt, \gmm : \\ \abs{\alp'+\bt} \leq m + \abs{\gmm}, \\ \alp' \leq \alp, \, \abs{\bt} \geq 1, \, \abs{\gmm} \leq m_{1}}} (-1)^{\abs{\alp+\bt}} \frac{(\alp'+\bt)! }{(\alp-\alp')! \alp'! \bt!}\frac{\rd_{s} \ul{\bfx}(y, y_{1}, 1)^{\bt}}{\abs{\rd_{s} \ul{\bfx}(y, y_{1}, 1)}^{2 \abs{\bt}}}   \\
&\relphantom{= \sum_{\substack{\alp' \leq \alp, \, \abs{\bt} \geq 1, \, \abs{\gmm} \leq m_{1}}}}
\times \rd_{s}^{\abs{\bt}} \left( \tensor{S}{_{J}^{(\gmm, K')}}(y, y_{1}, s) \tensor{c[\rd_{x}^{\gmm} \calP^{\ast}]}{_{K'}^{(\alp+\bt, J')}} (\ul{\bfx}(y, y_{1}, s)) (\ul{\bfx}(y, y_{1}, s) -y_{1})^{\alp-\alp'}\right) \Big|_{s = 1}.
\end{align*}
Again, we omit the details of the proof, as we will not use this in the remainder of the paper.
\end{remark}

\subsubsection{Conditions on the smooth averaging weight $\eta$, quantitative version} \label{subsubsec:eta-q}
Given $M_{\eta}, M_{\eta}' \in \bbZ_{\geq 0}$, $A_{\eta} > 0$, and $L_{\eta} > 0$, we will consider a smooth function $\eta : W \to \bbR$ satisfying the following properties:
\begin{enumerate}[label=($\eta$-\arabic*), series=etacondition]
\item \label{hyp:eta-1} $\int \eta(y, y_{1}) \, \ud y_{1} = 1$ for all $y$.
\item \label{hyp:eta-2} $\supp \eta \subseteq W$ and $\supp \eta(y, \cdot) \subseteq B_{L_{\eta}}(y)$.
\item \label{hyp:eta-3} We have
\begin{equation*}
	\abs{L_{y}^{\abs{\alp}} \abs{y_{1}-y}^{\abs{\bt}} (\rd_{y} + \rd_{y_{1}})^{\alp} \rd_{y_{1}}^{\bt}\eta(y, y_{1})} \leq A_{\eta} L_{\eta}^{-d} \hbox{ for } \abs{\alp} + \abs{\bt} \leq M_{\eta}.
\end{equation*}
\item \label{hyp:eta-smth}
We have
\begin{align*}
	\abs{L_{y}^{\abs{\alp}} L_{b}^{\abs{\bt}} (\rd_{y} + \rd_{y_{1}})^{\alp} \rd_{y_{1}}^{\bt} \eta(y, y_{1})} \leq A_{\eta} L_{\eta}^{-d} \quad \hbox{ for } \abs{\alp} + \abs{\bt} \leq M_{\eta}'.
\end{align*}
\end{enumerate}

\begin{remark}[On the hypotheses \ref{hyp:eta-1}--\ref{hyp:eta-3}] \label{rem:eta}
Hypotheses \ref{hyp:eta-1}--\ref{hyp:eta-3} are motivated by the assumptions on $\upeta$ in Section~\ref{sec:sio}. Note that \ref{hyp:eta-2} precludes the choice of smooth averaging weight that occurs in the conic operator in Section~\ref{subsec:ideas-div}. However, it can be easily worked around; see Examples~\ref{ex:conic-bog-sio} and \ref{ex:div-full-sol-conic}.
\end{remark}
\begin{remark}[On the hypothesis \ref{hyp:eta-smth}] \label{rem:eta-smth}
Note that \ref{hyp:eta-smth} improves upon the bounds for $\eta$ in \ref{hyp:eta-3} when $\abs{y_{1} - y} \leq L_{b}$. At a technical level, we note that \ref{hyp:eta-smth} is \emph{not} used in the proof of Theorem~\ref{thm:smth-ker}, our core analytic result. It is only used in Section~\ref{subsec:cokernel} below.
\end{remark}

\subsection{Smooth averaging} \label{subsec:crc-2}
We now carry out the construction of a smoothly averaged integral kernel $K_{\eta}$ outlined in Step~2 of Section~\ref{subsec:method}, which is at the heart of our approach.

\subsubsection{Construction of smoothly averaged kernels} \label{subsubsec:smth-kernel}
The \emph{smoothly averaged kernel} is defined by the equation
\begin{equation} \label{eq:smth-avg-ker}
	\int \tensor{(K_{\eta})}{^{K}_{J}}(x, y) \psi_{K}(x) \, \ud x
	= \int \int_{0}^{1} \sum_{\gmm} \tensor{S}{_{J}^{(\gmm, K)}}(y, y_{1}, s) \eta(y, y_{1}) \rd_{x}^{\gmm} \psi_{K}(\bfx(y, y_{1}, s)) \, \ud s \, \ud y_{1},
\end{equation}
which may be (somewhat informally) also written as $K_{\eta}(x, y) = \int K_{y_{1}} (x, y) \eta(y, y_{1}) \, \ud y_{1}$. Note that, in view of $\supp \eta \subseteq W$, the right-hand side of \eqref{eq:smth-avg-ker} is well-defined for every $y \in \bbR^{d}$ and $\varphi \in C^{\infty}_{c}(\bbR^{d})$, although it is trivial unless $y \in \set{y \in \bbR^{d} : \exists y_{1} \in \bbR^{d} \hbox{ s.t. } (y, y_{1}) \in W}$.

\begin{theorem}[Smoothly averaged integral kernels] \label{thm:smth-ker}
Let $\bfx : W \times [0, 1] \to \bbR^{d}$ and $\eta : W \to \bbR$ satisfy \ref{hyp:x-1}--\ref{hyp:x-2} and \ref{hyp:eta-1}--\ref{hyp:eta-3}, respectively, and assume that $\calP$ and $\tensor{S}{_{J}^{(\gmm, K)}}(y, y_{1}, s)$ satisfy \ref{hyp:crc-x-q}. Assume also that, either
\begin{enumerate}
\item $\br{U} \cap \br{\cup_{y \in U} \supp_{y_{1}} \eta(y, y_{1})} = \0$; or
\item $\calP$, $\tensor{S}{_{J}^{(\gmm, K)}}(y, y_{1}, s)$ and $\tensor{(b_{y_{1}})}{^{J'}_{J}}(\cdot, y)$ satisfy \ref{hyp:b-y1-q}.
\end{enumerate}
Then the integral kernel $K_{\eta} = \tensor{(K_{\eta})}{^{K}_{J}}$ defined by \eqref{eq:smth-avg-ker} satisfies
\begin{equation*}
	\calP K_{\eta}(\cdot, y) = \dlt_{0}(\cdot-y) - b_{\eta}(\cdot, y) \quad \hbox{ on } U,
\end{equation*}
where $\tensor{(b_{\eta})}{^{J'}_{J}}(x, y) = 0$ in Case~(1), and in Case~(2), $\tensor{(b_{\eta})}{^{J'}_{J}}(x, y)$ is a distribution on $\bbR^{d}_{x} \times \bbR^{d}_{y}$ defined by
\begin{align*}
\tensor{(b_{\eta})}{^{J'}_{J}}(x, y) &\ceq \sum_{\bfA \in \calA} \int (Z^{\bfA})_{J}(y, y_{1}) \eta(y, y_{1}) (g_{\bfA})^{J'}(x, y_{1}) \, \ud y_{1} \\
&= \sum_{\bfA \in \calA} \sum_{\alp} (-1)^{\abs{\alp}} \rd_{x}^{\alp} \left( c[g_{\bfA}]^{(\alp, J')}(x) (Z^{\bfA})_{J}(y, x) \eta(y, x) \right),
\end{align*}
which moreover satisfies
\begin{align*}
	\abs{L_{y}^{\abs{\alp}} \abs{x-y}^{\abs{\bt}} (\rd_{y} + \rd_{x})^{\alp} \rd_{x}^{\bt} \tensor{(b_{\eta})}{^{J'}_{J}}(x, y)} & \leq C_{\# \calA, m_{K}+m'_{K}}  A_{g} A_{Z} A_{\eta} \quad \hbox{ for all } x \neq y, \\
	\supp_{x} \tensor{(b_{\eta})}{^{J'}_{J}}(x, y) &\subseteq \supp_{y_{1}} \eta(y, y_{1}), \quad \hbox{ for all } y \in \bbR^{d},
\end{align*}
where $\abs{\alp} + \abs{\bt} \leq \min\set{M_{\bfx}, M_{Z}, M_{g}, M_{\eta}} - C_{m_{K}, m_{K}'}$.

In both cases, the integral kernel $K_{\eta}(x, y)$ is a locally integrable function on $\bbR^{d} \times \bbR^{d}$ satisfying
\begin{align*}
	\abs{\tensor{(K_{\eta})}{^{K}_{J}}(x, y)} &\leq C_{m_{K}, m_{K}'} A_{S} A_{\eta} (1+A_{\bfx})^{2(d+m_{K'})} \abs{x-y}^{-d+m_{K}} \quad \hbox{ for all } x \neq y, \\
	\supp \tensor{(K_{\eta})}{^{K}_{J}}(\cdot, y) &\subseteq \br{\bigcup_{y_{1} : (y, y_{1}) \in \supp \eta} \bfx(y, y_{1}, [0, 1])}, \\
	\supp \tensor{(K_{\eta})}{^{K}_{J}}(x, \cdot) &\subseteq \br{\set{y \in \bbR^{d} : \exists y_{1} \in \bbR^{d} \hbox{ s.t. } (y, y_{1}) \in \supp \eta, \, x \in \bfx(y, y_{1}, [0, 1])}},
\end{align*}
and the symbol $\tensor{(a_{\eta})}{^{K}_{J}}(\xi, y) = \int \tensor{(K_{\eta})}{^{K}_{J}}(y+z, y) e^{- i \xi \cdot z} \, \ud z$ obeys
\begin{align*}
	\abs{\rd_{y}^{\alp} \rd_{\xi}^{\bt} \tensor{(a_{\eta})}{^{K}_{J}}(\xi, y)} \leq C_{m_{K}, m_{K}', \alp, \bt} A_{S} A_{\eta} (1 + A_{\bfx})^{10 (d+\abs{\alp}+\abs{\bt}+m_{K}+m_{K}')} L_{y}^{-\abs{\alp}} (((1+A_{\bfx}) L_{\eta})^{-1} + \abs{\xi})^{-m-\abs{\bt}} 
\end{align*}
where $y \in \bbR^{d}$, $\abs{\alp} + \abs{\bt} \leq \min \set{M_{\bfx}, M_{S}, M_{\eta}} - C_{m_{K}, m_{K}'}$.
\end{theorem}

\begin{proof}
The expressions for $K_{\eta}$ and $b_{\eta}$, as well as the estimates for $b_{\eta}(x, y)$, follow from the defining relation \eqref{eq:crc-x-trunc} and hypothesis \ref{hyp:b-y1-q}. To obtain the claimed properties of $K_{\eta}$ and $a_{\eta}$, we apply Proposition~\ref{prop:sio} with $\uW = \set{(y, z_{1}) \in \bbR^{d} \times \bbR^{d} : (y, y + z_{1}) \in W}$, $z(y, z_{1}, s) = \bfx(y, z_{1} + y, s) - y$, $\uk_{\gmm}(y, z_{1}, s) = \tensor{S}{_{J}^{(\gmm, K)}}(y, z_{1} + y, s)$, and $\ueta(y, z_{1}) = \eta(y, z_{1} + y)$. Indeed, observe that \ref{hyp:x-1}--\ref{hyp:x-2}, \ref{hyp:crc-x-q} and \ref{hyp:eta-1}--\ref{hyp:eta-3} imply that the assumptions in Section~\ref{subsec:sio-hyp} are all satisfied (with appropriate parameters). \qedhere
\end{proof}

\subsubsection{Generalization of the conic operator and proof of Theorem~\ref{thm:main-summary-conic}} \label{subsubsec:conic}
From Theorem~\ref{thm:smth-ker} we immediately obtain the following result, which generalizes the construction of conic solution operators. 

\begin{theorem} [Full solution operator and Friedrich-type inequality, non-compact support] \label{thm:full-sol-conic}
Assume that Case~(1) of the hypothesis of Theorem~\ref{thm:smth-ker} holds, i.e.,
\begin{equation} \label{eq:U-nontrapping}
	\br{U} \cap \br{\bigcup_{y \in U} \supp_{y_{1}} \eta(y, y_{1})} = \0.
\end{equation}
Then the following statements hold.
\begin{enumerate}
\item {\bf Full solution operator.} For every $1 < p < \infty$ and $\abs{s} < \min\set{M_{\bfx}, M_{S}, M_{\eta}} - C_{d, m_{K}, m_{K}'}$, the operator $\tensor{(\calS_{\eta})}{^{K}_{J}}$ with integral kernel $\tensor{(K_{\eta})}{^{K}_{J}}$ defines a bounded operator $\td{W}^{s, p}(U) \to W^{s+m_{K}, p}(U)$ with
\begin{equation*}
\calP \calS_{\eta} f = f \quad \hbox{ for all } f \in \td{W}^{s, p}(U).
\end{equation*}

\item {\bf Representation formula and Poincar\'e-type inequality.} For every $1 < p < \infty$ and $\abs{s} < \min\set{M_{\bfx}, M_{S}, M_{\eta}} - C_{d, m_{K}, m_{K}'}$, we have the representation formula
\begin{equation*}
	\varphi_{J} = \tensor{(\calS_{\eta}^{\ast})}{_{J}^{K}} (\calP^{\ast} \varphi)_{K} \quad \hbox{ for all } \varphi \in \td{W}^{-s, p'}(U),
\end{equation*}
where $\calP^{\ast} \varphi$ is defined in the sense of $\calD'(\bbR^{d})$, so that $(\calP^{\ast} \varphi)_{K} \in \td{W}^{-s-m_{K}, p'}(U)$. Moreover, we have the Friedrich-type inequality
\begin{equation*}
	\nrm{\varphi}_{W^{-s, p'}(U)} \aleq \sum_{K} \nrm{(\calP^{\ast} \varphi)_{K}}_{\td{W}^{-s-m_{K}, p'}(U)} \quad \hbox{ for all } \varphi \in \td{W}^{-s, p'}(U).
\end{equation*}

\item {\bf Cokernel in $\td{W}^{-s, p'}$.}
For any open subset $V$ such that $\br{V} \subseteq U$, if $\bfZ \in \td{W}^{-s, p'}(V)$ with $\calP^{\ast} \bfZ = 0$ in $\calD'(U)$, then $\bfZ = 0$; in short,
\begin{equation*}
	\ker_{\td{W}^{-s, p'}(V)} \calP^{\ast} = \set{0}.
\end{equation*}
\end{enumerate}
Here, $\nrm{\tensor{(\calS_{\eta})}{^{K}_{J}}}_{\td{W}^{s, p}(U) \to W^{s+m_{K}, p}(U)}$, $\nrm{\tensor{(\calS_{\eta}^{\ast})}{_{J}^{K}}}_{\td{W}^{-s-m_{K}, p'}(U) \to W^{-s, p'}(U)}$ and the implicit constant in (2) are all bounded by $C_{d, m_{K}, m_{K}', s, p} A_{S} A_{\eta} (1 + A_{\bfx})^{C (d + m_{K} + m_{K'} + s)}$.
\end{theorem}

Note carefully that even for $f \in C_c^{\infty}(U)$ vanishing near $\rd U$, the theorem does \emph{not} ensure that $\calS_{\eta} f$ vanishes near $\rd U$; correspondingly, the representation formula require $\varphi \in \td{W}^{s, p}(U)$, leading to a Friedrich-type inequality for $\calP^{\ast}$. In order to ensure that $\calS_{\eta} f$ vanishes near $\rd U$ (correspondingly, to prove a representation formula and Poincar\'e-type inequality for $\varphi$ in $W^{s, p}(U)$), we need to take into account the cokernel of $\calP$ in $W^{s, p}(U)$ into account, as we will in Section~\ref{subsec:cokernel} below.

\begin{proof}
Part~(1) is an immediate consequence of Theorem~\ref{thm:smth-ker}, Case~(1). Indeed, the bounds for the symbol $a_{\eta}$ in Theorem~\ref{thm:smth-ker} implies that
\begin{equation*}
\nrm{\tensor{(\calS_{\eta})}{^{K}_{J}}}_{W^{s, p} \to W^{s+m_{K}, p}} \leq C_{d, m_{K}, m_{K}', s, p} A_{S} A_{\eta} (1 + A_{\bfx})^{C (d + m_{K} + m_{K'} + s)}
\end{equation*}
by the standard theory of pseudodifferential operators \cite{Stein}. Restricting to inputs $f \in \td{W}^{s, p}(U)$, and composing with the surjection $W^{s+m_{K}, p} \to W^{s+m_{K}, p}(U)$, the same bound for $\nrm{\tensor{(\calS_{\eta})}{^{K}_{J}}}_{\td{W}^{s, p} \to W^{s+m_{K}, p}(U)}$ follows.

To prove Part~(2), note that we have, by duality \eqref{eq:sob-duality},
\begin{equation*}
\nrm{\tensor{(\calS_{\eta}^{\ast})}{_{J}^{K}}}_{\td{W}^{-s - m_{K}, p'}(U) \to W^{-s, p'}(U)} = \nrm{\tensor{(\calS_{\eta})}{^{K}_{J}}}_{W^{s, p} \to W^{s+m_{K}, p}} \leq C_{d, m_{K}, m_{K}', s, p} A_{S} A_{\eta} (1 + A_{\bfx})^{C (d + m_{K} + m_{K'} + s)}.
\end{equation*}
Moreover, $\calP K_{\eta}(\cdot, y) = \dlt_{0}(\cdot - y)$ for $y \in U$ is equivalent to
\begin{equation*}
	\varphi = \calS_{\eta}^{\ast} \calP^{\ast} \varphi \hbox{ for all } \varphi \in C^{\infty}_{c}(U),
\end{equation*}
which extends to all $\varphi \in \td{W}^{-s, p'}(U)$ by approximation. The Friedrich-type inequality for $\nrm{\vphi}_{W^{-s, p'}(U)}$ now follows. 

For $\bfZ \in \td{W}^{s, p}(U)$ with $\calP^{\ast} \bfZ = 0$ in $\calD'(U)$, Part~(2) only implies that $\bfZ$ corresponds to a zero element in $W^{s, p}(U)$, which still leaves the possibility that $\bfZ \neq 0$ as an element of $\td{W}^{s, p}(U)$ (i.e., it may happen that $\supp \bfZ \subseteq \rd U$; see the discussion in Section~\ref{subsec:ftn-sp}). However, since we assume in addition that $\bfZ \in \td{W}^{s, p}(V)$ for an open subset $V$ with $\br{V} \subseteq U$, $\bfZ = 0$ in $W^{s, p}(U)$ is sufficient to conclude that $\bfZ = 0$ in $\td{W}^{s, p}(V)$, which proves Part~(3). \qedhere
\end{proof}

\begin{example} \label{ex:div-full-sol-conic}
For $\calP u = (\rd_{j} + B_{j}) u^{j}$, $\bfx = \ul{\bfx}$ (straight line segments), and a suitable choice of $\eta$, Theorem~\ref{thm:full-sol-conic} specializes to the conic integral kernel $K_{\slashed{\eta}}(x, y)$ in Section~\ref{subsec:ideas-div}, Step~2. Indeed, given $\slashed{\eta} \in C^{\infty}(\bbS^{d-1})$ with $\int \slashed{\eta} \, \ud S = 1$, take $\eta(y, y_{1}) = \psi(y_{1} - y) \slashed{\eta}(\frac{y_{1}-y}{\abs{y_{1} - y}})$, where $\psi$ is as in Example~\ref{ex:conic-bog-sio}. Then \ref{hyp:eta-1}--\ref{hyp:eta-3} are clearly satisfied (with $L_{\eta} = R_{\ueta}$, which can be taken to be arbitrarily large), and $K_{\eta}(x, y) = K_{\slashed{\eta}}(x, y)$ for $\abs{x-y} \leq \frac{1}{2} R_{\ueta}$.
\end{example}

We are now ready to give a precise formulation and proof of Theorem~\ref{thm:main-summary-conic}.
\begin{proof}[Precise formulation and proof of Theorem~\ref{thm:main-summary-conic}]
To make Theorem~\ref{thm:main-summary-conic} precise, we replace ``\ref{hyp:crc}'' and ``an admissibility family of curves $\bfx$'' in the statement of the theorem by \ref{hyp:crc-x-q} and \ref{hyp:x-1}--\ref{hyp:x-2}, respectively, on an open subset $\td{U}$ that contains $\br{U}$. 

Then Parts~(2) and (3), except for the Friedrich-type inequality, follow from Theorem~\ref{thm:full-sol-conic} by simply choosing $\eta(y, y_{1}) = \eta_{1}(y_{1})$ with $\eta_{1} \in C^{\infty}_{c}(U_{1})$ satisfying $\int \eta_{1}(y_{1}) \, \ud y_{1} = 1$. 
To prove the Friedrich-type inequality, we apply Theorem~\ref{thm:main-summary-conic}.(2) with $U$ replaced by $\td{U}$ (the larger open subset that contains $\br{U}$ on which the hypotheses hold), and observe that $\nrm{\vphi}_{\td{W}^{-s, p'}(U)} \aleq \nrm{\vphi}_{W^{-s, p'}(\td{U})}$ while $\nrm{(\calP^{\ast} \vphi)_{K}}_{\td{W}^{-s-m_{K}, p'}(\td{U})} \leq \nrm{(\calP^{\ast} \vphi)_{K}}_{\td{W}^{-s-m_{K}, p'}(U)}$. Similarly, for Part~(1), i.e., the triviality of $\ker_{\td{W}^{-s, p'}(U)} \calP^{\ast}$, we apply Theorem~\ref{thm:main-summary-conic}.(3) with $U$ and $V$ replaced by $\td{U}$ (the larger open subset that contains $\br{U}$ on which the hypotheses hold) and $U$, respectively. \qedhere
\end{proof}

\subsection{Solutions with compact support} \label{subsec:cokernel}
Finally, we carry out the procedure outlined in Step~3 of Section~\ref{subsec:method} for constructing an operator that produces solutions $u$ to $\calP u = f$ with $\supp u$ compact (provided, of course, that $f$ has compact support).

\subsubsection{Local regularity conditions for $Z^{\bfA}$ and $\eta$, and a simplifying assumption} \label{subsubsec:Z-eta-smth}

Given an extra parameter $L_{b} > 0$ (with the dimension of length) and $M_{Z}', M_{\eta}' \in \bbZ_{\geq 0}$, we {\bf assume \ref{hyp:b-y1-smth} holds} for $Z^{\bfA}(y, y_{1})$. Furthermore, for simplicity, we will make the following smoothness assumption:
\begin{enumerate}[label=($C^{\infty}$)]
\item {\bf Smoothness assumption.} \label{hyp:smth} $U$, $W$, $\bfx$, $\calP$, $\tensor{S}{_{J}^{(\gmm, K)}}$ and $\tensor{(b_{y_{1}})}{^{J'}_{J}}$ satisfy \ref{hyp:x-1}--\ref{hyp:x-3}, \ref{hyp:eta-1}--\ref{hyp:eta-smth}, \ref{hyp:crc-x-q}, \ref{hyp:b-y1-q}--\ref{hyp:b-y1-smth} for arbitrary $M_{\bfx}, M_{\eta}, M_{\eta}', M_{S}, M_{Z}, M_{Z}', M_{g}$ (where the constants $A_{\bfx}$, $A'_{\bfx}$ $A_{\eta}$, $A_{S}$, $A_{Z}$, $A'_{Z}$ and $A_{g}$ may depend on $M_{\bfx}, M_{\eta}, M_{\eta}', M_{S}, M_{Z}, M_{Z}', M_{g}$). 
\end{enumerate}
This simplifying assumption allows us not to worry about the number of derivatives in the ensuing discussion. Without \ref{hyp:smth}, the proofs below may be modified in a straightforward manner to cover the appropriate finite regularity case.

\subsubsection{$(Z^{A})_{A \in \calA}$, $\ker \calP^{\ast}$ and generalization of the Bogovskii operator} \label{subsubsec:bogovskii}
We define the \emph{formal kernel} of $\calP^{\ast}$ (or equivalently, \emph{formal cokernel} of $\calP$) on $U$ to be
\begin{equation*}
\ker \calP^{\ast} \ceq \set{\bfZ \in C^{\infty}(\br{U}) : \calP^{\ast} \bfZ = 0 \hbox{ in } U}.
\end{equation*}

\begin{lemma} \label{lem:b-y1-coker}
Assume that $U$, $W$, $\bfx$, $\calP$, $S_{\gmm}$ and $b_{y_{1}}$ satisfy \ref{hyp:smth}\footnote{We point out that there is no need for any assumptions on $\eta$ for this lemma.}.
\begin{enumerate}
\item For every $\bfZ \in \ker \calP^{\ast}$ and $y, y_{1} \in U$ such that $\bfx(y, y_{1}, [0, 1]) \subseteq U$,
we have
\begin{align*}
\brk{b_{y_{1}}(\cdot, y), \bfZ} = \bfZ(y).
\end{align*}
In particular, we have
\begin{equation*}
\dim \ker \calP^{\ast} \leq \# \calA.
\end{equation*}

\item If $\dim \ker \calP^{\ast} = \# \calA$, then for every basis $\set{\bfZ^{\bfA}}_{\bfA \in \calA}$ of $\ker \calP^{\ast}$ and $y, y_{1} \in U$ such that $\bfx(y, y_{1}, [0, 1]) \subseteq U$, we have the decomposition
\begin{equation*}
	\tensor{(b_{y_{1}})}{^{J'}_{J}}(x, y) = \sum_{\bfA \in \calA} \tensor{(\bfZ^{\bfA})}{_{J}}(y) \tensor{(\zt_{\bfA})}{^{J'}}(x, y_{1}).
\end{equation*}
Here, $\zt_{\bfA}$ and $g_{\bfA}$ in \ref{hyp:b-y1-q} are related by the identity
\begin{equation*}
	\brk{\tensor{(g_{\bfA})}{^{J'}}(\cdot, y_{1}), \tensor{(\bfZ^{\bfA'})}{_{J'}}} \tensor{(\zt_{\bfA'})}{^{J}}(x, y_{1}) = \tensor{(g_{\bfA})}{^{J}}(x, y_{1}).
\end{equation*}
\end{enumerate}
\end{lemma}
\begin{proof}
Part~(1) follows by testing $\bfZ(x) \in C^{\infty}(\br{U})$ against \eqref{eq:trun-ker-y1-gen}, which is possible thanks to \eqref{eq:K-y1-supp} and $\bfx(y, y_{1}, [0, 1]) \subseteq U$. The claim $\dim \ker \calP^{\ast} \leq \# \calA$ then follows. If $\dim \ker \calP^{\ast} = \# \calA$, then given a basis $\set{\bfZ^{\bfA}}_{\bfA \in \calA}$ of $\ker \calP^{\ast}$, Part~(1) implies that
\begin{align*}
	\sum_{\bfA'} Z^{\bfA'}(y, y_{1}) \brk{g_{\bfA'}(\cdot, y_{1}), \bfZ^{\bfA}} = \bfZ^{\bfA}(y)
\end{align*}
for every $\bfA \in \calA$. It follows that the square matrix $\brk{g_{\bfA'}(\cdot, y_{1}), \bfZ^{\bfA}}$ is full rank, and therefore is invertible. \qedhere
\end{proof}

From Theorem~\ref{thm:smth-ker} and Lemma~\ref{lem:b-y1-coker}, we already obtain the following construction of the solution operator in an important special case, namely, when $\dim \ker \calP^{\ast} = \# \calA$. Motivated by Proposition~\ref{prop:zero-curv} and Definition~\ref{def:aug-int}, we call this the \emph{completely integrable} case. This procedure generalizes the construction of the Bogovskii operator for the divergence operator on $\bbR^{d}$.
\begin{theorem} [Full solution operator and Poincar\'e-type inequality, completely integrable case] \label{thm:full-sol-max}
Assume that $U$, $W$, $\bfx$, $\eta$, $\calP$, $S$ and $b_{y_{1}}$ satisfy \ref{hyp:smth}. Assume furthermore that
\begin{equation*}
	\eta = \eta(y_{1}),
\end{equation*}
and that $\br{U}$ is \emph{$\bfx$-star-shaped} with respect to $\supp \eta$, i.e.,
\begin{equation*}
	\bigcup_{y \in \br{U}, \, y_{1} \in \supp \eta} \bfx(y, y_{1}, [0, 1]) \subseteq \br{U}.
\end{equation*}
If the formal cokernel of $\calP$ has the maximal dimension, i.e.,
\begin{equation*}
\dim \ker \calP^{\ast} = \# \calA,
\end{equation*}
then the following holds.
\begin{enumerate}
\item {\bf Full solution operator.} The operator $\calS_{\eta}$ in Theorem~\ref{thm:smth-ker} 
defines a bounded operator $\tensor{(\calS_{\eta})}{^{K}_{J}} : \td{W}^{s, p}(U) \to \td{W}^{s+m_{K}, p}(U)$ for every $1 < p < \infty$ and $s \in \bbR$. Moreover, we have
\begin{align*}
	\calP \calS_{\eta} f = f - \sum_{\bfA \in \calA} (\zt_{\eta})_{\bfA}  \brk{\bfZ^{\bfA}, f} \quad \hbox{ for all } f \in \td{W}^{s, p}(U),
\end{align*}
where $(\zt_{\eta})_{\bfA}$ is a smooth function with $\supp (\zt_{\eta})_{\bfA} \subseteq \supp \eta$ characterized by
\begin{equation*}
	\brk{(\zt_{\eta})_{\bfA}, \varphi} = \brk{\zt_{\bfA}(x, y_{1}), \varphi(x) \eta(y_{1})} \quad \hbox{ for all } \varphi \in C^{\infty}_{c}(U),
\end{equation*}
with $\zt_{\bfA}$ and $\bfZ^{\bfA}$ as in Lemma~\ref{lem:b-y1-coker}. We also have the following support property:
\begin{equation*}
	\supp \calS_{\eta} f \subseteq \br{\bigcup_{y \in \supp f, \, y_{1} \in \supp \eta} \bfx(y, y_{1}, [0, 1])} \subseteq \br{U}.
\end{equation*}

\item {\bf Representation formula and Poincar\'e-type inequality.} For $1 < p < \infty$ and $s \in \bbR$, we have the representation formula
\begin{equation*}
	\varphi_{J} = \tensor{(\calS_{\eta}^{\ast})}{_{J}^{K}} (\calP^{\ast} \varphi)_{K} + \sum_{\bfA \in \calA} \brk{(\zt_{\eta})_{\bfA}, \varphi} \bfZ^{\bfA} \quad \hbox{ for all } \varphi \in W^{-s, p'}(U),
\end{equation*}
where $\calP^{\ast} \varphi$ is defined in the sense of $\calD'(U)$, so that $(\calP^{\ast} \varphi)_{K} \in W^{-s-m_{K}, p'}(U)$.
Moreover, we have the Poincar\'e-type inequality
\begin{equation*}
	\nrm*{\varphi - \sum_{\bfA \in \calA} \brk{(\zt_{\eta})_{\bfA}, \varphi} \bfZ^{\bfA} }_{W^{-s, p'}(U)} \aleq \sum_{K} \nrm{(\calP^{\ast} \varphi)_{K}}_{W^{-s - m_{K}, p'}(U)} \hbox{ for all } \varphi \in W^{-s, p'}(U).
\end{equation*}

\item {\bf Cokernel in $W^{-s, p'}(U)$.}
If $\bfZ \in \td{W}^{-s, p'}(U)$ with $\calP^{\ast} \bfZ = 0$ in $\calD'(U)$, then $\bfZ \in \ker \calP^{\ast}$; in short,
\begin{equation*}
	\ker_{\td{W}^{-s, p'}(U)} \calP^{\ast} = \ker \calP^{\ast}.
\end{equation*}
\end{enumerate}
Here, $\nrm{\tensor{(\calS_{\eta})}{^{K}_{J}}}_{\td{W}^{s, p}(U) \to W^{s+m_{K}, p}(U)}$, $\nrm{\tensor{(\calS_{\eta}^{\ast})}{_{J}^{K}}}_{\td{W}^{-s-m_{K}, p'}(U) \to W^{-s, p'}(U)}$ and the implicit constant in (2) are all bounded by $C_{d, m_{K}, m_{K}', s, p} A_{S} A_{\eta} (1 + A_{\bfx})^{C (d + m_{K} + m_{K'} + s)}$, with $A_{S}$, $A_{\eta}$, $A_{\bfx}$ corresponding to $M_{\bfx}, M_{\eta}, M_{S} \geq s + C_{d, m_{K}, m_{K'}}$.
\end{theorem}
In particular, in Part~(1), $\calP \calS_{\eta} f = f$ if $f \in \td{W}^{s, p}(U)$ is orthogonal to $\ker \calP^{\ast}$, and in Part~(2), we have $\nrm*{\varphi}_{W^{-s, p'}(U)} \aleq \sum_{K} \nrm{(\calP^{\ast} \varphi)_{K}}_{W^{-s-m_{K}, p'}(U)}$ if $\varphi \in W^{-s, p'}(U)$ is orthogonal to $((\zt_{\eta})_{\bfA})_{\bfA \in \calA}$.

\begin{proof}
Part~(1) is an immediate consequence of Lemma~\ref{lem:b-y1-coker} and Theorem~\ref{thm:smth-ker}, Case~(2). We remark that the operator bounds on $\tensor{(\calS_{\eta})}{^{K}_{J}}$ are obtained as sketched in the proof of Theorem~\ref{thm:full-sol-conic}, but we have the additional mapping property $C^{\infty}_{c}(\bbR^{d}) \to C^{\infty}_{c}(\bbR^{d})$ in view of the support property of $K_{\eta}(\cdot, y)$; by completion, we obtain $\td{W}^{s, p}(U) \to \td{W}^{s+m_{K}, p}(U)$. We also remark that the smoothness of $\zt_{\eta}$ follows from the algebraic formulas in Lemma~\ref{lem:b-y1-coker} and the structure of $g_{\bfA}$ in \ref{hyp:b-y1-q}.  

To prove Part~(2), the key observation is that $\calS_{\eta}^{\ast} \calP^{\ast} \varphi$ is well-defined for any $\varphi \in W^{-s, p'}(U)$ thanks to the obvious mapping property $(\calP^{\ast} \varphi)_{K} \in W^{-s-m_{K}, p'}$ and the duality property \eqref{eq:sob-duality}. The desired identity and inequality are now immediate consequences of Part~(1).

Finally, Part~(3) is an immediate consequence of Part~(2): indeed, if $\bfZ \in \td{W}^{-s, p'}(U)$ with $\calP^{\ast} \bfZ = 0$, then by Part~(2), we have $\bfZ = \sum_{\bfA \in \calA} \brk{(\zt_{\eta})_{\bfA}, \bfZ} \bfZ^{\bfA} \in \ker \calP^{\ast}$.
\qedhere
\end{proof}

\begin{example} \label{ex:div-full-sol-max}
For $\calP u = (\rd_{j} + B_{j}) u^{j}$, $\bfx = \ul{\bfx}$ (straight line segments), and $\eta(y, y_{1}) = \eta_{1}(y_{1})$, $\calS_{\eta}$ in Theorem~\ref{thm:full-sol-max} specializes to the Bogovskii solution operator $\calS_{\eta_{1}}(x, y)$ in Section~\ref{subsec:ideas-div}, Step~3 (completely integrable case), with $\# \calA = 1$, $\bfZ = e^{z}$, and $\zt_{\eta} = \bfZ^{-1} \eta_{1}$.
\end{example}

\subsubsection{Soft arguments for the general case and proof of Theorem~\ref{thm:main-summary-bogovskii}} \label{subsubsec:full-sol}
We introduce the following assumption (in addition to the objects that have been already introduced), which is a slight generalization of one of the hypotheses of Theorem~\ref{thm:full-sol-max}:
\begin{enumerate}[label=($\br{U}\star$),]
\item {\bf $\br{U}$ is $\bfx$-star-shaped with respect to $\supp \eta$.} \label{hyp:U-star}
We have
\begin{equation*}
	\br{\bigcup_{(y, y_{1}) \in \supp \eta} \bfx(y, y_{1}, [0, 1])} \subseteq \br{U}.
\end{equation*}
\end{enumerate}

In view of Theorem~\ref{thm:smth-ker}, this assumption ensures that the integral kernel $K_{\eta}(\cdot, y)$ is supported in $\br{U}$ for $y \in \br{U}$.

We begin with a more detailed study of the properties of $\calB_{\eta}$ under our assumptions (in particular, \ref{hyp:b-y1-smth}).
\begin{lemma} \label{lem:B-eta}
Assume that $U$, $W$, $\bfx$, $\eta$, $\calP$, $S$ and $b_{y_{1}}$ satisfy \ref{hyp:smth}. Assume furthermore that \ref{hyp:U-star} is satisfied for $U$, $\bfx$ and $\eta$.
\begin{enumerate}
\item {\bf $\calB_{\eta}$ is smoothing.} The integral kernel $\tensor{(b_{\eta})}{^{J'}_{J}}(x, y)$ for $\calB_{\eta}$ is smooth and satisfies
\begin{align*}
	\abs{L_{y}^{\abs{\alp}} L_{b}^{\abs{\bt}} (\rd_{y} + \rd_{x})^{\alp} \rd_{x}^{\bt} \tensor{(b_{\eta})}{^{J'}_{J}}(x, y)} & \leq C_{\# \calA, m_{K}+m_{K}'}  A_{g} A_{Z} A_{\eta}, \\
	\supp_{x} \tensor{(b_{\eta})}{^{J'}_{J}}(x, y) &\subseteq \supp_{y_{1}} \eta(y, y_{1}), \quad \hbox{ for all } y \in U,
\end{align*}
where $A_{g}$, $A_{Z}$, $A_{\eta}$ correspond to $M_{g}, M_{Z}, M_{\eta}, M_{Z}', M_{\eta}' \geq \abs{\alp} + \abs{\bt} + C_{d, m+m_{1}}$.

\item {\bf Approximation of $\calB_{\eta}$ by finite rank operators.} Assume, in addition, that $U$ is a bounded open subset of $\bbR^{d}$. For every $\eps_{0} > 0$ and an open set $W_{0}$ such that
\begin{equation*}
\supp b_{\eta} \subseteq W_{0} \subseteq \br{W_{0}} \subseteq U \times U,
\end{equation*}
there exists a linear operator ${}^{(\eps_{0})} \calE_{0}$ with a smooth integral kernel ${}^{(\eps_{0})} e_{0} : U \times U \to \bbC^{r_{0} \times r_{0}}$, as well as an index set ${}^{(\eps_{0})} \td{\bfA}$ and smooth functions ${}^{(\eps_{0})} (g_{0})_{\td{\bfA}} : U \to \bbC^{r_{0}}$ and ${}^{(\eps_{0})} Z_{0}^{\td{\bfA}} : U \to \bbC^{r_{0}}$ such that, for all $f \in C^{\infty}_{c}(U)$,
\begin{align*}
	\calP \calS_{\eta} f = f - {}^{(\eps_{0})}\calE_{0} f - \sum_{\td{\bfA} \in {}^{(\eps_{0})}\td{\calA}} {}^{(\eps_{0})}(g_{0})_{\td{\bfA}}\brk{{}^{(\eps_{0})} Z_{0}^{\td{\bfA}}, f}.
\end{align*}
Moreover, it can be arranged so that kernel ${}^{(\eps_{0})} e_{0}$ satisfies
\begin{gather}
	\sup_{x \in U} \int_{U} \abs*{{}^{(\eps_{0})} e_{0}(x, y)} \, \ud y
	+ \sup_{y \in U} \int_{U} \abs*{{}^{(\eps_{0})} e_{0}(x, y)} \, \ud x \leq \eps_{0}, \label{eq:e0-eps} \\
	\abs{\rd_{x}^{\alp} \rd_{y}^{\bt} ({}^{(\eps_{0})} e_{0}(x, y))}\leq C_{d, \eps_{0}, U, \supp b_{\eta}} A_{g} A_{Z} A_{\eta}, \notag \\
	\supp {}^{(\eps_{0})} e_{0}(x, y) \subseteq W_{0}; \notag
\end{gather}
the functions ${}^{(\eps_{0})} (g_{0})_{\td{\bfA}}$, ${}^{(\eps_{0})} Z_{0}^{\td{\bfA}}$ satisfy
\begin{gather*}
	\abs*{\rd_{x}^{\alp} ({}^{(\eps_{0})}(g_{0})_{\td{\bfA}})(x)} \leq C_{d, \eps_{0}, U, \supp b_{\eta}} A_{g} A_{\eta},
	\quad \abs*{\rd_{y}^{\bt} ({}^{(\eps_{0})}Z_{0}^{\td{\bfA}})(y)} \leq C_{d, \eps_{0}, U, \supp b_{\eta}} A_{Z}, \\
	\supp \left( \sum_{\td{\bfA} \in {}^{(\eps_{0})} \td{\calA}} {}^{(\eps_{0})}(g_{0})_{\td{\bfA}}(x) {}^{(\eps_{0})} Z_{0}^{\td{\bfA}}(y) \right) \subseteq W_{0}.
\end{gather*}
and $\sharp ({}^{(\eps_{0})} \td{\calA})$ is bounded by $C_{d, \eps_{0}, U, \supp b_{\eta}}$. Here, $A_{g}$, $A_{Z}$, $A_{\eta}$ correspond to $M_{g}, M_{Z}, M_{\eta}, M_{Z}', M_{\eta}' \geq \abs{\alp} + \abs{\bt} + C_{d, m_{K}+m_{K}'}$
\end{enumerate}

\end{lemma}
\begin{proof}
Recall from Theorem~\ref{thm:smth-ker} that
\begin{align*}
b_{\eta}(x, y) = \sum_{\bfA \in \calA} \sum_{\alp} (-1)^{\abs{\alp}} \rd_{x}^{\alp} \left( c[g_{\bfA}]^{(\alp, J')}(x) (Z^{\bfA})_{J}(y, x) \eta(y, x) \right).
\end{align*}
Part~(1) immediately follows from this expression, \ref{hyp:b-y1-q}, \ref{hyp:b-y1-smth}, \ref{hyp:eta-3} and \ref{hyp:eta-smth} (which are a part of \ref{hyp:smth}). Part~(2) is, of course, a direct consequence of Part~(1). In what follows, we describe a construction of ${}^{(\eps_{0})} \calE_{0}$, ${}^{(\eps_{0})} (g_{0})_{\td{\bfA}}$ and ${}^{(\eps_{0})} Z_{0}^{\td{\bfA}}$.

To ease the notation, we suppress the superscript ${}^{(\eps_{0})}$ in what follows. Let $(\chi_{x_{\bfG}}(x)\chi_{y_{\bfG}}(y))_{\bfG \in \calG_{0}}$ be a finite partition of unity on $\supp b_{\eta}$ but vanishing outside of $W_{0}$ which will be fixed at the end of the construction (see also Remark~\ref{rem:B-eta} below). We have
\begin{align*}
\tensor{(b_{\eta})}{^{J'}_{J}}(x, y)
&= \sum_{\bfG \in \calG_{0}} \sum_{\bfA \in \calA} \sum_{\alp} (-1)^{\abs{\alp}} \rd_{x}^{\alp} \left( c[g_{\bfA}]^{(\alp, J')}(x) (Z^{\bfA})_{J}(y, x) \eta(y, x) \chi_{x_{\bfG}}(x) \chi_{y_{\bfG}}(y) \right).
\end{align*}
Defining $\td{\calA} \ceq \calA \times \bfG_{0}$ and
\begin{align*}
\tensor{(e_{0})}{^{J'}_{J}}(x, y) &\ceq \sum_{\td{\bfA} \in \td{\calA}} \sum_{\alp} (-1)^{\abs{\alp}} \rd_{x}^{\alp} \left( c[g_{\bfA}]^{(\alp, J')}(x) Z^{\bfA} (y, x) \left[ \eta(y, x) - \eta(y_{\bfG}, x)\right] \chi_{x_{\bfG}}(x) \chi_{y_{\bfG}}(y)\right) \\
&\peq + \sum_{\td{\bfA} \in \td{\calA}} \sum_{\alp} (-1)^{\abs{\alp}} \rd_{x}^{\alp} \left( c[g_{\bfA}]^{(\alp, J')}(x) \left[ Z^{\bfA}(y, x) - Z^{\bfA}(y, x_{\bfG}) \right] \eta(y_{\bfG}, x) \chi_{x_{\bfG}}(x) \chi_{y_{\bfG}}(y) \right), \\
\tensor{(g_{0})}{^{J'}_{\td{\bfA}}} (x) &\ceq \sum_{\alp} (-1)^{\abs{\alp}} \rd_{x}^{\alp} \left( c[g_{\bfA}]^{(\alp, J')}(x) \eta(y_{\bfG}, x) \chi_{x_{\bfG}}(x) \right), \\
Z_{0}^{\td{\bfA}}(y) &\ceq Z^{\bfA}(y, x_{\bfG}) \chi_{y_{\bfG}}(y),
\end{align*}
where $\td{\bfA} = (\bfA, \bfG)$, we obtain the desired decomposition and support properties. Moreover, in view of the presence of the differences in each sum in the definition of $e_{0}(x, y)$, as well as \ref{hyp:b-y1-smth} and \ref{hyp:eta-smth}, \eqref{eq:e0-eps} holds if we ensure that the supports of each $\chi_{x_{\bfG}}$ and $\chi_{y_{\bfG}}$ sufficiently small depending on $\eps_{0}$, $d$, $L_{y}$, $L_{b}$, $A_{Z}$ and $A_{\eta}$; at this point, we fix the choice of these functions. The remaining bounds then follow from \ref{hyp:b-y1-q}, \ref{hyp:b-y1-smth}, \ref{hyp:eta-3}, and \ref{hyp:eta-smth}. \qedhere
\end{proof}

\begin{remark} \label{rem:B-eta}
The finite partition of unity $(\chi_{x_{\bfG}}(x)\chi_{y_{\bfG}}(y))_{\bfG \in \calG_{0}}$ used in the proof always exists since $U$ is bounded (hence $\supp b_{\eta}$ is compact). The precise quantitative bounds on the functions $\chi_{x_{\bfG}}$ and $\chi_{y_{\bfG}}$, which determines the constant $C_{d, \eps_{0}, U, \supp b_{\eta}}$, would depend on the regularity assumptions on $U$ (alternatively, on $\supp \eta$ or $\supp b_{\eta}$).
\end{remark}

Combining Lemma~\ref{lem:B-eta}.(1) with a standard contradiction argument, we obtain the following Poincar\'e-type inequality for $\calP^{\ast}$ with optimal orthogonality conditions (i.e., formulated with respect to $\ker \calP^{\ast}$), but with a non-effective constant.

\begin{proposition}[Optimal Poincar\'e-type inequality with a non-effective constant] \label{prop:poincare}
Let $U$ be a connected bounded open subset of $\bbR^{d}$, and assume that there exist $W$, $\bfx$, $\eta$, $\calP$, $S$ and $b_{y_{1}}$ satisfying \ref{hyp:smth}. Assume furthermore that \ref{hyp:U-star} is satisfied for $U$, $\bfx$, and $\eta$.

\begin{enumerate}
\item {\bf Cokernel in $W^{-s, p'}(U)$.} For any $1 < p_{0}, p_{1} < \infty$ and $s_{0}, s_{1} \in \bbR$, we have
\begin{align*}
	\ker_{\td{W}^{-s_{0}, p_{0}'}(U)} \calP^{\ast} = \ker_{\td{W}^{-s_{1}, p_{1}'}(U)} \calP^{\ast}.
\end{align*}
Both are finite-dimensional and, in fact, coincide with $\ker \calP^{\ast}$.
\item {\bf Poincar\'e-type inequality.} For $1 < p < \infty$ and $s \in \bbR$, consider a family $w_{\bfA}(x) \in \td{W}^{s, p}(U)$ $(\bfA = \set{1, \ldots, \dim \ker_{W^{-s, p'}(U)} \calP^{\ast}})$ satisfying $\brk{w_{\bfA}, \bfZ^{\bfA'}} = \dlt_{\bfA}^{\bfA'}$ for some basis $\set{\bfZ^{\bfA'}}$ of $\ker_{W^{-s, p'}(U)} \calP^{\ast}$. Then there exists $C > 0$ such that
\begin{equation*}
	\nrm{\varphi}_{W^{-s, p'}(U)} \leq C \sum_{K} \nrm{(\calP^{\ast} \varphi)_{K}}_{W^{-s-m_{K}, p'}(U)} \quad \hbox{ for all } \varphi \in W^{-s, p'}(U) \hbox{ with } \brk{w_{\bfA}, \varphi} = 0,
\end{equation*}
where $\bfA = 1, \ldots, \dim \ker_{W^{-s, p'}(U)} \calP^{\ast}$.
\end{enumerate}
\end{proposition}
By \emph{non-effective}, we mean that there is no quantitative relationship between the constant $C$ and the parameters of our construction (e.g., $A_{\bfx}$, $A_{S}$ etc.). This feature is due to the use of a compactness argument in the proof below.

\begin{proof} 
We begin by observing that, as in the proof of Theorem~\ref{thm:full-sol-max}, we have the mapping properties $\tensor{(\calS_{\eta})}{^{K}_{J}} : \td{W}^{s, p}(U) \to \td{W}^{s+m_{K}, p}(U)$, and thus $\tensor{(\calS_{\eta}^{\ast})}{_{J}^{K}} : W^{-s-m_{K}, p'}(U) \to W^{-s, p'}(U)$. Thus, for every $\varphi \in W^{-s, p'}(U)$, we have
\begin{equation*}
\varphi = \calS_{\eta}^{\ast} \calP^{\ast} \varphi + \calB_{\eta}^{\ast} \varphi,
\end{equation*}
where $\calB_{\eta}$ is a smoothing operator according to Lemma~\ref{lem:B-eta}. As a consequence, for any $\dlt > 0$, $1 < p < \infty$ and $s \in \bbR$,
\begin{equation*}
	\nrm{\calB_{\eta} \varphi}_{W^{-s, p'}(U)} \aleq_{\dlt, s, p} \nrm{\varphi}_{W^{-s-\dlt, p'}(U)}.
\end{equation*}
From this estimate, it immediately follows that for $\bfZ \in \ker_{W^{-s, p'}(U)} \calP^{\ast}$, we have $\bfZ = \calB_{\eta}^{\ast} \bfZ$, and thus the cokernels in different Sobolev spaces are the same; this implies Part~(1). Moreover, combined with Theorem~\ref{thm:smth-ker}, we obtain the elliptic estimate
\begin{equation} \label{eq:poincare:pf:elliptic}
	\nrm{\varphi}_{W^{-s, p'}(U)} \aleq_{s, p, \dlt} \sum_{K} \nrm{(\calP^{\ast} \varphi)_{K}}_{W^{-s-m_{K}, p'}(U)} + \nrm{\varphi}_{W^{-s-\dlt, p'}(U)},
\end{equation}
for every $\varphi \in W^{-s, p'}(U)$.

We are ready to start the proof of Part~(2). Suppose that the conclusion does not hold; then there exists a sequence $\varphi^{(n)} \in W^{-s, p'}(U)$ such that
\begin{equation*}
	\nrm{\varphi^{(n)}}_{W^{-s, p'}(U)} = 1, \quad \nrm{(\calP^{\ast} \varphi^{(n)})_{K}}_{W^{-s - m_{K}, p'}(U)} \leq \frac{1}{n}, \quad \brk{w_{\bfA}, \varphi^{(n)}} = 0 \hbox{ for } \bfA = 1, \ldots, \dim \ker_{W^{-s, p'}(U)} \calP^{\ast}.
\end{equation*}
By Rellich--Kondrachov (Lemma~\ref{lem:sob-cpt}), there exists $\varphi \in W^{-s, p'}(U)$ such that, after passing to a subsequence,
\begin{equation*}
	\varphi^{(n)} \weakto \varphi \hbox{ in } W^{-s, p'}(U), \quad \varphi^{(n)} \to \varphi \hbox{ in } W^{-s-\dlt, p'}(U).
\end{equation*}
By the weak $W^{-s, p'}(U)$-convergence, we have $\calP^{\ast} \varphi = 0$ (i.e., $\varphi \in \ker_{W^{-s, p'}(U)} \calP^{\ast}$) and $\brk{w_{\bfA}, \varphi} = 0$ for all $\bfA = 1, \ldots, \dim \ker_{W^{-s, p'}(U)} \calP^{\ast}$. By the properties of $w_{\bfA}$, we have $\varphi = 0$. But, by the strong $W^{s-\dlt, p}(U)$-convergence and \eqref{eq:poincare:pf:elliptic}, we have $\varphi \neq 0$, which is a contradiction. \qedhere
\end{proof}

By a standard duality argument, Proposition~\ref{prop:poincare} may be turned into an existence statement for the equation $\calP u = f$.
\begin{corollary} \label{cor:abstract-solvability}
Let $U$ be a connected bounded open subset of $\bbR^{d}$, and assume that there exist $W$, $\bfx$, $\eta$, $\calP$, $S$ and $b_{y_{1}}$ satisfying \ref{hyp:smth}. Assume furthermore that \ref{hyp:U-star} is satisfied for $U$, $\bfx$, and $\eta$. For every $f \in \td{H}^{s}(U)$ satisfying $f \perp \ker_{H^{-s}(U)} \calP^{\ast}$, there exists $u = (u^{K})_{K \in \set{1, \ldots, s_{0}}}$ with $u^{K} \in \td{H}^{s+m_{K}}(U)$ such that $\calP u = f$ and $\nrm{u^{K}}_{\td{H}^{s+m_{K}}(U)} \leq C \nrm{f}_{\td{H}^{s}(U)}$, where $C$ is the constant in Proposition~\ref{prop:poincare} with $p = 2$.
\end{corollary}

\begin{proof}
By $\td{H}^{s+m_{K}}(U) = (H^{-s-m_{K}}(U))^{\ast}$ from \eqref{eq:sob-duality}, it suffices to construct bounded linear functionals $u^{K}$ on $H^{-s-m_{K}}(U)$ such that $\brk{u^{K}, (\calP^{\ast} \varphi)_{K}} = \brk{f^{J}, \varphi_{J}}$ for all $\varphi_{J} \in H^{-s}(U)$. We will use the preceding proposition and the Hahn--Banach theorem.

Let $(\bfZ^{\bfA})_{J}$ and $(w_{\bfA})^{J'}$ be as in Proposition~\ref{prop:poincare}, and let $\varphi_{J} \in H^{-s}(U)$. Since $\brk{f^{J}, (\bfZ^{\bfA})_{J}} = 0$ for all $\bfA = 1, \ldots, \dim \ker \calP^{\ast}$, we have 
\begin{align*}
\abs{\brk{f^{J}, \varphi_{J}}}
&= \abs*{\brk{f^{J}, \varphi_{J} - \sum_{\bfA} (\bfZ^{\bfA})_{J} \brk{(w_{\bfA})^{J'}, \varphi_{J'}}}}
\leq \sum_{J} \nrm{f^{J}}_{\td{H}^{s}(U)} \nrm*{\varphi_{J} - \sum_{\bfA} (\bfZ^{\bfA})_{J} \brk{(w_{\bfA})^{J'}, \varphi_{J'}}}_{H^{-s}(U)} \\
&\leq C \sum_{J, K} \nrm{f^{J}}_{\td{H}^{s}(U)} \nrm{(\calP^{\ast} \varphi)_{K}}_{H^{-s-m_{K}}(U)},
\end{align*}
where $C$ is the constant in Proposition~\ref{prop:poincare} with $p = 2$. It follows that the linear functional $\brk{u^{K}, (\calP^{\ast} \varphi)_{K}} \ceq \brk{f^{J}, \varphi_{J}}$ is well-defined and bounded on the vector subspace $M = \calP^{\ast} (H^{-s}(U; \bbC^{r_{0}}))$ of $H^{-s-m_{1}} \times \cdots \times H^{-s-m_{s_{0}}}(U)$. By the Hahn--Banach theorem, it may be extended to an element (with an abuse of notation) $u \in (H^{-s-m_{1}} \times \cdots \times H^{-s-m_{s_{0}}}(U))^{\ast} = \td{H}^{s+m_{1}} \times \cdots \times \td{H}^{s+m_{s_{0}}}(U)$ with $\nrm{u^{K}}_{\td{H}^{s+m_{K}}(U)} \leq C \nrm{f}_{\td{H}^{s}(U)}$, as desired.
\end{proof}

Finally, we upgrade Corollary~\ref{cor:abstract-solvability} to the existence of a linear operator $\calQ$ that completes $\calS_{\eta}$ to a full solution operator, which is moreover smoothing and has a prescribed support property, but with non-effective bounds on the operator norms.
\begin{theorem} [Full solution operator and representation formula] \label{thm:full-sol}
Let $U$ be a bounded open subset of $\bbR^{d}$, and assume that $W$, $\bfx$, $\eta$, $\calP$, $S$ and $b_{y_{1}}$ satisfy \ref{hyp:smth}. Assume furthermore that \ref{hyp:U-star} is satisfied for $U$, $\bfx$, and $\eta$. Consider a family $w_{\bfA}(x) \in C^{\infty}_{c}(U)$ $(\bfA = \set{1, \ldots, \dim \ker \calP^{\ast}})$ satisfying $\brk{w_{\bfA}, \bfZ^{\bfA'}} = \dlt_{\bfA}^{\bfA'}$ for some basis $\set{\bfZ^{\bfA'}}$ of $\ker \calP^{\ast}$. Consider also an open subset $V$ of $U$ satisfying, for all $\bfA \in \set{1, \ldots, \dim \ker \calP^{\ast}}$ and $y \in U$,
\begin{equation*}
	\supp w_{\bfA} \subseteq V, \quad \supp_{x} \eta(y, x) \subseteq V.
\end{equation*}
Then there exists a linear operator $\calQ$ such that
\begin{equation*}
	\calP (\calS_{\eta} - \calQ) f = f - \sum_{\bfA \in \set{1, \ldots, \dim \ker \calP^{\ast}}} w_{\bfA} \brk{\bfZ^{\bfA}, f} \quad \hbox{ for all } f \in C^{\infty}_{c}(U),
\end{equation*}
where the integral kernel $q(x, y)$ of $\calQ$ has uniformly bounded derivatives of all order and, for every $y \in U$,
\begin{equation*}
	\supp_{x} q(x, y) \subseteq V \cup \br{\bigcup_{y' \in V, \, (y', y_{1}) \in \supp \eta} \bfx(y', y_{1}, [0, 1])}.
\end{equation*}
By duality, we also have the representation formula
\begin{equation*}
	\varphi = (\calS_{\eta} - \calQ)^{\ast} \calP^{\ast} \varphi + \sum_{\bfA \in \set{1, \ldots, \dim \ker \calP^{\ast}}} \bfZ^{\bfA} \brk{w_{\bfA}, \varphi} \quad \hbox{ for all } \varphi \in C^{\infty}(\br{U}).
\end{equation*}
\end{theorem}

As a consequence, for any $1 < p < \infty$ and $s \in \bbR$, the full solution operator $\calS_{\eta} - \calQ$ extends to a bounded operator $\td{W}^{s, p}(U) \to \td{W}^{s+m_{1}, p} \times \cdots \td{W}^{s+m_{s_{0}}, p} (U)$. Hence, the representation formula holds for all $\varphi \in W^{-s, p'}(U)$, for any $1 < p < \infty$ and $s \in \bbR$. The operator norm of $\calQ$ (and thus that of $\calS_{\eta} - \calQ$) produced in our proof below is non-effective, but only through the application of Corollary~\ref{cor:abstract-solvability}.

\begin{proof}
Let $\eps_{0}, \eps_{1} > 0$ be small parameters to be fixed later. Our starting point is Lemma~\ref{lem:B-eta}.(2), which provides us with an approximation of $\calB_{\eta}$ by a finite rank operator $\sum_{\td{\bfA}} (g_{0})_{\td{\bfA}} \brk{Z_{0}^{\td{\bfA}}, \cdot}$ up to an error operator $\calE_{0}$ obeying, in particular, \eqref{eq:e0-eps} with $\eps_{0}$ on the right-hand side (to ease the notation, we omit the superscript ${}^{(\eps_{0})}$). Let us project each $(g_{0})_{\td{\bfA}}$ to ${}^{\perp}(\ker \calP^{\ast})$ using $w_{\bfA}$; i.e., we introduce
\begin{equation*}
	g_{\td{\bfA}} \ceq (g_{0})_{\td{\bfA}} - \sum_{\bfA} \brk{\bfZ^{\bfA}, (g_{0})_{\td{\bfA}}} w_{\bfA},
\end{equation*}
where, here and throughout the proof, the index $\bfA$ is summed over $\set{1, \ldots, \dim \ker \calP^{\ast}}$. Then we arrive at
\begin{equation} \label{eq:full-sol:pf:1}
\calP \calS_{\eta} f = f - \calE_{0} f - \sum_{\bfA} w_{\bfA} \ell^{\bfA}(f) - \sum_{\td{\bfA} \in \td{\calA}} g_{\td{\bfA}} \brk{Z^{\td{\bfA}}_{0}, f},
\end{equation}
where $\ell^{\bfA}(f)$ is a linear functional for each $\bfA$ (it may be computed explicitly, but it does not matter) and $g_{\td{\bfA}} \perp \ker \calP^{\ast}$ for each $\td{\bfA} \in \td{\calA}$.

For each such $g_{\td{\bfA}}$, we now look for $u_{\td{\bfA}} \in C^{\infty}_{c}(V)$ that solves the approximate equation
\begin{equation} \label{eq:full-sol:pf:2}
	\calP u_{\td{\bfA}} = g_{\td{\bfA}} + e_{\td{\bfA}},
\end{equation}
with $\supp u_{\td{\bfA}}, \supp e_{\td{\bfA}} \subseteq V$ and an error bound (in terms of $\eps_{1}$) for $e_{\td{\bfA}}$. For this purpose, we first fix an open subset $V_{1}$ such that
\begin{equation*}
	\bigcup_{\td{\bfA} \in \td{\calA}} \supp g_{\td{\bfA}} \subseteq V_{1} \subseteq \br{V_{1}} \subseteq V,
\end{equation*}
and apply Corollary~\ref{cor:abstract-solvability} to find a solution $v_{\td{\bfA}}$ in, say, $\td{H}^{1+m_{1}} \times \cdots \times \td{H}^{1+m_{s_{0}}}(V_{1})$ to $\calP v_{\td{\bfA}} = g_{\td{\bfA}}$. We fix $\psi_{1} \in C^{\infty}_{c}(B_{1}(0))$ such that $\int \psi_{1} = 1$ and define
\begin{equation*}
	u_{\td{\bfA}} \ceq \psi_{\eps_{1}} \ast v_{\td{\bfA}},
\end{equation*}
where $v_{\td{\bfA}}$ is viewed as an element of $H^{1+m_{1}} \times \cdots \times H^{1+m_{s_{0}}}(\bbR^{d})$ that is supported in $\br{V_{1}}$ (recall the definition of $\td{H}^{s}(V_{1})$ as a subspace of $H^{s}$) and $\psi_{\eps}(\cdot) \ceq \eps^{-d} \psi_{1}(\eps^{-1}(\cdot))$. Then \eqref{eq:full-sol:pf:2} holds with
\begin{equation*}
	e_{\td{\bfA}} = - (g_{\td{\bfA}} - \psi_{\eps_{1}} \ast g_{\td{\bfA}}) - [\psi_{\eps_{1}} \ast, \calP] v_{\td{\bfA}}.
\end{equation*}
We now verify that such $u_{\td{\bfA}}$ and $e_{\td{\bfA}}$ have desirable properties. First, as long as $\eps_{1} < \dist(\br{V_{1}}, \rd V)$, we have
\begin{equation} \label{eq:full-sol:pf:3-1}
	\supp u_{\td{\bfA}}, \supp e_{\td{\bfA}} \subseteq V.
\end{equation}
Moreover, rewriting $[\psi_{\eps_{1}} \ast, \calP] v_{\td{\bfA}} = \psi_{\eps_{1}} \ast \calP v_{\td{\bfA}} - \calP v_{\td{\bfA}} + \calP (v_{\td{\bfA}} - \psi_{\eps_{1}} \ast v_{\td{\bfA}})$ and using the property of convolution, we obtain
\begin{equation} \label{eq:full-sol:pf:3-2}
	\nrm{e_{\td{\bfA}}}_{L^{2}} \aleq \nrm{g_{\td{\bfA}} - \psi_{\eps_{1}} \ast g_{\td{\bfA}}}_{L^{2}}
	+ \nrm{\calP v_{\td{\bfA}} - \psi_{\eps_{1}} \ast \calP v_{\td{\bfA}}}_{L^{2}}
	+ \nrm{v_{\td{\bfA}} - \psi_{\eps_{1}} \ast v_{\td{\bfA}}}_{H^{m_{1}} \times \cdots \times H^{m_{s_{0}}}} \aleq \eps_{1} \nrm{g_{\td{\bfA}}}_{H^{1}},
\end{equation}
where the implicit constants depend on that in Corollary~\ref{cor:abstract-solvability} with $s_{0} = 1$ and $\nrm{\calP}_{H^{s+m_{1}} \times \cdots \times H^{s+m_{s_{0}}} \to H^{s}}$ with $s = 0, 1$. Finally, we clearly have, for every $N \geq 0$,
\begin{equation} \label{eq:full-sol:pf:3-3}
	\nrm{u_{\td{\bfA}}}_{H^{1+N+m_{1}} \times \cdots \times H^{1+N+m_{s_{0}}}} \leq C_{N} \eps_{1}^{-N} \nrm{g_{\td{\bfA}}}_{H^{1}}, \quad
	\nrm{e_{\td{\bfA}}}_{H^{1+N}} \leq C \nrm{g_{\td{\bfA}}}_{H^{1+N}} + C_{N} \eps_{1}^{-N} \nrm{g_{\td{\bfA}}}_{H^{1}}.
\end{equation}

We are now ready to conclude the proof. Introducing the operators
\begin{equation*}
	\calQ_{1} f \ceq \sum_{\td{\bfA} \in \td{\calA}} u_{\td{\bfA}} \brk{Z_{0}^{\td{\bfA}}, f}, \quad
	\calE_{1} f \ceq \sum_{\td{\bfA} \in \td{\calA}} e_{\td{\bfA}} \brk{Z_{0}^{\td{\bfA}}, f},
\end{equation*}
we arrive at
\begin{align*}
	\calP (\calS_{\eta} - \calQ_{1}) f= f - (\calE_{0} + \calE_{1}) f - \sum_{\bfA \in \set{1, \ldots, \dim \ker \calP^{\ast}}} w_{\bfA} \ell^{\bfA}(f).
\end{align*}
We define the operator $\calQ$ by
\begin{equation*}
	\calS_{\eta} - \calQ = (\calS_{\eta} - \calQ_{1})(I - \calE)^{-1}, \quad \hbox{ where } \calE \ceq \calE_{0} + \calE_{1}.
\end{equation*}
To finish the proof, it remains to verify the desired properties of $\calQ$ (including that it is well-defined). First, by Lemma~\ref{lem:B-eta}.(2) and \eqref{eq:full-sol:pf:3-3}, we see that the integral kernels $e(x, y)$ and $q_{1}(x, y)$ of $\calE$ and $\calQ_{1}$, respectively, have uniformly bounded derivatives of all order and $\supp e(\cdot, y), \, \supp q_{1}(\cdot, y) \subseteq V$ for all $y \in U$. Next, in view of Lemma~\ref{lem:B-eta}.(2) (especially \eqref{eq:e0-eps}) and \eqref{eq:full-sol:pf:3-2}, we may choose $\eps_{0}$ and $\eps_{1}$ small enough so that $\nrm{\calE}_{L^{2}(U) \to L^{2}(U)} < 1$. As a consequence, $I - \calE$ is invertible on $L^{2}(U)$. Moreover, in view of the identities
\begin{equation*}
	(I - \calE)^{-1} = I + \calE + \calE^{2} + \cdots = I + \calE (1-\calE)^{-1}, \quad
	\left[(I - \calE)^{-1}\right]^{\ast} = I + \calE^{\ast} \left[ (1-\calE)^{-1}\right]^{\ast},
\end{equation*}
it follows that the integral kernel $K(x, y)$ of $(I - \calE)^{-1} - I$ also has uniformly bounded derivatives of all order and $\supp K(\cdot, y) \subseteq V$ for all $y \in U$. Now writing
\begin{equation*}
	\calQ = - (\calS_{\eta} - \calQ_{1})\left( (I - \calE)^{-1} - I \right) + \calQ_{1},
\end{equation*}
the desired properties of the integral kernel $q(x, y)$ of $\calQ$ follow from those for $\calS_{\eta}$, $\calQ_{1}$ and $(I - \calE)^{-1} - I$. \qedhere
\end{proof}

\begin{example} \label{ex:div-full-sol}
For $\calP u = (\rd_{j} + B_{j}) u^{j}$, $\bfx = \ul{\bfx}$ (straight line segments), and $\eta(y, y_{1}) = \eta_{1}(y_{1})$, Theorem~\ref{thm:full-sol} is the precise version of the result outlined in Step~4 (non-completely integrable case).
\end{example}

\begin{remark} [Comparison with Theorem~\ref{thm:full-sol-max}] \label{rem:full-sol-w}
Other than the non-effective nature of the operator bounds, another difference between Theorems~\ref{thm:full-sol-max} and \ref{thm:full-sol} is the latter's ability to prescribe the dual family $w_{\bfA}$ to $\ker \calP^{\ast}$. This feature allow us to have more freedom in prescribing the support properties of the solution operator.
\end{remark}

We are ready to give a precise formulation and proof of Theorem~\ref{thm:main-summary-bogovskii}.
\begin{proof}[Precise formulation and proof of Theorem~\ref{thm:main-summary-bogovskii}]
Theorem~\ref{thm:main-summary-bogovskii}.(2) and (3) -- with the admissibility of $\bfx$ and \ref{hyp:crc} replaced by \ref{hyp:x-1}--\ref{hyp:x-2} and \ref{hyp:crc-x-q} on $U$ -- follows from Theorem~\ref{thm:full-sol} by simply choosing $\eta(y, y_{1}) = \eta_{1}(y_{1})$ with $\eta_{1} \in C^{\infty}_{c}(U_{1})$ that satisfies $\int \eta_{1}(y_{1}) \, \ud y_{1} = 1$. The full solution operator $\td{\calS}$ is precisely $\calS_{\eta} + \calQ$. For the properties of $\ker_{W^{-s, p'}(U)} \calP^{\ast}$, we apply Proposition~\ref{prop:poincare} (for the invariance and finite-dimensional properties), and Lemma~\ref{lem:b-y1-coker}.(1) for paths that connect $y_{1} \in V$ and $y \in U$ for the invariance of the dimension under restrictions.
\end{proof}

\section{From graded augmented system to \ref{hyp:crc}} \label{sec:ode}
In Section~\ref{sec:crc}, we identified some abstract conditions needed to construct the operator $\calS_{\eta}$ with prescribed support properties. In this section, we formulate and prove a precise version of Proposition~\ref{prop:aug2crc} (see Proposition~\ref{prop:ode2crc} below), i.e., that these conditions follow from the existence of a graded augmented system with adequate quantitative bounds. We also provide a proof of Proposition~\ref{prop:zero-curv}.

This section is structured as follows. In Section~\ref{subsec:1st-order-pde}, we record the precise formulas for the objects $\tensor{S}{_{J}^{(\gmm, K)}}$ and $\tensor{(b_{y_{1}})}{^{J'}_{J}}$ arising in \ref{hyp:crc-full} in terms of a graded augmented system. To prove Proposition~\ref{prop:aug2crc}, it remains to establish the quantitative bounds in \ref{hyp:crc-x-q}, \ref{hyp:b-y1-q} and \ref{hyp:b-y1-smth}. In Section~\ref{subsec:ode}, we formulate and prove an abstract ODE lemma for this purpose. Then in Section~\ref{subsec:ode2crc}, we prove the main result of this section (Proposition~\ref{prop:ode2crc}). Finally, in Section~\ref{subsec:conn}, we give a geometric interpretation of cokernel elements as parallel vectors with respect to a suitable connection on a vector bundle constructed from the augmented system, and prove Proposition~\ref{prop:zero-curv} as a simple byproduct.

\subsection{Structure of $S(y, y_{1}, t)$ and $b_{y_{1}}$ for a given augmented system} \label{subsec:1st-order-pde}
Let $\calP$ be an $r_{0} \times s_{0}$-matrix-valued differential operator on an open subset $U \subseteq \bbR^{d}$, with $m_{K}$ denoting the order of $(\calP^{\ast} \varphi)_{K}$ for each $K \in \set{1, \ldots, s_{0}}$. Given $m_{K}' \in \bbZ_{\geq 0}$ for each $K \in \set{1, \ldots, s_{0}}$ and an $\bbC^{r_{0}}$-valued function $(\varphi_{J})_{J \in \set{1, \ldots, r_{0}}}$ on $U$, let $(\Phi_{\bfA} = \Phi_{\bfA}(y))_{\bfA \in \calA}$ be (graded) augmented variables as in Definition~\ref{def:aug} (see also the discussion above this definition for our conventions for indices). The aim of this short subsection is to record the formulas for $S(y, y_{1}, t)$ and $b_{y_{1}}$ in \ref{hyp:crc-full} in terms of the augmented system.

Let $\tensor{(\bfB_{i})}{_{\bfA}^{\bfA''}}$ and $\tensor{(\bfC_{i})}{_{\bfA}^{(\gmm, K)}}$ be the coefficients of the system of first-order PDEs in \ref{hyp:aug3}. Let ${}^{(\bfx_{y, y_{1}})} \tensor{\bfPi}{_{\bfA}^{\bfA'}}(s, t)$ be the fundamental matrix solving the ODE
\begin{align*}
	\rd_{s} \left[ {}^{(\bfx_{y, y_{1}})} \tensor{\bfPi}{_{\bfA}^{\bfA'}}(s, t)\right] &= \dot{\bfx}_{y, y_{1}}^{i} \left( \tensor{(\bfB_{i})}{_{\bfA}^{\bfA''}} \circ \bfx_{y, y_{1}} \right)(y, y_{1}, s) \, {}^{(\bfx_{y, y_{1}})} \tensor{\bfPi}{_{\bfA''}^{\bfA'}}(s, t), \\
	 {}^{(\bfx_{y, y_{1}})}\tensor{\bfPi}{_{\bfA}^{\bfA'}}(t, t) &= \dlt_{\bfA}^{\bfA'},
\end{align*}
or equivalently,
\begin{equation} \label{eq:fund-mat-eq}
	{}^{(\bfx_{y, y_{1}})} \tensor{\bfPi}{_{\bfA}^{\bfA'}}(s, t) = \dlt_{\bfA}^{\bfA'} - \int_{s}^{t} \dot{\bfx}_{y, y_{1}}^{i}  \left( \tensor{(\bfB_{i})}{_{\bfA}^{\bfA''}} \circ \bfx_{y, y_{1}} \right) (y, y_{1}, s') \, {}^{(\bfx_{y, y_{1}})} \tensor{\bfPi}{_{\bfA''}^{\bfA'}}(s', t) \, \ud s'.
\end{equation}
Recall that, by Duhamel's principle, \ref{hyp:aug1} and \ref{hyp:aug3}, we have
\begin{equation*} \tag{\ref{eq:varphi-duhamel}}
\begin{aligned}
\varphi_{J}(y) &= - \int_{0}^{1}  {}^{(\bfx_{y, y_{1}})} \tensor{\bfPi}{_{J}^{\bfA}}(0, s) \dot{\bfx}_{y, y_{1}}^{i} \left( \tensor{(\bfC_{i})}{_{\bfA}^{(\gmm, K)}}  \circ \bfx_{y, y_{1}} \right) \left( (\rd^{\gmm} \calP^{\ast} \varphi)_{K} \circ \bfx_{y, y_{1}} \right) (y, y_{1}, s) \, \ud s  \\
&\peq + {}^{(\bfx_{y, y_{1}})} \tensor{\bfPi}{_{J}^{\bfA}}(0, 1) \Phi_{\bfA}(y_{1}).
\end{aligned}
\end{equation*}

From this expression, we immediately obtain expressions for $\tensor{S}{_{J}^{(\gmm, K)}}(y, y_{1}, s)$ and  $b_{y_{1}}(x, y)$ in \ref{hyp:crc-full}, as well as  $K_{y_{1}}(x, y)$ in \eqref{eq:K-y1-def}. We record them in the following proposition.

\begin{proposition} \label{prop:crc-ODE}
Let $(\varphi_{J})_{J \in \set{1, \ldots, r_{0}}} \mapsto (\Phi_{\bfA})_{\bfA \in \calA}$ satisfy \ref{hyp:aug1}--\ref{hyp:aug3} with smooth $c[\Phi_{\bfA}]^{(\alp, J)}$, $\tensor{(\bfB_{i})}{_{\bfA}^{\bfA'}}(y)$ and $\tensor{(\bfC_{i})}{_{\bfA}^{(\gmm, K)}}(y)$. Then \ref{hyp:crc-full} is satisfied with the following objects:
\begin{align}
	\tensor{S}{_{J}^{(\gmm, K)}}(y, y_{1}, t) &\ceq - {}^{(\bfx_{y, y_{1}})} \tensor{\bfPi}{_{J}^{\bfA}}(0, t) \dot{\bfx}_{y, y_{1}}^{i} \left(\tensor{(\bfC_{i})}{_{\bfA}^{(\gmm, K)}} \circ \bfx_{y, y_{1}}\right)(y, y_{1}, t), \label{eq:S-gmm-ODE} \\
	\tensor{(b_{y_{1}})}{^{J'}_{J}}(x, y) &\ceq \sum_{\bfA \in \calA} \tensor{(Z^{\bfA})}{_{J}}(y, y_{1})\tensor{(g_{\bfA})}{^{J'}}(x, y_{1}), \label{eq:b-y1-ODE} \\
	(Z^{\bfA})_{J}(y, y_{1}) &\ceq  {}^{(\bfx_{y, y_{1}})} \tensor{\bfPi}{_{J}^{\bfA}}(0, 1), \label{eq:Z-ODE} \\
	\brk{\tensor{(g_{\bfA})}{^{J}}(\cdot, y_{1}), \varphi_{J}} &\ceq \Phi_{\bfA}(y_{1}) \quad (\hbox{i.e., } c[g_{\bfA}]^{(\alp, J)}(y_{1}) = c[\Phi_{\bfA}]^{(\alp, J)}(y_{1})), \label{eq:g-ODE}
\end{align}
In particular, $b_{y_{1}}$ clearly satisfies \eqref{eq:b-y1-properties}. Moreover, defining the kernel $\tensor{(K_{y_{1}})}{^{K}_{J}}(x, y)$ by
\begin{equation*}
	\brk{\tensor{(K_{y_{1}})}{^{K}_{J}}(x, y), \psi_K(x)} = \sum_{\gmm} \int_{0}^{1} \tensor{S}{_{J}^{(\gmm, K)}}(y, y_{1}, s) \rd_{x}^{\gmm} \psi_{K} (\bfx(y, y_{1}, s)) \, \ud s,
\end{equation*}
we have \eqref{eq:trun-ker-y1-gen}, i.e., $\calP K_{y_{1}}(x, y) = \dlt_{0}(x-y) - b_{y_{1}}(x, y)$.
\end{proposition}

\subsection{Abstract ODE estimates} \label{subsec:ode}

Let $\calA$ be a finite index set of size $N$, which we label $\set{1, \ldots, N}$ without loss of generality. Consider the following $N \times N$ system of ODEs for $t \in [0, T]$ for some $T > 0$:
\begin{equation} \label{eq:ode}
	\frac{\ud}{\ud t} \psi_{\bfA}(t) = \sum_{\bfA'} \tensor{B}{_{\bfA}^{\bfA'}}(t) \psi_{\bfA'}(t) + g_{\bfA}(t),
\end{equation}
where $\bfA, \bfA' \in \set{1, \ldots, N}$. As in \ref{hyp:aug4}, we associate to each index $\bfA \in \set{1, \ldots, N}$ a \emph{degree} $d_{\bfA} \in \bbZ$, and without loss of generality, we assume that
\begin{equation} \label{eq:ode-dmax}
	\max_{\bfA \in \set{1, \ldots, N}} d_{\bfA} = 0.
\end{equation}
We assume that $B$ obeys the following vanishing condition:
\begin{equation} \label{eq:ode-B}
	\tensor{B}{_{\bfA}^{\bfA'}}(s) \equiv 0 \quad \hbox{ if } d_{\bfA} > d_{\bfA'} + 1.
\end{equation}

\begin{remark} \label{rem:ode-B}
In our applications, we take $\tensor{B}{_{\bfA}^{\bfA'}}(s) = \dot{\bfx}^{i} \left(\tensor{(\bfB_{i})}{_{\bfA}^{\bfA'}} \circ \bfx\right)(y, y_{1}, s)$ for a fixed $(y, y_{1}) \in W$. Hence, \eqref{eq:ode-B} follows from the properties of degree introduced in Definition~\ref{def:aug}.
\end{remark}

The solution to \eqref{eq:ode} with the final data $\psi_{\bfA}(t) = 0$ is given by
\begin{equation} \label{eq:ode-sol}
	\psi_{\bfA}(t) = - \sum_{\bfA'} \int_{t}^{T} \tensor{\Pi}{_{\bfA}^{\bfA'}}(t, s) g_{\bfA'}(s) \, \ud s,
\end{equation}
where $\tensor{\Pi}{_{\bfA}^{\bfA'}}(s, t)$ is the fundamental solution solving,
\begin{equation*}
	\frac{\ud}{\ud s} \Pi(s, t) = B(s) \Pi(s, t), \quad \Pi(t, t) = I,
\end{equation*}
or equivalently,
\begin{equation} \label{eq:ode-Pi}
	\tensor{\Pi}{_{\bfA}^{\bfA'}}(s, t) = \tensor{\dlt}{_{\bfA}^{\bfA'}} - \int_{s}^{t} \sum_{\bfA''} \tensor{B}{_{\bfA}^{\bfA''}}(s') \tensor{\Pi}{_{\bfA''}^{\bfA'}}(s', t) \, \ud s'.
\end{equation}
Eventually, the kernel $S_{\gmm}(y, y_{1}, s)$ in \ref{hyp:crc-x} will be constructed out of $\tensor{\Pi}{_{\bfA}^{\bfA'}}(0, s)$ for $s \in [0, T]$  (see Proposition~\ref{prop:crc-ODE} and Remark~\ref{rem:ode-B}). Our goal in this subsection is to estimate the size of $\Pi(t, s)$ in terms of the coefficients $B$ in \eqref{eq:ode}.

For this purpose, we first formulate and prove a result for an abstract ODE (or more precisely, an integral equation; see \eqref{eq:abs-ode}). We introduce the following scale of norms for $b : [0, T] \to \bbR$ and $m \in \bbZ_{\leq 0}$:
\begin{align*}
	\nnrm{b}_{0} & \ceq \nrm{b}_{L^{\infty}[0, T]} & & \hbox{ if } m = 0, \\
	\nnrm{b}_{m} &: = \int_{0}^{T} \abs{b(s)} s^{-m-1} \, \ud s & & \hbox{ if } m < 0.
\end{align*}
We introduce
\begin{align*}
	\nnrm{B}_{(\infty)} &\ceq \sum_{\bfA, \bfA' : d_{\bfA} = d_{\bfA'} + 1} \nnrm{\tensor{B}{_{\bfA}^{\bfA'}}}_{d_{\bfA} - d_{\bfA'} - 1} = \sum_{\bfA, \bfA' : d_{\bfA} = d_{\bfA'} + 1} \nrm{\tensor{B}{_{\bfA}^{\bfA'}}}_{L^{\infty}[0, T]} , \\
	\nnrm{B}_{(1)} &\ceq \sum_{\bfA, \bfA' : d_{\bfA} \leq d_{\bfA'}} \nnrm{\tensor{B}{_{\bfA}^{\bfA'}}}_{d_{\bfA} - d_{\bfA'} - 1} = \sum_{\bfA, \bfA' : d_{\bfA} \leq d_{\bfA'}} \int_{0}^{T} \abs{\tensor{B}{_{\bfA}^{\bfA'}}(s)} s^{-(d_{\bfA}-d_{\bfA'})} \, \ud s,
\end{align*}
and define
\begin{equation*}
	\nnrm{B} \ceq \sum_{\bfA, \bfA'} \nnrm{\tensor{B}{_{\bfA}^{\bfA'}}}_{d_{\bfA} - d_{\bfA'} - 1} = \nnrm{B}_{(\infty)} + \nnrm{B}_{(1)}.
\end{equation*}

The main result of this subsection is the following.
\begin{proposition}[Abstract ODE estimates] \label{prop:ode}
Consider any $T > 0$ and degrees $d_{\bfA}$ ($\bfA \in \set{1, \ldots, N}$) and an $N \times N$ matrix-valued function $B$ on $[0, T]$ satisfying \eqref{eq:ode-dmax}, \eqref{eq:ode-B} and $\nnrm{B} < + \infty$. Let $\tensor{\Psi}{_{\bfA}^{\bfA'}} = \tensor{\Psi}{_{\bfA}^{\bfA'}}(s, t)$ ($\bfA, \bfA' \in \set{1, \ldots, N}$) satisfy the equation
\begin{equation} \label{eq:abs-ode}
	\tensor{\Psi}{_{\bfA}^{\bfA'}}(s, t) = \int_{s}^{t} \tensor{G}{_{\bfA}^{\bfA'}}(s', t) \, \ud s' + \int_{s}^{t} \sum_{\bfA''} \tensor{B}{_{\bfA}^{\bfA''}}(s') \tensor{\Psi}{_{\bfA''}^{\bfA'}}(s', t) \, \ud s'
\end{equation}
for $0 < s < t < T$, where
\begin{align}
	\int_{s}^{t} (s')^{- d_{\bfA}} t^{d_{\bfA'}} \abs{\tensor{G}{_{\bfA}^{\bfA'}}(s', t)} \, \ud s'  \leq A_{0} \label{eq:abs-ode-G}
\end{align}
for some $A_{0} > 0$. Then for all $0 < s < t < T$, we have
\begin{align}
	s^{- d_{\bfA}} t^{d_{\bfA'}}\abs{\tensor{\Psi}{_{\bfA}^{\bfA'}}(s, t)} & \leq C A_{0}  & & \hbox{ if } d_{\bfA} \leq 0, \label{eq:abs-ode-infty} \\
	\int_{s}^{t} (s')^{- d_{\bfA} - 1} t^{d_{\bfA'}} \abs{\tensor{\Psi}{_{\bfA}^{\bfA'}}(s', t)} \, \ud s' & \leq C A_{0} & & \hbox{ if } d_{\bfA} < 0, \label{eq:abs-ode-1}
\end{align}
where we have the following bound for the constant $C$ in \eqref{eq:abs-ode-infty}--\eqref{eq:abs-ode-1}:
\begin{align*}
C \leq C_{N, \max_{\bfA} (- d_{\bfA})} (2+\nnrm{B}_{(\infty)})^{C_{N} [ 1 + (1+\nnrm{B}_{(\infty)})^{N} \nnrm{B}_{(1)} ]}.
\end{align*}
\end{proposition}
\begin{proof}
In this proof, we suppress the dependence of constants on $N$ and $\max_{\bfA}(-d_{\bfA})$. Exceptions to this rule are the constants denoted by $C_{N}$, which depend only on $N$ but not on $\max_{\bfA}(-d_{\bfA})$. Moreover, some constants that are independent of $N$ and $\max_{\bfA}(-d_{\bfA})$ (i.e., absolute) will be pointed out in the argument.

\smallskip
\noindent  {\it Step~1.} We first work under an additional assumption
\begin{equation} \label{eq:abs-ode:pf:B1}
	\nnrm{B}_{(1)} < \eps,
\end{equation}
where $\eps < 1$ will be specified at the end of the step. Let $t \in (0, T)$ be fixed, and consider $s \in (0, t)$. Define
\begin{align*}
	\tensor{A}{_{\bfA}^{\bfA'}}(s, t) &\ceq \left[ \sup_{s' \in [s, t]} (s')^{- d_{\bfA}} t^{d_{\bfA'}}\abs{\tensor{\Psi}{_{\bfA}^{\bfA'}}(s', t)}
	+ \int_{s}^{t} (s')^{- d_{\bfA} - 1} t^{d_{\bfA'}} \abs{\tensor{\Psi}{_{\bfA}^{\bfA'}}(s', t)} \, \ud s'\right], \\
	A(s, t) &\ceq \sum_{\bfA, \bfA'} \tensor{A}{_{\bfA}^{\bfA'}}(s, t).
\end{align*}
By the assumptions on $B$, it easily follows that $\Psi(s, t)$ is continuous for $s \in (0, t]$, and that $A(s, t) < +\infty$ for $s \in (0, t]$. The goal of this step is to prove that, under \eqref{eq:abs-ode:pf:B1}, we have, for all $s \in (0, t)$,
\begin{equation*}
	A(s, t) \aleq (1+\nnrm{B}_{(\infty)})^{N} A_{0}.
\end{equation*}

Let us use \eqref{eq:abs-ode} to estimate $\Psi(s, t)$. We begin by estimating the contribution of $G$:
\begin{align}
	s^{-d_{\bfA}} t^{d_{\bfA'}} \int_{s}^{t} \abs{\tensor{G}{_{\bfA}^{\bfA'}}(s', t)} \, \ud s' &\leq A_{0} & & \hbox{ if } d_{\bfA} \leq 0, \label{eq:abs-ode:pf:G-infty} \\
	\int_{s}^{t} (s'')^{-d_{\bfA}-1} t^{d_{\bfA'}} \int_{s''}^{t} \abs{\tensor{G}{_{\bfA}^{\bfA'}}(s', t)} \, \ud s' \, \ud s'' &\leq \frac{1}{(-d_{\bfA})} A_{0} & & \hbox{ if } d_{\bfA} < 0. \label{eq:abs-ode:pf:G-1}
\end{align}
Indeed, recalling that $d_{\bfA}, d_{\bfA'} \leq 0$ by \eqref{eq:ode-dmax}, bound \eqref{eq:abs-ode:pf:G-infty} follows simply from \eqref{eq:abs-ode-G} and the obvious inequality $s^{-d_{\bfA}} \leq (s')^{-d_{\bfA}}$ for $s' \in [s, t]$. To prove \eqref{eq:abs-ode:pf:G-1}, note that its LHS is bounded (via Fubini) by
\begin{align*}
\int_{s}^{t} \left( \int_{s}^{s'} (s'')^{-d_{\bfA}-1} \, \ud s''\right) t^{d_{\bfA'}} \abs{\tensor{G}{_{\bfA}^{\bfA'}}(s', t)} \, \ud s'
\leq \frac{1}{(-d_{\bfA})} \int_{s}^{t} (s')^{-d_{\bfA}} t^{d_{\bfA'}} \abs{\tensor{G}{_{\bfA}^{\bfA'}}(s', t)} \, \ud s' \leq \frac{1}{(-d_{\bfA})} A_{0},
\end{align*}
where we used $d_{\bfA} < 0$ in the first inequality.

Next, we claim that
\begin{align}
	s^{-d_{\bfA}} t^{d_{\bfA'}} \int_{s}^{t} \abs{\tensor{B}{_{\bfA}^{\bfA''}} \tensor{\Psi}{_{\bfA''}^{\bfA'}}(s', t)} \, \ud s' &\leq \begin{cases}
	\nnrm{B}_{(\infty)} \tensor{A}{_{\bfA''}^{\bfA'}}(s, t) & \hbox{ if } d_{\bfA} \leq 0, \, d_{\bfA} = d_{\bfA''} + 1, \\
	\nnrm{B}_{(1)} A(s, t) & \hbox{ if } d_{\bfA} \leq 0, \, d_{\bfA} \leq d_{\bfA''} ,
	\end{cases} \label{eq:abs-ode:pf:BPsi-infty} \\
	\int_{s}^{t} (s'')^{-d_{\bfA}-1} t^{d_{\bfA'}} \int_{s''}^{t} \abs{\tensor{B}{_{\bfA}^{\bfA''}} \tensor{\Psi}{_{\bfA''}^{\bfA'}}(s', t)} \, \ud s' \, \ud s'' &\leq \begin{cases}
	\frac{1}{(-d_{\bfA})} \nnrm{B}_{(\infty)} \tensor{A}{_{\bfA''}^{\bfA'}}(s, t) & \hbox{ if } d_{\bfA} < 0, \, d_{\bfA} = d_{\bfA''} + 1, \\
	\frac{1}{(-d_{\bfA})} \nnrm{B}_{(1)} A(s, t) & \hbox{ if } d_{\bfA} < 0, \, d_{\bfA} \leq d_{\bfA''}.
	\end{cases} \label{eq:abs-ode:pf:BPsi-1}
\end{align}
Indeed, for \eqref{eq:abs-ode:pf:BPsi-infty} in the case $d_{\bfA} = d_{\bfA''} + 1$, we have
\begin{align*}
	s^{-d_{\bfA}} t^{d_{\bfA'}} \int_{s}^{t} \abs{\tensor{B}{_{\bfA}^{\bfA''}} \tensor{\Psi}{_{\bfA''}^{\bfA'}}(s', t)} \, \ud s'
	\leq \nnrm{B}_{(\infty)} \int_{s}^{t} (s')^{-d_{\bfA''}-1} t^{d_{\bfA'}} \abs{\tensor{\Psi}{_{\bfA''}^{\bfA'}}(s', t)} \, \ud s' \leq \nnrm{B}_{(\infty)} \tensor{A}{_{\bfA''}^{\bfA'}}(s, t),
\end{align*}
and in the case $d_{\bfA} \leq d_{\bfA''}$, we have
\begin{align*}
	s^{-d_{\bfA}} t^{d_{\bfA'}} \int_{s}^{t} \abs{\tensor{B}{_{\bfA}^{\bfA''}} \tensor{\Psi}{_{\bfA''}^{\bfA'}}(s', t)} \, \ud s'
	\leq \int_{s}^{t} (s')^{-d_{\bfA}+d_{\bfA''}} \abs{\tensor{B}{_{\bfA}^{\bfA''}}} A(s, t) \, \ud s' \leq \nnrm{B}_{(1)} A(s, t).
\end{align*}
Similarly, for \eqref{eq:abs-ode:pf:BPsi-1} in the case $d_{\bfA} = d_{\bfA''} + 1$, we have (by Fubini and $- d_{\bfA} > 0$)
\begin{align*}
\int_{s}^{t} (s'')^{-d_{\bfA}-1} t^{d_{\bfA'}} \int_{s''}^{t} \abs{\tensor{B}{_{\bfA}^{\bfA''}} \tensor{\Psi}{_{\bfA''}^{\bfA'}}(s', t)} \, \ud s' \, \ud s''
&\leq \nnrm{B}_{(\infty)} \int_{s}^{t} (s'')^{-d_{\bfA}-1} \int_{s''}^{t} t^{d_{\bfA'}} \abs{\tensor{\Psi}{_{\bfA''}^{\bfA'}}(s', t)} \, \ud s' \, \ud s''  \\
&\leq \frac{1}{(-d_{\bfA})} \nnrm{B}_{(\infty)} \int_{s}^{t} (s')^{-d_{\bfA}} t^{d_{\bfA'}} \abs{\tensor{\Psi}{_{\bfA''}^{\bfA'}}(s', t)} \, \ud s' \, \ud s''
\end{align*}
and in the case $d_{\bfA} \leq d_{\bfA''}$, we have (again by Fubini and $-d_{\bfA} > 0$)
\begin{align*}
\int_{s}^{t} (s'')^{-d_{\bfA}-1} t^{d_{\bfA'}} \int_{s''}^{t} \abs{\tensor{B}{_{\bfA}^{\bfA''}} \tensor{\Psi}{_{\bfA''}^{\bfA'}}(s', t)} \, \ud s' \, \ud s''
&\leq \int_{s}^{t} (s'')^{-d_{\bfA}-1} \int_{s''}^{t} (s')^{d_{\bfA''}} \abs{\tensor{B}{_{\bfA}^{\bfA''}}} A(s, t) \, \ud s' \, \ud s'' \\
& \leq \frac{1}{(-d_{\bfA})} \nnrm{B}_{(1)} A(s, t).
\end{align*}

Using \eqref{eq:abs-ode}, \eqref{eq:abs-ode:pf:G-1}--\eqref{eq:abs-ode:pf:BPsi-1}, we arrive at the inequality
\begin{align*}
	\tensor{A}{_{\bfA}^{\bfA'}}(s, t) \aleq A_{0} + \nnrm{B}_{(1)} A(s, t) + \sum_{\bfA'' : d_{\bfA''} = d_{\bfA} - 1} \nnrm{B}_{(\infty)} \tensor{A}{_{\bfA''}^{\bfA'}}(s, t).
\end{align*}
The term with $\nnrm{B}_{(1)}$ will eventually be absorbed into the left-hand side using \eqref{eq:abs-ode:pf:B1}, but the term with $\nnrm{B}_{(\infty)}$ needs care, since we are \emph{not} assuming smallness. Nevertheless, there is a reductive (or nilpotent) structure for this term. More precisely, it is not present when $d_{\bfA} = \min_{\bfA''} d_{\bfA''}$, and for $d_{\bfA} > \min_{\bfA''} d_{\bfA''}$, it only involves $\tensor{A}{_{\bfA''}^{\bfA'}}$ with $d_{\bfA''} = d_{\bfA} - 1$. Therefore, iterating this bound (no more than $N$ times), we arrive at
\begin{align*}
	A(s, t) \leq C_{1} (1+\nnrm{B}_{(\infty)})^{N} (A_{0} + \nnrm{B}_{(1)} A(s, t)).
\end{align*}
Taking $\eps = c C_{1}^{-1} (1+\nnrm{B}_{(\infty)})^{-N}$ for a sufficiently small absolute constant $c > 0$, we may absorb the contribution of $\nnrm{B}_{(1)} A(s, t)$ into the LHS and obtain the desired estimate for $A(s, t)$.

\smallskip \noindent {\it Step~2.} Next, we consider the general case when $\nnrm{B}_{(1)}$ is finite but possibly large. Let us split $[0, t] = [t_{m}, t_{m-1}] \cup \cdots \cup [t_{1}, t_{0}]$ with $t_{m} = 0$ and $t_{0} = t$ so that
\begin{equation} \label{eq:abs-ode:pf:B1-i}
	\nnrm{\chf_{[t_{i}, t_{i-1}]} B}_{(1)} < \eps \qquad \hbox{ for each } i = 1, \ldots, m,
\end{equation}
where $\eps$ is as in Step~1. Since $\nnrm{\cdot}_{(1)}$ consists of $L^{1}$-type norms, such a splitting with $m \leq C_{N} \eps^{-1} \nnrm{B}_{(1)}$ exists.
For $t_{i} \leq s < t_{i-1}$, define
\begin{align*}
A_{(i)}(s, t_{i-1})
&\ceq \sum_{\bfA, \bfA'} \bb[ \sup_{s' \in [s, t_{i-1}]} (s')^{- d_{\bfA}} t^{d_{\bfA'}}\abs{\tensor{\Psi}{_{\bfA}^{\bfA'}}(s', t) - \tensor{\Psi}{_{\bfA}^{\bfA'}}(t_{i-1}, t) } \\
	&\phantom{\ceq \sum_{\bfA, \bfA'} \bb[ }
	+ \int_{s}^{t_{i-1}} (s')^{- d_{\bfA} - 1} t^{d_{\bfA'}} \abs{\tensor{\Psi}{_{\bfA}^{\bfA'}}(s', t) - \tensor{\Psi}{_{\bfA}^{\bfA'}}(s', t_{i-1})} \, \ud s' \bb].
\end{align*}
For $t_{i} \leq s < t_{i-1}$, we claim that
\begin{align}
s^{-d_{\bfA}} t^{d_{\bfA'}} \abs{\tensor{\Psi}{_{\bfA}^{\bfA'}}(s, t)} &\leq A_{(i)}(s, t_{i-1}) + \sum_{i' \leq i-1} A_{(i')}(t_{i'}, t_{i'-1}), \label{eq:abs-ode:pf:cont-infty} \\
\int_{s}^{t}(s')^{-d_{\bfA}-1} t^{d_{\bfA'}} \abs{\tensor{\Psi}{_{\bfA}^{\bfA'}}(s', t)} \, \ud s' &\aleq A_{(i)}(s, t_{i-1}) + \sum_{i' \leq i-1} (i - i' + 1) A_{(i')}(t_{i'}, t_{i'-1}), \label{eq:abs-ode:pf:cont-1}
\end{align}
where we take the sum to be vacuous if $i = 1$, and the latter bound is only for $\bfA, \bfA'$ such that $d_{\bfA} < 0$. Indeed, estimate \eqref{eq:abs-ode:pf:cont-infty} follows easily by writing $\tensor{\Psi}{_{\bfA}^{\bfA'}}(s, t)$ as a telescopic sum
\begin{equation} \label{eq:abs-ode:pf:Psi-sum}
\tensor{\Psi}{_{\bfA}^{\bfA'}}(s, t) = \tensor{\Psi}{_{\bfA}^{\bfA'}}(s, t) - \tensor{\Psi}{_{\bfA}^{\bfA'}}(t_{i-1}, t) + \sum_{i' \leq i-1} (\tensor{\Psi}{_{\bfA}^{\bfA'}}(t_{i'}, t) - \tensor{\Psi}{_{\bfA}^{\bfA'}}(t_{i'-1}, t)).
\end{equation}
For the proof of \eqref{eq:abs-ode:pf:cont-1}, let us assume that $s = t_{i}$ for simplicity; the general case is a minor modification of this case. We split the integration domain into $[t_{i}, t_{i-1}] \cup \cdots \cup [t_{1}, t_{0}]$ and also the integrand using a telescopic sum akin to \eqref{eq:abs-ode:pf:Psi-sum} as follows:
\begin{align*}
	&\int_{s}^{t}(s')^{-d_{\bfA}-1} t^{d_{\bfA'}} \abs{\tensor{\Psi}{_{\bfA}^{\bfA'}}(s', t)} \, \ud s' \\
	&\leq \sum_{i' \leq i} \int_{t_{i'}}^{t_{i'-1}} (s')^{-d_{\bfA} - 1} t^{d_{\bfA'}} \left(\abs{\tensor{\Psi}{_{\bfA}^{\bfA'}}(s', t) - \tensor{\Psi}{_{\bfA}^{\bfA'}}(t_{i'-1}, t)} + \sum_{i'' \leq i'-1} \abs{\tensor{\Psi}{_{\bfA}^{\bfA'}}(t_{i''}, t) - \tensor{\Psi}{_{\bfA}^{\bfA'}}(t_{i''-1}, t)} \right) \, \ud s'.
\end{align*}
Estimating the contribution of $\Psi(s', t) - \Psi(t_{i'-1}, t)$ on $[t_{i'}, t_{i'-1}]$ by $A_{(i')}(t_{i'}, t_{i'-1})$, and the rest using \eqref{eq:abs-ode:pf:cont-infty} and the obvious integral
\begin{align*}
	\int_{t_{i'}}^{t_{i'-1}} (s')^{-d_{\bfA} -1} \, \ud s' = \frac{1}{(-d_{\bfA})} t_{i'-1}^{-d_{\bfA}},
\end{align*}
the desired estimate \eqref{eq:abs-ode:pf:cont-1} follows.

 For $s \in [t_{i+1}, t_{i})$, we rewrite \eqref{eq:abs-ode} as
\begin{align*}
	\tensor{\Psi}{_{\bfA}^{\bfA'}}(s, t) - \tensor{\Psi}{_{\bfA}^{\bfA'}}(t_{i}, t)
	&= \int_{s}^{t_{i}} \tensor{G}{_{\bfA}^{\bfA'}}(s', t) \, \ud s' + \int_{s}^{t_{i}} \sum_{\bfA''} \tensor{B}{_{\bfA}^{\bfA''}}(s') \tensor{\Psi}{_{\bfA''}^{\bfA'}}(t_{i}, t) \, \ud s' \\
	&\peq + \int_{s}^{t_{i}} \sum_{\bfA''} \tensor{B}{_{\bfA}^{\bfA''}}(s') (\tensor{\Psi}{_{\bfA''}^{\bfA'}}(s', t) - \tensor{\Psi}{_{\bfA''}^{\bfA'}}(t_{i}, t)) \, \ud s'\\
	&= \int_{s}^{t} \chf_{[t_{i+1}, t_{i}]}(s') \left( \tensor{G}{_{\bfA}^{\bfA'}}(s', t) + \sum_{\bfA''} \tensor{B}{_{\bfA}^{\bfA''}}(s') \tensor{\Psi}{_{\bfA''}^{\bfA'}}(t_{i}, t)\right)  \, \ud s' \\
	&\peq + \int_{s}^{t} \sum_{\bfA''} \chf_{[t_{i+1}, t_{i}]}(s') \tensor{B}{_{\bfA}^{\bfA''}}(s') (\tensor{\Psi}{_{\bfA''}^{\bfA'}}(s', t) - \tensor{\Psi}{_{\bfA''}^{\bfA'}}(t_{i}, t))  \, \ud s'.
\end{align*}
Using the last identity, we may trivially extend $\tensor{\Psi}{_{\bfA}^{\bfA'}}(s, t) - \tensor{\Psi}{_{\bfA}^{\bfA'}}(t_{i}, t)$ to all $s \in (0, t)$. Then, in view of $\nnrm{\chf_{[t_{i+1}, t_{i}]} B}_{(\infty)} \leq \nnrm{B}_{(\infty)}$ and \eqref{eq:abs-ode:pf:B1-i}, the argument in Step~1 is applicable to $\tensor{\Psi}{_{\bfA}^{\bfA'}}(s, t) - \tensor{\Psi}{_{\bfA}^{\bfA'}}(t_{i}, t)$. We claim that for every $\bfA, \bfA'$, we have
\begin{align*}
\int_{s}^{t} (s')^{-d_{\bfA}} t^{d_{\bfA'}} \chf_{[t_{i+1}, t_{i}]}(s')  \abs{\tensor{G}{_{\bfA}^{\bfA'}}(s', t)}  \, \ud s' &\leq A_{0}, \\
\int_{s}^{t} (s')^{-d_{\bfA}} t^{d_{\bfA'}} \chf_{[t_{i+1}, t_{i}]}(s')  \sum_{\bfA''} \abs{\tensor{B}{_{\bfA}^{\bfA''}}(s') \tensor{\Psi}{_{\bfA''}^{\bfA'}}(t_{i}, t)}  \, \ud s' &\aleq \nnrm{\chf_{[t_{i+1}, t_{i}]} B} \sum_{i' \leq i} A_{(i')}(t_{i'}, t_{i'-1}).
\end{align*}
The first bound is an immediate consequence of the hypothesis for $G$. To establish the second bound, note that
\begin{align*}
\int_{s}^{t} (s')^{-d_{\bfA}} t^{d_{\bfA'}} \chf_{[t_{i+1}, t_{i}]}(s') \abs{\tensor{B}{_{\bfA}^{\bfA''}}(s') \tensor{\Psi}{_{\bfA''}^{\bfA'}}(t_{i}, t)}  \, \ud s' &
\leq \int_{s}^{t_{i}} (s')^{-d_{\bfA}} t_{i}^{d_{\bfA''}} \abs{\tensor{B}{_{\bfA}^{\bfA''}}(s')} \, \ud s' \left( t_{i}^{-d_{\bfA''}} t^{d_{\bfA'}} \abs{\tensor{\Psi}{_{\bfA''}^{\bfA'}}(t_{i}, t)} \right) \\
&\aleq \nnrm{\chf_{[t_{i+1}, t_{i}]} B} \sum_{i' \leq i} A_{(i')}(t_{i'}, t_{i'-1}),
\end{align*}
where we used \eqref{eq:abs-ode:pf:cont-infty}. We remark that the integral can be bounded easily by dividing into cases $d_{\bfA} = d_{\bfA''} + 1$ and $d_{\bfA} \leq d_{\bfA''}$. Now, applying Step~1, we derive
\begin{equation*}
	A_{(i+1)}(s, t_{i}) \aleq (1+\nnrm{B}_{(\infty)})^{N} \bb( A_{0} + \nnrm{\chf_{[t_{i+1}, t_{i}]}B} \sum_{i' \leq i} A_{(i')} (t_{i'}, t_{i'-1}) \bb),
\end{equation*}
where the last term inside the parentheses does not exist for $i = 0$. By a simple induction argument on $i \in \set{1, \ldots, m-1}$ (observe also that $\nnrm{\chf_{[t_{i+1}, t_{i}]}B} \leq \nnrm{B}_{(\infty)} + \eps$), we obtain
\begin{align*}
\sum_{i \leq m} A_{(i)} (t_{i}, t_{i-1}) \aleq (2+\nnrm{B}_{(\infty)})^{C_{N}(1+ m)} A_{0}.
\end{align*}
Combined with \eqref{eq:abs-ode:pf:cont-infty}--\eqref{eq:abs-ode:pf:cont-1}, the desired estimates \eqref{eq:abs-ode-infty}--\eqref{eq:abs-ode-1} follow. \qedhere
\end{proof}

The following result for the fundamental matrix $\Pi$ is an immediate corollary of Proposition~\ref{prop:ode}:

\begin{corollary} \label{cor:ode-fundsol}
Consider degrees $d_{\bfA}$ ($\bfA \in \set{1, \ldots, N}$) and an $N \times N$ matrix-valued function $B$ on $[0, T]$ satisfying \eqref{eq:ode-dmax}, \eqref{eq:ode-B} and $\nnrm{B} < + \infty$. Let $\tensor{\Pi}{_{\bfA}^{\bfA'}} = \tensor{\Pi}{_{\bfA}^{\bfA'}}(s, t)$ be given by \eqref{eq:ode-Pi} for $0 < s < t < T$. Then, for $0 < s < t < T$, we have
\begin{align}
	s^{- d_{\bfA}} t^{d_{\bfA'}}\abs{\tensor{\Pi}{_{\bfA}^{\bfA'}}(s, t) - \tensor{\dlt}{_{\bfA}^{\bfA'}}}  & \leq C_{N} (1+\nnrm{B}_{(\infty)})^{N} \exp(C_{N} \nnrm{B}_{(1)}) \nnrm{B}  & & \hbox{ if } d_{\bfA} \leq 0, \\
	\int_{s}^{t} (s')^{- d_{\bfA} - 1} t^{d_{\bfA'}} \abs{\tensor{\Pi}{_{\bfA}^{\bfA'}}(s', t) - \tensor{\dlt}{_{\bfA}^{\bfA'}}} \, \ud s' & \leq C_{N} (1+\nnrm{B}_{(\infty)})^{N} \exp(C_{N} \nnrm{B}_{(1)}) \nnrm{B} & & \hbox{ if } d_{\bfA} < 0.
\end{align}
\end{corollary}
This corollary will be sufficient to verify \ref{hyp:crc-x} for many divergence-type equations Appendix~\ref{sec:appl}. However, to establish the higher derivative bounds in \ref{hyp:crc-x-q}, we will directly apply Proposition~\ref{prop:ode} to a suitably differentiated ODE system.

\subsection{From augmented system to quantitative \ref{hyp:crc}} \label{subsec:ode2crc}

We now show that graded augmented system in the sense of \ref{hyp:aug1}--\ref{hyp:aug4} leads to the quantitative version of \ref{hyp:crc}, namely, \ref{hyp:crc-x-q}, \ref{hyp:b-y1-q} and \ref{hyp:b-y1-smth}.

Let $\bfx : W \times [0, 1] \to \bbR^{d}$ satisfy assumptions \ref{hyp:x-1}, \ref{hyp:x-2}, and \ref{hyp:x-3} in Sections~\ref{subsubsec:x-q}. Given $M \in \bbZ_{\geq 0}$ and $\dlt \geq 0$, we introduce the following norm adapted to $(\bfx, W)$ ($G$ is for the underlying Geometry):
\begin{align*}
    \nrm{\bfb}_{\dot{G}^{M, \dlt}(\bfx, W)} &\ceq \sup_{(y, y_{1}) \in W} \sum_{\alp : \abs{\alp} \leq M} \int_{0}^{1} \abs{(\rd^{\alp} \bfb)(\bfx(y, y_{1}, s))} \abs{y_{1} - y}^{\dlt+\abs{\alp}}  s^{\dlt+\abs{\alp}-1} \, \ud s,\quad \delta>0, \\
    \nrm{\bfb}_{\dot{G}^{M, \dlt}_{\infty}(\bfx, W)} &\ceq \sup_{(y, y_{1}) \in W} \sum_{\alp : \abs{\alp} \leq M} \sup_{0\leq s\leq 1} \left[ \abs{(\rd^{\alp} \bfb)(\bfx(y, y_{1}, s))} \abs{y_{1} - y}^{\dlt+\abs{\alp}}  s^{\dlt+\abs{\alp}} \right] .
\end{align*}
When $\dlt = 0$, we will only use the $L^{\infty}$-based norm; hence we introduce the shorthand
\begin{align*}
    \nrm{\bfb}_{\dot{G}^{M, 0}(\bfx, W)} \ceq \nrm{\bfb}_{\dot{G}^{M, 0}_{\infty}(\bfx, W)} &\ceq \sup_{(y, y_{1}) \in W} \sum_{\alp : \abs{\alp} \leq M} \sup_{0\leq s\leq 1} \left[ \abs{(\rd^{\alp} \bfb)(\bfx(y, y_{1}, s))} \abs{y_{1} - y}^{\abs{\alp}}  s^{\abs{\alp}} \right] .
\end{align*}

Given $L_{b} > 0$, we also consider the following \emph{inhomogeneous} norms:
\begin{align*}
    \nrm{\bfb}_{G^{M, \dlt}(\bfx, W; L_{b})} &\ceq \sup_{(y, y_{1}) \in W} \sum_{\alp : \abs{\alp} \leq M} \int_{0}^{1} \abs{(\rd^{\alp} \bfb)(\bfx(y, y_{1}, s))} \max \set{L_{b}, \abs{y_{1} - y}}^{\dlt+\abs{\alp}} s^{\dlt-1} \, \ud s,\quad \delta>0, \\
    \nrm{\bfb}_{G^{M, \dlt}_{\infty}(\bfx, W; L_{b})} &\ceq \sup_{(y, y_{1}) \in W} \sum_{\alp : \abs{\alp} \leq M} \sup_{0\leq s\leq 1} \left[ \abs{(\rd^{\alp} \bfb)(\bfx(y, y_{1}, s))} \max\set{L_{b}, \abs{y_{1}-y}}^{\dlt+\abs{\alp}}  s^{\dlt} \right] .
\end{align*}
As before, we write $\nrm{\bfb}_{G^{M, 0}(\bfx, W; L_{b})} \ceq \nrm{\bfb}_{G^{M, 0}_{\infty}(\bfx, W; L_{b})}$.

\begin{remark} \label{rem:coeff-nrm-Ck}
Observe that all these norms are easily bounded by (appropriate) standard $C^{k}$-norms. Indeed, for an open set $V \subseteq \bbR^{d}$ with $\bfx(W \times [0, 1]) \subseteq V$, we have (for any $\dlt \geq 0$ and $M \in \bbZ_{\geq 0}$)
\begin{align*}
	\nrm{\bfb}_{\dot{G}^{M, \dlt}(\bfx, W)}, \nrm{\bfb}_{\dot{G}^{M, \dlt}_{\infty}(\bfx, W)} &\aleq_{M} \sup_{(y, y_{1}) \in W} \abs{y_{1} - y}^{\dlt+M} \nrm{\bfb}_{C^{M}(V)}, \\
	\nrm{\bfb}_{G^{M, \dlt}(\bfx, W; L_{b})}, \nrm{\bfb}_{G^{M, \dlt}_{\infty}(\bfx, W; L_{b})} &\aleq_{M} \sup_{(y, y_{1}) \in W} \max \set{L_{b}, \abs{y_{1} - y}}^{\dlt+M} \nrm{\bfb}_{C^{M}(V)}.
\end{align*}
\end{remark}

We are now ready to formulate the main result of this subsection. For $\bfB = (\tensor{(\bfB_{i})}{_{\bfA}^{\bfA'}})$ and $\bfC_{i} = (\tensor{(\bfC_{i})}{_{\bfA}^{(\gmm, K)}})$ as in \ref{hyp:aug3} and \ref{hyp:aug4}, we define
\begin{align*}
\nrm{\bfB}_{\dot{G}^{M}(\bfx, W)} &= \sum_{i, \bfA, \bfA'} \nrm{\tensor{(\bfB_{i})}{_{\bfA}^{\bfA'}}}_{\dot{G}^{M, d_{\bfA'} + 1 - d_{\bfA}}(\bfx, W)}, \\
\nrm{\bfB}_{G^{M}(\bfx, W; L_{b})} &= \sum_{i, \bfA, \bfA'} \nrm{\tensor{(\bfB_{i})}{_{\bfA}^{\bfA'}}}_{G^{M, d_{\bfA'} + 1 - d_{\bfA}}(\bfx, W; L_{b})}, \\
\nrm{\bfC}_{\dot{G}^{M}_{\infty}(\bfx, W)} &= \sum_{i, \bfA, K, \gmm} \nrm{\tensor{(\bfC_{i})}{_{\bfA}^{(\gmm, K)}}}_{\dot{G}^{M, d_{K} - \abs{\gmm} + 1 - d_{\bfA}}_{\infty}(\bfx, W)}.
\end{align*}
In view of Remark~\ref{rem:coeff-nrm-Ck}, the following statement holds: If $\bfB = (\tensor{(\bfB_{i})}{_{\bfA}^{\bfA'}})$ and $\bfC_{i} = (\tensor{(\bfC_{i})}{_{\bfA}^{(\gmm, K)}})$ as in \ref{hyp:aug3} and \ref{hyp:aug4} belong to $C^{M}(V)$ for some bounded open set $V \subseteq \bbR^{d}$ with $\bfx(W \times [0, 1]) \subseteq V$, then
\begin{equation} \label{eq:aug-nrm-Ck}
	\nrm{\bfB}_{\dot{G}^{M}(\bfx, W)},
	\nrm{\bfB}_{G^{M}(\bfx, W; L_{b})},
	\nrm{\bfC}_{\dot{G}^{M}_{\infty}(\bfx, W)},
	\leq C (\nrm{\bfB}_{C^{M}(V)} + \nrm{\bfC}_{C^{M}(V)}),
\end{equation}
where $C$ depends on $L_{b}$, $\diam V$, $d$, $\# \calA$, $s_{0}$, $\max_{\bfA} \abs{d_{\bfA}}$, $\max_{K} \abs{d_{K}}$, $m_{K}'$, and $M$.

\begin{proposition} [From graded augmented system to \ref{hyp:crc-x-q}, \ref{hyp:b-y1-q}, and \ref{hyp:b-y1-smth}] \label{prop:ode2crc}
Given an $\bbC^{r_{0}}$-valued function $\varphi$ on $U_{0}$, consider augmented variables $(\Phi_{\bfA})_{\bfA \in \calA}$ satisfying \ref{hyp:aug1}, \ref{hyp:aug2}, \ref{hyp:aug3}, and \ref{hyp:aug4} (with $N_{0} = \max_{\bfA \in \calA} \abs{d_{\bfA}} + 1$). Let $\bfx : W \times [0, 1] \to U$ satisfy \ref{hyp:x-1} and \ref{hyp:x-2}. Then the following holds.
\begin{enumerate}
\item If, for some $M, M_{g} \in \bbZ_{\geq 0}$ sufficiently large, we have
\begin{equation*}
\nrm{\bfB}_{\dot{G}^{M}(\bfx, W)} + \nrm{\bfC}_{\dot{G}^{M}_{\infty}(\bfx, W)} < + \infty, \quad \sup_{\bfA, \alp, J}\nrm{c[\Phi_{\bfA}]^{(\alp, J)}}_{C^{M_{g}}(U)} < + \infty,
\end{equation*}
then \ref{hyp:crc-x-q} and \ref{hyp:b-y1-q} hold with $M_{S}, M_{Z} \leq \min \set{M_{\bfx}, M} - C_{d}$ and
\begin{align*}
A_{S} &\leq C\left( d, A_{\bfx}, M_{\bfx}, M, \nrm{\bfB}_{\dot{G}^{M}(\bfx, W)}, \nrm{\bfC}_{\dot{G}^{M}_{\infty}(\bfx, W)} \right), \\
A_{Z} &\leq C\left( d, A_{\bfx}, M_{\bfx}, M, \nrm{\bfB}_{\dot{G}^{M}(\bfx, W)} \right) \sup_{(y, y_{1}) \in W} \abs{y - y_{1}}^{N_{0} - 1}, \\
A_{g} &\leq \sup_{\bfA, \alp, J} \nrm{c[\Phi_{\bfA}]^{(\alp, J)}}_{C^{M_{g}}(U)}.
\end{align*}

\item Assume in addition that \ref{hyp:x-3} holds for some $0 \leq M_{\bfx}'  \leq M_{\bfx}$. If, for some $M' \geq 0$ and $L_{b} > 0$,
\begin{equation*}
\nrm{\bfB}_{G^{M}(\bfx, W; L_{b})} < + \infty,
\end{equation*}
then \ref{hyp:b-y1-smth} holds with the same $L_{b}$, $M_{Z}' \leq \min \set{M_{\bfx}', M'} - C_{d}$ and
\begin{equation*}
A_{Z} \leq C\left( d, A_{\bfx}, M_{\bfx}', M', \nrm{\bfB}_{G^{M}(\bfx, W; L_{b})} \right) \sup_{(y, y_{1}) \in W} \max \set{L_{b}, \abs{y - y_{1}}}^{N_{0} - 1}.
\end{equation*}

\end{enumerate}
\end{proposition}
\begin{proof}
In view of Proposition~\ref{prop:crc-ODE}, we may express $S$ and $Z$ in terms of the coefficients $\bfB_{i}$ and $\bfC_{i}$ in \eqref{eq:aug-pde}; we review this process as follows. From \eqref{eq:aug-pde}, we obtain the following ODE system along each curve $\bfx(y, y_{1}, t)$:
\begin{equation*}\tag{\ref{eq:aug-ode}}
\frac{\ud}{\ud s} \left( \Phi_{\bfA} \circ \bfx \right) = \dot{\bfx}^{i} \left( \tensor{(\bfB_{i})}{_{\bfA}^{\bfA'}} \circ \bfx \right) \left( \Phi_{\bfA'}  \circ \bfx \right) + \dot{\bfx}^{i} \left( \tensor{(\bfC_{i})}{_{\bfA}^{(\gmm, K)}}\circ \bfx\right) \left( \rd^{\gmm} (\calP^{\ast} \varphi)_{K} \circ \bfx \right).
\end{equation*}
In what follows, we use the shorthand $\tensor{\bfPi}{_{\bfA}^{\bfA'}}(y, y_{1}; s, t) \ceq {}^{(\bfx_{y, y_{1}})} \tensor{\bfPi}{_{\bfA}^{\bfA'}}(s, t)$ for the fundamental matrix. By \eqref{eq:fund-mat-eq}, it solves
\begin{equation} \label{eq:ode2crc-fund-mat}
	\tensor{\bfPi}{_{\bfA}^{\bfA'}}(y, y_{1}; s, t) = \dlt_{\bfA}^{\bfA'} - \int_{s}^{t} B(y, y_{1}, s')  \tensor{\bfPi}{_{\bfA''}^{\bfA'}}(y, y_{1}; s', t) \, \ud s',
\end{equation}
where $B(y, y_{1}, s') \ceq \dot{\bfx}^{i}  \left( \tensor{(\bfB_{i})}{_{\bfA}^{\bfA''}} \circ \bfx \right) (y, y_{1}, s')$.

We now make the following claims:
\begin{enumerate}[label=(\roman*)]
\item Under the hypothesis of part (1), we have
\begin{equation} \label{eq:ode2crc-1}
    |L_{y}^{\abs{\alp}} \abs{y - y_{1}}^{\abs{\bt}} (\rd_{y} + \rd_{y_{1}})^{\alp} \rd_{y_{1}}^{\bt}  \tensor{\bfPi}{_{J}^{\bfA}}(y, y_{1}; s, t)|\lesssim |y_1-y|^{-d_{\bfA}}t^{-d_{\bfA}}
\end{equation}
for $\abs{\alp} + \abs{\bt} \leq \min\set{M_{\bfx}, M} - C_{d}$, where the implicit constant depends on $d$, $A_{\bfx}$, $M_{\bfx}$, $M$, and $\nrm{\bfB}_{\dot{G}^{M}(\bfx, W)}$.
\item Under the hypothesis of part (2), we have
\begin{equation} \label{eq:ode2crc-2}
    |L_{y}^{\abs{\alp}} L_{b}^{\abs{\bt}} (\rd_{y} + \rd_{y_{1}})^{\alp} \rd_{y_{1}}^{\bt} \tensor{\bfPi}{_{J}^{\bfA}}(y, y_{1}; s, t)|\lesssim L_{b}^{-d_{\bfA}} t^{-d_{\bfA}} \quad \hbox{ for } \abs{y_{1} - y} \leq L_{b}
\end{equation}
and for $\abs{\alp} + \abs{\bt} \leq \min\set{M_{\bfx}', M'} - C_{d}$, where the implicit constant depends on $d$, $A_{\bfx}$, $M_{\bfx}'$, $M'$, and $\nrm{\bfB}_{G^{M'}(\bfx, W; L_{b})}$.

\end{enumerate}

Assuming these claims, let us first complete the proof. Applying Proposition~\ref{prop:crc-ODE} to the present case, we may express $\tensor{S}{_{J}^{(\gmm, K)}}$ and $Z^{\bfA}$ ($\bfA \in \calA$) as follows:
\begin{align*}
	\tensor{S}{_{J}^{(\gmm, K)}}(y, y_{1}, t)  &= -\tensor{\bfPi}{_{J}^{\bfA}}(y, y_{1}; 0, t) \, \dot{\bfx}^{i} \left( \tensor{(\bfC_{i})}{_{\bfA}^{(\gmm, K)}} \circ \bfx \right) (y, y_{1}, t), \\
	(Z^{\bfA})(y, y_{1}) &= \tensor{\bfPi}{_{\varphi}^{\bfA}}(y, y_{1}; 0, 1), \\
	c[g_{\bfA}]^{(\alp, K)}(y_{1}) &= c[\Phi_{\bfA}]^{(\alp, K)}(y_{1}).
\end{align*}
In view of these expressions, as well as \ref{hyp:x-1} (to estimate $\dot{\bfx}^{i} = \rd_{s} \bfx^{i}$), the assertions in part (1) concerning $A_{S}$, $A_{Z}$, $M_{S}$, and $M_{Z}$ follow from the above claims. The assertion concerning $A_{g}$ is obvious from the last identity. Finally, under the hypothesis of part (2), which is stronger than that of part (1), we may combine \eqref{eq:ode2crc-1} and \eqref{eq:ode2crc-2} to verify \ref{hyp:b-y1-smth} with the above bounds on $M_{Z}'$ and $A_{Z}$. 

It remains to verify the claims. For (i), we set up an induction on $M_{\bfPi} \in \bbZ_{\geq 0}$ with the following induction hypothesis for $0 \leq \abs{\alp} + \abs{\bt} \leq M_{\bfPi}$:
\begin{equation}\label{eq:ode2crc-induction-1}
    |L_{y}^{\abs{\alp}} \abs{y - y_{1}}^{\abs{\bt}} (\rd_{y} + \rd_{y_{1}})^{\alp} \rd_{y_{1}}^{\bt} \tensor{\bfPi}{_{\bfA}^{\bfA'}}(y, y_{1}; s, t)| \leq C |y_1-y|^{d_{\bfA}-d_{\bfA'}+1}s^{d_{\bfA}}t^{-d_{\bfA'}},
\end{equation}
\begin{equation}\label{eq:ode2crc-induction-2}
     \int_s^t (s')^{-d_{\bfA}-1}t^{d_{\bfA'}}|L_{y}^{\abs{\alp}} \abs{y - y_{1}}^{\abs{\bt}} (\rd_{y} + \rd_{y_{1}})^{\alp} \rd_{y_{1}}^{\bt} \tensor{\bfPi}{_{\bfA}^{\bfA'}}(y, y_{1}; s', t)|\,\ud s'\leq C |y_1-y|^{d_{\bfA}-d_{\bfA'}+1},\quad d_{\bfA}<0.
\end{equation}
where
\begin{equation*}
	 C = C\left(d, \alp, \bt, A_{\bfx}, M_{\bfx}, M, \nrm{\bfB}_{\dot{G}^{M}(\bfx, W)}\right).
\end{equation*}

To carry out the induction, we introduce the renormalized parameter $\tilde{t}=|y_1-y|t$ and write
\begin{equation*}
    \td{\bfPi}(y, y_{1}; \tilde{s}, \tilde{t} )=\bfPi(y, y_{1}; |y_1-y|^{-1}\tilde{s},|y_1-y|^{-1}\tilde{t}).
\end{equation*}
Then we have
\begin{equation}\label{eq:ode2crc-rescale}
    \td{\bfPi}(y, y_{1}; \tilde{s},\tilde{t})=I-\int_{\tilde{s}}^{\tilde{t}}\td{B}(y, y_{1}, s') \td{\bfPi}(y, y_{1}; \tilde{s}',\tilde{t})\, \ud \tilde{s}',
\end{equation}
where $\td{B}(y, y_{1}, \tilde{s}) \ceq |y_1-y|^{-1}B(y_1,y,|y_1-y|^{-1}\tilde{s})$.
Thanks to the renormalization, we obtain $\abs{y_{1} - y}^{-1}$ in $\td{B}$, which offsets the factor $\abs{y_{1} - y}$ in \ref{hyp:x-1} for $\dot{\bfx} = \rd_{s} \bfx$. Indeed, a direct computation using \ref{hyp:x-1} and the definition of $\nrm{\cdot}_{\dot{G}^{M, \dlt}(\bfx, W)}$ shows
\begin{equation} \label{eq:ode2crc-B}
	\nnrm{L_{y}^{\abs{\alp}} \abs{y - y_{1}}^{\abs{\bt}}( (\rd_{y} + \rd_{y_{1}})^{\alp} \rd_{y_{1}}^{\bt} B){\td{\,\,\,}}(y, y_{1}, \cdot)} \aleq 1 + \nrm{\bfB}_{\dot{G}^{M}(\bfx, W)},
\end{equation}
for $|\alp|+|\bt|\leq \min\set{M_{\bfx}, M} - C_{d}$, where the implicit constant depends on $d$, $\alp$, $\bt$, $A_{\bfx}$. We carefully note that the variable change $t \mapsto \tilde{t}$ on the left-hand side is done after taking $(\rd_{y} + \rd_{y_{1}})^{\alp} \rd_{y_{1}}^{\bt}$.

The base case $M_{\bfPi} = 0$ follows from Proposition~\ref{prop:ode} applied to the rescaled ODE system \eqref{eq:ode2crc-rescale}.

Now we assume \eqref{eq:ode2crc-induction-1} \eqref{eq:ode2crc-induction-2} holds for all $|\alp|+|\bt|\leq M_{\bfPi}-1$, we would like to show \eqref{eq:ode2crc-induction-1}  \eqref{eq:ode2crc-induction-2} for $|\alp|+|\bt|=M_{\bfPi}$. We differentiate \eqref{eq:ode2crc-fund-mat}:
\begin{equation*}
\begin{aligned}
       & L_{y}^{\abs{\alp}} \abs{y - y_{1}}^{\abs{\bt}} (\rd_{y} + \rd_{y_{1}})^{\alp} \rd_{y_{1}}^{\bt} \bfPi(y, y_{1}; s, t)\\
        &=-L_{y}^{\abs{\alp}} \abs{y - y_{1}}^{\abs{\bt}}\sum\limits_{\alp=\alp'+\alp'',\beta=\beta'+\beta''}\binom{\alp}{\alp'}\binom{\bt}{\bt'}\int_s^t (\rd_{y}+\rd_{y_1})^{\alp'}\rd_{y_1}^{\bt'}B(y,y_1,s')(\rd_{y}+\rd_{y_1})^{\alp''}\rd_{y_1}^{\bt''}\bfPi(y,y_1;s',t)\,\ud s'\\
        &= -\int_s^t G(s',t)\,\ud s'-\int_s^t L_{y}^{\abs{\alp}} \abs{y - y_{1}}^{\abs{\bt}} B(y,y_1,s')(\rd_{y} + \rd_{y_{1}})^{\alp} \rd_{y_{1}}^{\bt} \bfPi(y, y_{1}; s', t)\,\ud s'
\end{aligned}
\end{equation*}
where
\begin{equation*}
    G(s,t)=L_{y}^{\abs{\alp}} \abs{y - y_{1}}^{\abs{\bt}}\sum\limits_{\alp=\alp'+\alp'',\beta=\beta'+\beta'',|\alp'|+|\beta'|>0}\binom{\alp}{\alp'}\binom{\bt}{\bt'} (\rd_{y}+\rd_{y_1})^{\alp'}\rd_{y_1}^{\bt'}B(y,y_1,s)(\rd_{y}+\rd_{y_1})^{\alp''}\rd_{y_1}^{\bt''}\bfPi(y,y_1;s,t).
\end{equation*}
At this point, we pass to the renormalized variable $\tilde{t} = \abs{y - y_{1}} t$ (and similarly use $\tilde{s}$, $\tilde{s}'$). By our induction hypothesis and \eqref{eq:ode2crc-B},
\begin{equation*}
    \int_{\tilde{s}}^{\tilde{t}} (\tilde{s}')^{-d_{\bfA}}\tilde{t}^{d_{\bfA'}}\tensor{\td{G}}{_{\bfA}^{\bfA'}}(\tilde{s}',\tilde{t})\,\ud \tilde{s}'\leq C.
\end{equation*}
Thus, applying Proposition~\ref{prop:ode}, we prove the induction hypothesis \eqref{eq:ode2crc-induction-1}  \eqref{eq:ode2crc-induction-2} for $|\alp|+|\bt|=M_{\bfPi}$ as long as $M_{\bfPi} \leq \min\set{M_{\bfx}, M} - C_{d}$. Then from the case $\bfA = J$ of \eqref{eq:ode2crc-induction-1}, we obtain \eqref{eq:ode2crc-1}.

Estimate \eqref{eq:ode2crc-2} is proved in the same way where $\abs{y_{1} - y}$ is replaced by $L_{b}$. More precisely, under the restriction $\abs{y_{1} - y} \leq L_{b}$, we use \ref{hyp:x-3} instead of \ref{hyp:x-1} and renormalized parameter is defined as $\tilde{t} = L_{b} t$. We omit the details. \qedhere

\end{proof}

\subsection{Connection and holonomy} \label{subsec:conn}
The graded augmented system as in Definition~\ref{def:aug} gives a connection on the trivial bundle $U\times \bbC^{\# \calA}$:
\begin{equation*}
 	D_{i} \Phi_{\bfA} \ud x^{i} \ceq\left(\partial_i \Phi_{\bfA}-\tensor{(\bfB_i)}{_{\bfA}^{\bfA'}}(x) \Phi_{\bfA'}(x) \right) \ud x^i.
\end{equation*}
A parallel section $\Phi_{\bfA}$ of this connection gives an element in $\ker\calP^*$. Given a point $y \in U$, let $\bfx(t):[0,1]\to U$ be a closed curve with $\bfx(0) = y$. The parallel transport along this curve induces a linear map on the fiber over $y$, $\alp_{\bfx}:\bbC^{\# \calA}\to \bbC^{\# \calA}$, called a \emph{holonomy map based at $y$}. If $D_{i} \Phi_{\bfA}=0$ for all $i$ for some $\Phi_{\bfA}$ on $U$, then one must have $\alp_{\bfx} \Phi_{\bfA}(y) =\Phi_{\bfA}(y)$. On the other hand, if there is a point $y \in U$ at which $\phi_{\bfA} \in \bbC^{\# \calA}$ is invariant under all holonomy maps $\alp_{\bfx}$ based at $y$, then the section $\Phi_{\bfA}$ defined by solving ODEs along (arbitrary) paths with $\Phi_{\bfA}(y) = \phi_{\bfA}$ (for all $\bfA \in \calA$) is well-defined and parallel, i.e., $\Phi_{\bfA} \in \ker \calP^{*}$. Therefore we conclude the following.
\begin{proposition}
    Let $U$ be a connected subset of $\bbR^{d}$, and $\calP$ an $r_{0} \times s_{0}$-matrix valued differential operator with $C^{\infty}(\br{U})$ coefficients. Let $(\Phi_{\bfA})_{\bfA \in \calA}$ be graded augmented variables as in Definition~\ref{def:aug}, with $C^{\infty}(\br{U})$ coefficients $\bfB_{i}, \bfC_{i}$ in \ref{hyp:aug3}. Then the  cokernel $\ker \calP^{*}$ may be identified with the space of vectors $\phi_{\bfA}$ at a point $y$ invariant under all holonomy maps $\alp_{\bfx}$ associated with $D_{i} \ceq \rd_{i} - \bfB_{i}$ based at $y$.
\end{proposition}
The infinitesimal behavior of the holonomy is controlled by the curvature of the connection:
\begin{equation*}
    \tensor{\bfR}{_{ij\bfA}^{\bfA'}}\ud x^i\wedge \ud x^j 
    = -\frac{1}{2} \left( \partial_i\tensor{(\bfB_j)}{_{\bfA}^{\bfA'}} - \partial_j \tensor{(\bfB_i)}{_{\bfA}^{\bfA'}} + \tensor{(\bfB_i)}{_{\bfA}^{\bfA''}}\tensor{(\bfB_j)}{_{\bfA''}^{\bfA'}} - \tensor{(\bfB_j)}{_{\bfA}^{\bfA''}}\tensor{(\bfB_i)}{_{\bfA''}^{\bfA'}} \right) \ud x^i\wedge  \ud x^j
\end{equation*}
The Ambrose--Singer theorem \cite{ambrosesinger} states that the Lie algebra of the holonomy group is spanned by the curvature forms $\set{\tensor{\bfR}{_{ij\bfA}^{\bfA'}}}_{i, j = 1, \ldots, d}$. Hence, if $ \tensor{\bfR}{_{ij\bfA}^{\bfA'}}(y)$ has trivial kernel for some $i, j$ and $y \in U$ (which is the generic case), then $\ker\calP^*=\{0\}$. On the other hand, Proposition~\ref{prop:zero-curv} now admits a simple proof:

\begin{proof}[Proof of Proposition~\ref{prop:zero-curv}]
Note that \eqref{eq:zero-curv} is nothing but the condition that $\tensor{\bfR}{_{ij \bfA}^{\bfA'}} = 0$ vanishes in $U$. If $\dim \ker \calP^{\ast} = \# \calA$, then the holonomy group must be trivial, which implies that the curvature vanishes. Conversely, if the curvature vanishes identically and $U$ is simply connected, then the holonomy group is trivial, and $\dim \ker \calP^* =\# \calA$ (this also follows directly from the Frobenius theorem, as in \cite[\S3.1]{ReshBook}). \qedhere 
\end{proof}

\section{\ref{hyp:fc} implies \ref{hyp:crc}} \label{sec:fc2crc}
In this section, we give a precise formulation and proof of Theorem~\ref{thm:alg-cond-intro}. A basic ingredient is Hilbert's Nullstellensatz (see, e.g., \cite[p.~85]{AtiMac}), which gives a correspondence between the set of common zeros of a family of polynomials and the ideal they generate. We recall its statement below:
\begin{proposition}[Hilbert's Nullstellensatz]\label{prop:nullstellensatz}
    Suppose $f_1(x_1,\cdots, x_d),\cdots, f_N(x_1,\cdots, x_d)\in \bbC[x_1,\cdots, x_d]$ are polynomials with the set of common zeros
    \begin{equation*}
        Z=\{(z_1,\cdots,z_d)\in \bbC^d:f_j(z_1,\cdots, z_d)=0,\,j=1,2,\cdots, N\}.
    \end{equation*}
    Then for any polynomial $h(x_1.\cdots, x_d)\in\bbC[x_1,\cdots,x_d]$, if $h$ vanishes on $Z$, then there exist $M>0$ and polynomials $g_j(x_1,\cdots,x_d)\in \bbC[x_1,\cdots, x_d]$, $j=1,2,\cdots, N$ such that
    \begin{equation*}
        h(x_1,\cdots,x_d)^{M}=\sum_j g_j(x_1,\cdots, x_d) f_j(x_1,\cdots, x_d).
    \end{equation*}
\end{proposition}

Our main result -- which is the precise version of Theorem~\ref{thm:alg-cond-intro}.(1) -- is the following. \begin{theorem} \label{thm:fc2crc}
Let $U$ be an open subset of $\bbR^{d}$. Suppose that $\calP$ has smooth coefficients. Assume that the $s_{0} \times r_{0}$-matrix-valued principal symbol $p^{\ast}(x, \xi)$ of $\calP^{\ast}$ (where $s_{0} \geq r_{0}$) satisfies
\begin{enumerate}[label=\ref{hyp:fc}]
\item $p^{\ast}(x, \xi)$ is injective for all $x \in U$ and $\xi \in \bbC^{d} \setminus \set{0}$.
\end{enumerate}
Then there exist augmented variables $(\Phi_{\bfA})_{\bfA \in \calA}$ on $U$ satisfying \ref{hyp:aug1}--\ref{hyp:aug4} with index set $\calA = \set{((\alp, J) : 1 \leq J \leq r_{0}, \, \abs{\alp} \leq N_{0} - 1}$ (which is \emph{maximal} in the sense of Definition~\ref{def:aug-max}), degree $d_{(\alp, J)} = -\abs{\alp}$, and coefficients $\bfB_{i},\bfC_{i} \in C^{\infty}(U)$, which have the following properties:
\begin{enumerate}
\item We have $\tensor{(\bfC_{i})}{_{(\alp, J)}^{(\gmm, K)}} = 0$ unless $m_{K} + \abs{\gmm} = \abs{\alp} + 1$.

\item If $\calP$ has \emph{constant coefficients}, then $\bfB_{i},\bfC_{i}$ are constant. In addition, if $\calP$ is also \emph{homogeneous} in the sense that $\calP = \calP_{\prin}$ (see Definition~\ref{def:prin-symb}), then
\begin{equation} \label{eq:fc2rc-const-hom-Bi}
\tensor{(\bfB_{i})}{_{(\alp, J)}^{(\alp', J')}} = \begin{cases} 1 & \hbox{ if } J = J', \, \alp' = \alp + \bfe_{i} \\
0 & \hbox{ otherwise}. \end{cases}
\end{equation}

\item In general, on any compact set $X \Subset U$, we have
\begin{equation}\label{eq:FC-RC-bfB}
\begin{aligned}
    \|\tensor{(\bfB_{i})}{_{(\alp, J)}^{(\alp', J')}}\|_{C^k(X)}\lesssim_{\|c[p^*]\|_{C^k(X)}} \left\{\begin{array}{ll}
    1, & d_{(\alp, J)}=d_{(\alp', J')}+1, \\
    \sum\limits_{\alp,J,K}\|\tensor{c[\calP^{\ast}]}{_{K}^{(\alp,J)}}\|_{C^{k+|\alp|-m_K+d_{\bfA'}-d_{\bfA}+1}(X)}, & d_{(\alp, J)}<d_{(\alp', J')}+1,
    \end{array}\right.
\end{aligned}
\end{equation}
and
\begin{equation}\label{eq:FC-RC-bfC}
    \|\tensor{(\bfC_{i})}{_{(\alp', J')}^{(\gmm, K)}}\|_{C^k(X)}\lesssim_{\|c[p^{\ast}]\|_{C^k(X)}} \sum_{|\alp|=m_K}\|\tensor{c[\calP^{\ast}]}{_{K}^{(\alp,J)}}\|_{C^k(X)},
\end{equation}
where $\|c[p^*]\|_{C^k(X)} \ceq \sum_{|\alp|=m_K}\|\tensor{c[\calP^{\ast}]}{_{K}^{(\alp,J)}}\|_{C^k(X)}$.
\end{enumerate}
\end{theorem}
\begin{proof}
\noindent{\bf Step 1:} For any $x_0\in U$, there exist $N_0>0$ and polynomial valued $r_0\times s_0$ matrices $g_{\alp}(\xi)$ for $|\alp|=N_0$ such that
\begin{equation}\label{eq:Nullstellensatz-x0}
    \xi^{\alp} I_{r_0}= g_{\alp}(\xi) p^*(x_0,\xi),\quad |\alp|=N_0.
\end{equation}
Moreover, $\tensor{(p^{\ast})}{_{K}^{J}}(\xi)$ is homogeneous of degree $m_K$ and $\tensor{(g_{\alp})}{_{J}^{K}}(\xi)$ is homogeneous of degree $N_0-m_K$.

The claim follows from applying Hilbert's Nullstellensatz (Proposition~\ref{prop:nullstellensatz}) to $\det(M_j(\xi))$ where $M_j(\xi)$ goes through all the $r_0\times r_0$ minors of $p^*(x_0,\xi)$. By assumption, the common zeros of $\det(M_j(\xi))$ in $\bbC^d$ is contained in $\{0\}$. Since $\xi_{\ell}$ vanishes at $0$, there exist $N_1,\cdots, N_d$ and polynomials $h_{\ell}^{j}(\xi)$ such that
\begin{equation*}
    \xi_\ell^{N_\ell}=\sum_{\ell} h^{j}_{\ell}(\xi) \det (M_j(\xi)).
\end{equation*}
We notice that $\det(M_j(\xi))I_{r_0}$ can be written as the product of its adjugate matrix ${\rm adj}(M_j(\xi))$ and itself:
\begin{equation*}
    \det(M_j(\xi))I_{r_0}={\rm adj}(M_j(\xi)) M_j(\xi).
\end{equation*}
Moreover, $M_j(\xi)$ is the product of a constant $0-1$ matrix and $p^*(x_0,\xi)$. Therefore we conclude that there are polynomial-valued $r_0\times s_0$ matrices $\td{g}^{\ell}(\xi)$ so that
\begin{equation*}
    \xi_{\ell}^{N_{\ell}} I_{r_0}= \td{g}^{\ell}(\xi) p^*(x_0,\xi).
\end{equation*}
Taking $N_0>N_1+\cdots+N_d$, we conclude \eqref{eq:Nullstellensatz-x0}. We may assume $\tensor{(g_{\alp})}{_{J}^{K}}(\xi)$ is homogeneous of degree $N_0-m_K$ since terms of other degrees do not contribute.

\noindent{\bf{Step 2:}} There exist $\xi$-polynomial-valued $r_0\times s_0$ matrices $g_{\alp}(x,\xi)$ for $|\alp|=N_0$ such that
\begin{equation}\label{eq:Nullstellensatz-x}
    \xi^{\alp} I_{r_0}= g_{\alp}(x,\xi) p^*(x,\xi),\quad |\alp|=N_0.
\end{equation}
The matrix $g_{\alp}(x,\xi)$ depends smoothly on the coefficients of $p^*(x,\xi)$ and some smooth cutoff functions.
Moreover, $\tensor{(p^{\ast})}{_{K}^{J}}(\xi)$ is homogeneous of degree $m_K$ and $\tensor{(g_{\alp})}{_{J}^{K}}(\xi)$ is homogeneous of degree $N_0-m_K$.

When $\calP$ has constant coefficients, then it suffices to take $g_{\alp}(x, \xi) = g_{\alp}(\xi)$ (from Step~1); hence, it suffices to only consider the case when $\calP$ does not have constant coefficients. After a partition of unity, it suffices to prove the claim in a small neighborhood of a given point $x_0$. By Step 1, we have
\begin{equation}\label{eq:Nullstellensatz-x0-x0}
    \xi^{\alp} I_{r_0}= g_{\alp}(x_0,\xi) p^*(x_0,\xi),\quad |\alp|=N_0.
\end{equation}
Thus
\begin{equation*}
    \xi^{\alp} I_{r_0}= g_{\alp}(x_0,\xi) p^*(x,\xi)+g_{\alp}(x_0,\xi)(p^*(x_0,\xi)-p^*(x,\xi)).
\end{equation*}
Since each entry of $g_{\alp}(x_0,\xi)(p^*(x_0,\xi)-p^*(x,\xi))$ is a homogeneous polynomial of degree $N_0$ in $\xi$, we apply \eqref{eq:Nullstellensatz-x0-x0} again:
\begin{equation*}
    \xi^{\alp} I_{r_0}= g_{\alp}(x_0,\xi) p^*(x,\xi)+\sum_{\beta}\eps^{\alp
    }_{\beta}(x)g^{\beta}(x_0,\xi)p^*(x_0,\xi)
\end{equation*}
where $\eps_{\beta}^{\alp}(x)$ are $r_0\times r_0$ matrices depending smoothly on coefficients of $p^*(x,\xi)$ such that $\eps(x,\xi)\to 0$ as $x\to x_0$. This can be done repeatedly and we get
\begin{equation*}
\begin{aligned}
   \xi^{\alp} I_{r_0}&= g_{\alp}(x_0,\xi) p^*(x,\xi)+\sum_{\beta}\eps^{\bt}_{\alp}(x) g_{\beta}(x_0,\xi)p^*(x,\xi)+\sum_{\beta,\gamma} \eps^{\bt}_{\alp}(x) \eps^{\gmm}_{\bt}(x) g_{\gmm}(x_0,\xi)p^*(x,\xi)
+\cdots\\
&= \left(g_{\alp}(x_0,\xi)
+ \sum_{\beta}\eps^{\bt}_{\alp}(x) g_{\beta}(x_0,\xi)
+ \sum_{\beta,\gamma} \eps^{\bt}_{\alp}(x) \eps^{\gmm}_{\bt}(x) g_{\gmm}(x_0,\xi)
+\cdots\right) p^*(x,\xi).
\end{aligned}
\end{equation*}
For $x$ in a small neighborhood of $x_0$, the Neumann series converges and we conclude \eqref{eq:Nullstellensatz-x}.

\smallskip \noindent{\it Step 3.}
We claim that taking $\Phi_{\bfA}= \partial^{\alp}\varphi_J$ to be the jet of $\varphi_J$ up to order $N_0-1$ indexed by $\bfA = (\alp, J)$ in the set $\calA \ceq \set{(\alp, J) : 1 \leq J \leq r_{0}, \,  |\alp| \leq N_0-1}$ gives an augmented system satisfying \ref{hyp:aug1}-\ref{hyp:aug4}. First, \ref{hyp:aug1} and \ref{hyp:aug2} are obvious from the definition. For \ref{hyp:aug3}, we take the trivial relation $\rd_{i} \partial^{\alp} \varphi_{J}(y) = \rd^{\alp+\bfe_{i}} \varphi_{J}(y)$ to be the equation if $\abs{\alp} < N_{0} - 1$, i.e.,
\begin{equation*}
\tensor{(\bfB_{i})}{_{(\alp, J)}^{(\alp', J')}} = \begin{cases} 1 &\hbox{if } J = J', \, \alp' = \alp + \bfe_{i}, \\ 0 & \hbox{otherwise}, \end{cases}
\quad \hbox{ and } \quad
\tensor{(\bfC_{i})}{_{(\alp, J)}^{(\gmm, K)}} = 0 \quad \hbox{ for } \abs{\alp} < N_{0} - 1.
\end{equation*}
It remains to check that, for $\abs{\alp} = N_{0} - 1$, there exist $\tensor{(\bfB_{i})}{_{(\alp, J)}^{(\alp', J')}}(y)$ and $\tensor{(\bfC_{i})}{_{(\alp, J)}^{(\gamma, K)}}(y)$ such that
\begin{equation}\label{eq:FC-RC-Aug}
    \rd_{i} \partial^{\alp}\varphi_J(y)= \tensor{(\bfB_{i})}{_{(\alp, J)}^{(\alp', J')}}(y) \rd^{\alp'} \varphi_{J'}(y)+\tensor{(\bfC_{i})}{_{(\alp, J)}^{(\gamma, K)}}(y) \partial^{\gamma}(\calP^*\varphi)_K(y).
\end{equation}
This follows from \eqref{eq:Nullstellensatz-x} and putting $(\calP_{\prin}^{\ast} - \calP^{\ast}) \varphi$ into the first term on the right-hand side. Indeed, $\bfC_{i}$ is determined by rewriting \eqref{eq:Nullstellensatz-x} in the form
\begin{equation} \label{eq:fc2crc-Ci-top}
\xi_{i} \xi^{\alp} \tensor{(I_{r_{0}})}{_{J}^{J'}} = \tensor{(\bfC_{i})}{_{(\alp, J)}^{(\gmm, K)}}(y) \xi^{\gmm} \tensor{(p^{\ast})}{_{K}^{J'}}(y, \xi) \quad \hbox{ for } \abs{\alp} = N_{0} - 1,
\end{equation}
(i.e., $\tensor{(g_{\alp+\bfe_{i}})}{_{J}^{K}}(y, \xi) = \tensor{(\bfC_{i})}{_{(\alp, J)}^{(\gmm, K)}}(y) \xi^{\gmm}$) and $\bfB_{i}$ is determined by
\begin{equation} \label{eq:fc2crc-Bi-top}
\tensor{(\bfB_{i})}{_{(\alp, J)}^{(\alp', J')}}(y) \rd^{\alp'} \varphi_{J'}(y) = \tensor{(\bfC_{i})}{_{(\alp, J)}^{(\gamma, K)}}(y) \tensor{(\calP^{\ast}_{\prin}\partial^{\gamma}-\partial^{\gamma}\calP^{\ast})}{_{K}^{J'}}\varphi_{J'}(y) \quad \hbox{ for } \abs{\alp} = N_{0} - 1.
\end{equation}
It remains to check \ref{hyp:aug4}. We define the degree of $\partial^{\alp}\varphi_J$ to be $d_{(\alp, J)} \ceq -|\alp|$. The condition \ref{hyp:aug4} for terms with degree $d_{(\alp, J)}>-N_0+1$ is obvious. It suffices to check $\tensor{(\bfB_{i})}{_{(\alp, J)}^{(\alp', J')}}$ and $\tensor{(\bfC_{i})}{_{(\alp, J)}^{(\gamma,K)}}(y)$ which corresponds to $d_{(\alp, J)}=-N_0+1$ in \ref{hyp:aug4}. Since $d_{(\alp', J')}\geq -N_0+1$, the condition for $\bfB$ is automatic. Moreover, $m_K+|\gamma|\leq N_0$, so $d_{(\alp, J)}\leq -m_K-|\gamma|+1$ and the condition for $\bfC$ follows.

It remains to verify (1), (2), and (3). The first statement follows from \eqref{eq:fc2crc-Ci-top}. For (2) and (3), the only nontrivial case to consider is $\abs{\alp} = N_{0} - 1$. When $\calP$ has constant coefficients, \eqref{eq:Nullstellensatz-x} holds with a polynomial $g_{\alp}(\xi)$. It follows that each coefficient $\tensor{(\bfC_{i})}{_{(\alp, J)}^{(\gmm, K)}}$ in \eqref{eq:fc2crc-Ci-top} is constant.  The claim about $\bfB_{i}$ (both when $\calP$ is homogeneous or not) then follows from \eqref{eq:fc2crc-Bi-top}. In general, estimate $\eqref{eq:FC-RC-bfC}$ holds because $g_{\alp}(x,\xi)$ depends smoothly on the coefficients $\tensor{c[\calP^*]}{_{K}^{(\alp,J)}}$ for $|\alp|=m_K$. Estimate \eqref{eq:FC-RC-bfB} follows from inspecting the lower order terms in \eqref{eq:FC-RC-Aug}.
\end{proof}

We may now give a precise formulation and proof of Theorem~\ref{thm:alg-cond-intro}.(2) as well.
\begin{proposition} \label{prop:const-coeff-fc}
Let $U$ be an connected open subset of $\bbR^{d}$, and $\calP$ an $r_{0} \times s_{0}$-matrix-valued operator on $U$ such that $\calP_{\prin}$ has constant coefficients. Then the following are equivalent:
\begin{enumerate}
\item $p^{\ast}(\xi)$ satisfies \ref{hyp:fc}.
\item There exists a maximal graded augmented system $\set{\Phi_{\bfA}}_{\bfA \in \calA}$ for $\calP$ on $U$.
\item $\calP_{\prin}$ satisfies \ref{hyp:crc-x-q} for $\ul{\bfx}(y, y_{1}, s) \ceq y + s(y_{1} - y)$ on any convex open subset $U'$ of $U$.
\item $\dim \ker \calP_{\prin}^{\ast}(U) < + \infty$.
\end{enumerate}
Moreover, if any (thus all) of the above conditions holds, then $\calP$ satisfies \ref{hyp:crc-x-q} on any $\bfx(y, y_{1}, s)$ satisfying \ref{hyp:x-1}, \ref{hyp:x-2}, and \ref{hyp:x-3}.
\end{proposition}
\begin{proof}
The implication $(1) \imp (2)$ is exactly Theorem~\ref{thm:fc2crc}. Moreover, $(2) \imp (3)$, as well as the last assertion, follows from  Proposition~\ref{prop:ode2crc} and Definition~\ref{def:aug-max}, and $(3) \imp (4)$ is a consequence of Lemma~\ref{lem:b-y1-coker} applied on the convex open subset $U'$ with $\bfx = \ul{\bfx}$, and the obvious observation that $\ker \calP^{\ast}(U)$ is a subspace of $\ker \calP^{\ast}(U')$ (by restriction). It remains to prove $(4) \imp (1)$, or equivalently, its contrapositive. Indeed, if $p^{\ast}$ does \emph{not} satisfy \ref{hyp:fc}, then there exists a constant vector $\set{\varphi_{J}}_{J \in \set{1, \ldots, r_{0}}} \neq 0$ and $\xi \in \bbC^{d} \setminus \set{0}$ such that $p^{\ast}(\xi) \varphi = 0$, which means that $e^{i \lmb \xi x} \varphi \in \ker \calP_{\prin}^{\ast}$ for every $\lmb \in \bbC$. Thus $\ker \calP_{\prin}^{\ast}$ is infinite dimensional.
\end{proof}

\appendix
\section{Examples of $\calP$ from geometry and physics} \label{sec:appl}
In Section~\ref{subsec:ideas-div} (see also Examples~\ref{ex:div-full-sol-conic}, \ref{ex:div-full-sol-max}, and \ref{ex:div-full-sol}), we have demonstrated how our method in Section~\ref{sec:crc} applies to the operator $\calP u = (\rd_{j} + B_{j}) u^{j}$. In this appendix, we write down graded augmented systems that are \emph{completely integrable} on a constant sectional curvature background. For each example, we also give explicit computations of (i)~the rough integral kernel $K_{y_{1}}(x, y)$, as well as the corresponding point distribution $b_{y_{1}}(x, y)$, on a geodesic segment in a constant sectional curvature background; and (ii)~the integral kernels for the conic and Bogovskii-type operators on $\bbR^{d}$.

\subsection{Preliminaries}
Before reading any of the following subsections, we advise the reader to go over the preliminaries below.

\subsubsection{Notation and conventions for this appendix} \label{subsubsec:appl-ex-prelim}
In this appendix, it is conceptually (and algebraically) advantageous to work with the notation and conventions of differential geometry. We work on a $d$-dimensional smooth manifold $\calM$ equipped with a Riemannian metric $\bfg$. We denote by $\nb$ and $\ud V$ the corresponding Levi-Civita connection and the Riemannian volume form, respectively. Given tensor fields $g$ and $\psi$ on (an open subset of) $\calM$ of types $(r, s)$ and $(s, r)$, respectively, the duality pairing with respect to $\ud V$ is defined as
\begin{equation} \label{eq:duality-pair-geom}
	\brk{g, \psi} = \int g \cdot \psi \, \ud V,
\end{equation}
where $g \cdot \psi$ is the natural pointwise contraction of an $(r, s)$-tensor field $g = \tensor{g}{^{j_{1} \cdots j_{r}}_{k_{1} \cdots k_{s}}}$ and an $(s, r)$-tensor field $\psi = \tensor{\psi}{_{j_{1} \cdots j_{r}}^{k_{1} \cdots k_{s}}}$. We also use the standard convention of lowering and raising indices using $\bfg_{ij}$ and $\bfg^{ij} = (\bfg^{-1})^{ij}$. We write $\abs{\cdot}_{\bfg}$ for the induced norm on tensors, and $\tr_{\bfg} \bfh = \bfg_{ij} \bfh^{ij}$.

In this appendix, unlike the main body of the paper, {\bf the formal adjoint $\calP^{\ast}$ of a differential operator $\calP$ acting on tensor fields on (an open subset of) $\calM$ is defined with respect to $\ud V$}. This convention leads to simpler formulas. Since $\ud V = \sqrt{\det \bfg} \ud x$ in local coordinates, the two different definitions are related by the formula
\begin{equation} \label{eq:adj-difff}
	\calP^{\ast_{\ud V}} \varphi = \frac{1}{\sqrt{\det \bfg}} \calP^{\ast_{\ud x}} (\sqrt{\det \bfg} \varphi)
\end{equation}
In particular, observe that \emph{the principal symbols are identical}.

\subsubsection{Preliminaries for explicit computations in the constant curvature case} \label{subsubsec:const-curv-prelim}
In Sections~\ref{subsubsec:d-div-explicit}, \ref{subsubsec:symm-div-explicit}, \ref{subsubsec:symm-div-tf-explicit}, we obtain explicit covariant formulas for the rough integral kernel $K_{y_{1}}(x, y)$, as well as the corresponding point distribution $b_{y_{1}}(x, y)$, on a geodesic segment in a constant sectional curvature background (i.e., space form). These furnish a starting point for the construction of various smoothly averaged integral kernels, such as the conic and Bogovskii-type operators on $\bbR^{d}$ (cf.~Examples~\ref{ex:div-full-sol-conic} and \ref{ex:div-full-sol-max}).

Under the convention $\tensor{\bfR}{_{i j}^{k}_{\ell}} u^{\ell} = (\nb_{i} \nb_{j} - \nb_{j} \nb_{i})u^{k}$ for the Riemann curvature tensor, recall that the Riemann, Ricci, and scalar curvature tensors on a Riemannian manifold $(\calM, \bfg)$ with constant sectional curvature $\kpp$ are given by
\begin{equation*}
	\bfR_{i j k \ell} = \kpp (\bfg_{ik} \bfg_{j \ell} - \bfg_{i\ell} \bfg_{jk}), \quad
	\Ric = (d-1) \kpp \bfg, \quad
	R = d (d-1) \kpp,
\end{equation*}
where our convention is $\tensor{\bfR}{_{i j}^{k}_{\ell}} u^{\ell} = (\nb_{i} \nb_{j} - \nb_{j} \nb_{i})u^{k}$.

We also introduce the \emph{generalized sine function} that is defined for $\kpp \in \bbR$ as
\begin{equation*}
s_{\kpp}(t) = \begin{cases}
t & \kpp = 0, \\
\frac{\sinh (\sqrt{-\kpp} t)}{\sqrt{-\kpp}} & \kpp < 0, \\
\frac{\sin (\sqrt{\kpp} t)}{\sqrt{\kpp}} &  \kpp > 0.
\end{cases}
\end{equation*}
The generalized cosine function is defined as
\begin{equation*}
c_{\kpp}(t) = s_{\kpp}'(t) = \begin{cases}
1 & \kpp = 0, \\
\cosh (\sqrt{-\kpp} t) & \kpp < 0, \\
\cos (\sqrt{\kpp} t) &  \kpp > 0.
\end{cases}
\end{equation*}
Fix a point $y \in \calM$, and consider the polar coordinate system $(r, \omg)$ (where $\omg = (\omg^{1}, \ldots \omg^{d-1}$) is a parametrization of $\bbS^{d-1}$) centered at $y$ (i.e., $r$ is the geodesic distance from $y$). The metric can be written as
\begin{equation*}
    \bfg= \ud r^2 + s_{\kpp}(r)^2 \bfg_{\bbS^{d-1}}.
\end{equation*}
Using ordinary capital latin letters $A, B, \ldots$ to denote the angular variables $\omg^{A}$ (so $A, B, \ldots \in \set{1, \ldots, d-1}$), the Christoffel symbols take the form
\begin{equation*}
	\Gmm^{A}_{r B}
	= \Gmm^{A}_{B r} = \frac{s_{\kpp}'(r)}{s_{\kpp}(r)} \dlt^{A}_{B}, \quad
	\Gmm^{r}_{A B} = - s_{\kpp}'(r) s_{\kpp}(r) (\bfg_{\bbS^{d-1}})_{AB}, \quad
	\Gmm^{A}_{B C} = {}^{(\bfg_{\bbS^{d-1}})} \Gmm^{A}_{B C},
\end{equation*}
where all the other components vanish. The normalized angular forms and vector fields are defined as $d\td{\omg}^{A}=s_{\kappa}(r)d\omg^A$ and $\partial_{\td{\omg}}^A=s_{\kappa}(t)^{-1}\partial_{\omega^A}$.

When performing explicit computations in the constant curvature case, it is convenient to work with \emph{$(r, s)$-tensor-valued distributions} on an open subset $U$ of $(\calM, \bfg)$, which have the advantage of being covariant. They are defined to be continuous linear functionals on the space of smooth and compactly supported $(s, r)$-tensor-valued densities of the form $\tensor{\psi}{_{j_{1} \cdots j_{r}}^{k_{1} \cdots k_{s}}} \ud V$. A smooth $(r, s)$-tensor field $\tensor{g}{^{j_{1} \cdots j_{r}}_{k_{1} \cdots k_{s}}}$ is identified with an $(r, s)$-tensor-valued distribution via the pairing $\tensor{\psi}{_{j_{1} \cdots j_{r}}^{k_{1} \cdots k_{s}}} \ud V \mapsto \brk{g, \psi}$ as in \eqref{eq:duality-pair-geom}.

\subsection{Double divergence operator (or linearized scalar curvature operator)} \label{subsec:d-div}
We consider
\begin{equation} \label{eq:d-div-eq}
	\calP \bfh = \nb_{j} \nb_{k} \bfh^{j k} - \left( \mathrm{Ric}_{jk} - \frac{1}{d-1} R \mathbf{g}_{jk}  \right) \bfh^{jk} \qquad \hbox{where } \bfh^{jk} =  \bfh^{k j},
\end{equation}
for $d \geq 1$. The operator $\calP$ is obtained by making a simple change of variables to the linearization of the scalar curvature operator. In fact, the \emph{linearized scalar curvature operator} is directly given by
\begin{equation*}
\mathrm{D} R(\bfg) \dot{\bfg} = - \nb^{\ell} \nb_{\ell} \tr_{\bfg} \dot{\bfg} + \nb^{j} \nb^{k} \dot{\bfg}_{jk} - \Ric^{jk} \dot{\bfg}_{jk},
\end{equation*}
where $\tr_{\bfg} \dot{\bfg} = \bfg^{jk} \dot{\bfg}_{jk}$. We introduce\footnote{We carefully note that $\dot{\bfg}^{jk}$, according to our conventions in Section~\ref{subsubsec:appl-ex-prelim}, is the metric dual of $\dot{\bfg}_{jk}$, \emph{not} the first order variation of $(\bfg^{-1})^{jk}$. These two objects differ by a sign. \label{fn:g-upper}}
\begin{equation*}
	\bfh^{jk} = \dot{\bfg}^{jk} - \bfg^{jk} \tr_{\bfg} \dot{\bfg}.
\end{equation*}
Then since $\tr_{\bfg} \bfh = (1-d) \tr_{\bfg} \dot{\bfg}$, this change of variables is invertible and we have
\begin{equation*}
	\dot{\bfg}^{jk} = \bfh^{jk} - \frac{1}{d-1} \bfg^{jk} \tr_{\bfg} \bfh.
\end{equation*}
Under this change of variables, we have
\begin{equation*}
\mathrm{D} R(\bfg) \dot{\bfg} = \nb_{j} \nb_{k} \bfh^{jk} - \left(\Ric_{jk} - \frac{1}{d-1} R \bfg_{jk} \right) \bfh^{jk}
\end{equation*}
where the RHS is equal to $\calP (\bfh)$ defined above.

We first compute the formal adjoint operator $\calP^*$ (with respect to $\ud V$). For a smooth compactly supported symmetric $2$-tensor $\bfh$ on $U$ and $\varphi\in  C_c^\infty(U)$, we have
\begin{align*}
	\int_{U} \calP(\bfh) \varphi \, \ud x
	&= \int_{U} \left(\nb_{j} \nb_{k} \bfh^{jk}-\left(\Ric_{jk}-\frac{1}{d-1}R\bfg_{jk}\bfh^{jk}\right)\right) \varphi \, \ud x \\
	&=
	 \int_{U}  \bfh^{jk}  \left(\nb_{j} \nb_{k} -\left(\Ric_{jk}-\frac{1}{d-1}R\bfg_{jk}\right)\right)\varphi\, \ud x.
\end{align*}
So the formal adjoint is
\begin{equation} \label{eq:d-div-dual}
	(\calP^{\ast} \varphi)_{jk} =\nb_{j} \rd_{k} \varphi - \left(\Ric_{jk} - \frac{1}{d-1} R \bfg_{jk} \right)\varphi.
\end{equation}
Its principal symbol is
\begin{equation} \label{eq:d-div-dual-symb}
	(p^{\ast})_{jk}(\xi) = - \xi_{j} \xi_{k},
\end{equation}
which clearly satisfies \ref{hyp:fc} for all $d \geq 1$.

\subsubsection{Covariant graded augmented system and $\ker \calP^{\ast}$} \label{subsubsec:d-div-key}
Given a function $\varphi$, define
\begin{equation} \label{eq:d-div-aug}
	\bfomg_{j} = \rd_{j} \varphi.
\end{equation}
Then
\begin{equation}\label{eq:d-div-system}
\left\{
   \begin{aligned}
	\nb_{i} \varphi &= \bfomg_{i}, \\
	\nb_{i} \bfomg_{j} &= \left(\Ric_{jk} - \frac{1}{d-1} R \bfg_{jk} \right)\varphi + (\calP^{\ast} \varphi)_{ij}.
\end{aligned}
\right.
\end{equation}
As we will see in Section~\ref{subsubsec:d-div-explicit}, this system of covariant first-order PDEs \eqref{eq:d-div-system} is useful for performing explicit computations on space forms.

We note that \eqref{eq:d-div-system} immediately leads to graded augmented variables in the sense of Definition~\ref{def:aug}. Indeed, if we specialize \eqref{eq:d-div-system} to the Euclidean space in rectangular coordinates (so that $\calP = \calP_{\prin}$ and $\calP^{\ast} = \calP_{\prin}^{\ast_{\ud x}}$), then we see that  $\Phi_{\varphi} \ceq \varphi$, $\Phi_{\bfomg_{j}} \ceq \bfomg_{j}$ define augmented variables for $\calP_{\prin}$ that satisfy \ref{hyp:aug1}--\ref{hyp:aug4}, where $\calA = \set{\varphi, \bfomg_{1}, \ldots, \bfomg_{d}}$ (in particular, $\# \calA = d+1$), $d_{\varphi} = 0$, $d_{\bfomg_{j}} = -1$, $m_{ij} = 2$, and $m_{ij}' = 0$ for all $i, j \in \set{1, \ldots, d}$. These augmented variables also satisfy \ref{hyp:aug1}--\ref{hyp:aug4} for any lower order perturbations of $\calP_{\prin}$ (viewed as an operator on an open subset of $\bbR^{d}$, the Euclidean space in rectangular coordinates).

In view of \eqref{eq:d-div-system}, we see that $\dim \ker \calP^*\leq \# \calA = d+1$. This bound is optimal, and the maximal dimension is reached (i.e., the augmented system is completely integrable) on space forms:
\begin{itemize}
    \item For $\bbR^d$, $\calP^*\varphi=\partial_j\partial_k\varphi$ and $\ker \calP^*=\mathrm{span}\,\set{1, x^{1}, \ldots, x^{d}}$.
    \item For the sphere $\bbS^d$, $\calP^*\varphi=\nb_{j} \rd_{k} \varphi+\bfg_{jk}\varphi$. If we embed the sphere into $\bbR^{d+1}$ by $\{(x^0)^2+\cdots+(x^d)^2=1\}\subseteq \bbR^d$, then $\ker \calP^*$ is spanned by the restrictions on $\bbS^{d}$ of the ambient coordinates $x^0,x^1,\cdots, x^d$, which are linearly independent. Indeed, for each $\mu \in \set{0, 1, \ldots, d}$, that $\calP^{\ast} x^{\mu} = 0$ can be seen by a direct computation using the fact that the second fundamental form $\Pi(X, Y)$ of the embedding $\bbS^{d} \hookrightarrow \bbR^{d+1}$, which proceeds as follows:
    \begin{equation*}
     (\calP^*x^{\mu})(X,Y)=X(Y x^{\mu})-(\overline{\nb}_{X}Y)^{\parallel} x^{\mu}+\bfg(X,Y) x^{\mu}=(\Pi(X,Y)+\bfg(X,Y))x^{\mu}=0,
    \end{equation*}
    where $\overline{\nb}$ is the covariant derivative in the ambient coordinates, $X, Y \in T \bbS^{d}$, and $(\overline{\nb}_{X}Y)^{\parallel}$ is the tangential part of $\overline{\nb}_{X}Y$. The expressions $X (Y x^{\mu})$ and $(\overline{\nb}_{X}Y)^{\parallel}$ are computed using any extension of $Y$ as a smooth vector field; as is well-known, the identity holds independently of this choice. The linear independence claim is clear. Since $\dim \calP^{\ast} \leq d + 1$, $\set{x^{0}, x^{1}, \ldots, x^{d}}$ indeed also spans $\calP^{\ast}$.

    \item For the hyperbolic space $\bbH^d$, $\calP^*\varphi=\nb_{j} \rd_{k} \varphi-\bfg_{jk}\varphi$. We embed the hyperbolic space into the Minkowski space $\bbR^{1,d}$ with metric $-(\ud x^0)^2+(\ud x^1)^2+\cdots+(\ud x^d)^2$ by $\{-(x^0)^2+(x^1)^2+\cdots+(x^d)^2=-1\}$. Then $\ker \calP^*$ is again spanned by the restrictions to $\bbH^{d}$ of the ambient coordinates $x^0,x^1,\cdots,x^d$. This statement can be verified in a similar manner to the case of $\bbS^{d}$. In particular, for $\mu \in \set{0, 1, \ldots, d}$ and any $X,Y\in T\bbH^d$, we have
    \begin{equation*}
        (\calP^*x^{\mu})(X,Y)=X(Y x^{\mu})-(\overline{\nb}_{X}Y)^{\parallel} x^\mu-\bfg(X,Y) x^{\mu}=(\Pi(X,Y)-\bfg(X,Y))x^\mu=0,
    \end{equation*}
 where we used that the second fundamental form $\Pi(X, Y)$ of the embedding $\bbH^{d} \hookrightarrow \bbR^{1, d}$ equals $\bfg$.
\end{itemize}

\subsubsection{Explicit computations in the constant curvature case} \label{subsubsec:d-div-explicit}
Let $\kpp$ be the constant sectional curvature of $(\calM, \bfg)$. Fix $y, y_{1} \in \calM$ and a geodesic segment $\bfx$ from $y$ to $y_{1}$. We may write $\bfx(t) = (t, \omg_{0})$ in polar coordinates at $y$ for some $\omg_{0} \in \bbS^{d-1}$. From \eqref{eq:d-div-system} and the formulas in Section~\ref{subsubsec:const-curv-prelim}, we have
\begin{equation*}
    \rd_{r} \varphi=\bfomg_{r},\quad \rd_{r} \bfomg_{r}=-\kpp \varphi+(\calP^*\varphi)_{rr}.
\end{equation*}
So
\begin{equation*}
    \frac{\ud^2}{\ud t^2}\varphi(\bfx(t))=\left(-\kpp\varphi+(\calP^*\varphi)_{rr}\right) .
\end{equation*}
The solution is given by
\begin{equation*}
    \varphi(\bfx(0))=\int_{0}^{d(y_{1}, y)} s_{\kpp}(s)\dot{\bfx}^{i} \dot{\bfx}^{j} (\calP^*\varphi)_{i j} \circ \bfx(s) \, \ud s + \varphi \circ \bfx(d(y_{1}, y)) c_{\kappa}(d(y_1,y)) - \dot{\bfx}^{j} (\rd_{j} \varphi) \circ \bfx (d(y_{1}, y)) s_{\kappa}(d(y_1,y)).
\end{equation*}
Thus,
\begin{equation*}
\begin{split}
    \brk{K_{y_{1}}(\cdot, y), \psi} &= \int_{0}^{d(y_{1}, y)} s_{\kpp}(s)\dot{\bfx}^{i} \dot{\bfx}^{j} \psi_{i j} \circ \bfx(s) \, \ud s, \\
	\brk{b_{y_{1}}(\cdot, y), \varphi} &= \varphi (y_{1}) c_{\kappa}(d(y_1,y)) - \dot{\bfx}^{j}(y, y_{1}, d(y_{1}, y)) (\rd_{j} \varphi) (y_{1}) s_{\kappa}(d(y_1,y)),
\end{split}
\end{equation*}
which are tensor-valued distributions in the sense of Section~\ref{subsubsec:appl-ex-prelim}.

\subsubsection{Explicit formulas on flat spaces}
For the readers' convenience, we record the explicit formulas for the Bogovskii-type and conic solution operators on $\bbR^d$. We average over straight line segments $\bfx(y,y_1,s)=y+s(y_1-y)$ for the Bogovski-type solution operator and over straight rays $\bfx(y,\omg,s)=y+s\omg$ for the conic solution operator.
\begin{enumerate}
    \item Let $\eta\in C_c^{\infty}(\bbR^d)$ with $\int_{\bbR^d}\eta =1$. The Bogovskii-type solution operator for $\calP \bfh =\nb_j\nb_k \bfh^{jk}$ where $\bfh^{ij}=\bfh^{ji}$ on $\bbR^d$ with flat metric is given by
    \begin{equation*}
        K_{\eta}^{ij}(z+y,y)=\left(\int_{|z|}^{\infty}\eta\left(r\frac{z}{|z|}+y\right)r^{d-1} \ud r\right)\frac{z^iz^j}{|z|^d}
    \end{equation*}
    with
    \begin{equation*}
        b_{\eta}(x,y)=(d+1)\eta(x)+(x-y)^i\partial_i\eta(x).
    \end{equation*}
    \item Let $\slashed{\eta} \in C^{\infty}(\bbS^{d-1})$ with $\int_{\bbS^{d-1}}\slashed{\eta} =1$. The conic solution operator for $\calP \bfh =\nb_j\nb_k \bfh^{jk}$ where $\bfh^{ij}=\bfh^{ji}$ on $\bbR^d$  with flat metric is given by
    \begin{equation*}
        K_{\slashed{\eta}}^{ij}(z+y,y)=\frac{z^iz^j}{|z|^d}\slashed{\eta}\left(\frac{z}{|z|}\right).
    \end{equation*}
\end{enumerate}

\subsection{Trace-free double divergence operator} \label{subsec:d-div-tf}
We consider
\begin{equation} \label{eq:d-div-tf-eq}
	\calP \bfh = \nb_{j} \nb_{k} \bfh^{j k} \qquad \hbox{where } \bfh^{jk} =  \bfh^{k j} \hbox{ and } \tr_{\bfg} \bfh = 0,
\end{equation}
for $d \geq 2$.  We will soon show that the adjoint operator $\calP^*$ is the traceless part of the Hessian of a scalar function.  Geometrically, a function in the kernel of $\calP^*$ has level sets that locally give rise to a warped product decomposition of the manifold on which they are defined\footnote{See for instance \url{https://www.math.ucla.edu/~petersen/233.1.10s/BLWformulas.pdf}.}.  Also, both $\calP$ and $\calP^*$ arise naturally in the study of the 2D Euler equations in vorticity form and the surface quasi-geostrophic equation.  For example, on $\bbR^2$ the kernel of the adjoint $\calP^*$ is spanned by $1, x, y$ and $ x^2 + y^2$.  Each of these functions can be integrated against a solution to 2D Euler or SQG to define the mean, impulse and angular momentum of a solution, all of which are conserved by the corresponding evolution.  The operator $\calP$ itself then arises naturally in the construction of weak solutions to these equations \cite{isett2021direct,isett2024proof}.  

We now compute the formal adjoint operator $\calP^*$ (with respect to $\ud V$). For a smooth compactly supported trace-free symmetric $2$-tensor $\bfh$ on $U$ and $\varphi\in  C_c^\infty(U)$, we have
\begin{align*}
	\int_{U} \calP(\bfh) \varphi \, \ud x
	&= \int_{U} \nb_{j} \nb_{k} \bfh^{jk} \varphi \, \ud x \\
	&=
	 \int_{U}  \bfh^{jk}  \left(\nb_{j} \nb_{k} - \frac{1}{d} \bfg_{jk} \nb^{\ell} \nb_{\ell} \right)\varphi\, \ud x.
\end{align*}
So the formal adjoint is
\begin{equation} \label{eq:d-div-tf-dual}
	(\calP^{\ast} \varphi)_{jk} = \nb_{j} \rd_{k} \varphi - \frac{1}{d} \bfg_{jk} \nb^{\ell} \rd_{\ell} \varphi.
\end{equation}
Its principal symbol is
\begin{equation} \label{eq:d-div-tf-dual-symb}
	(p^{\ast})_{jk}(x, \xi) = - \xi_{j} \xi_{k} + \frac{1}{d} \xi_{\ell} \xi_{m} \bfg^{\ell m}(x) \bfg_{jk}(x).
\end{equation}

\subsubsection{Covariant graded augmented system and $\ker \calP^{\ast}$} \label{subsubsec:d-div-tf-key}
Given a function $\varphi$, define
\begin{equation} \label{eq:d-div-tf-aug}
	\bfomg_{j} = \rd_{j} \varphi, \quad w = \frac{1}{d} \nb^{\ell} \bfomg_{\ell}.
\end{equation}
Then
\begin{equation}\label{eq:d-div-tf-system}
\left\{
   \begin{aligned}
	\nb_{i} \varphi &= \bfomg_{i}, \\
	\nb_{i} \bfomg_{j} &= (\calP^{\ast} \varphi)_{ij} + w \bfg_{ij}, \\
	\nb_{i} w &= \frac{1}{d-1} \nb^{\ell} (\calP^{\ast} \varphi)_{i \ell} - \frac{1}{d-1} \tensor{\Ric}{_{i}^{j}} \bfomg_{j}.
\end{aligned}
\right.
\end{equation}
All of these identities are obvious except for the last one, which follows from:
\begin{equation*}
	\nb_{i} w
	= \frac{1}{d} \nb^{\ell} \nb_{i} \bfomg_{\ell} - \frac{1}{d} \tensor{\bfR}{_{i}^{\ell}^{j}_{\ell}} \bfomg_{j}
	= \frac{1}{d} \nb^{\ell} (\calP^{\ast} \varphi)_{i \ell} + \frac{1}{d} \nb_{i} w - \frac{1}{d} \tensor{\Ric}{_{i}^{j}} \bfomg_{j}.
\end{equation*}

We note that \eqref{eq:d-div-tf-system} immediately leads to graded augmented variables in the sense of Definition~\ref{def:aug}. Indeed, if we specialize \eqref{eq:d-div-tf-system} to the Euclidean space in rectangular coordinates (so that $\calP = \calP_{\prin}$ and $\calP^{\ast} = \calP_{\prin}^{\ast_{\ud x}}$), then we see that  $\Phi_{\varphi} \ceq \varphi$, $\Phi_{\bfomg_{j}} \ceq \bfomg_{j}$, $\Phi_{w} \ceq w$ define augmented variables for $\calP_{\prin}$ that satisfy \ref{hyp:aug1}--\ref{hyp:aug4}, where $\calA = \set{\varphi, \bfomg_{1}, \ldots, \bfomg_{d}, w}$ (in particular, $\# \calA = d+2$), $d_{\varphi} = 0$, $d_{\bfomg_{j}} = -1$, $d_{w} = -2$, $m_{ij} = 2$, and $m_{ij}' = 1$ for all $i, j \in \set{1, \ldots, d}$. These augmented variables also satisfy \ref{hyp:aug1}--\ref{hyp:aug4} for any lower order perturbations of $\calP_{\prin}$ (viewed as an operator on an open subset of $\bbR^{d}$, the Euclidean space in rectangular coordinates).

In view of \eqref{eq:d-div-tf-system}, we see that $\dim \ker \calP^*\leq \# \calA = d+2$. The maximal dimension is reached (i.e., the augmented system is completely integrable) on space forms:
\begin{itemize}
    \item For $\bbR^{d}$, $\calP^{\ast}\varphi=\partial_j\partial_k\varphi -\frac{1}{d}\delta_{jk}\Delta\varphi$ and $\ker \calP^{\ast} = \mathrm{span} \, \set{1, x^{1}, \ldots, x^{d}, \abs{x}^{2}}$.
    \item For the sphere $\bbS^{d}$, $\calP^*\varphi=\nb_j\partial_{k}\varphi-\frac{1}{d}\bfg_{jk}\Delta_{\bfg} \varphi$. If we embed the sphere into $\bbR^{d+1}$ by $\{(x^0)^2+\cdots+(x^d)^2=1\}\subseteq \bbR^d$ as in \S\ref{subsubsec:d-div-key}, then $\ker \calP^*$ is spanned by the ambient coordinates $x^0,x^1,\cdots, x^d$ and constant $1$, which are linearly independent. This is because our computation in \S\ref{subsubsec:d-div-key} shows that $\nb_{j}\partial_k x^{\mu}=-\bfg_{jk}x^{\mu}$.
    \item For the hyperbolic space $\bbH^{d}$, $\calP^{\ast}\varphi=\nb_{j}\partial_k \varphi -\frac{1}{d}\bfg_{jk}\Delta_{\bfg}\varphi$. We embed the hyperbolic space into the Minkowski space $\bbR^{1,d}$ with metric $-(\ud x^0)^2+(\ud x^1)^2+\cdots+(\ud x^d)^2$ by $\{-(x^0)^2+(x^1)^2+\cdots+(x^d)^2=-1\}$ as in \S\ref{subsubsec:d-div-key}. Then $\ker \calP^{\ast}$ is again spanned by the ambient coordinates $x^0,x^1,\cdots, x^d$ and constant $1$, which are linearly independent. This is because our computation in \S\ref{subsubsec:d-div-key} shows that $\nb_{j}\partial_k x^{\mu}=\bfg_{jk}x^{\mu}$.
\end{itemize}
We remark that one can also give a derivation of the kernel of $\calP^*$ that is independent of \S\ref{subsubsec:d-div-key} as follows.  Motivated by the fact that functions in the kernel of $\calP^*$ induce warped product decompositions, we first look for a solution that is radial in geodesic normal coordinates.  Integrating the system of ODE's in \S\ref{subsubsec:d-div-tf-explicit} below gives a two-dimensional space of solutions spanned by constants and $\phi_\kpp(r) \ceq \int s_\kpp(r) \ud r$, where $\kpp$ is the constant sectional curvature of the manifold.  To find $d$ other linearly independent solutions, the next idea is to consider infinitesimal translations of the base point for the polar coordinates.  Namely, we consider the functions $\calL_K \phi_\kpp$ where $\calL_K$ denotes the Lie derivative with respect to a Killing vector field $K$.  Because Lie derivatives with respect to Killing vector fields commute with contraction with the metric and with covariant derivatives (see, e.g., \cite[Lemma 7.1.3]{CK1993}), we find that $\calL_K \phi_\kpp$ also has vanishing traceless Hessian.  Choosing a set of Killing vector fields whose flow maps translate the base point in $d$ linearly independent directions gives the desired $d$ remaining linearly independent solutions.

\begin{itemize}
    \item On $\bbR^d$, $\phi_0(r) = r^2/2$ and this operation is the observation that $x^i =  \partial_i \frac{|x|^2}{2}$.
    \item On $\bbS^d$, $\phi_{1}(r) = x^0$ and this operation is the observation that $x^i = (x^i \partial_0 - x^0 \partial_i)x^0$.
    \item On $\bbH^d$, $\phi_{-1}(r) = x^0$ and this operation is the observation that $ x^i = (x^i \partial_0 + x^0 \partial_i)x^0$.
\end{itemize}

\subsubsection{Explicit computations in the constant curvature case} \label{subsubsec:d-div-tf-explicit}
Let $\kpp$ be the constant sectional curvature of $(\calM, \bfg)$. Fix $y, y_{1} \in \calM$ and a geodesic segment $\bfx$ from $y$ to $y_{1}$. We may write $\bfx(t) = (t, \omg_{0})$ in polar coordinates at $y$ for some $\omg_{0} \in \bbS^{d-1}$. From \eqref{eq:d-div-tf-system} and the formulas in Section~\ref{subsubsec:const-curv-prelim}, we have
\begin{equation*}
    \frac{\ud}{\ud t} \varphi(\bfx(t)) = \bfomg_{r}(\bfx(t)), \quad
      \frac{\ud}{\ud t} \bfomg_{r}(\bfx(t)) = (\calP^{\ast} \varphi)_{rr} + w, \quad
           \frac{\ud}{\ud t} w(\bfx(t)) = \frac{1}{d-1}\nb^{\ell} (\calP^{\ast} \varphi)_{r \ell} - \kpp \bfomg_{r}, \quad
\end{equation*}
The solution is given by
\begin{equation*}
\begin{split}
    \varphi(\bfx(0)) &=\varphi\circ\bfx(d(y_1,y))-s_{\kappa}(d(y_1,y))(\bfomg_r\circ\bfx)(d(y_1,y))-\kappa^{-1}(c_{\kappa}(d(y_1,y))-1)(w\circ\bfx)(d(y_1,y)) \\
    &+ \int_0^{d(y_1,y)}s_{\kappa}(s)\dot{\bfx}^i\dot{\bfx}^j(\calP^*\varphi)_{ij}\circ\bfx(s) \ud s+\frac{1}{(d-1)\kappa}\int_0^{d(y_1,y)}(c_{\kpp}(s)-1) \dot{\bfx}^i\nabla^j(\calP^*\varphi)_{ij}\circ\bfx(s)  \ud s
\end{split}
\end{equation*}
Thus
\begin{equation*}
    \begin{split}
        \langle K_{y_1}(\cdot,y),\psi\rangle = & \int_0^{d(y_1,y)}s_{\kappa}(s)\dot{\bfx}^i\dot{\bfx}^j\psi_{ij}\circ\bfx(s) \ud s+\frac{1}{(d-1)\kappa}\int_0^{d(y_1,y)}(c_{\kpp}(s)-1) \dot{\bfx}^i\nabla^j \psi_{ij}\circ\bfx(s)  \ud s,\\
        \langle b_{y_1}(\cdot, y),\varphi\rangle = &\varphi(y_1) -s_{\kappa}(d(y_1,y))\dot{\bfx}^j(y,y_1,d(y_1,y))(\rd_j\varphi)(y_1)-(d\kappa)^{-1}(c_{\kappa}(d(y_1,y))-1)(\Delta\varphi)(y_1).
    \end{split}
\end{equation*}

\subsubsection{Explicit formulas on flat spaces}
For the readers' convenience, we record the explicit formulas for the Bogovskii-type and conic solution operators on $\bbR^d$. We average over straight line segments $\bfx(y,y_1,s)=y+s(y_1-y)$ for the Bogovski-type solution operator and over straight half-lines $\bfx(y,\omg,s)=y+s\omg$ for the conic solution operator.

To state our results, we introduce
\begin{equation*}
       (\calT^*f)_{ij}\ceq\frac{1}{2}(f_{ij}+f_{ji})-\frac{1}{d}(\tr f)\delta_{ij}.
\end{equation*}
\begin{enumerate}
    \item Let $\eta\in C_c^{\infty}(\bbR^d)$ with $\int_{\bbR^d}\eta =1$. The Bogovskii-type solution operator for $\calP \bfh =\nb_j\nb_k \bfh^{jk}$ where $\bfh^{ij}=\bfh^{ji}$ and $\tr_{\bfg}\bfh=0$ on $\bbR^d$ with flat metric is given by
    \begin{equation*}
    \begin{split}
        K_{\eta}^{ij}(z+y,y)=&\mathcal{T}^*\left(\left(\int_{|z|}^{\infty}\eta\left(r\frac{z}{|z|}+y\right)r^{d-1} \ud r\right)\frac{z^iz^j}{|z|^d}\right)\\
        &+\frac{1}{2(d-1)} \mathcal{T}^*\left(\partial_{z^j}\left(\left(\int_{|z|}^{\infty}\eta\left(r\frac{z}{|z|}+y\right)r^{d-1} \ud r\right)\frac{z^i}{|z|^{d-2}}\right)\right)
    \end{split}
    \end{equation*}
    with
    \begin{equation*}
        b_{\eta}(x,y)=\eta(x)+\partial_j((x-y)^j\eta(x))+\frac{1}{2d}\Delta(|x-y|^2\eta(x)).
    \end{equation*}
    \item Let $\slashed{\eta}\in C^{\infty}(\bbS^{d-1})$ with $\int_{\bbS^{d-1}}\slashed{\eta} =1$. The conic solution operator for $\calP \bfh =\nb_j\nb_k \bfh^{jk}$ where $\bfh^{ij}=\bfh^{ji}$ and $\tr_{\bfg}\bfh = 0$ on $\bbR^d$  with flat metric is given by
    \begin{equation*}
        K_{\slashed{\eta}}^{ij}(z+y,y)=\mathcal{T}^*\left(\frac{z^iz^j}{|z|^d}\slashed{\eta}\left(\frac{z}{|z|}\right)\right)+\frac{1}{2(d-1)}\mathcal{T}^*\left(\partial_{z^j}\left(\frac{z^i}{|z|^{d-2}}\slashed{\eta}\left(\frac{z}{|z|}\right)\right)\right).
    \end{equation*}
\end{enumerate}

\subsection{Symmetric divergence operator (or adjoint Killing operator)}\label{subsec:symm-div}

We consider the \emph{symmetric divergence operator}
\begin{equation} \label{eq:symm-div-eq}
	\calP \bfh = \nb_{j} \bfh^{j k} \qquad \hbox{where } \bfh^{jk} =  \bfh^{k j},
\end{equation}
for $d \geq 1$.

We first compute the formal adjoint operator $\calP^{\ast}$ (with respect to $\ud V$). For a smooth compactly supported symmetric $2$-tensor $\bfh$ and $1$-form $\bfomg$ on $U$, we have
\begin{align*}
	\int_{U} (\calP \bfh)^{k} \bfomg_{k} \, \ud V
	= \int_{U} (\nb_{j} \bfh^{jk}) \bfomg_{k} \, \ud V
	= - \int_{U} \bfh^{jk} \nb_{j} \bfomg_{k} \, \ud V
	= - \int_{U} \bfh^{jk} \frac{1}{2} \left(\nb_{j} \bfomg_{k} + \nb_{k} \bfomg_{j} \right) \, \ud V.
	\end{align*}
The formal $L^{2}$-adjoint is
\begin{equation} \label{eq:symm-div-dual}
(\calP^{\ast} \bfomg)_{jk} = -\frac{1}{2} \left(\nb_{j} \bfomg_{k} + \nb_{k} \bfomg_{j} \right).
\end{equation}
Observe that $\calP^{\ast} \bfomg = 0$ is precisely the condition that the vector field $\bfomg_{\sharp}$ is a Killing vector field of $(\calM, \bfg)$; for this reason, we will call $\calP^{\ast}$ the \emph{Killing operator}. Its principal symbol is
\begin{equation} \label{eq:symm-div-dual-symb}
\tensor{(p^{\ast})}{_{jk}^{\ell}}(\xi) = - \frac{i}{2} (\xi_{j} \dlt_{k}^{\ell} + \xi_{k} \dlt_{j}^{\ell}).
\end{equation}

\subsubsection{Covariant graded augmented system and $\ker \calP^{\ast}$} \label{subsubsec:symm-div-key}
Given a $1$-form $\bfomg_{j}$, define
\begin{equation} \label{eq:symm-div-aug}
\bfeta_{jk} = \tfrac{1}{2}(\ud \bfomg)_{jk} = \tfrac{1}{2} (\nb_{j} \bfomg_{k} - \nb_{k} \bfomg_{j}).
\end{equation}
Then
\begin{equation} \label{eq:symm-div-system}
\left\{
\begin{aligned}
	\nb_{i} \bfomg_{j} &= \bfeta_{i j} - (\calP^{\ast} \bfomg)_{i j}, \\
	\nb_{i} \bfeta_{j k}
	&=
	-\tensor{\bfR}{_{ j k i}^{\ell}} \bfomg_{\ell}
	+ \nb_{k}(\calP^{\ast} \bfomg)_{i j} - \nb_{j}(\calP^{\ast} \bfomg)_{ki}.
\end{aligned}
\right.
\end{equation}
We postpone the proof of \eqref{eq:symm-div-system} and discuss its consequences first.

We note that \eqref{eq:symm-div-system} immediately leads to graded augmented variables in the sense of Definition~\ref{def:aug}. Indeed, if we specialize \eqref{eq:symm-div-system} to the Euclidean space in rectangular coordinates (so that $\calP = \calP_{\prin}$ and $\calP^{\ast} = \calP_{\prin}^{\ast_{\ud x}}$), then we see that  $\Phi_{\bfomg_{j}} \ceq \bfomg_{j}$, $\Phi_{\bfeta_{jk}} \ceq \bfeta_{jk}$ define augmented variables for $\calP_{\prin}$ that satisfy \ref{hyp:aug1}--\ref{hyp:aug4}, where $\calA = \set{\bfomg_{1}, \ldots, \bfomg_{d}, \bfeta_{12}, \ldots, \bfeta_{(d-1) d}} $ (in particular, $\# \calA = d + \frac{d(d-1)}{2} = \frac{d(d+1)}{2}$), $d_{\bfomg_{j k}} = 0$, $d_{\bfeta_{j}} = -1$, $m_{ij} = 1$, and $m_{ij}' = 1$ for all $i, j \in \set{1, \ldots, d}$. These augmented variables also satisfy \ref{hyp:aug1}--\ref{hyp:aug4} for any lower order perturbations of $\calP_{\prin}$ (viewed as an operator on an open subset of $\bbR^{d}$, the Euclidean space in rectangular coordinates).

As is well-known, $\calP^{\ast}$ has a finite dimensional kernel with $\dim \ker \calP^{\ast} \leq \# \calA = \frac{d(d+1)}{2}$. The maximal dimension is reached (i.e., the augmented system is completely integrable) on space forms:
\begin{itemize}
\item For $\bbR^{d}$, $\ker \calP^{\ast}$ consists of the metric duals of the Killing vector fields, which are
\begin{equation*}
\mathrm{span}\,\left( \set{\bfe_{J}}_{J=1, \ldots, d} \cup \set{x_{K} \bfe_{J} - x_{J} \bfe_{K}}_{1\leq J < K \leq d}\right).
\end{equation*}
Geometrically, these correspond to translation and rotation vector fields on $\bbR^{d}$.

\item For the sphere $\bbS^{d}$, $\ker \calP^{\ast}$ consists of the metric duals of
\begin{equation*}
\ker \calP^{\ast} =
    \mathrm{span}\,\left(  \set{x_{K} \bfe_{J} - x_{J} \bfe_{K}}_{0\leq J < K \leq d}\right).
\end{equation*}
Geometrically, these correspond to the rotation vector fields on the ambient Euclidean space $\bbR^{d+1}$.

\item For the hyperbolic space $\bbH^{d}$, $\ker \calP^{\ast}$ consists of the metric duals of
\begin{equation*}
 \ker \calP^{\ast} = \mathrm{span}\,\left( \set{x^{0} \bfe_{J} + x^{J} \bfe_{0}}_{1\leq  J \leq d} \cup \set{x_{K} \bfe_{J} - x_{J} \bfe_{K}}_{1\leq J < K \leq d}\right).
\end{equation*}
Geometrically, these correspond to the Lorentz boost and rotation vector fields on the ambient Minkowski space $\bbR^{1, d}$.
\end{itemize}

Finally, we give a proof of \eqref{eq:symm-div-system}, which is a covariant generalization of \cite{Resh}. Indeed, the first identity is obvious. To prove the second identity, we begin by computing $\nb_{i}(\calP^{\ast} \bfomg)_{jk}$ and cycling the indices $i, j, k$:
\begin{align*}
- 2 \nb_{i}(\calP^{\ast} \bfomg)_{jk} &=\nb_{i} \nb_{j} \bfomg_{k} + \nb_{i} \nb_{k} \bfomg_{j},  \\
- 2 \nb_{j}(\calP^{\ast} \bfomg)_{ki} &=\nb_{j} \nb_{k} \bfomg_{i} + \nb_{j} \nb_{i} \bfomg_{k}  =\nb_{k} \nb_{j} \bfomg_{i} + \nb_{i} \nb_{j} \bfomg_{k}   - \tensor{\bfR}{_{j k}^{\ell}_{i}} \bfomg_{\ell}  + \tensor{\bfR}{_{i j}^{\ell}_{k}} \bfomg_{\ell}, \\
- 2 \nb_{k}(\calP^{\ast} \bfomg)_{i j} &=\nb_{k} \nb_{i} \bfomg_{j} + \nb_{k} \nb_{j} \bfomg_{i}  = \nb_{i} \nb_{k} \bfomg_{j} + \nb_{k} \nb_{j} \bfomg_{i}- \tensor{\bfR}{_{k i}^{\ell}_{j}} \bfomg_{\ell}.
\end{align*}
Then we subtract the second equation from the third equation. We obtain
\begin{align*}
	\nb_{i} \bfeta_{j k}
	&=
	 \tfrac{1}{2} (\tensor{\bfR}{_{j k}^{\ell}_{i}} -\tensor{\bfR}{_{i j }^{\ell}_{k}}  -\tensor{\bfR}{_{k i}^{\ell}_{j}} )\bfomg_{\ell}
	+ \nb_{k}(\calP^{\ast} \bfomg)_{i j} - \nb_{j}(\calP^{\ast} \bfomg)_{ki} \\
	&= \tensor{\bfR}{_{jk}^{\ell}_{i}} \bfomg_{\ell}
	+ \nb_{k}(\calP^{\ast} \bfomg)_{i j} - \nb_{j}(\calP^{\ast} \bfomg)_{ki},
\end{align*}
where we used the first Bianchi identity for the last equality.

\subsubsection{Explicit computations in the constant curvature case} \label{subsubsec:symm-div-explicit}
Let $\kpp$ be the constant sectional curvature of $(\calM, \bfg)$. Fix $y, y_{1} \in \calM$ and a geodesic segment $\bfx$ from $y$ to $y_{1}$. We may write $\bfx(t) = (t, \omg_{0})$ in polar coordinates at $y$ for some $\omg_{0} \in \bbS^{d-1}$. From \eqref{eq:symm-div-system} and the formulas in Section~\ref{subsubsec:const-curv-prelim}, we have
\begin{align*}
   \frac{\ud}{\ud t} \bfomg_{r}(\bfx(t))&=-(\calP^*\bfomg)_{r r}, \\
   \frac{\ud}{\ud t} \bfomg_{A}(\bfx(t))&=\Gamma_{r A}^{B} \bfomg_{B}+\bfeta_{r A}-(\calP^*\bfomg)_{r A}, \\
   \frac{\ud}{\ud t} \bfeta_{r A}(\bfx(t))&=\Gamma_{r A}^{B} \bfeta_{r B} - \kpp \bfomg_{A}+ \nb_{A}(\calP^{\ast} \bfomg)_{rr} - \nb_{r}(\calP^{\ast} \bfomg)_{rA}.
\end{align*}

For the first equation we can solve directly:
\begin{equation*}
    \bfomg_r(\bfx(t))=\bfomg_r(y_1)+\int_t^{d(y_1,y)}(\calP^*\bfomg)_{rr}(\bfx(s)) \, \ud s.
\end{equation*}
The equation for $\bfomg_{A}$ can be reduced:
\begin{equation*}
    \frac{\ud^2}{\ud t^2}(s_{\kappa}(t)^{-1}\bfomg_{A}(\bfx(t))=-\kappa s_{\kappa}(t)^{-1}\bfomg_A +s_{\kappa}(t)^{-1}\nb_A(\calP^*\bfomg)_{rr}-2s_{\kappa}(t)^{-1}\nb_r(\calP^*\bfomg)_{rA}.
\end{equation*}
Therefore,
\begin{equation*}
    \begin{aligned}
    s_{\kappa}(t)^{-1}\bfomg_A(\bfx(t))&=\int_t^{d(y_1,y)} s_{\kappa}(t-s) s_{\kappa}(s)^{-1}(\nb_r(\calP^*\bfomg)_{rA}-\nb_A(\calP^*\bfomg)_{rr})(\bfx(s)) \ud s\\
    &+\int_t^{d(y_1,y)}c_{\kappa}(t-s) s_{\kappa}(s)^{-1}(\calP^*\bfomg)_{rA}\ud s\\&+\bfomg_A(y_1)c_{\kappa}(t-d(y_1,y))/s_\kappa(d(y_1,y))+\bfeta_{rA}(y_1)s_\kappa(t-d(y_1,y))/s_\kappa(d(y_1,y)).
    \end{aligned}
\end{equation*}
Thus
\begin{equation*}
\begin{split}
    \langle K_{y_1}(\cdot,y),\psi \rangle &= \left(\int_0^{d(y_1,y)} \dot{\bfx}^{i}\dot{\bfx}^j\psi_{ij}(\bfx(s)) \, \ud s\right) \ud r+\left(\int_0^{d(y_1,y)}c_{\kappa}(s) \dot{\bfx}^j\psi_{j\td{A}}(\bfx(s)) \ud s\right) \, \ud \td{\omg}^{A}\\
    &- \left(\int_0^{d(y_1,y)}s_{\kappa}(s) \dot{\bfx}^{i}\dot{\bfx}^{j}(\nb_i\psi_{j\td{A}}-\nb_{\td{A}}\psi_{ij})(\bfx(s))\, \ud s\right) \, \ud \td{\omg}^{A}
\end{split}
\end{equation*}
(we recall that $\ud \td{\omega}^A = s_\kappa(r) \, \ud \omega^A$ is the normalized angular $1$-form and $\psi_{j\td{A}}=s_{\kappa}(r)^{-1}\psi_{jA}$ since $\partial_{\td{\omega}^A}=s_{\kappa}^{-1}\partial_{\omega^A}$)
and
\begin{equation*}
    \langle b_{y_1}(\cdot,y),\varphi\rangle = \varphi_r(y_1) \, \ud r +\left(\varphi_{\td{A}}(y_1)c_{\kappa}(d(y_1,y))-\frac{1}{2}\dot{\bfx}^j(d(y_1,y))(\nb_j\varphi_{\td{A}}-\nb_{\td{A}}\varphi_j)(y_1)s_{\kappa}(d(y_1,y))\right) \, \ud \td{\omg}^A.
\end{equation*}

\subsubsection{Explicit formulas on flat spaces}
For the readers' convenience, we record the explicit formulas for the Bogovskii-type and conic solution operators on $\bbR^d$. We average over straight line segments $\bfx(y,y_1,s)=y+s(y_1-y)$ for the Bogovski-type solution operator and over straight half-lines $\bfx(y,\omg,s)=y+s\omg$ for the conic solution operator.
\begin{enumerate}
    \item Let $\eta\in C_c^{\infty}(\bbR^d)$ with $\int_{\bbR^d}\eta =1$. The Bogovskii-type solution operator for $\calP \bfh =\nb_j \bfh^{jk}$ where $\bfh^{ij}=\bfh^{ji}$ on $\bbR^d$ with flat metric is given by
    \begin{equation*}
    \begin{split}
        (K_{\eta})_k^{ij}(z+y,y)=&\frac{1}{2}\left(\int_{|z|}^{\infty}\eta\left(r\frac{z}{|z|}+y\right)r^{d-1} \, \ud r\right)\frac{z^i\bfdlt_k^j+z^j\bfdlt_k^i}{|z|^d}\\
        &+\frac{1}{2}\partial_{z^m}\left(\left(\int_{|z|}^{\infty}\eta\left(r\frac{z}{|z|}+y\right)r^{d-1} \, \ud r\right)\frac{z^m(z^i\bfdlt_k^j+z^j\bfdlt_k^i)}{|z|^d}\right)\\
        &-\partial_{z^k}\left(\left(\int_{|z|}^{\infty}\eta\left(r\frac{z}{|z|}+y\right)r^{d-1}\, \ud r\right)\frac{z^iz^j}{|z|^d}\right)
    \end{split}
    \end{equation*}
    with
    \begin{equation*}
        (b_{\eta})_k^j (x,y)=\frac{d+1}{2}\eta(x)\bfdlt^j_k+\frac{1}{2}(x-y)^{\ell}\partial_{\ell}\eta(x)\bfdlt^j_k -\frac{1}{2}(x-y)^j\partial_k\eta(x).
    \end{equation*}
    \item Let $\slashed{\eta}\in C^{\infty}(\bbS^{d-1})$ with $\int_{\bbS^{d-1}}\slashed{\eta} =1$. The conic solution operator for $\calP \bfh =\nb_j \bfh^{jk}$ where $\bfh^{ij}=\bfh^{ji}$ on $\bbR^d$ with flat metric is given by
    \begin{equation*}
        (K_{\slashed{\eta}})_k^{ij}(z+y,y)=\frac{1}{2}\frac{z^i\bfdlt^j_k+z^j\bfdlt^i_k}{|z|^d}\slashed{\eta}\left(\frac{z}{|z|}\right)+\frac{1}{2}\partial_{z^m}\left(\frac{z^m(z^i\bfdlt^j_k+z^j\bfdlt^i_k)}{|z|^d}\slashed{\eta}\left(\frac{z}{|z|}\right)\right)-\partial_{z^k}\left(\frac{z^iz^j}{|z|^d}\slashed{\eta}\left(\frac{z}{|z|}\right)\right).
    \end{equation*}
\end{enumerate}

\subsection{Trace-free symmetric divergence operator (or adjoint conformal Killing operator)} \label{subsec:symm-div-tf}
We consider the \emph{trace-free symmetric divergence operator}
\begin{equation} \label{eq:symm-div-tf-eq}
	\calP \bfh = \nb_{j} \bfh^{j k} \qquad \hbox{where } \bfh^{jk} =  \bfh^{k j} \hbox{ and } \tr_{\bfg} \bfh = 0,
\end{equation}
for $d \geq 3$.

We first compute the formal adjoint $\calP^{\ast}$ (with respect to $\ud V$). For a smooth compactly supported trace-free symmetric $2$-tensor $\bfh$ and $1$-form $\bfomg$ on $U$, we have
\begin{align*}
	\int_{U} (\calP \bfh)^{k} \bfomg_{k} \, \ud V
	&= \int_{U} (\nb_{j} \bfh^{jk}) \bfomg_{k} \, \ud V
	= - \int_{U} \bfh^{jk} \nb_{j} \bfomg_{k} \, \ud V  \\
	&= - \int_{U} \bfh^{jk} \left[ \frac{1}{2} \left(\nb_{j} \bfomg_{k} + \nb_{k} \bfomg_{j}\right) - \frac{1}{d} \nb^{\ell} \bfomg_{\ell} \bfg_{jk} \right] \, \ud V.
	\end{align*}
The formal $L^{2}$-adjoint is
\begin{equation} \label{eq:symm-div-tf-dual}
(\calP^{\ast} \bfomg)_{j k} = -\frac{1}{2} \left(\nb_{j} \bfomg_{k} + \nb_{k} \bfomg_{j}\right) + \frac{1}{d} \nb^{\ell} \bfomg_{\ell} \bfg_{jk}.
\end{equation}
Observe that $\calP^{\ast} \bfomg = 0$ is precisely the condition that the vector field $\bfomg_{\sharp}$ is a conformal Killing vector field of $(\calM, \bfg)$ (i.e., the infinitesimal generator of a one-parameter family of conformal isometries); for this reason, we will call $\calP^{\ast}$ the \emph{conformal Killing operator}. Its principal symbol is
\begin{equation} \label{eq:symm-div-tf-dual-symb}
\tensor{(p^{\ast})}{_{jk}^{\ell}}(x, \xi) = - \frac{i}{2} (\xi_{j} \dlt_{k}^{\ell} + \xi_{k} \dlt_{j}^{\ell}) + \frac{i}{d} \xi_{m} \bfg_{j k}(x)(\bfg^{-1})^{\ell m}(x)
\end{equation}

\subsubsection{Covariant graded augmented system and $\ker \calP^{\ast}$}\label{subsubsec:symm-div-tf-key}
Given a $1$-form $\bfomg_{j}$, define
\begin{equation} \label{eq:symm-div-tf-aug}
\bfeta_{jk} = \tfrac{1}{2}(\ud \bfomg)_{jk} = \tfrac{1}{2} (\nb_{j} \bfomg_{k} - \nb_{k} \bfomg_{j}), \quad w = \frac{1}{d} \nb^{\ell} \bfomg_{\ell}, \quad \bfzt_{j} = \rd_{j} w.
\end{equation}
Note that $\tfrac{1}{2} (\nb_{j} \bfomg_{k} + \nb_{k} \bfomg_{j}) = - (\calP^{\ast} \bfomg)_{jk} + w \bfg_{jk}$. Then
\begin{equation} \label{eq:symm-div-tf-system}
\left\{
\begin{aligned}
	\nb_{i} \bfomg_{j} &= \bfeta_{ij} + w \bfg_{ij} - (\calP^{\ast} \bfomg)_{ij}, \\
	\nb_{i} \bfeta_{j k}
	&=
	 -\tensor{\bfR}{_{j k i}^{\ell}}  \bfomg_{\ell}
	 + \bfzt_{j} \bfg_{i k} - \bfzt_{k} \bfg_{i j}
	+ \nb_{k}(\calP^{\ast} \bfomg)_{i j} - \nb_{j}(\calP^{\ast} \bfomg)_{ki}, \\
	\nb_{i} w &= \bfzt_{i}, \\
	\nb_{i} \bfzt_{j}
	&= - \frac{1}{d-2} \nb^{m} \left( \tensor{\Ric}{_{i j}}
- \frac{1}{2(d-1)}  R \bfg_{ij}\right) \bfomg_{m}  \\
&\peq - \frac{1}{d-2} \left[ \tensor{\Ric}{_{i}^{m}} \bfeta_{j m} + \tensor{\Ric}{_{j}^{m}} \bfeta_{i m}
+ 2 \left(\tensor{\Ric}{_{i j}} - \frac{1}{2(d-1)} R \bfg_{i j}\right)  w \right] \\
&\peq + \frac{1}{d-2} \calC (\calP^{\ast} \bfomg)_{i j}.
\end{aligned}
\right.
\end{equation}
where
\begin{align*}
\calC (\calP^{\ast} \bfomg)_{i j}
&= - \nb^{\ell} \nb_{i} (\calP^{\ast} \bfomg)_{\ell j} - \nb^{\ell} \nb_{j} (\calP^{\ast} \bfomg)_{\ell i} + \nb^{\ell} \nb_{\ell} (\calP^{\ast} \bfomg)_{i j} + \frac{1}{d-1} \nb^{\ell} \nb^{m} (\calP^{\ast} \bfomg)_{\ell m} \bfg_{ij} \\
&\peq + \tensor{\Ric}{_{i}^{\ell}} (\calP^{\ast} \bfomg)_{j \ell} + \tensor{\Ric}{_{j}^{\ell}} (\calP^{\ast} \bfomg)_{i \ell}
 - \frac{1}{d-1} \tensor{\Ric}{^{\ell m}} (\calP^{\ast} \bfomg)_{\ell m} \bfg_{i j}.
\end{align*}
We postpone the proof of \eqref{eq:symm-div-tf-system} and discuss its consequences first. We note that \eqref{eq:symm-div-tf-system} immediately leads to graded augmented variables in the sense of Definition~\ref{def:aug}. Indeed, if we specialize \eqref{eq:symm-div-tf-system} to the Euclidean space in rectangular coordinates (so that $\calP = \calP_{\prin}$ and $\calP^{\ast} = \calP_{\prin}^{\ast_{\ud x}}$), then we see that  $\Phi_{\bfomg_{j}} \ceq \bfomg_{j}$, $\Phi_{\bfeta_{jk}} \ceq \bfeta_{jk}$, $\Phi_{w} \ceq w$ and $\Phi_{\bfzt_{j}} \ceq \bfzt_{j}$ define augmented variables for $\calP_{\prin}$ that satisfy \ref{hyp:aug1}--\ref{hyp:aug4}, where $\calA = \set{\bfomg_{1}, \ldots, \bfomg_{d}, \bfeta_{12}, \ldots, \bfeta_{(d-1) d}, w, \bfzt_{1}, \ldots, \bfzt_{d}} $ (in particular, $\# \calA = d + \frac{d(d-1)}{2} + 1 + d = \frac{(d+1)(d+2)}{2}$), $d_{\bfomg_{j}} = 0$, $d_{\bfeta_{jk}} = d_{w} = -1$, $d_{\bfzt_{j}} = -2$, $m_{ij} = 1$, and $m_{ij}' = 2$ for all $i, j \in \set{1, \ldots, d}$. These augmented variables also satisfy \ref{hyp:aug1}--\ref{hyp:aug4} for any lower order perturbations of $\calP_{\prin}$ (viewed as an operator on an open subset of $\bbR^{d}$, the Euclidean space in rectangular coordinates).

The operator $\calP^{\ast}$ has a finite dimensional kernel with $\dim \ker \calP^{\ast} \leq \# \calA = \frac{(d+1)(d+2)}{2}$. This bound is optimal, and the maximal dimension is reached (i.e., the augmented system is completely integrable) on space forms:
\begin{itemize}
\item For $\bbR^{d}$, $\ker \calP^{\ast}$ consists of the metric duals of the conformal Killing vector fields, which are
\begin{equation}\label{eq:conformal-killing-euc}
\mathrm{span}\,\left( \set{\bfe_{J}}_{J=1, \ldots, d} \cup \set{x_{K} \bfe_{J} - x_{J} \bfe_{K}}_{1\leq J < K \leq d} \cup \set{x^{j} \bfe_{j}} \cup \set{ x_{J} x^{j} \bfe_{j} - \abs{x}^{2} \bfe_{J}}_{1 \leq J \leq d }\right).
\end{equation}

\item For $\bbS^d$, the stereographic projection
\begin{equation*}
    y^j=\frac{x^j}{1-x^0}, \quad x^0=\frac{\sum (y^j)^2-1}{\sum (y^j)^2+1},\quad x^k=\frac{2y^k}{\sum (y^j)^2+1},\,k=1,\cdots,d
\end{equation*}
gives $\ud s^2=\frac{4((dy^1)^2+\cdots+(dy^d)^2)}{1+\sum (y^j)^2}$, so the conformal Killing vector fields are given by \eqref{eq:conformal-killing-euc} in the coordinate $y^j$.

\item For $\bbH^d$, in the upper half space model, the metric on $\bbH^d$ is $\frac{(d x^0)^2+(dx^1)^2+\cdots+(dx^d)^2}{(x^0)^2}$, which is confromal to the Euclidean metric, so the conformal Killing vector fields are also given by \eqref{eq:conformal-killing-euc}.

\end{itemize}

Finally, we give a proof of \eqref{eq:symm-div-tf-system}, which is a covariant generalization of \cite{Resh}. The first identity is obvious, whereas the second identity follows from Section~\ref{subsubsec:symm-div-key}. The third identity is again a restatement of the definition of $\bfzt_{i}$. It remains to establish the last identity.

To simplify the computation, we introduce the notation
\begin{equation*}
\bfA_{ijk\ell} \ceq - \nb^{2}_{i j} (\calP^{\ast} \bfomg)_{k \ell}
= \frac{1}{2} \left(\nb^{3}_{ijk} \bfomg_{\ell} + \nb^{3}_{i j \ell} \bfomg_{k}\right) - \frac{1}{d} \nb^{2}_{ij} \nb^{m} \bfomg_{m} \bfg_{k \ell},
\end{equation*}
where $\nb^{2}_{ij} = \nb_{i} \nb_{j}$ and $\nb^{3}_{ijk} = \nb_{i} \nb_{j} \nb_{k}$.
Indeed,
\begin{align*}
\bfA_{ijk\ell} &= \frac{1}{2} \left(\nb^{3}_{ijk} \bfomg_{\ell} + \nb^{3}_{i j \ell} \bfomg_{k}\right) - \frac{1}{d} \nb^{2}_{ij} \nb^{m} \bfomg_{m} \bfg_{k \ell} \\
&= \frac{1}{2} \left(\nb^{3}_{ik j} \bfomg_{\ell} + \nb^{3}_{j \ell i } \bfomg_{k}\right) - \frac{1}{d} \nb^{2}_{ij} \nb^{m} \bfomg_{m} \bfg_{k \ell} \\
&\peq - \frac{1}{2} \nb_{i} (\tensor{\bfR}{_{jk}^{m}_{\ell}} \bfomg_{m}) - \frac{1}{2} (\tensor{\bfR}{_{i j}^{m}_{\ell}} \nb_{m} \bfomg_{k} +\tensor{\bfR}{_{i j}^{m}_{k}} \nb_{\ell} \bfomg_{m}) - \frac{1}{2} \nb_{j}(\tensor{\bfR}{_{i \ell}^{m}_{k}} \bfomg_{m}),
\end{align*}
so after contracting $i$ and $k$ using the inverse metric, we have
\begin{align*}
\tensor{\bfA}{^{k}_{jk\ell}}
&= \frac{1}{2} \left(\lap \nb_{j} \bfomg_{\ell} + \nb^{2}_{j \ell} \nb^{k} \bfomg_{k}\right) - \frac{1}{d} \nb^{2}_{\ell j} \nb^{m} \bfomg_{m} \\
&\peq - \frac{1}{2} \nb^{k} (\tensor{\bfR}{_{jk}^{m}_{\ell}} \bfomg_{m}) - \frac{1}{2} (\tensor{\bfR}{^{k}_{j}^{m}_{\ell}} \nb_{m} \bfomg_{k} +\tensor{\bfR}{^{k}_{j}^{m}_{k}} \nb_{\ell} \bfomg_{m}) - \frac{1}{2} \nb_{j}(\tensor{\bfR}{^{k}_{\ell}^{m}_{k}} \bfomg_{m})
\\
&= \frac{1}{2} \lap \nb_{j} \bfomg_{\ell} + \frac{d-2}{2d}\nb^{2}_{j \ell} \nb^{k} \bfomg_{k} \\
&\peq - \frac{1}{2} \nb^{k} \tensor{\bfR}{_{jk}^{m}_{\ell}} \bfomg_{m} - \frac{1}{2} (-\tensor{\bfR}{_{j}^{k}_{\ell}^{m}} \nb_{k} \bfomg_{m} + \tensor{\bfR}{_{j}^{k}_{\ell}^{m}} \nb_{m} \bfomg_{k} - \tensor{\Ric}{_{j}^{m}} \nb_{\ell} \bfomg_{m}) \\
&\peq + \frac{1}{2} \nb_{j} \tensor{\Ric}{_{\ell}^{m}} \bfomg_{m} + \frac{1}{2} \tensor{\Ric}{_{\ell}^{m}} \nb_{j} \bfomg_{m}
\\
&= \frac{1}{2} \lap \nb_{j} \bfomg_{\ell} + \frac{d-2}{2d}\nb^{2}_{j \ell} \nb^{k} \bfomg_{k} \\
&\peq + \frac{1}{2} (\nb^{m} \tensor{\Ric}{_{j\ell}} - \nb_{\ell} \tensor{\Ric}{_{j}^{m}} + \nb_{j} \tensor{\Ric}{_{\ell}^{m}}) \bfomg_{m}
+ \tensor{\bfR}{_{j}^{k}_{\ell}^{m}} \bfeta_{k m} + \frac{1}{2} (\tensor{\Ric}{_{j}^{m}} \nb_{\ell} \bfomg_{m} + \tensor{\Ric}{_{\ell}^{m}} \nb_{j} \bfomg_{m}).
\end{align*}
where, for the last equality, we used the second Bianchi identity to write
\begin{align*}
-\nb^{k} \tensor{\bfR}{_{jk}^{m}_{\ell}}
= \nb^{m} \tensor{\bfR}{_{jk \ell}^{k}} + \nb_{\ell} \tensor{\bfR}{_{jk}^{k m}}
= \nb^{m} \tensor{\Ric}{_{j\ell}} - \nb_{\ell} \tensor{\Ric}{_{j}^{m}} .
\end{align*}
Thus
\begin{align*}
\tensor{\bfA}{^{k}_{jk\ell}}  + \tensor{\bfA}{^{k}_{\ell k j}}
&= \frac{1}{2} \lap (\nb_{j} \bfomg_{\ell} + \nb_{\ell} \bfomg_{j}) + \frac{d-2}{d}\nb^{2}_{j \ell} \nb^{k} \bfomg_{k} \\
&\peq + \nb^{m} \tensor{\Ric}{_{j\ell}} \bfomg_{m}
+ \tensor{\Ric}{_{j}^{m}} \nb_{\ell} \bfomg_{m} + \tensor{\Ric}{_{\ell}^{m}} \nb_{j} \bfomg_{m}.
\end{align*}
Contracting also $j$ and $\ell$ using the inverse metric, we have
\begin{align*}
\tensor{\bfA}{^{k \ell}_{k\ell}}
&= \frac{d-1}{d} \lap \nb^{k} \bfomg_{k} + \frac{1}{2} \nb^{m} R \bfomg_{m}
+ \tensor{\Ric}{^{\ell m}} \nb_{\ell} \bfomg_{m}.
\end{align*}
Since
\begin{align*}
\tensor{\bfA}{^{\ell}_{\ell j k}} = \frac{1}{2} \left(\lap \nb_{j} \bfomg_{k} + \lap \nb_{k} \bfomg_{j}\right) - \frac{1}{d} \lap \nb^{m} \bfomg_{m} \bfg_{jk},
\end{align*}
we have
\begin{align*}
\tensor{\bfA}{^{\ell}_{\ell j k}} &+ \frac{1}{d-1} \tensor{\bfA}{^{\ell m}_{\ell m}} \bfg_{jk}
=\tensor{\bfA}{^{\ell}_{\ell j k}} + \frac{1}{d} \lap \nb^{m} \bfomg_{m} \bfg_{jk} + \frac{1}{2(d-1)} \nb^{m} R \bfomg_{m} \bfg_{jk}
+ \frac{1}{d-1}\tensor{\Ric}{^{\ell m}} \nb_{\ell} \bfomg_{m} \bfg_{jk} \\
&= \frac{1}{2} \lap \left(\nb_{j} \bfomg_{k} + \nb_{k} \bfomg_{j}\right)
+ \frac{1}{2(d-1)} \nb^{m} R \bfomg_{m} \bfg_{jk}
+ \frac{1}{2(d-1)}\tensor{\Ric}{^{\ell m}} (\nb_{\ell} \bfomg_{m} + \nb_{m} \bfomg_{\ell})\bfg_{jk}.
\end{align*}
Therefore,
\begin{align*}
&\tensor{\bfA}{^{k}_{jk\ell}}  + \tensor{\bfA}{^{k}_{\ell k j}} - \tensor{\bfA}{^{k}_{k j \ell}} - \frac{1}{d-1} \tensor{\bfA}{^{k m}_{k m}} \bfg_{j \ell} \\
&= \frac{d-2}{d}\nb^{2}_{j \ell} \nb^{k} \bfomg_{k}
+ \nb^{m} \tensor{\Ric}{_{j\ell}} \bfomg_{m}
- \frac{1}{2(d-1)} \nb^{m} R \bfomg_{m} \bfg_{j\ell}\\
&\peq + \tensor{\Ric}{_{j}^{m}} \nb_{\ell} \bfomg_{m} + \tensor{\Ric}{_{\ell}^{m}} \nb_{j} \bfomg_{m}
- \frac{1}{2(d-1)}\tensor{\Ric}{^{k m}} (\nb_{k} \bfomg_{m} + \nb_{m} \bfomg_{k}) \bfg_{j\ell}.
\end{align*}
Note also that
\begin{align*}
\tensor{\Ric}{_{j}^{m}} \nb_{\ell} \bfomg_{m} + \tensor{\Ric}{_{\ell}^{m}} \nb_{j} \bfomg_{m}
&= \tensor{\Ric}{_{j}^{m}} \bfeta_{\ell m} + \tensor{\Ric}{_{\ell}^{m}} \bfeta_{j m}
+ 2 \tensor{\Ric}{_{j \ell}} w \\
&\peq - \tensor{\Ric}{_{j}^{m}} (\calP^{\ast} \bfomg)_{\ell m} - \tensor{\Ric}{_{\ell}^{m}} (\calP^{\ast} \bfomg)_{j m}, \\
- \frac{1}{2(d-1)}\tensor{\Ric}{^{k m}} (\nb_{k} \bfomg_{m} + \nb_{m} \bfomg_{k})
&= - \frac{1}{(d-1)} R w + \frac{1}{d-1} \tensor{\Ric}{^{k m}} (\calP^{\ast} \bfomg)_{k m}.
\end{align*}
We arrive at
\begin{align*}
\nb_{j} \bfzt_{\ell}
&=  \frac{1}{d}\nb^{2}_{j \ell} \nb^{k} \bfomg_{k} \\
&= - \frac{1}{d-2} \left(\nb^{m} \tensor{\Ric}{_{j \ell}}
- \frac{1}{2(d-1)} \nb^{m} R \bfg_{j \ell}\right) \bfomg_{m}  \\
&\peq - \frac{1}{d-2} \left[ \tensor{\Ric}{_{j}^{m}} \bfeta_{\ell m} + \tensor{\Ric}{_{\ell}^{m}} \bfeta_{j m}
+ 2 \left(\tensor{\Ric}{_{j \ell}} - \frac{1}{2(d-1)} R \bfg_{j \ell}\right)  w \right] \\
&\peq + \frac{1}{d-2} \calC_{j \ell}
\end{align*}
where
\begin{align*}
\calC_{j \ell}
&= \tensor{\bfA}{^{k}_{jk\ell}}  + \tensor{\bfA}{^{k}_{\ell k j}} - \tensor{\bfA}{^{k}_{k j \ell}} - \frac{1}{d-1} \tensor{\bfA}{^{k m}_{k m}} \bfg_{j \ell} \\
&\peq + \tensor{\Ric}{_{j}^{m}} (\calP^{\ast} \bfomg)_{\ell m} + \tensor{\Ric}{_{\ell}^{m}} (\calP^{\ast} \bfomg)_{j m}
 - \frac{1}{d-1} \tensor{\Ric}{^{k m}} (\calP^{\ast} \bfomg)_{k m} \bfg_{j \ell},
\end{align*}
which completes the proof.

\subsubsection{Explicit computations in the constant curvature case} \label{subsubsec:symm-div-tf-explicit}

Let $\kpp$ be the constant sectional curvature of $(\calM, \bfg)$. Fix $y, y_{1} \in \calM$ and a geodesic segment $\bfx$ from $y$ to $y_{1}$. We may write $\bfx(t) = (t, \omg_{0})$ in polar coordinates at $y$ for some $\omg_{0} \in \bbS^{d-1}$. From \eqref{eq:symm-div-tf-system} and the formulas in Section~\ref{subsubsec:const-curv-prelim}, we have
\begin{equation*}
    \begin{aligned}
        \frac{\ud}{\ud t } \bfomg_r(\bfx(t)) &= w -(\calP^*\bfomg)_{rr},\\
         \frac{\ud}{\ud t } w(\bfx(t)) &= \bfzt_r,\\
          \frac{\ud}{\ud t } \bfzt_r(\bfx(t)) &= -\kappa w+\frac{1}{d-2}\calC(\calP^*\bfomg)_{rr}
    \end{aligned}
\end{equation*}
and
\begin{equation*}
    \begin{aligned}
         \frac{\ud}{\ud t } \bfomg_A(\bfx(t)) &= \Gamma_{rA}^{B}\bfomg_{B}+\bfeta_{rA}-(\calP^*\bfomg)_{rA},\\
          \frac{\ud}{\ud t } \bfeta_{rA}(\bfx(t)) &= \Gamma_{rA}^{B}\bfeta_{rB}-\kappa \bfomg_A-\bfzt_A+\nb_A(\calP^*\bfomg)_{rr}-\nb_r(\calP^*\bfomg)_{rA},\\
 \frac{\ud}{\ud t } \bfzt_{A}(\bfx(t)) &= \Gamma_{rA}^{B}\bfzt_A+ \frac{1}{d-2}\calC(\calP^*\bfomg)_{rA}.
    \end{aligned}
\end{equation*}
From the first three equations we get
\begin{equation*}
\begin{aligned}
    w(\bfx(t)) &= w(y_1)c_{\kappa}(t-d(y_1,y))+\bfzt_r(y_1)s_{\kappa}(t-d(y_1,y)) \\
    &-\frac{1}{d-2} \int_{t}^{d(y_1,y)}s_{\kappa}(t-s) \calC(\calP^*\bfomg)_{rr}(\bfx(s)) \, \ud s
\end{aligned}
\end{equation*}
and
\begin{equation*}
    \begin{aligned}
        \bfomg_r(\bfx(0)) &= \bfomg_r(y_1)- w(y_1) s_{\kappa}(d(y_1,y)) +\frac{1}{\kappa}\bfzt_r(y_1)(1-c_{\kappa}(d(y_1,y))\\
        &+\int_{0}^{d(y_1,y)}(\calP^*\bfomg)_{rr}(\bfx(s)) \, \ud s-\frac{1}{(d-2)\kappa}\int_0^{d(y_1,y)}(1-c_{\kappa}(s))\calC(\calP^*\bfomg)_{rr}(\bfx(s))\, \ud s.
    \end{aligned}
\end{equation*}
From the last three equations, we get
\begin{equation*}
    s_{\kappa}(t)^{-1}\bfzt_A(\bfx(t)) = s_{\kappa}(d(y_1,y))^{-1}\bfzt_A(y_1)-\frac{1}{d-2} \int_t^{d(y_1,y)}s_{\kappa}(s)^{-1}\calC(\calP^*\bfomg)_{rA}(\bfx(s)) \, \ud s
\end{equation*}
and
\begin{equation*}
   \frac{\ud^2}{\ud t^2} \left(s_{\kpp}(t)^{-1}\bfomg_{A}(\bfx(t)) \right) = -\kappa  s_{\kpp}(t)^{-1}\bfomg_{rA} +s_{\kpp}(t)^{-1} ( - \bfzt_A+ \nb_A(\calP^*\bfomg)_{rr}-2\nb_{r}(\calP^*\bfomg)_{rA}).
\end{equation*}
Therefore
\begin{equation*}
\begin{aligned}
    s_{\kpp}(t)^{-1}\bfomg_{A}(\bfx(t)) &= \int_{t}^{d(y_1,y)} s_{\kappa}(t-s)s_{\kappa}(s)^{-1}(\nb_r(\calP^*\bfomg)_{rA}-\nb_A(\calP^*\bfomg)_{rr})(\bfx(s)) \, \ud s\\
    &+\int_t^{d(y_1,y)} c_{\kappa}(t-s) s_{\kappa}(s)^{-1}(\calP^*\bfomg)_{rA} \ud s\\
    &+\frac{1}{(d-2)\kappa}\int_{t}^{d(y_1,y)} (1-c_{\kappa}(t-s))s_{\kpp} (s)^{-1}\calC(\calP^*\bfomg)_{rA}(\bfx(s))\, \ud s\\
    &-\kappa^{-1}(1-c_{\kpp}(t-d(y_1,y))s_{\kappa}(d(y_1,y))^{-1}\bfzt_A(y_1)\\
    &+\bfomg_A(y_1)c_{\kappa}(t-d(y_1,y))s_{\kappa}(d(y_1,y))^{-1}+\bfeta_{rA}(y_1)s_{\kappa}(t-d(y_1,y))s_{\kappa}(d(y_1,y))^{-1}.
\end{aligned}
\end{equation*}
Therefore,
\begin{equation*}
    \begin{aligned}
        \langle K_{y_1}(\cdot,y), \psi \rangle &=\left(\int_{0}^{d(y_1,y)}\dot{\bfx}^i\dot{\bfx}^j\psi_{ij}(\bfx(s)) \ud s-\frac{1}{(d-2)\kappa}\int_0^{d(y_1,y)}(1-c_{\kappa}(s))\dot{\bfx}^i\dot{\bfx}^j(\calC \psi)_{ij}(\bfx(s))\ud s\right) \, \ud r \\
        &
       +\left(\int_0^{d(y_1,y)} c_{\kappa}(s)\dot{\bfx}^{j} \psi_{j\td{A}}(\bfx(s)) \ud s+\frac{1}{(d-2)\kappa}\int_0^{d(y_1,y)}(1-c_{\kappa}(s))\dot{\bfx}^j(\calC \psi)_{j\td{A}}(\bfx(s))\ud s\right)  \, \ud\td{\omg}^{A} \\&+\left(\int_{0}^{d(y_1,y)} s_{\kappa}(s)\dot{\bfx}^i\dot{\bfx}^j(\nb_{\td{A}}(\calP^*\bfomg)_{ij}-\nb_i(\calP^*\bfomg)_{j\td{A}})(\bfx(s)) \ud s\right) \, \ud \td{\omg}^A
    \end{aligned}
\end{equation*}
and
\begin{equation*}
\begin{aligned}
    \langle b_{y_1}(\cdot,y), \bfomg \rangle &=\left(\bfomg_r(y_1)- w(y_1) s_{\kappa}(d(y_1,y)) +\frac{1}{\kappa}\bfzt_r (y_1)(1-c_{\kappa}(d(y_1,y))\right) \, \ud r\\
    &+\left(-\kappa^{-1}(1-c_{\kpp}(d(y_1,y))\bfzt_{\td{A}}(y_1)+\bfomg_{\td{A}}(y_1)c_{\kappa}(d(y_1,y))-\bfeta_{r\td{A}}(y_1)s_{\kappa}(d(y_1,y))\right) \, \ud \td{\omg}^A.
\end{aligned}
\end{equation*}

\subsubsection{Explicit formulas on flat spaces}
For the readers' convenience, we record the explicit formulas for the Bogovskii-type and conic solution operators on $\bbR^d$. We average over straight line segments $\bfx(y,y_1,s)=y+s(y_1-y)$ for the Bogovski-type solution operator and over straight half-lines $\bfx(y,\omg,s)=y+s\omg$ for the conic solution operator.

To state our results, we introduce
\begin{equation*}
        (\calC^*f)_{ij}\ceq-\partial^\ell\partial_i f_{\ell j}-\partial^{\ell}\partial_if_{j\ell}+\partial^{\ell}\partial_{\ell}f_{ij}+\frac{1}{d-1}\partial_i\partial_j\tr f,\quad (\calT^*f)_{ij}\ceq\frac{1}{2}(f_{ij}+f_{ji})-\frac{1}{d}(\tr f)\delta_{ij}.
\end{equation*}
\begin{enumerate}
    \item Let $\eta\in C_c^{\infty}(\bbR^d)$ with $\int_{\bbR^d}\eta =1$. The Bogovskii-type solution operator for $\calP \bfh =\nb_j \bfh^{jk}$ where $\bfh^{ij}=\bfh^{ji}$ and $\tr_{\bfg}\bfh =0$ on $\bbR^d$ with flat metric is given by
    \begin{equation*}
    \begin{split}
        (K_{\eta})_k^{ij}(z+y,y)=&\calT^*\left(\int_{|z|}^{\infty} \eta\left(y+r\frac{z}{|z|}\right)r^{d-1}\frac{z^i}{|z|^d}\bfdlt_{k}^{j}\, \, \ud r\right)\\
        &+\frac{1}{2(d-2)}\calT^*\calC^*\left(\int_{|z|}^{\infty} \eta\left(y+r\frac{z}{|z|}\right)r^{d-1}\frac{z^{i}}{|z|^{d-2}}\bfdlt_{k}^{j}\, \, \ud r\right)\\
        &-\frac{1}{(d-2)}\calT^*\calC^*\left(\int_{|z|}^{\infty} \eta\left(y+r\frac{z}{|z|}\right)r^{d-1}\frac{z^{i}z^{j}z_{k}}{|z|^{d}}\, \, \ud r\right)\\
        &-(\calT^*\bfdlt_{\ell}^{j}\nabla_k-\calT^*\bfdlt_{k}^{j}\nabla_\ell)\left(\int_{|z|}^{\infty}\eta\left(y+r\frac{z}{|z|}\right)r^{d-1}\frac{z^{i}z^\ell}{|z|^d}\, \, \ud r\right)
    \end{split}
    \end{equation*}
    with
    \begin{equation*}
    \begin{split}
        (b_{\eta})_k^j (x,y)=&\eta(x)\bfdlt^j_k+\frac{1}{2}\partial_{\ell}(\eta(x)(x-y)^{\ell})\bfdlt^{j}_{k}-\frac{1}{2}\partial_k(\eta(x)(x-y)^j)+\frac{1}{d}\partial^j(\eta(x)(x-y)_k)\\
    &-\frac{1}{2d}\partial^j\partial_k(\eta(x)|x-y|^2)+\frac{1}{d}\partial^j\partial_{\ell}(\eta(x)(x-y)^{\ell}(x-y)_k).
    \end{split}
    \end{equation*}
    \item Let $\slashed{\eta}\in C^{\infty}(\bbS^{d-1})$ with $\int_{\bbS^{d-1}}\slashed{\eta} =1$. The conic solution operator for $\calP \bfh =\nb_j \bfh^{jk}$ where $\bfh^{ij}=\bfh^{ji}$ and $\tr_{\bfg}\bfh =0$ on $\bbR^d$ with flat metric is given by
    \begin{equation*}
    \begin{split}
        (K_{\slashed{\eta}})_k^{ij}(z+y,y)=&\calT^*\left(\bfdlt_{k}^{j}\frac{z^i}{|z|^d}\slashed{\eta}\left(\frac{z}{|z|}\right)\right)+\frac{1}{2(d-2)}\calT^*\calC^* \left(\left(\bfdlt_{k}^{j}\frac{z^{i}}{|z|^{d-2}}-2\frac{z^iz^jz_k}{|z|^d}\right)\slashed{\eta}\left(\frac{z}{|z|}\right)\right)\\
    &-(\calT^*\bfdlt_{\ell}^{i}\nb_{k}-\calT^*\bfdlt_{k}^{i}\nb_\ell )\left(\frac{z^{\ell}z^j}{|z|^d}\slashed{\eta}\left(\frac{z}{|z|}\right)\right).
    \end{split}
    \end{equation*}
\end{enumerate}

\subsection{Linearized Einstein constraint equation} \label{subsec:einstein}
We consider the linearized Einstein constraint equation, first by itself and second under the constant mean curvature condition.

\subsubsection{Linearized Einstein constraint operator}
The vacuum \emph{Einstein constraint equation} on $(\calM, \bfg)$ is a nonlinear underdetermined system of PDEs for $\bfg$ and a symmetric $2$-tensor $\bfk$ of the form
\begin{equation} \label{eq:einstein-constraint}
\left\{
\begin{aligned}
	R_{\bfg} + (\tr_{\bfg} \bfk)^{2} - \abs{\bfk}_{\bfg}^{2} &= 0, \\
	\nb^{i} \bfk_{ij} - \rd_{j} \tr_{\bfg} \bfk &= 0.
\end{aligned}
\right.
\end{equation}
The linearization of the operator on the left-hand side \eqref{eq:einstein-constraint} takes the form
\begin{equation} \label{eq:einstein-constraint-lin}
	\begin{pmatrix}
	\begin{array}{c} \mathrm{D} R(\bfg) \dot{\bfg} + 2 (\tr_{\bfg} \bfk) \bfk^{ij} \dot{\bfg}_{ij} - 2 \bfk^{ii'} \bfk^{jj'} \bfg_{i'j'} \dot{\bfg}_{ij} + 2(\tr_{\bfg}\bfk)\bfg^{ij}\dot{\bfk}_{ij} -2\bfg^{ii'}\bfg^{jj'}\bfk_{i'j'}\dot{\bfk}_{ij} \\
	\end{array} \\
	\nb^{i} \left( \dot{\bfk}_{ij} - \bfg_{ij} \tr_{\bfg} \dot{\bfk}\right) - \tensor{{\dot{\bfGmm}}}{^{\ell}_{{ ii'}}}\bfg^{ii'}\bfk_{\ell j} - \tensor{{\dot{\bfGmm}}}{^{\ell}_{{i'j}}} \bfg^{ii'}\bfk_{i\ell}-\partial_j(\bfk^{i\ell}\dot{\bfg}_{i\ell})
	\end{pmatrix}.
\end{equation}
As in Section~\ref{subsec:d-div}, the following change of variables simplify the principal terms (recall footnote~\ref{fn:g-upper} in Section~\ref{subsec:d-div} for our convention for $\dot{\bfg}^{ij}$):
\begin{equation} \label{eq:}
	\bfh^{ij} = \dot{\bfg}^{ij} - \bfg^{ij} \tr_{\bfg} \dot{\bfg}, \quad
	\bfpi^{ij} = \dot{\bfk}^{ij} - \bfg^{ij} \tr_{\bfg} \dot{\bfk}.
\end{equation}
In terms of $(\bfh, \bfpi)$, we may rewrite \eqref{eq:einstein-constraint-lin} as
\begin{align*}
	\calP(\bfh, \bfpi)
	&\ceq \left( \nabla_{i} \nabla_{j} \bfh^{ij}+(\bfC^{(1)})_{i j} \bfpi^{i j}+ (\bfR^{(1)})_{i j} \bfh^{i j} , \nabla_{i} \bfpi^{i j} + \nb_{i} (\tensor{(\bfC^{(2)})}{^{i j}_{k \ell}}\bfh^{k \ell}) +\tensor{(\bfR^{(2)})}{^{j}_{k \ell}} \bfh^{k \ell}\right),
\end{align*}
for some tensor fields $\bfC^{(1)}$, $\bfC^{(2)}$, $\bfR^{(1)}$, $\bfR^{(2)}$ determined by $(\bfg, \bfk)$, which are all zero if $(\bfg, \bfk) = (\bfdlt, 0)$. The adjoint $\calP^*$ is then given by
\begin{align*}
    \calP^{\ast} (\varphi, \bfomg)
    &= \left( \calP^*(\varphi,\bfomg)_{[1] ij}, \calP^*(\varphi,\bfomg)_{[2] ij} \right) \\
    &\ceq \left( \nabla_{i} \partial_{j} \varphi+(\bfR^{(1)})_{i j} \varphi+ \tensor{(\bfR^{(2)})}{^{k}_{ij}}\bfomg_{k} - \tensor{(\bfC^{(2)})}{^{k \ell}_{ij}} \nb_{k} \bfomg_{\ell}, -\frac{1}{2}(\nabla_{i}\bfomg_{j}+\nabla_{j}\bfomg_{i})+(\bfC^{(1)})_{i j} \varphi \right).
\end{align*}
As is well-known, the kernel of $\calP^{\ast}$ consist of Killing Initial Data sets (KIDs) \cite{ChrDel}.

The principal symbol of $\calP^{\ast}$ is given by
\begin{equation*}
    p^{\ast}(x,\xi) = \begin{pmatrix}
    p_{\rm ddiv}^{\ast}(x,\xi)& 0 \\
    0&p_{\rm sdiv}^{\ast}(x,\xi)
    \end{pmatrix}
\end{equation*}
where $p_{\rm ddiv}^{\ast}(x,\xi)$ is given by \eqref{eq:d-div-dual-symb} and $p_{\rm sdiv}^{\ast}(x,\xi)$ is given by \eqref{eq:symm-div-dual-symb}.
By the computations in Sections~\ref{subsubsec:d-div-key} and \ref{subsubsec:symm-div-key}, we can write down an augmented system with the augmented variables
\begin{equation*}
	\bfalp_{i} \ceq \rd_{i} \varphi, \quad
	\bfeta_{ij} \ceq \frac{1}{2} (\ud \bfomg)_{ij} = \frac{1}{2} (\nb_{i} \bfomg_{j} - \nb_{j} \bfomg_{i}).
\end{equation*}
Since the augmented system is stable under lower order perturbations, the same ODE system also works here with extra lower order terms.

Proceeding as in Sections~\ref{subsubsec:d-div-key} and \ref{subsubsec:symm-div-key}, we obtain the following system:
\begin{equation} \label{eq:einstein-system}
\left\{
\begin{aligned}
        \nabla_i\varphi&=\bfalp_i,\\
        \nabla_i\bfalp_j&=  (\tilde{\bfR}^{(1)})_{ij} \varphi + \tensor{(\tilde{\bfR}^{(2)})}{^{k}_{ij}} \bfomg_{k} + \tensor{(\tilde{\bfC}^{(1)})}{^{k \ell}_{ij}}\bfeta_{k \ell}
        + \calP^*_{[1]}(\varphi,\bfomg)_{ij} + \tensor{(\tilde{\bfC}^{(4)})}{^{k \ell}_{i j}} \calP^*_{[2]}(\varphi,\bfomg)_{k \ell},\\
        \nabla_i\bfomg_j&= \bfeta_{ij}+ (\tilde{\bfC}^{(2)})_{ij} \varphi-\calP^*_{[2]}(\varphi,\bfomg)_{ij},\\
        \nabla_i\bfeta_{jk}&= (\tilde{\bfR}^{(3)})_{ijk} \varphi + \tensor{(\tilde{\bfR}^{(4)})}{^{\ell}_{ijk}} \bfomg_{\ell} +\tensor{(\tilde{\bfC}^{(3)})}{^{\ell}_{ijk}} \bfalp_{\ell}+\nabla_k\calP^*_{[2]}(\varphi, \bfomg)_{ij}-\nabla_j\calP^*_{[2]}(\varphi, \bfomg)_{ki}.
\end{aligned}
\right.
\end{equation}
As before, \eqref{eq:einstein-system} immediately leads to graded augmented variables in the sense of Definition~\ref{def:aug}. Indeed, if we specialize \eqref{eq:einstein-system} to the Euclidean space in rectangular coordinates (so that $\calP = \calP_{\prin}$ and $\calP^{\ast} = \calP_{\prin}^{\ast_{\ud x}}$), then we see that  $\Phi_{\varphi} = \varphi$, $\Phi_{\bfalp_{i}} = \bfalp_{i}$, $\Phi_{\bfomg_{i}} \ceq \bfomg_{i}$, and $\Phi_{\bfeta_{ij}} \ceq \bfeta_{ij}$ define augmented variables for $\calP_{\prin}$ that satisfy \ref{hyp:aug1}--\ref{hyp:aug4}, where $\calA = \set{\varphi, \bfalp_{1}, \ldots, \bfalp_{d}, \bfomg_{1}, \ldots, \bfomg_{d}, \bfeta_{12}, \ldots, \bfeta_{(d-1) d}} $ (in particular, $\# \calA = 1 + d + d + \frac{d(d-1)}{2}= \frac{(d+1)(d+2)}{2}$), $d_{\varphi} = d_{\bfomg_{i}} = 0$, $d_{\bfalp_{i}} = d_{\bfeta_{ij}} = -1$, $m_{[1] ij} = 2$, and $m_{[1] ij}' = 0$, $m_{[2] ij} = 1$, and $m_{[2] ij}' = 1$. These augmented variables also satisfy \ref{hyp:aug1}--\ref{hyp:aug4} for any lower order perturbations of $\calP_{\prin}$ (viewed as an operator on an open subset of $\bbR^{d}$, the Euclidean space in rectangular coordinates).

\subsubsection{Linearized Einstein constraint operator under constant mean curvature gauge}
The method of this paper is also applicable to the linearized Einstein vacuum constraint equation under constant mean curvature gauge $\tr_{\bfg}\bfk=c$.

Introduce the new variables
\begin{equation*}
    \bfh^{ij}= \dot{\bfg}^{ij}-\bfg^{ij}\tr_{\bfg}\dot{\bfg},\quad \wh{\bfpi}^{ij} = \dot{\bfk}^{ij}- \frac{1}{d}\bfg_{ij}\tr_{\bfg}\dot{\bfk},\quad
 \rho=\tr_{\bfg}\dot{\bfk}.
\end{equation*}
The Linearized Einstein constraint operator under constant mean curvature gauge can be written as
\begin{equation*}
\begin{aligned}
    \calP (\bfh,\wh{\bfpi},\rho) =& \left(\nabla_i\nabla_j \bfh^{ij}+(\bfC^{(1)})_{ij}\wh{\bfpi}^{ij}+(\bfR^{(1)})_{ij}\bfh^{ij}+C^{(2)}\rho,\right.\\
    &\left. \nabla_i \wh{\bfpi}^{ij}+\tensor{(\bfR^{(2)})}{^{j}_{k\ell}}\bfh^{k\ell}+\nb_i(\tensor{(\bfC^{(3)})}{^{ij}_{k \ell}}\bfh^{k \ell})-\frac{d-1}{d}\nb^j\rho,\rho\right),
\end{aligned}
\end{equation*}
for some tensor fields $\bfC^{(1)}, C^{(2)}, \bfC^{(3)}, \bfR^{(1)}, \bfR^{(2)}$ determined by $(\bfg,\bfk)$.

We may use the last component $\rho$ to eliminate $\rho$ in the equations and only consider (by an abuse of notation)
\begin{equation*}
    \wh{\calP} (\bfh,\wh{\bfpi})= (\nabla_i\nabla_j \bfh^{ij}+\tensor{(\bfC^{(1)})}{_{ij}}\wh{\bfpi}^{ij}+(\bfR^{(1)})_{ij}\bfh^{ij}, \nabla_i \wh{\bfpi}^{ij}+\tensor{(\bfR^{(2)})}{^{j}_{k\ell}}\bfh^{k \ell}+\nb_i(\tensor{(\bfC^{(2)})}{^{ij}_{k \ell}}\bfh^{k \ell})).
\end{equation*}
The adjoint is given by
\begin{equation*}
\begin{split}
     &\wh{\calP}^*(\varphi,\bfomg)=(\calP^*(\varphi,\bfomg)_{[1]ij}, \calP^*(\varphi,\bfomg)_{[2]ij})\\
     &=\left(\nabla_i\partial_j\varphi+(\bfR^{(1)})_{ij}\varphi+\tensor{(\bfR^{(2)})}{^{k}_{ij}}\bfomg_{k}- \tensor{(\bfC^{(2)})}{^{k\ell}_{ij}}\nb_{k}\bfomg_{\ell},-\frac{1}{2}(\nabla_i\bfomg_j+\nabla_j\bfomg_i)+\frac{1}{d}\nb^{\ell}\bfomg_{\ell} \bfg_{ij}+\tensor{(\bfC^{(1)})}{_{ij}}\varphi\right).
\end{split}
\end{equation*}
As before, the principal symbol of $\wh{\calP}^{\ast}$ is given by
\begin{equation*}
    p^{\ast}(x,\xi) = \begin{pmatrix}
        p_{\rm ddiv}^{\ast}(x,\xi) & 0 \\
        0 & p^{\ast}_{\rm tsdiv}(x,\xi)
    \end{pmatrix}
\end{equation*}
where $p_{\rm ddiv}^{\ast}(x,\xi)$ is given by \eqref{eq:d-div-dual-symb} and $p^{\ast}_{\rm tsdiv}(x,\xi)$ is given by \eqref{eq:symm-div-tf-dual-symb}.

Let $\bfalp_i=\nb_i\varphi$, $\bfeta_{ij}=\frac{1}{2}(\nb_j\bfomg_j-\nb_j\bfomg_i)$, $w=\frac{1}{d}\nb^{\ell}\bfomg_{\ell}$ and $\bfzt_j=\partial_j w$, we have
\begin{equation}\label{eq:einstein-max-aug}
    \begin{aligned}
    \nabla_i\varphi&=\bfalp_i,\\
        \nabla_i\bfalp_j&=   (\tilde{\bfR}^{(1)})_{ij}\varphi+\tensor{(\tilde{\bfR}^{(2)})}{^k_{ij}}\bfomg_k+\tensor{(\tilde{\bfC}^{(1)})}{^{k\ell}_{ij}}(\bfeta_{k\ell}+w\bfg_{k\ell}-(\wh{\calP}^*(\varphi,\bfomg)_{[2]k\ell})+(\wh{\calP}^*(\varphi,\bfomg))_{[1]ij},\\
        \nabla_i\bfomg_j&=\bfeta_{ij}+w\bfg_{ij}+(\tilde{\bfC}^{(2)})_{ij}\varphi-(\wh{\calP}^*(\varphi,\bfomg))_{[2]ij},\\
        \nabla_i\bfeta_{jk}&=\tensor{(\tilde{\bfR}^{(3)})}{_{ijk}}\varphi+\tensor{(\tilde{\bfR}^{(4)})}{^{\ell}_{ijk}}\bfomg_{\ell}+\tensor{(\tilde{\bfC}^{(3)})}{^{\ell}_{ijk}}\bfalp_{\ell}+\bfzt_j \bfg_{ik}-\bfzt_k \bfg_{ij}+\nabla_k(\wh{\calP}^*\bfomg)_{[2]ij}-\nabla_j(\wh{\calP}^*\bfomg)_{[2]ki},\\
        \nb_i w&=\bfzt_j,\\
        \nb_i\bfzt_j&=\tensor{(\tilde{\bfA}^{(1)})}{_{ij}}\varphi+\tensor{(\tilde{\bfR}^{(5)})}{^k_{ij}}\bfalp_k+\tensor{(\tilde{\bfA}^{(2)})}{^k_{ij}}\bfomg_k+\tensor{(\tilde{\bfR}^{(6)})}{^{k\ell}_{ij}}\bfeta_{k\ell}+\tensor{(\tilde{\bfR}^{(7)})}{_{ij}}w\\
        &\peq +\tensor{(\tilde{\bfC}^{(4)})}{^{k\ell}_{ij}}(\wh{\calP}^*(\varphi,\bfomg))_{[1]k \ell}+\tensor{(\tilde{\bfR}^{(8)})}{^{k \ell}_{ij}}(\wh{\calP}^*(\varphi,\bfomg))_{[2]k \ell}+\frac{1}{d-1}\calC(\wh{\calP}^*\bfomg)_{[2]ij}.
    \end{aligned}
\end{equation}
As before, \eqref{eq:einstein-max-aug} gives graded augmented variables in the sense of Definition~\ref{def:aug}. Indeed, if we specialize \eqref{eq:einstein-max-aug} to the Euclidean space in rectangular coordinates (so that $\wh{\calP} = \wh{\calP}_{\prin}$ and $\wh{\calP}^{\ast} = \wh{\calP}_{\prin}^{\ast_{\ud x}}$), then we see that  $\Phi_{\varphi} = \varphi$, $\Phi_{\bfalp_{i}} = \bfalp_{i}$, $\Phi_{\bfomg_{i}} \ceq \bfomg_{i}$, $\Phi_{\bfeta_{ij}} \ceq \bfeta_{ij}$, $\Phi_{w}=w$, and $\Phi_{\bfzt_j}=\bfzt_j$ define augmented variables for $\wh{\calP}_{\prin}$ that satisfy \ref{hyp:aug1}--\ref{hyp:aug4}, where
$$\calA=\{\varphi,\bfalp_1,\cdots,\bfalp_d,\bfomg_1,\cdots,\bfomg_d,\bfeta_{12},\cdots, \bfeta_{(d-1)d},w,\bfzt_1,\cdots,\bfzt_d\}$$
and $\#\calA=1+d+d+\frac{d(d-1)}{2}+1+d=\frac{(d+1)(d+4)}{2}$. We have $d_{\varphi}=d_{\bfomg_i}=0$, $d_{\alp_i}=d_{\bfeta_{ij}}=d_{w}=-1$, $d_{\bfzt_j}=-2$, $m_{[1]ij}=2$, and $m_{[1]ij}'=0$, $m_{[2]ij}=1$, and $m_{[2]ij}'=2$.  When $\bfk=0$, the equations decouple to the linearized scalar curvature equation and symmetric divergence equation in the maximal gauge, studied in Section~\ref{subsec:d-div} and Section~\ref{subsec:symm-div-tf}, respectively.

\bibliographystyle{amsplain}
\bibliography{div-eq}

\end{document}